\documentclass[10pt,a4paper]{article}
\usepackage{amsthm,amsfonts,amsmath,amssymb,latexsym}
\usepackage{epsfig,graphics,color}
\usepackage[utf8]{inputenc}
\usepackage[dvipsnames]{xcolor}
\usepackage[matrix,arrow,curve]{xy}
\usepackage[breaklinks=true]{hyperref}
\usepackage{tikz-cd}
\usepackage{verbatim}
\usepackage{bm}
\usepackage{ulem}
\usepackage{enumitem}
\newcommand{\ri}{\rightarrow}

\newcommand{\calm}{{\cal M}}
\newcommand{\calf}{{\cal F}}

\newcommand{\call}{{\cal L}}
\newcommand{\calp}{{\cal P}}

\newcommand{\crit}{\mathrm{ Crit}}

\newtheorem{PARA}{}[section]
\newtheorem{theorem}[PARA]{Theorem}
\newtheorem{corollary}[PARA]{Corollary}
\newtheorem{lemma}[PARA]{Lemma}
\newtheorem{proposition}[PARA]{Proposition}
\newtheorem{definition}[PARA]{Definition}
\newtheorem{conjecture}[PARA]{Conjecture}
\theoremstyle{definition}
\newtheorem{remark}[PARA]{Remark}
\theoremstyle{plain}
\newtheorem{example}[PARA]{Example}

\newtheorem{definition-proposition}[PARA]{Definition/Proposition}
\newtheorem{innercustomthm}{Theorem}
\newenvironment{customthm}[1]
  {\renewcommand\theinnercustomthm{#1}\innercustomthm}
  {\endinnercustomthm}
\newtheorem{innercustomcor}{Corollary}
\newenvironment{customcor}[1]
  {\renewcommand\theinnercustomcor{#1}\innercustomcor}
  {\endinnercustomcor}
\newcommand{\para}{\begin{PARA}\rm}
\newcommand{\arap}{\end{PARA}\rm}
\newcommand{\dfn}{\begin{definition}\rm}
\newcommand{\nfd}{\end{definition}\rm}
\newcommand{\rmk}{\begin{remark}\rm}
\newcommand{\kmr}{\end{remark}\rm}
\newcommand{\xmpl}{\begin{example}\rm}
\newcommand{\lpmx}{\end{example}\rm}

\newcommand{\Iso}{\mathrm{Iso}}

\newcommand{\cA}{\mathcal{A}}
\newcommand{\cB}{\mathcal{B}}

\newcommand{\cC}{\mathcal{C}}
\newcommand{\cD}{\mathcal{D}}
\newcommand{\cF}{\mathcal{F}}
\newcommand{\cE}{\mathcal{E}}
\newcommand{\cG}{\mathcal{G}}
\newcommand{\cH}{\mathcal{H}}

\newcommand{\cL}{\mathcal{L}}
\newcommand{\cM}{\mathcal{M}}

\newcommand{\cN}{\mathcal{N}}

\newcommand{\cP}{\mathcal{P}}

\newcommand{\cS}{\mathcal{S}}

\newcommand{\cY}{\mathcal{Y}}

\newcommand{\ueta}{{\underline{\eta}}}
\newcommand{\oev}{\overline{\mathrm{ev}}}

\newcommand{\h}{\boldsymbol{\mathrm{h}}}
\newcommand{\bolddelta}{\boldsymbol{\delta}}
\newcommand{\q}{\boldsymbol{q}}

\newcommand{\one}
{{{\mathchoice \mathrm{ 1\mskip-4mu l} \mathrm{ 1\mskip-4mu l}
\mathrm{ 1\mskip-4.5mu l} \mathrm{ 1\mskip-5mu l}}}}

\newcommand{\C}{{\mathbb{C}}}
\newcommand{\D}{{\mathbb{D}}}

\renewcommand{\H}{{\mathbb{H}}}

\newcommand{\N}{{\mathbb{N}}}
\newcommand{\Q}{{\mathbb{Q}}}
\newcommand{\R}{{\mathbb{R}}}
\renewcommand{\SS}{{\mathbb{S}}}
\newcommand{\T}{{\mathbb{T}}} 

\newcommand{\Z}{{\mathbb{Z}}}
\newcommand{\caly}{{\cal Y}}
\newcommand{\coker}{\mathrm{ coker }}  
\newcommand{\colim}{\mathrm{ colim}\, }  
\newcommand{\im}{\mathrm{im}\,}        
\newcommand{\geo}{\mathrm{ geo}}   
\newcommand{\id}{\mathrm{ id}}         
\newcommand{\Id}{\mathrm{ Id}}
\newcommand{\ind}{\mathrm{ind}\,}
\renewcommand{\Im}{\mathrm{ Im\,}}       

\newcommand{\pr}{{\mathrm{pr}}}
\newcommand{\ev}{\mathrm{ev}}
\newcommand{\Crit}{\mathrm{ Crit}}
\newcommand{\Diff}{\mathrm{ Diff}}        
\newcommand{\End}{\mathrm{ End}}          

\newcommand{\Per}{\mathrm{Per}}

\newcommand{\Sp}{\mathrm{Sp}}

\newcommand{\Hom}{\mathrm{Hom}}
\newcommand{\Map}{\mathrm{Map}}
\newcommand{\eps}{{\varepsilon}}

 \newcommand{\CZ}{\mathrm{CZ}}
\def\NABLA#1{{\mathop{\nabla\kern-.5ex\lower1ex\hbox{$#1$}}}}
\def\Nabla#1{\nabla\kern-.5ex{}_{#1}}
\def\Tabla#1{\Tilde\nabla\kern-.5ex{}_{#1}}
\renewcommand{\Tilde}{\widetilde}
\newcommand{\wh}{\widehat}

\newcommand{\p}{{\partial}}

\newcommand{\ol}{\overline}

\definecolor{vincent}{rgb}{0,0,1.0}

\newcommand{\footremember}[2]{
\footnote{#2}
\newcounter{#1}
\setcounter{#1}{\value{footnote}}
}
\newcommand{\footrecall}[1]{
\footnotemark[\value{#1}]
} 
\title{Floer Homology with DG Coefficients. Applications to Cotangent Bundles}
\author{%
  Jean-Fran\c{c}ois Barraud\footremember{Toulouse}{IMT, Universit\'e de Toulouse}
  \and Mihai Damian\footremember{Strasbourg}{IRMA, Universit\'e de Strasbourg}
  \and Vincent Humili\`ere\footremember{Jussieu}{IMJ-PRG, Sorbonne Universit\'e}
  \and Alexandru Oancea\footrecall{Strasbourg}
}
\date{\today}

\begin{document}

\maketitle


\begin{abstract}
We define Hamiltonian Floer homology with differential graded (DG) local coefficients for symplectically aspherical manifolds. The differential of the underlying complex involves chain representatives of the fundamental classes of the moduli spaces of Floer trajectories of arbitrary dimension. This setup allows in particular to define and compute Floer homology with coefficients in chains on fibers of fibrations over the free loop space of the underlying symplectic manifold. We develop the DG Floer toolset, including continuation maps and homotopies, and we also define and study symplectic homology groups with DG local coefficients. We define spectral invariants and establish general criteria for almost existence of contractible periodic orbits on regular energy levels of Hamiltonian systems inside Liouville domains.      

In the case of cotangent bundles, we prove a Viterbo isomorphism theorem with DG local coefficients. This serves as a stepping stone for applications to the almost existence of contractible closed characteristics on closed smooth hypersurfaces. In this context, our methods allow to access for the first time the dichotomy between closed manifolds that are $K(\pi,1)$ and those that are not.     
\end{abstract}

\setcounter{tocdepth}{2}
\tableofcontents

\section{Introduction and main results}

\subsection{Motivation and overview}\label{sec:overview}

The study of closed characteristics on   hypersurfaces of symplectic manifolds is a fundamental topic in symplectic geometry. When such a closed  hypersurface is of contact type the famous Weinstein conjecture \cite{Weinstein79}, first proved by Viterbo in $\R^{2n}$~\cite{Viterbo-WeinsteinR2n}, claims the existence of at least one closed characteristic.  Without assuming the contact type hypothesis,  given some symplectic manifold $(W,\omega)$ and some {\it proper} autonomous smooth Hamiltonian $H:W\ri \R$, one can ask which regular levels $H^{-1}(c)$ contain at least one closed characteristic. The question whether {\it all} the regular levels of any Hamiltonian satisfy this property is known as the Hamiltonian Seifert conjecture. It was disproved already for $W=\R^{2n}$ by Ginzburg~\cite{Ginzburg-Ham-Seifert1} \cite{Ginzburg-Ham-Seifert2} and Herman~\cite{Herman} for $2n\geq6$; G\"urel and Ginzburg later provided a  ${\cal C}^2$ counterexample for $2n=4$ \cite{Ginzburg-Gurel1}, \cite{Ginzburg-Gurel2}. Once the conjecture is disproved it is natural to ask whether  at least some of the regular levels $H^{-1}(c)$ contain closed characteristics. In the particular case $W=\R^{2n}$ endowed with the standard symplectic form  it actually turns out that many of them do. More precisely, a celebrated result of Hofer-Zehnder~\cite{HZ} and Struwe~\cite{Struwe} asserts that almost all (regular) levels of a smooth proper Hamiltonian $H:\R^{2n}\ri \R$ carry at least one closed characteristic. This result is known as {\it the almost existence property}. 
  
The study of the almost existence property in various settings has been a driving question in Hamiltonian dynamics. We give in~\S\ref{sec:intro-other_work} a historical overview of this problem and discuss the relationship between our results and other work. In this paper we focus on the existence of {\it contractible} closed characteristics in the neighborhood of a given closed hypersurface.

 \begin{definition}\label{def:almost-sure} Let $(W,\omega)$ be a symplectic manifold and $H: W\ri \R$ a proper 
  smooth\footnote{We assume by default that our Hamiltonians are smooth, but ${\cal C}^2$ regularity would suffice for our results.} 
  Hamiltonian. A regular level $\Sigma= H^{-1}(c)\subset \mathrm{int}(W)$ \emph{has the almost existence property  for contractible closed characteristics}  if there exists $\epsilon >0$ such that almost every (in the sense of measure theory) regular  level $H^{-1}(c')$,  $c'\in (c-\epsilon, c+\epsilon)$  carries a closed characteristic that is contractible in $W$. 
  \end{definition}
  
We will sometimes use the shorthand formulation ``the contractible almost existence property holds near $\Sigma$",  or ``$\Sigma$ has the contractible almost existence property"  instead of ``$\Sigma$ has the almost existence property for contractible closed characteristics". 

 Two general  questions arise naturally: 
\begin{quote} QUESTION 1: Under which assumptions  does a closed hypersurface $\Sigma= H^{-1}(c)$ satisfy this property ?\\
\\ 
QUESTION 2: For which symplectic manifolds does  {\it any} closed hypersurface $\Sigma$ satisfy this property ?
\end{quote} 

We mainly focus on the case where the manifold $W$ is the cotangent bundle $T^*Q$ endowed with the standard symplectic form $\omega_Q=\sum_i dp_i\wedge dq_i$. A first remark is that we cannot expect a positive answer to Question~2 for any $Q$: if $Q$ is a manifold that admits a metric with nonpositive curvature, the boundary of its unit cotangent bundle $\Sigma=S^*Q$ does not satisfy the contractible almost existence property 
   since $Q$ does not admit contractible closed geodesics. It is therefore reasonable to work under the assumption that $Q$ is not a $K(\pi, 1)$: under this hypothesis and various topological conditions on $\Sigma$ we give a positive answer to Question 1. Then, we show that for some compact manifolds $Q$ the answer to Question 2 is positive as well. Our ultimate goal is the following statement that we conjecture: 
   
   \begin{conjecture}\label{magic} If $Q$ is not a $K(\pi,1)$ then any closed hypersurface $\Sigma \subset T^*Q$ has the contractible almost existence property.
   \end{conjecture}
   
   We point out that our paper provides a general method to approach this
   conjecture. We briefly explain it here:  Suppose more generally that
   $W^{2n}$ is a Weinstein domain.   Going back to the seminal
     paper of Viterbo~\cite{Viterbo99}, in order to prove existence
     results for closed characteristics it is enough for symplectic
     homology $SH_*(W)$ to be ``different enough" from the  usual
     homology $H_*(W, \p W)$.   Considering the canonical map $i_{W*}:
   H_{*+n}(W,\p W)\ri SH_*(W)$ we adapt the work of
     Irie~\cite{Irie14} to formulate two criteria for positive answers
     respectively to the questions stated above. 
   \begin{quote} CRITERION 1: Let $j:U\hookrightarrow W$ be a bounded domain with smooth boundary $\Sigma=\p U$. 
   If $\ker(i_{W*}) \not\subset \ker(j_!)$ then $\Sigma$ has the contractible almost existence property.
   \end{quote} Here $j_! : H_*(W, \p W)\ri H_*(U, \p U)$ is the shriek map, i.e., the Poincaré dual of the restriction map $H^*(W)\ri H^*(U)$. 
   \begin{quote} CRITERION 2: Suppose that $\ker(i_{W*})$ contains the fundamental class $[W]$. 
   Then any 
   hypersurface $\Sigma$ that bounds a domain in $W$ has the contractible almost existence property.
   \end{quote}
   Criterion~2 is a corollary of Criterion~1 since $j_![W]= [U]\neq 0$.
   As we will emphasize in the sequel, the assumption of Criterion 2 actually implies more generally that the $\pi_1$-sensitive Hofer-Zehnder capacity $c^{\circ}_{HZ}(W)$ is finite. 
   
   Now the problem with these two criteria is that they  cannot be fulfilled as stated in cotangent bundles. Indeed, for $W=T^*Q$, the map $i_{W*}$ is identified via the Viterbo isomorphism   with the map  induced by the inclusion $i_Q$ of constant loops $Q$ into the space of contractible free loops $\cL_0Q$, which is injective. 
   The approach taken in~\cite{Albers-Frauenfelder-Oancea} is to prove Criterion 2 for some well-chosen {\it  local coefficients} under the assumption that the Hurewicz map $\pi_2(Q)\ri H_2(Q)$ does not vanish.  A different approach is taken in~\cite{Frauenfelder-Pajitnov}, where the key role is played by $S^1$-localization and Criterion 2 is proved for rationally inessential manifolds~$Q$.
   
   In the present paper we establish and apply these criteria for differential graded (DG) coefficients. We developed a  Morse theory with DG coefficients in our previous work \cite{BDHO} following the work of Barraud and Cornea \cite{BC07}.  The following example illustrates how the use of DG coefficients changes  the kernel of $i_{Q*}:H_*(Q)\ri H_*(\call_0 Q)$ enabling us to apply our criteria in some cases.  We proved in \cite[Theorem 7.2]{BDHO}   that, for any Hurewicz fibration $F\ri E\ri X$, the homology of $X$ with DG coefficients in the cubical chains $C_*(F)$ is isomorphic to the singular homology of $E$. In our case, if $E\ri \call_0 Q$ is a fibration and  $E|_Q$ is the total space of its pullback by $i_Q$, then  $i_{Q*}$ becomes the map associated to the inclusion  $E|_Q\hookrightarrow E$ in singular homology. Let us now  specify our example:  consider  the path-loop fibration $E=\calp\call_0Q \ri \call_0 Q$, with $E$ being the space of paths in $\call_0Q$ starting at a fixed basepoint $\star$; we may depict the elements of $E$ as cones in $Q$.  Then $E|_Q$ is formed of  cones that end up in a point of $Q$, i.e., spheres with north pole in $\star$.
  
  \begin{figure}
  \centering
  \includegraphics[scale=.6]{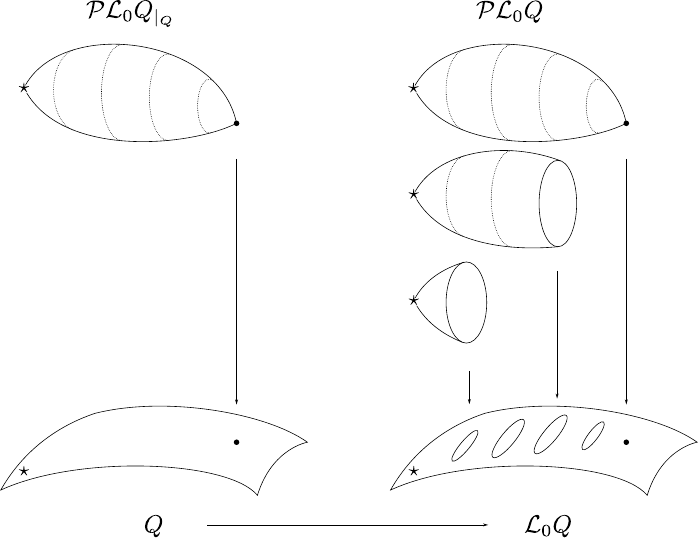} 
  \caption{Inclusion of spheres into the (contractible) space of cones.}
  \label{fig:bananes-cones}
\end{figure}

   The space $E$ is clearly contractible whereas $E|_Q$ is not, unless $Q$ is a $K(\pi, 1)$ --- this is obvious if $\pi_2(Q)\neq 0$; we prove it in general in Proposition~\ref{prop:Ker-i}. Therefore we have some elements in the kernel of $i_{Q*}$.  See Figure \ref{fig:bananes-cones}.
   
   In order to establish the above criteria with DG coefficients the first step was of course to develop the DG coefficients setup in Morse theory. Other steps were needed, most of them of independent interest. We present them in the following  ``to do list'': 
   \begin{enumerate} 
   \item Establish a  Morse theory with DG coefficients for compact manifolds $Q$ and generalise it to the case of $\call_0 Q$. Define direct and shriek maps in this setting:  done in \cite{BDHO}, recalled in \S\ref{sec:DGMorse} of this paper.
   \item Define Hamiltonian Floer homology and symplectic homology with DG coefficients: stated in \S\ref{sec:intro-DGFloer} below and done in \S\ref{sec:DGFloer}.
   \item State and prove the Viterbo isomorphism theorem for DG coefficients: stated in \S\ref{sec:intro-Viterbo} below and proved in  \S\ref{sec:DGViterbo}.
   \item Establish Criterion 1 for DG coefficients: stated in \S\ref{sec:intro-Liouville} and done in \S\ref{sec:criterion1}.
   \item In view of the statement of  Criterion 2, define the notion of homology class with DG coefficients ``living over the fundamental class'': done in \S\ref{sec:classes_over_fund_class}.
   \item Establish Criterion 2 for DG coefficients: stated at the end of  \S\ref{sec:intro-Liouville} and proved in \S\ref{sec:criterion2}.
   \item Apply Criterion 1 for well chosen DG coefficients to get some answers to Question 1: Our results are stated in \S\ref{sec:intro-almost-existence} and proved in \S\ref{sec:almost-existence-cotangent} and in \S\ref{sec:knotted}.
   \item Apply Criterion 2  for well chosen DG coefficients to get some answers to  Question 2: Our results are stated in \S\ref{sec:intro-finiteness-HZ} and proved in \S\ref{sec:case-cotangents}-\S\ref{sec:abundant_examples}.
  
   \end{enumerate}

 In view of the above,  the remaining part of the introduction is structured as follows: in~\S\ref{sec:intro-almost-existence} and \S\ref {sec:intro-finiteness-HZ} we present our dynamical results in cotangent bundles, in~\S\ref{sec:intro-DGFloer} we  discuss Floer homology with DG coefficients, and in~\S\ref{sec:intro-Viterbo} we state the Viterbo isomorphism with DG coefficients. We complement the introduction with   another section \S\ref{sec:intro-Liouville}, that contains the  dynamical criteria 1 and  2   for DG coefficients. In Section~\S\ref{sec:intro-other_work} we discuss the relationship between our results and other work.

\subsection[Almost existence in cotangent bundles]{Almost existence in cotangent bundles. \\ Answers to Question 1} \label{sec:intro-almost-existence}

Given a closed manifold $Q$ we consider its cotangent bundle $T^*Q$ endowed with the standard symplectic form $\omega_Q= \sum_i dp_i\wedge dq_i$. We establish the  contractible almost existence property for cotangent bundles in the following setting.

    \begin{customthm}{A} \label{thm:almost-existence-simplified} Let $Q^n$ be a closed orientable $n$-dimensional manifold, and let $\varphi:S\to T^*Q$ be a continuous map defined on a 
  closed orientable $n$-dimensional manifold $S$. Let $d=\deg(\pi\circ\varphi)$ and assume one of the following: 
 \begin{enumerate}
 \item  $Q$ is not a $K(\pi,1)$ and $d=\pm 1$, or 
 \item $Q$ is not a rational $K(\pi,1)$ and $d$ is nonzero. 
\end{enumerate}
Then, any hypersurface $\Sigma\subset T^*Q$ that bounds a relatively compact domain $U$ that contains $\varphi(S)$ has the almost existence property for closed characteristics that are contractible in $T^*Q$.
\end{customthm}    
    
That $Q$ is {\it not a rational} $K(\pi,1)$ means that $\pi_\ell(Q)\otimes \Q\neq 0$ for some $\ell\geq 2$.

\begin{remark}
Let $Q^n$ be closed connected of dimension $n\ge 2$. Then any closed connected hypersurface $\Sigma^{2n-1}\subset T^*Q$ is necessarily orientable and separating, meaning that $T^*Q\setminus \Sigma$ consists of two connected components of which exactly one is relatively compact.\footnote{This is proved as follows: (i) if the complement of $\Sigma$ was connected, we could find a loop in $T^*Q$ that has exactly one transverse intersection point with $\Sigma$, which implies that $[\Sigma]\neq 0\in H_{2n-1}(T^*Q;\Z_2)$. This contradicts the fact that the homology of $Q$ is supported in degree $\le n$ with $n\ge 2$. 
(ii) the complement of $\Sigma$ inside $T^*Q$ actually has  exactly  two connected components because any point in $T^*Q\setminus \Sigma$ can be connected by a path disjoint from $\Sigma$ to a point in the boundary of a tubular neighborhood and since this boundary fibers over $\Sigma$ with fiber $\SS^0$ it has at most two connected components. (iii) the complement of a disc bundle in $T^*Q$ is connected for $n\ge 2$, and this implies that one (and only one) of the two connected components of $T^*Q\setminus \Sigma$ is bounded. In particular $\Sigma$ is orientable  because $T^*Q$ is oriented. The argument works more generally for hypersurfaces inside Weinstein domains of dimension $2n\ge 4$. We thank K.~Niederkr\"uger and J.-Y.~Welschinger for a useful discussion on this topic.} Accordingly, in the statement of Theorem~\ref{thm:almost-existence-simplified} and also in the sequel, the fact that $\Sigma$ bounds a relatively compact domain $U$ should be understood in a notational sense.  
\end{remark}

\begin{remark}
One notable situation in which Theorem~\ref{thm:almost-existence-simplified} applies is for $S\subset T^*Q$ a closed exact Lagrangian. Then, by work of Abouzaid-Kragh~\cite{Abouzaid-nearby,Kragh2013} and Guillermou~\cite{Guillermou2023}, the projection $\pi:S\to Q$ is a homotopy equivalence and therefore has degree $+1$ (for a suitable orientation of $S$). 
\end{remark}

 We also prove another version of Theorem~\ref{thm:almost-existence-simplified} under a stronger homotopical condition but without referring to any auxiliary manifold $S$.

\begin{customthm}{B} \label{thm:almost-existence-Thom-simplified} Let $Q^n$ be a closed orientable $n$-dimensional  manifold that is not a rational $K(\pi, 1)$. 
Any hypersurface $\Sigma\subset T^*Q$ that bounds a relatively compact domain $U$ such that the projection $\pi :U\ri Q$ induces a nonzero map $H_n(U;\Z)\ri H_n(Q;\Z)$ has the  contractible almost existence property. 
\end{customthm}
    
 Until now $Q$ was  supposed orientable, but we can also prove the following statement on the contractible almost existence property in the nonorientable case. 
 
   \begin{customthm}{C}\label{thm:almost-existence-nonorientable} Let $Q^n$ be a closed non-orientable $n$-dimensional manifold that is not a $\Z_2$-$K(\pi,1)$,
   and let $\varphi:S \ri T^*Q$ be a continuous map defined on a closed $n$-dimensional manifold $S$ such that $\pi\circ\varphi$ has non-zero mod-2 degree. 
   
   Then, any hypersurface $\Sigma\subset T^*Q$ that bounds a relatively compact domain $U$ that contains $\varphi(S)$ has the  contractible almost existence property.
   \end{customthm} 
   
   That $Q$ is not a $\Z_2$-$K(\pi,1)$ means that   $\pi_\ell(Q)\otimes_\Z \Z_2\neq 0$ for some $\ell\geq 2$, which arises for instance  if $\pi_\ell(Q)$ has  $\Z$, or  $\Z_{2m}$ for some $m >0$ as a direct summand. 
     
    Note that the above theorems imply the  contractible almost existence property for boundaries of domains $U$ that contain the zero section, or, more generally, that contain  the graph of a (not necessarily  closed) $1$-form (take  $S=Q$ and the embedding $\varphi:Q\ri T^*Q$ whose image is this graph  to get degree $=1$ in Theorem~\ref{thm:almost-existence-simplified} and Theorem~\ref{thm:almost-existence-nonorientable}). 
    
A \emph{magnetic cotangent bundle} is the cotangent bundle $T^*Q$ endowed with a symplectic form $\omega_Q+\pi^*\beta$, where $\beta\in\Omega^2(Q)$ is a closed $2$-form called \emph{magnetic form}. If $\beta$ is exact we speak of an \emph{exact} magnetic cotangent bundle. In this case, we get from the previous observation as an immediate consequence\footnote{We thank Leonardo Macarini for pointing out this corollary.}: 
    
    \begin{customcor}{D}[contractible almost existence in exact magnetic cotangent bundles] \label{cor:exact-magnetic} Let $Q$ be a closed orientable manifold and $\sigma$ a $1$-form on $Q$.  Theorems~\ref{thm:almost-existence-simplified}, \ref{thm:almost-existence-Thom-simplified} and \ref{thm:almost-existence-nonorientable} are valid for $T^* Q$ endowed with  the magnetic symplectic form 
    $\omega_\sigma=\omega_Q+\pi^*d\sigma$.
    
    In particular if $Q$ is not a $K(\pi,1)$ and $U\subset (T^*Q, \omega_\sigma)$ is a bounded domain that contains the zero section, its boundary $\Sigma=\p U$ has the contractible almost existence property. 
    \end{customcor} 
    
\begin{proof}
Let $U$ be a relatively compact domain in $(T^*Q, \omega_Q+\pi^*d\sigma)$ with smooth boundary $\Sigma =\p U$, where $\omega_Q$ is the standard symplectic form on $T^*Q$ and $\sigma$ is an arbitrary 1-form on $Q$. Let $\omega_\sigma=\omega_Q+\pi^*d\sigma$ be the twisted symplectic form and remark that $\Psi(p,q) =(p+\sigma(q),q)$ is a symplectomorphism between $(T^*Q,\omega_\sigma)$ and the standard cotangent bundle $(T^*Q,\omega_Q)$. It therefore suffices to prove the contractible almost existence property for $\Psi(\Sigma)=\p (\Psi(U))$ in the standard cotangent bundle. 

It is straightforward that if $U$ satisfies the assumptions of Theorem~\ref{thm:almost-existence-simplified}, respectively Theorem~\ref{thm:almost-existence-Thom-simplified}, or Theorem~\ref{thm:almost-existence-nonorientable}, then $\Psi(U)$ also does: the composition with $\Psi$ does not alter the projection $\pi$.  We apply these theorems to $\Psi(U)$ in the standard cotangent and we get the corresponding results for $\Sigma= \p U$ in the exact magnetic cotangent. 
\end{proof}
 
As another straightforward consequence we can prove instances of the Weinstein conjecture. If $\Sigma$ is of contact type the  contractible almost existence property immediately implies the Weinstein conjecture for $\Sigma$ endowed with the contact form given by the restriction of the Liouville form, i.e., the existence of at least one closed characteristic for $\Sigma$ that is contractible in $T^*Q$. Indeed, the characteristic foliation on nearby hypersurfaces can be chosen to be conjugate to the one on $\Sigma$. It is natural to ask whether such a contact type hypersurface satisfies the Weinstein conjecture for \emph{all} contact forms that define the same contact structure $\xi$. We say in this case that \emph{$(M,\xi)$ satisfies the Weinstein conjecture}. We say that $\Sigma$ is of \emph{restricted contact type} if it admits a transverse Liouville vector field that is globally defined on $T^*Q$.

\begin{customcor}{E} \label{cor:intro-Weinstein}
Let $\Sigma=\p U$ be as in Theorems~\ref{thm:almost-existence-simplified}, \ref{thm:almost-existence-Thom-simplified} or \ref{thm:almost-existence-nonorientable}, and assume that it is of restricted contact type with complete Liouville vector field. Then $(\Sigma,\xi)$ satisfies the Weinstein conjecture for closed Reeb orbits that are contractible in $T^*Q$.
\end{customcor}

\begin{proof}
Let $\alpha_0$ be the contact form induced on $\Sigma$. The fact that $\Sigma$ has restricted contact type with complete Liouville vector field $Z$ ensures that its symplectization $((0,\infty)\times \Sigma,d(r\alpha_0))$ admits a Liouville embedding in $T^*Q$ by the flow of $Z$, and the positive half of the symplectization $(1,\infty)\times\Sigma$ embeds in the complement of $U$. Up to global scaling by a constant, an arbitrary contact form $\alpha$ on $\Sigma$ such that $\ker\alpha=\ker\alpha_0=\xi$ can be written $\alpha=f_\alpha\alpha_0$ for some function $f_\alpha:\Sigma\to (1,\infty)$. As such, $\alpha$ is induced by the restriction of the Liouville form $r\alpha_0$ to the graph of $f_\alpha$ in $(1,\infty)\times\Sigma$, denoted $\Sigma_\alpha$. The interior of $\Sigma_\alpha$ contains $U$ and retracts onto $U$, and therefore the assumptions of Theorems~\ref{thm:almost-existence-simplified}, \ref{thm:almost-existence-Thom-simplified}, and~\ref{thm:almost-existence-nonorientable} are satisfied by $\Sigma_\alpha$ if and only if they are satisfied by $\Sigma$. From this we conclude.  
\end{proof}

An obvious example of such a hypersurface is $S^*Q$ (the standard Liouville vector field on $T^*Q$ is complete). Starting from this, the embedded isotropic surgery method of Laudenbach~\cite{Laudenbach97} allows us to construct restricted contact type hypersurfaces in $T^*Q$ with complete Liouville vector field (equal to the standard one at infinity) and arbitrarily complicated topology. These contact manifolds $(\Sigma,\xi)$ satisfy the Weinstein conjecture for closed Reeb orbits that are contractible in $T^*Q$.

If $\Sigma$ is of contact type and $H^1(\Sigma;\R)=0$, then $\Sigma$ is of restricted contact type with complete Liouville vector field. Indeed, a local Liouville vector field can be extended to a global one that coincides with the canonical Liouville vector field on $T^*Q$ outside of a neighborhood of $\Sigma$. 

\begin{remark}
Theorems~\ref{thm:almost-existence-simplified} and~\ref{thm:almost-existence-Thom-simplified} are particular cases of the more general Theorems~\ref{thm:almost-existence} and~\ref{thm:almost-existence-Thom}, which we prove in~\S\ref{sec:almost-existence}. As such, Corollaries~\ref{cor:exact-magnetic} and~\ref{cor:intro-Weinstein} remain true under the more general assumptions of those theorems, with the same proofs. 
\end{remark}

We end this section with another easy corollary of our results based on the work of Ginzburg and Niche \cite{Ginzburg-Niche}. Recall that a 1-dimensional foliation is said to be \emph{uniquely ergodic} if it is generated by a vector field whose flow admits a unique invariant probability measure. A regular level set $\Sigma= H^{-1}(0)$ of an autonomous Hamiltonian $H$ always admits at least one invariant probability measure obtained by taking the limit of normalized Liouville measures on $H^{-1}([-\eps, \eps])$ as $\eps$ goes to 0 \cite[\S1.4]{HZ}. 
A closed characteristic of $\Sigma$, if it exists, provides another probability measure, obtained by pushing forward the Lebesgue measure of the circle by the parametrization of the closed characteristic given by the flow. Therefore, a hypersurface which admits a closed characteristic is not uniquely ergodic. 

  \begin{customcor}{F}\label{cor-unique-ergodicity}
In the settings of Theorems~\ref{thm:almost-existence-simplified}, \ref{thm:almost-existence-Thom-simplified}, \ref{thm:almost-existence-nonorientable} or Corollary \ref{cor:exact-magnetic}, the characteristic foliation of $\Sigma=\partial U$ is not uniquely ergodic.
\end{customcor}

In view of the previous discussion, this corollary can be seen as a very weak version of the Weinstein conjecture which holds for all (not necessarily contact) hypersurfaces. 

  \begin{proof} Ginzburg--Niche \cite{Ginzburg-Niche} shows the following. Consider a Hamiltonian structure, i.e., a $2n-1$-manifold endowed with an exact 2-form $\Omega$ such that $\Omega^{n-1}\neq 0$ pointwise. Assume 
  its characteristic foliation is uniquely ergodic and
  satisfies $\mathrm{lk}(\Omega)\neq 0$, where $\mathrm{lk}(\Omega)=\int\alpha\wedge\Omega^{n-1}$ for any primitive $\alpha$ of $\Omega$. 
    Then \cite{Ginzburg-Niche} prove that the Hamiltonian structure has contact type, i.e., $\Omega$ admits a primitive which is a contact form.

   Let $\Sigma=\partial U$ be a hypersurface as in Theorems~\ref{thm:almost-existence-simplified}, \ref{thm:almost-existence-Thom-simplified}, \ref{thm:almost-existence-nonorientable} or Corollary \ref{cor:exact-magnetic}. The restriction of $\omega$ to $\Sigma$ defines a Hamiltonian structure on $\Sigma$. Moreover, by the Stokes formula, \[\mathrm{lk}(\omega|_{\Sigma})=\int_U\omega^n\neq 0.\]
   Assume now that the characteristic foliation of $\Sigma$ were uniquely ergodic. By \cite{Ginzburg-Niche}, $\Sigma$ would have contact type, hence would admit a closed characteristic as explained in the paragraph that precedes Corollary \ref{cor:intro-Weinstein}. This would contradict unique ergodicity.   
  \end{proof}

\subsubsection{A generalization: Almost existence for hypersurfaces knotted with sphe\-res}

Our final result giving an answer to Question 1 is a theorem about  contractible almost existence for relatively compact open subsets $U\subset T^*Q$ that contain a submanifold that may not cover the base with nontrivial degree, but is somehow knotted with a
sphere. This generalizes Theorem~\ref{thm:almost-existence-simplified}.

Let $Q$ be a closed orientable manifold, and suppose that
\begin{itemize}
\item 
  $Q$ contains an embedded sphere $X$ of dimension $k\ge 2$,
\item
  $Q$ contains an orientable $(n-k)$-submanifold $Y$ that intersects $X$ transversely at a single point $Y\cap X=\{p\}$.
\item $Y$ is $\ell$-connected, with $\ell=2k-2$ if $k$ is even, and $\ell=k-1$ if $k$ is odd. 
\item the normal bundle of $Y$ is trivial.
\end{itemize}

\begin{customthm}{G}\label{thm:orbits-knotted-case-intro}
  In the above situation, consider a continuous map $\varphi:S\to
  \pi^{-1}(Y)\subset T^*Q$ defined on a smooth $(n-k)$-dimensional
  orientable closed manifold $S$, and suppose that the map
  $\pi\circ\varphi:S\to Y$ has degree $d\neq 0$.

  Then, any hypersurface $\Sigma\subset T^*Q$ that bounds a relatively
  compact domain $U$ that contains $\varphi(S)$ has the almost 
  existence property for contractible closed characteristics.
\end{customthm}

This result is proved in~\S\ref{sec:knotted} as Theorem~\ref{thm:orbits-knotted-case}. The example below is revisited in~\S\ref{sec:knotted} as Example~\ref{example:knotted-with-spheres-section-7}. 

\begin{example} \label{example:knotted_with_sphere}
  Let $k\ge 2$ and $Q=Q'\sharp (\SS^k\times Y^{n-k})$, with $Q'$ closed orientable and $Y$ closed orientable and $(2k-2)$-connected (in particular $n\ge 3k-1$). Then $Q$ satisfies the assumptions of Theorem~\ref{thm:orbits-knotted-case-intro}.
  
  Indeed, let $B\subset \SS^k\times Y$ be the ball used to
  construct the connected sum, and let 
  $X=\SS^{k}\times\{q\}$ and $Y\equiv \{p\}\times Y$ for some point $(p,q)\in \SS^{k}\times Y$ such that $X\cap B=Y\cap B=\emptyset$. 
  Then  
  $X$ and $Y$ meet transversely at the single point $(p,q)$ and
  the normal bundle of $Y$ is trivial. Also, $Y$ is connected enough.

  \medskip

  The case where $Q'$ is a $K(\pi, 1)$, for instance $Q=\T^{3k}\sharp(\SS^k\times\SS^{2k})$, is particularly interesting. Indeed, no technique from the literature, nor from this paper, allows to prove finiteness of the $\pi_1$-sensitive Hofer-Zehnder capacity for $D^*Q$, whereas Theorem~\ref{thm:orbits-knotted-case-intro} goes strictly beyond the scope of Theorems~\ref{thm:almost-existence-simplified} and~\ref{thm:almost-existence-Thom-simplified}.
\end{example}

\subsection[Finiteness of the $\pi_1$-sensitive Hofer-Zehnder capacity]{Finiteness of the $\pi_1$-sensitive Hofer-Zehnder capacity.\\ Answers to Question 2} \label{sec:intro-finiteness-HZ} 
The results presented in~\S\ref{sec:intro-almost-existence} are conditional on some assumption on the hypersurface $\Sigma$. To obtain results that are unconditional on the hypersurface $\Sigma=\p U$ one is led to consider the \emph{$\pi_1$-sensitive Hofer-Zehnder capacity} $c_{HZ}^\circ(U)\in (0,\infty]$.  
We exhibit a class of closed manifolds $Q$, which we call \emph{abundant with $2$-spheres}, for which $c_{HZ}^\circ(D^*Q)<\infty$, so that almost existence for contractible closed characteristics holds for all closed hypersurfaces in $T^*Q$. 

 To define manifolds abundant with $2$-spheres we introduce the space $\Map(\SS^2,Q)$, and the fibration $\Map(\SS^2,Q)\to Q$ given by evaluating each map $\SS^2\to Q$ at a fixed point in $\SS^2$. The fiber is the space $\Map_*(\SS^2,Q)\equiv \Omega^2Q$ of maps $\SS^2\to Q$ that preserve the basepoint. Let $j:\star\hookrightarrow Q$ be the inclusion of the basepoint.  

\begin{definition} \label{defi:abundant_with_2spheres-intro}
Let $Q^n$ be a closed connected oriented $n$-dimensional manifold. We say that \emph{$Q$ is abundant with $2$-spheres} if there exists a class $\alpha\in H_*(\Map(\SS^2,Q))$ in the total space of the fibration $\Map(\SS^2,Q)\to Q$ such that:
\begin{itemize}
\item[(i)] $\alpha$ lives over the fundamental class (see Definition~~\ref{def:living-above-fundamental}, Ex.~\ref{def:living-above-fundamental-fibration}), meaning that the shriek map to a point $j_!:H_*(\Map(\SS^2,Q))\to H_{*-n}(\Map_*(\SS^2,Q))$ satisfies $j_!\alpha\neq 0$,  and
\item[(ii)] if $\mathrm{deg}(\alpha)=n$  we also have $j_!\alpha\neq 0$ in reduced homology $\widetilde{H}_0(\Map_*(\SS^2,Q))$.
\end{itemize} 
\end{definition}

\begin{customthm}{H} \label{thm:abundant-intro}
Let $Q$ be a closed connected oriented manifold that is abundant with $2$-spheres. Then $c_{HZ}^\circ(D^*Q)<\infty$ and the almost existence property for closed contractible characteristics holds in $T^*Q$. 
\end{customthm}

Theorem~\ref{thm:abundant-intro} is proved in~\S\ref{sec:HZ} as Theorem~\ref{thm:abundant}. 

In~\S\ref{sec:abundant_examples} we exhibit several classes of manifolds that are abundant with $2$-spheres. That discussion is summarized, but not exhausted, by the following statement.

\begin{customthm}{I} \label{thm:examples-intro}
The following manifolds $Q$ are abundant with $2$-spheres, and therefore $c_{HZ}^\circ(D^*Q)<\infty$ and the  contractible almost existence property holds in $D^*Q$. 
\begin{itemize}
\item[(i)] Manifolds that admit fibered families of $2$-spheres (Definition~\ref{defi:fibered_2spheres}). In particular, the spheres $\SS^n$ of even dimension $n=2k$, $k\ge 1$ (Example~\ref{example:spheres}) and of odd dimension $n=4k-1$, $k\ge 1$ (Example~\ref{example:spheres-4k-1}), and also the total spaces of fiber bundles with such spherical fibers over highly connected bases (Example~\ref{example:fibrations_spherical}). 
\item[(ii)] Manifolds that are not a $K(\pi,1)$ and that are covered by diffeomorphisms in the sense of Definitions~\ref{defi:covered_by_diffeomorphisms} and~\ref{defi:ell_covered_by_diffeomorphisms}, in particular all compact connected Lie groups other than tori (Proposition~\ref{prop:Lie_groups}), and all odd-dimen\-sio\-nal spheres $\SS^{n=2k+1}$, $k\ge 1$ (Proposition~\ref{prop:odd-spheres}).
\item[(iii)] The compact rank 1 symmetric spaces $\C P^d$, $\H P^d$, $d\ge 1$, or $Ca P^2$ (Propositions~\ref{prop:CPd} and~\ref{prop:HPd-CaP2}). 
\item[(iv)] The product $Q=M\times N$, where $M$ is abundant with
  $2$-spheres and $N$ is an arbitrary closed oriented manifold 
  (Proposition~\ref{prop:abundant_products}). 
\end{itemize}
\end{customthm}

The examples from Theorem~\ref{thm:examples-intro} go well beyond (but do not cover entirely) those of~\cite{Albers-Frauenfelder-Oancea}. However, they are strictly contained in the class of rationally inessential manifolds of Frauenfelder-Pajitnov~\cite{Frauenfelder-Pajitnov}. Despite this fact, the method that we use is new and, for that reason, we believe that it has potential. Our result relies on the existence of the fibration 
$$
\Map_*(\SS^2,Q)\hookrightarrow \cP_Q\cL_0Q\to \cL_0Q,
$$
where $\cP_Q\cL_0Q$ is the space of paths in $\cL_0Q$ starting at a constant path.
The resulting symplectic
homology with DG coefficients is equal to the homology of $Q$. In contrast, the DG
Morse homology of the constant loops $Q$ is that of $\cP_Q\cL_0Q|_Q=\Map(\SS^2,Q)$, and this is typically larger than the homology of $Q$ when the latter is not a $K(\pi,1)$.  
Moreover, it contains suitable classes ``that live over the fundamental class" (see~\S\ref{sec:classes_over_fund_class}) whenever $Q$ is abundant
with $2$-spheres, which allows us to conclude.  This example echoes the one presented in~\S\ref{sec:overview}.

The playfield is wide open. It is absolutely conceivable that other fibrations may do the job in much more general situations and maybe even provide a proof of  Conjecture \ref{magic}. We point out that in view of Criterion 2 for DG coefficients presented in~\S\ref{sec:intro-Liouville} (Theorem \ref{thm:finiteness-HZ-intro}), this becomes a topological problem about the existence of suitable fibrations, or more general DG local systems, over $\cL_0 Q$. 


\subsection{Floer homology with DG coefficients} \label{sec:intro-DGFloer}

The results presented in the previous section are obtained using a new Floer homology theory with DG local coefficients that we develop in~\S\ref{sec:DGFloer}. Our construction builds crucially on~\cite{BC07} and it parallels the Morse homology theory with DG coefficients from~\cite{BDHO}. In this section we summarize it.

In one word, the difference between classical Floer homology and Floer homology with DG local coefficients is the following: the former uses the compactified moduli spaces of connecting Floer trajectories of \emph{dimension $0$} (to define maps) \emph{and $1$} (to infer relations), while the latter uses representatives of the fundamental classes rel boundary for the compactified moduli spaces of Floer trajectories \emph{of arbitrary dimension} (to define maps \emph{and} to infer relations).  

Given a based and path-connected topological space $B$, let $C_*(\Omega B)$ be the DGA of cubical chains on the based Moore loop space $\Omega B$.  A \emph{Moore loop} is a path $\gamma:[0,L]\to B$ with $L\ge 0$ arbitrary and $\gamma(0)=\gamma(L)$. If the last condition is dropped, we speak of a \emph{Moore path}. 

\begin{definition} A \emph{DG local system} on $B$ is a chain complex $\cF_*$ with an additional structure of DG right $C_*(\Omega B)$-module. 
\end{definition} 

The main source of DG local systems are Serre fibrations $F\hookrightarrow E\stackrel{\pi}{\longrightarrow} B$: up to a canonical fiberwise weak homotopy equivalence, the holonomy of such a fibration is encoded in a \emph{holonomy map} $F\times \Omega B\to F$ determined by a choice of \emph{lifting function} $E \times_B \cP B\to \cP E$ which preserves the concatenation of paths. Here $\cP B$ is the space of Moore paths in $B$. The holonomy map induces a right $C_*(\Omega B)$-module structure on $\cF=C_*(F)$. When $B$ is a manifold, and more generally a colimit of manifolds, the Morse complex of $B$ with DG coefficients in $C_*(F)$ is a model for the chains on the total space of $E$, see~\cite[\S7]{BDHO}. 

Given a manifold $X$, let $\cL X$ be its free loop space and $\cL_0 X$ be the component of contractible loops. 
We shall define Hamiltonian Floer homology with coefficients in a DG local system $\cF$ on $\cL_0 X$ for a symplectic manifold $(X,\omega)$ such that $\omega$ vanishes on $2$-spheres (\emph{symplectically aspherical}), and with coefficients in a DG local system $\cF$ on $\cL X$ for a symplectic manifold $(X,\omega)$ such that $\omega$ vanishes on $2$-tori (\emph{symplectically atoroidal}). We address both closed manifolds and Liouville domains. Our setup includes the particular case where the DG local system is supported in degree $0$ (we will refer to such DG local systems as \emph{classical local systems}). In this case, we recover the classical Floer homology with local coefficients.

{\bf The DG Floer toolset} provides a hierarchy of maps with the same formal structure as in classical Floer theory (differential, continuation maps, homotopies), but enriched with coefficients in a fixed DG local system $\cF$ over $\cL X$. (It is understood that, if $\cL X$ has several connected components, we work componentwise.) In addition to Hamiltonians $H$ and almost complex structures $J$, one needs to choose additional auxiliary data.\footnote{In list form: representing chain systems for the moduli spaces of Floer trajectories, embedded trees in $\cL X$ connecting the $1$-periodic orbits of the Hamiltonian to a chosen basepoint, homotopy inverses for the maps that collapse these trees, functions that allow to suitably modify the value of the action functional along continuation Floer trajectories. See~\S\ref{sec:DGFloer}.} We reserve the details for~\S\ref{sec:DGFloer} and, at this stage, we just record the necessity of this additional choice by the generic notation $\Xi_{(H,J)}$. The discussion that follows is similar to~\cite[\S2]{BDHO}.

Let $R_*=C_*(\Omega\cL X)$. 

\begin{itemize}
\item to a regular admissible pair $(H,J)$ consisting of a Hamiltonian and an almost complex structure, complemented by additional data $\Xi_{(H,J)}$, one associates \emph{the Floer complex with coefficients in the DG local system $\cF$} defined as 
$$
FC_*(H,J,\Xi_{(H,J)};\cF)=(\cF\otimes C_\bullet,D),
$$
where $C_\bullet = \langle \Per(H)\rangle$ and $\Per(H)$ is the set of $1$-periodic orbits of the Hamiltonian $H$. With respect to the canonical basis $\Per(H)$ of $C_\bullet$, the differential $D$ takes the explicit form
$$
D(\alpha\otimes x) = \p\alpha\otimes x + (-1)^{|\alpha|} \sum_y \alpha \cdot m_{x,y} \otimes y,
$$
with $m_{x,y}\in C_{|x|-|y|-1}(\Omega \cL X)$ such that the Maurer-Cartan condition   
$$
\p m_{x,y}-\sum_{z}(-1)^{|x|-|z|}m_{x,z}\cdot m_{z,y} = 0
$$
holds. Here the product $m_{x,z}\cdot m_{z,y}$ is induced by concatenation of based loops.

\begin{definition}\label{def:general-twisting} 
A family of chains $m_{x,y}\in C_{|x|-|y|-1}(\Omega \cL X)$ that satisfies the Maurer-Cartan equation above is called a \emph{twisting cocycle}.
\end{definition}

In practice, the chains $m_{x,y}$ are constructed from chain representatives of the fundamental classes of the compactified moduli spaces of Floer trajectories \emph{of arbitrary dimension} by evaluating them into chains on $\Omega \cL X$. This construction is the content of~\S\ref{sec:twisting-cocycle}.  We call the family of chains obtained by this method {\it Barraud-Cornea twisting cocycle}.

\item to a regular admissible homotopy $(H_s,J_s)$ that interpolates between regular pairs $(H_\pm,J_\pm)$ at $\pm\infty$, complemented by additional data $\Xi_{(H_s,J_s)}$ that interpolates between $\Xi_{(H_+,J_+)}$ and $\Xi_{(H_-,J_-)}$, one associates a degree $0$ chain map
$$
\Psi:FC_*(H_+,J_+,\Xi_{(H_+,J_+)};\cF)\to FC_*(H_-,J_-,\Xi_{(H_-,J_-)};\cF),
$$ 
called \emph{continuation map}. With respect to the canonical bases $\Per(H_+)$ and $\Per(H_-)$, the map $\Psi$ takes the explicit form 
$$
\Psi(\alpha\otimes x^+) = \sum_{y^-} \alpha\cdot \nu_{x^+,y^-}\otimes y^-,
$$
with $\nu_{x^+,y^-}\in C_{|x^+|-|y^-|}(\Omega \cL X)$ such that 
$$
\p \nu_{x^+,y^-}=\sum_{z^+} m_{x^+,z^+} \nu_{z^+,y^-}+ \sum_{z^-}(-1)^{|x^+|-|z^-|-1} \nu_{x^+,z^-} m_{z^-,y^-}.
$$

\begin{definition}\label{def:general-continuation} 
A family of chains $\nu_{x^+,y^-}\in C_{|x^+|-|y^-|}(\Omega \cL X)$ that satisfies the  equation above is called \emph{continuation  cocycle subordinated to the twisting cocycles $(m_{x^+, y^+})$ and $(m_{x^-, y^-})$}.
\end{definition}

In practice, the chains $\nu_{x^+,y^-}$ are constructed from chain representatives of the fundamental classes of the compactified moduli spaces of continuation Floer trajectories of \emph{arbitrary dimension} by evaluating them into chains on $\Omega \cL X$. This construction is the content of~\S\ref{sec:continuation} and the family of chains thus obtained is called {\it Barraud-Cornea continuation cocycle}. 

\item to a regular homotopy of homotopies $\{(H^\tau,J^\tau)\}$ connecting two regular homotopies $(H^i, J^i)$, $i=0,1$ with endpoints $(H_\pm,J_\pm)$, complemented by additional data $\Xi_{(H^\tau,J^\tau)}$ that interpolates between $\Xi_{(H^0,J^0)}$ and $\Xi_{(H^1,J^1)}$, one associates a degree $1$ map 
$$
\h:FC_*(H_+,J_+,\Xi_{(H_+,J_+)};\cF)\to FC_*(H_-,J_-,\Xi_{(H_-,J_-)};\cF)
$$
that realizes a chain homotopy between the continuation maps $\Psi^i$ determined by $(H^i, J^i,\Xi_{(H^i,J^i)})$, i.e.,  
\begin{equation} \label{eq:Psi10h-intro}
\Psi^1 - \Psi^0 = D^-\h + \h D^+.
\end{equation}
With respect to the canonical bases $\Per(H_+)$ and $\Per(H_-)$, the map $\h$ takes the explicit form 
$$
\h(\alpha\otimes x^+) = (-1)^{|\alpha|}\sum_{y^-} \alpha\cdot h_{x^+,y^-}\otimes y^-,
$$
where $h_{x^+,y^-}\in C_{|x^+|-|y^-|+1}(\Omega \cL X)$ are such that 
\begin{align*}
\partial h_{x^+,y^-}\, =\,  \nu^1_{x^+,y^-} & - \nu^0_{x^+, y^-} + \sum_{z^+} (-1)^{|x^+|-|z^+|} m_{x^+,z^+} h_{z^+,y^-}\\
& + \sum_{z^-}(-1)^{|x^+|-|z^-|} h_{x^+,z^-} m_{z^-,y^-}.
\end{align*}
\begin{definition}\label{def:general-parametrized} 
A family of chains $h_{x^+,y^-}\in C_{|x^+|-|y^-|+1}(\Omega \cL X)$ that satisfies the  equation above is called \emph{parametrized continuation  cocycle subordinated  to the twisting cocycles $(m_{x^+, y^+})$ and $(m_{x^-, y^-})$ and to the continuation cocycles $(\nu^1_{x^+, y^-})$ and $(\nu^2_{x^+, y^-})$}.
\end{definition} 

In practice, the chains $h_{x^+,y^-}$ are constructed from chain representatives of the fundamental classes of the compactified moduli spaces of solutions of a Floer problem parametrized by the interval $[0,1]$, of \emph{arbitrary dimension}, by evaluating them into chains on $\Omega \cL X$. We call the family of chains thus obtained {\it Barraud-Cornea parametrized continuation cocycle}. This is the content of~\S\ref{sec:homotopies}.
 
\end{itemize}

The continuation maps $\Psi=\Psi^{H_s,J_s,\Xi_{(H_s,J_s)}}$ satisfy the following two additional properties. 
\begin{itemize}
\item Let $\Id_{H,J}$ be the constant homotopy equal to a regular pair $(H,J)$. For any data $\Xi_{(H,J)}$ the continuation map $\Psi^{\mathrm{Id}_{H,J}, \Xi_{(H,J)}}$ is homotopic to the identity. See Proposition~\ref{prop:Id-Id} and Corollary~\ref{cor:Id-Id}.
\item Given three regular homotopies $(H^{ij},J^{ij})$, $1\le i<j\le 3$ interpolating between regular pairs $(H^i,J^i)$ at $+\infty$ and $(H^j,J^j)$ at $-\infty$, complemented by additional data $\Xi_{(H^{ij},J^{ij})}$, the chain maps $\Psi^{23}\circ\Psi^{12}$ and $\Psi^{13}$ are homotopic. See Proposition~\ref{prop:composition-v1}.
\end{itemize}

As in the classical case, this construction is compatible with action filtrations: 
\begin{itemize}
\item The Floer complex $FC_*(H,J,\Xi_{(H,J)};\cF)$ is filtered by the values of the action. We denote $FC_*^{<b}(H,J,\Xi_{(H,J)};\cF)$, $b\in\R$ the subcomplex generated by orbits of action $<b$. We emphasize that the complex of coefficients $\cF$ need not be filtered. See~\S\ref{sec:DGFloercomplex}.
\item The continuation maps induced by monotone homotopies $H_s$ such that $\p_s H_s\le 0$ preserve the action filtration. See Corollary~\ref{cor:DGcontinuation-monotone}.
\item For any two triples $(H_\pm,J_\pm,\Xi_{(H_\pm,J_\pm)})$ and any $\eps>0$, there exists a homotopy $(H_s,J_s,\Xi_{(H_s,J_s)})$ such that the induced  continuation map sends $FC_*^{<b}(H_+,J_+,\Xi_{(H_+,J_+)};\cF)$ to $FC_*^{<b+E}(H_-,J_-,\Xi_{(H_-,J_-)};\cF)$, where $E=\int_{0}^1\max (H_+^t-H_-^t)\,dt +\eps$. See Corollary~\ref{cor:DGcontinuation-Hofer-norm}.
\end{itemize}

This is the complete DG toolset at homology level. We do not address higher chain level structures in this paper, nor multiplicative structures. 

\begin{remark}
We would like to emphasize the universality of this procedure. It applies to any Floer theory (Lagrangian, instanton etc.) and can be used to enrich Floer maps in a variety of settings, for example in the context of Lagrangian correspondences. 
\end{remark}

We also define \emph{symplectic homology with DG coefficients} for a Liouville domain $X$. The usual construction adapts in a straightforward way once the DG Floer toolset has been put into place. See~\S\ref{sec:symplectic_homology}.

We close this subsection by discussing an important computational device. As in the case of  Morse theory with DG coefficients \cite {BDHO}, Hamiltonian Floer homology with DG coefficients is the limit of a canonical spectral sequence defined by filtering the DG Floer complex $FC_*(H,J,\Xi_{(H,J)};\cF)= \calf\otimes C_{\bullet}$ by the index of the generators in $\Per(H)$; more precisely this filtration is given by $F_p= \bigoplus_{i\leq p}\calf\otimes C_i$. See Theorem~\ref{thm:spectral_sequence}. 
We illustrate the use of this spectral sequence by the following theorem, which is proved in~\S\ref{sec:spectral-sequence-sympl-asph}.

\begin{customthm}{J} \label{thm:FH_for_fibrations-intro} Let $X$ be a closed symplectically aspherical manifold of dimension $2n$. Let $F\hookrightarrow E\to \cL_0 X$ be a Serre fibration and let $\cF=C_*(F)$ be the corresponding DG local system on $\cL_0 X$. For any regular pair $(H,J)$ and any choice of enriched Floer data $\Xi_{(H,J)}$, 
Floer homology in the component of contractible loops on $X$ is given by 
$$
FH_*(H,J,\Xi_{(H,J)};\cF)\simeq H_{*+n}(E|_X),
$$
where $H_{*+n}(E|_X)$ is the homology of the total space of the fibration $E|_X$. Moreover, this isomorphism is the limit of an isomorphism between the DG Floer spectral sequence and the Leray-Serre spectral sequence. 
\end{customthm}

\subsection{The Viterbo isomorphism with DG local coefficients} \label{sec:intro-Viterbo}

Let $Q$ be a closed smooth manifold and $T^*Q$ its cotangent bundle. We denote $D^*Q$ and $S^*Q$ the disc cotangent bundle, respectively the sphere cotangent bundle with respect to some Riemannian metric on $Q$. The celebrated Viterbo isomorphism identifies the symplectic homology of the cotangent bundle, denoted $SH_*(T^*Q)$, with the free loop space homology of the base, twisted by a suitable rank one local system $\underline\eta$. 
We call the latter \emph{loop homology} and denote it $H_*(\cL Q;\underline\eta)$. The precise statement is the following. 

{\bf Theorem} (``Classical" Viterbo isomorphism {\cite{Viterbo-cotangent,AS,AS-corrigendum,AS2,SW,Abouzaid-cotangent}})
{\it 
There is an isomorphism  
$$
\Psi:SH_*(T^*Q)\stackrel\simeq\longrightarrow H_*(\cL Q;\underline\eta) 
$$
that intertwines the canonical BV algebra structures defined on both sides.
}   

This theorem is established in this degree of generality in~\cite{Abouzaid-cotangent}, and we refer to \emph{loc.\@ cit.}\@ for a detailed history. The map $\Psi$ in the statement is a variant of a map first considered by Cieliebak and Latschev in the context of contact homology~\cite{Cieliebak-Latschev}. The necessity of involving a local system was first pointed out by Kragh~\cite{Kragh}. 
Cieliebak, Hingston and the fourth author of this paper proved that the map $\Psi$ descends to \emph{reduced} versions of symplectic homology and loop homology, where it intertwines the unital infinitesimal bialgebra structures defined on both sides~\cite{CHO-MorseFloerGH}. We describe the local system $\underline\eta$ in~\S\ref{sec:DGViterbo}.

As a matter of fact, Abouzaid's proof of the above theorem implies a Viterbo isomorphism with arbitrary (classical) local coefficients. Given a classical local system $M$ on $\cL T^*Q$ we denote $SH_*(T^*Q;M)$ symplectic homology with local coefficients in $M$. Denote $\pi:T^*Q\to Q$ the projection and 
$$
\Pi:\cL T^*Q\to\cL Q
$$ 
the induced map between free loop spaces. Since $\pi$ is a homotopy equivalence, so is the map $\Pi$ and therefore any local system on $\cL T^*Q$ is of the form $\Pi^*L $ for some local system $L$ on $\cL Q$. 

{\bf Theorem} (Viterbo isomorphism with local coefficients~{\cite{Abouzaid-cotangent}})
{\it 
For any classical local system $L$ on $\cL Q$, there is an isomorphism 
\begin{equation} \label{eq:Viterbo_iso_loc_sys}
\Psi:SH_*(T^*Q;\Pi^*L)\stackrel\simeq\longrightarrow H_*(\cL Q;L\otimes\underline\eta).  
\end{equation}
}

Whenever $L$ is \emph{transgressive} in the sense of~\cite{Abouzaid-cotangent}, each of the two sides carries a canonical BV structure and the map $\Psi$ intertwines those. More generally, whenever $L$ is \emph{compatible with products} in the sense of~\cite[Appendix~A]{CHO-MorseFloerGH}, the reduced versions of each of the two terms on the right carry unital infinitesimal bialgebra structures and the map $\Psi$ intertwines those. 

We prove in this paper the following generalization of the above theorem. 

\begin{customthm}{K}[DG Viterbo isomorphism] 
\label{thm:Viterbo_iso_loc_sys_DG}
Let $\cF$ be a DG local system on $\cL Q$  such that $H_*(\cF)$ is bounded from below. 

There is an isomorphism 
$$
\widetilde\Psi:SH_*(T^*Q;\Pi^*\cF)\stackrel\simeq\longrightarrow H_*(\cL Q;\cF\otimes\underline\eta).  
$$
The map $\widetilde\Psi$ is the limit of an isomorphism between the canonical spectral sequences associated to its source and target. At the second page, this isomorphism of spectral sequences is given by the map from~\eqref{eq:Viterbo_iso_loc_sys}:  
$$
\widetilde\Psi^2_{p,q}=\Psi:
SH_p(T^*Q;\Pi^*H_q(\cF))\stackrel\simeq\longrightarrow 
H_p(\cL Q;H_q(\cF)\otimes \underline\eta).
$$ 
\end{customthm}

In the second part of the statement we implicitly use the identifications 
$$
H_q(\Pi^*\cF)=\Pi^*H_q(\cF) \qquad \mbox{and} \qquad H_q(\cF\otimes\underline\eta)=H_q(\cF)\otimes\underline\eta.
$$ 

 This DG Viterbo isomorphism is a key ingredient for our dynamical applications to cotangent bundles, because it bridges Floer theory with DG coefficients in $T^*Q$ to Morse theory with DG coefficients on the free loop space $\cL Q$. For the more algebraically minded reader, note that~\cite[\S3, Remark 3.9]{BDHO}
$$
H_*(\cL Q;\cF\otimes\underline\eta) = \mathrm{Tor}_*^{C_*(\Omega \cL Q)}(\cF\otimes \underline\eta,\Z).
$$

\subsection{General criteria for Liouville domains} \label{sec:intro-Liouville}

We prove in the paper two dynamical criteria valid for general Liouville domains, based on symplectic homology with DG coefficients. The first is item (a) below and concerns the almost existence of contractible closed characteristics in the neighborhood of closed energy levels, the second is item (b) below and concerns the finiteness of the $\pi_1$-sensitive Hofer-Zehnder capacity. 
 A Liouville domain is a compact exact symplectic manifold such that the vector field dual to the primitive of the symplectic form points outwards along the boundary, a condition that has to be interpreted as a symplectic notion of convexity. For the purpose of this paper it is important that unit disc cotangent bundles are Liouville domains, but the latter class is much vaster~\cite{Seidel07,Cieliebak-Eliashberg-book}.

Let $W$ be a Liouville domain of dimension $2n$ and $\cF$ a DG local system on $\cL W$. Let $i_W:W\hookrightarrow \cL_0 W$ be the inclusion of constant loops.

(a) Given a relatively compact open subset $U\subset W$ with smooth boundary, the inclusion $j_U:U\hookrightarrow W$ induces a shriek map 
$$j_{U!}:H_*(W,\p W;i_W^*\cF)\to H_*(U,\p U;j_U^*i_W^*\cF)$$ between Morse homologies with DG coefficients~\cite[\S\S9--10]{BDHO}. 

On the other hand, a variant of Theorem~\ref{thm:FH_for_fibrations-intro} (see Proposition~\ref{prop:SH-action-zero}) shows that the canonical map from the zero energy sector of symplectic homology to the full symplectic homology with DG coefficients translates into a canonical map $i_{W*}: H_{*+n}(W,\p W;i_W^*\cF)\to SH_*(W;\cF)$. 

We consider the diagram
\begin{equation}\label{eq:diagram-criterion-intro}
  \xymatrix{
     H_{*+n}(U, \partial U;j_U^*i_W^*\cF) &\, \\
      H_{*+n}(W,\p W;i_W^*\cF) \ar[r]^-{i_{W*}}\ar[u]^-{j_{U!}}     
    & SH_*(W;\cF).}
\end{equation}

\begin{customthm}{L}\label{thm:almost-sure-existence-W-intro}[Criterion for  contractible almost existence] Let $\Sigma\subset W$ be a closed hypersurface bounding a relatively compact open subset $U$. We assume that there exists a DG local system $\cF$ on $\cL_0 W$ and a class $\alpha\in H_*(W,\p W;i_W^*\cF)$ such that $j_{U!}(\alpha) \neq 0$ and $i_{W*}(\alpha)=0$.

Then the  contractible almost existence property holds near $\Sigma$.
\end{customthm} 

Theorem~\ref{thm:almost-sure-existence-W-intro} is proved in~\S\ref{sec:almost-existence} as Theorem~\ref{thm:almost-sure-existence-W}.

(b) The inclusion of a point $j:\mathrm{pt}\hookrightarrow W$ induces a shriek map
$$j_!:H_{*+n}(W,\p W;i_W^*\cF)\to H_{*-n}(\mathrm{pt};j^*\cF)=H_{*-n}(\cF).$$ We consider the diagram  
\begin{equation}\label{eq:diagram-finiteness-HZ-intro}
  \xymatrix{
     H_{*-n}(\cF) &\, \\
      H_{*+n}(W,\p W;i_W^*\cF) \ar[r]^-{i_{W*}}\ar[u]^-{j_{!}}     
    & SH_*(W;\cF).}
\end{equation}

We denote $c_{HZ}^\circ(W)\in (0,\infty]$ the \emph{$\pi_1$-sensitive Hofer-Zehnder capacity} of $W$, and recall that the finiteness of $c_{HZ}^\circ(W)$ implies the  contractible almost existence for \emph{any} closed hypersurface $\Sigma\subset W$, see~\cite{HZ,Struwe}.

\begin{customthm}{M}\label{thm:finiteness-HZ-intro}[Criterion for finiteness of $c_{HZ}^\circ(W)$] Assume that there exists a DG local system $\cF$ on $\cL_0W$ and a class $\alpha\in H_*(W,\p W;i_W^*\cF)$ such that $j_{!}(\alpha) \neq 0$ and $i_{W*}(\alpha)=0$. Then 
$$
c_{HZ}^\circ(W)<\infty. 
$$
In particular, the  contractible almost existence property holds in $W$.
\end{customthm} 

Theorem~\ref{thm:finiteness-HZ-intro} is proved in~\S\ref{sec:HZ} as Theorem~\ref{thm:finiteness-HZ}.

\subsection{Relation to other work} \label{sec:intro-other_work}

{\it Dynamics.} We mentioned in~\S\ref{sec:intro-almost-existence} the foundational result by Hofer-Zehnder and Struwe on the almost existence property for hypersurfaces in $\R^{2n}$. The almost existence property is valid for other symplectic manifolds, such as $\C P^n$ and sub-critical Stein manifolds \cite{FHV90,Hofer-Viterbo-92,Lu98}, but false in general (e.g., in the ``Zehnder torus"~\cite{Ze87}). In the same vein, there are many results that prove the almost existence property for particular classes of closed hypersurfaces, e.g., hypersurfaces located in a neighborhood of a symplectic submanifold~\cite{CGK,Ginzburg-Gurel-relative-capacity,Usher}.

In the context of cotangent bundles $T^*Q$, the foundational result is due to Hofer-Viterbo~\cite{HV88}, who prove almost existence for hypersurfaces that enclose the 0-section in cotangent bundles (without any statement on contractibility). Further flavors of this result were proved by Biran-Polterovich-Salamon~\cite{BPS}, Weber~\cite{Weber06}, Gong-Xue~\cite{Gong-Xue-2020}, Asselle-Starostka~\cite{AS2020}. This almost existence result was recently enhanced by Benedetti-Kang~\cite{Benedetti-Kang} to almost existence of contractible closed characteristics for hypersurfaces that enclose the $0$-section, under the assumption that $H(\cL_0 Q,Q)\neq 0$. In comparison, our Theorem~\ref{thm:almost-existence-simplified} holds in optimal generality for all bases $Q$ that are not aspherical, and it enlarges significantly the class of hypersurfaces. Theorems~\ref{thm:almost-existence-Thom-simplified}, \ref{thm:almost-existence-nonorientable}, and~\ref{thm:orbits-knotted-case-intro} constitute further enhancements. 

 Contractible almost existence unconditional on the hypersurface is a consequence of the  finiteness of the $\pi_1$-sensitive Hofer-Zehnder capacity, and the latter has been established for $D^*Q$ under various conditions on $Q$ by Ritter~\cite{Ritter09}, Irie~\cite{Irie14}, Frauenfelder-Pajitnov~\cite{Frauenfelder-Pajitnov}, Albers, Frauenfelder and the fourth author~\cite{Albers-Frauenfelder-Oancea}. Our Theorems~\ref{thm:abundant-intro} and~\ref{thm:examples-intro} go significantly beyond~\cite{Ritter09,Irie14,Albers-Frauenfelder-Oancea}, but stay within the class of rationally inessential manifolds studied by Frauenfelder-Pajitnov~\cite{Frauenfelder-Pajitnov}. The method is however radically different, and this motivates our Conjecture~\ref{magic}. 
 
We refer to~\cite{Ginzburg-overview} for other related results and a comprehensive description of the problem of almost existence, and to the introduction of~\cite{Benedetti-Kang} for a detailed account of the various available results for cotangent bundles. 

The case of magnetic cotangent bundles  has also inspired a vast literature. Our Corollary \ref{cor:exact-magnetic} generalizes Contreras' work \cite{Contreras2006} which covers the case where the open subset $U$ is fiberwise convex with an exact magnetic form. A strong existence result for exact magnetic cotangent bundles on surfaces was proved by Contreras-Macarini-Paternain~\cite{CMP}. The case of a non exact magnetic form was settled by Groman-Merry \cite{Groman-Merry} who proved that if the magnetic form $\beta$ is not weakly-exact (i.e., does not vanish on $\pi_2(Q)$) then $D^*Q$  has  finite $\pi_1$-sensitive Hofer-Zehnder capacity in $(T^*Q, \omega+\pi^*\beta)$. Our result complements the latter. The case of a weakly exact but non exact magnetic form is studied for surfaces by Benedetti-Ritter~\cite{Benedetti-Ritter}, but remains largely unexplored in higher dimensions.

{\it Floer homology with DG coefficients.}  Differential graded local coefficients also appear in the literature under the name \emph{infinity local systems}, or \emph{derived local systems}. The equivalence between these different notions is proved by Holstein~\cite[Theorem~26]{Holstein}. We refer to the introduction of~\cite{BDHO} for a more detailed discussion of this equivalence and for further references. 

Our construction of the Floer complex with DG coefficients builds on the original ideas of Barraud and Cornea~\cite{BC07}. The key 
ingredient is to consider simultaneously all the compactified moduli spaces of Floer trajectories, of any dimension, and a coherent system of chain representatives of their fundamental classes rel boundary. In order to further extract geometric information, one needs to choose a suitable space of paths into which these moduli spaces evaluate. Our definition of the Barraud-Cornea cocycle is universal in the sense that we choose to evaluate into spaces of \emph{paths on the free loop space of the underlying symplectic manifold}. The connection with the Brown universal cocycle described in~\cite{BDHO} gives yet another meaning, of a more algebraic nature, to this universality property. The original Barraud-Cornea cocycle was obtained, in the case of Hamiltonian Floer homology, by evaluating into a space of \emph{paths on the symplectic manifold itself}, and can be viewed as a pull-back of our universal construction~\cite{BC07b}. Correspondingly, the Lagrangian Barraud-Cornea cocycle from~\cite{BC07} was obtained by evaluating moduli spaces of Floer strips into paths on a Lagrangian submanifold.   

Abouzaid reinterpreted in~\cite{Abouzaid2012a} the Lagrangian Barraud-Cornea cocycle as a twisted complex, and used this point of view in order to construct an $A_\infty$ functor from the wrapped Fukaya category of a Liouville domain to the category of twisted complexes on the path category of a given closed exact Lagrangian. The construction of this functor makes use of compactified moduli spaces of Floer punctured discs of arbitrary dimension, which is the next level of complexity. Abouzaid's functor plays a role in his proof of the nearby Lagrangian conjecture~\cite{Abouzaid-nearby}, though the Fukaya category enriched in complexes of local systems from~\cite[\S2]{Abouzaid-nearby} only uses moduli spaces of dimension 0.

The original construction of Barraud-Cornea has been also revisited in recent years by Zhou~\cite{Zhou-MB,Zhou-ring}, Rezchikov~\cite{rezchikov}, Charette~\cite{Charette2017}. 
 Compared to all these works, our perspective is to use DG coefficients as auxiliary data that feeds into a filtered homological Floer package, from which we extract useful information, e.g., spectral invariants.

{\it Floer homotopy theory.} The construction of the Floer complex with DG coefficients 
connects to Floer homotopy theory. This is a rapidly unfolding field envisioned by Cohen-Jones-Segal~\cite{cohen-jones-segal} and going back to Floer~\cite{Floer-Witten}. 

Floer homotopy theory aims to  extract generalized cohomology theories, or spectra, from higher dimensional trajectory spaces. Using the terminology of Cohen-Jones-Segal~\cite{cohen-jones-segal}, the ensemble of all trajectory spaces has the structure of a \emph{flow category}, see also Pardon~\cite{Pardon-algebraic} and Abouzaid~\cite{Abouzaid-flows}. 

We would like to argue that the Floer complex with DG coefficients sits midway between homology and homotopy, see also the discussion in~\cite[Introduction]{BDHO}. This can be seen by discussing how much information one retains from the Floer trajectory spaces: classical Floer theory use moduli spaces of dimension $0$ and $1$; Floer homotopy theory uses the full moduli spaces together with their framing; our Floer homology with DG coefficients uses (representatives of) the fundamental classes rel boundary of all the compactified moduli spaces.

In relation to Floer homotopy theory, we should note that the construction of DG Floer theory only requires for the compactified moduli spaces the structure of ``manifolds with corners" established in~\cite[Appendix~A]{BC07}, i.e., 
topological manifolds with smoothly stratified boundary. This is strictly weaker than the smooth structure needed in Floer homotopy theory. The latter has now been constructed in full generality for Hamiltonian Floer theory by Abouzaid and Blumberg~\cite{Abouzaid-Blumberg-FHT}, building on their previous work~\cite{Abouzaid-Blumberg-Morava} and on the work of Abouzaid-McLean-Smith~\cite{Abouzaid-McLean-Smith}, Bai-Xu~\cite{Bai-Xu}, and Rezchikov~\cite{Rezchikov-Arnold}.

\subsection{Structure of the paper} 

The paper is organized as follows. In~\S\ref{sec:DGMorse} we briefly recall the construction of Morse homology with DG coefficients. In~\S\ref{sec:DGFloer} we develop Hamiltonian Floer theory with DG coefficients, and we prove in particular Theorem~\ref{thm:FH_for_fibrations-intro}. In~\S\ref{sec:DGViterbo} we prove the Viterbo isomorphism from Theorem~\ref{thm:Viterbo_iso_loc_sys_DG}. In~\S\ref{sec:almost-existence} we prove our applications to almost existence, namely Theorems~\ref{thm:almost-existence-simplified}, \ref{thm:almost-existence-Thom-simplified} and~\ref{thm:almost-existence-nonorientable}, as well as the criterion from Theorem~\ref{thm:almost-sure-existence-W-intro}. In~\S\ref{sec:HZ} we prove the criterion from Theorem~\ref{thm:finiteness-HZ-intro} and our applications to the finiteness of the $\pi_1$-sensitive Hofer-Zehnder capacity from Theorems~\ref{thm:abundant-intro} and~\ref{thm:examples-intro}. In~\S\ref{sec:knotted} we prove our application to almost existence for hypersurfaces knotted with spheres from Theorem~\ref{thm:orbits-knotted-case-intro}.

\bigskip

{\it Acknowledgements.} We thank Leonardo Macarini for pointing out Corollary~\ref{cor:exact-magnetic}. 
We thank Gabriele Benedetti for help with the bibliography on magnetic cotangent bundles. We thank Kai Cieliebak for his careful reading of the manuscript and for his valuable suggestions. We have benefited from discussions with Mohammed Abouzaid, Colin Fourel, Klaus Niederkr\"uger, Robin Riegel, Ivan Smith, and Jean-Yves Welschinger.
  
The authors were able to meet in excellent conditions during the Covid pandemic at CIRM Marseille in February 2021 and at IHP Paris in May-June 2021. They acknowledge partial funding from the ANR, grant no. 21-CE40-0002 (COSY). The fourth author acknowledges support from a Fellowship of the University of Strasbourg Institute for Advanced Study (USIAS), within the French national programme ``Investment for the future" (IdEx-Unistra).

\section{Morse homology with DG local coefficients} \label{sec:DGMorse}

In this section we recall the definition and some properties of Morse homology with DG local coefficients following~\cite{BDHO}, we present an alternative formulation where the DG local system is twisted by a rank $1$ (classical) local system, and we discuss the new notion of a class ``living over the fundamental class". We work over a ground ring $A$. 

\subsection{Definition of Morse homology with DG coefficients}\label{sec:def-MorseDG}

Given a path-connected topological space $X$ with a basepoint $\star$, we call \emph{DG local system on $X$} the data of a right $C_*(\Omega X)$-module. When we work on a disconnected space, we fix a basepoint on each connected component, and by DG local system we mean the data of a DG right $C_*(\Omega X)$-module for each path-connected component $X$. Here $\Omega X$ denotes the space of  based Moore loops on $X$ and $C_*(\Omega X)$ denotes the DG algebra of cubical chains on $\Omega X$ with product given by the concatenation of loops.

We place ourselves in the setup where $X$ is a closed connected manifold (extensions to the case of manifolds with boundary, and to spaces that have the homotopy type of manifolds, are discussed in~\cite{BDHO}). Fix a basepoint $\star$ and let $\cF$ be a DG local system on $X$. Let $f:X\to \R$ be a Morse function and $\xi$ a Morse-Smale negative pseudo-gradient. Given $x,y\in\Crit(f)$ we denote by $\cL(x;y)$ the moduli space of $\xi$-trajectories connecting $x$ at $-\infty$ to $y$ at $+\infty$. Let $|x|$ denote the Morse index of $x\in\Crit(f)$. For $x\neq y$ the space $\cL(x;y)$ is a smooth manifold of dimension $|x|-|y|-1$, and it admits a compactification $\ol\cL(x;y)$ that is a manifold with boundary with corners such that, ignoring orientations, $\p\ol\cL(x;y)=\cup_z\ol\cL(x;z)\times\ol\cL(z;x)$. We call \emph{representing chain system for the moduli spaces $\ol\cL(x;y)$} a collection $s_{x,y}\in C_{|x|-|y|-1}(\ol\cL(x;y),\p \cL(x;y))$ such that $s_{x,y}$ is a cycle rel boundary representing the fundamental class, and there holds the relation
$$
\p s_{x,y} 	= \sum_z (-1)^{|x|-|z|} s_{x,z}\times s_{z,y},
$$
where the product of chains is defined via the inclusions $\ol\cL(x;z)\times\ol\cL(z;x)\subset \p\ol\cL(x;y)\subset \ol\cL(x;y)$. We call \emph{enriched Morse data for $(f,\xi)$} a tuple consisting of an orientation $o=(o_x)_{x\in\crit(f)}$ for the unstable manifolds, a choice of representing chain system $(s_{x,y})$, an embedded tree $\cY$ rooted at the basepoint $\star$ and whose leaves are the critical points of $f$, and a map $\theta:X/\cY\to X$ that is a based homotopy inverse for the projection $p:X\to X/\cY$. We refer to the tuple $\Xi=(f,\xi, o, s_{x,y},\cY,\theta)$ as \emph{enriched Morse data}. 

Given enriched Morse data $\Xi$, we construct a collection $m_{x,y}\in C_{|x|-|y|-1}(\Omega X)$ that satisfies the relation 
\begin{equation} \label{eq:DGMorsedifferential}
\p m_{x,y}=\sum_z (-1)^{|x|-|z|} m_{x,z}\cdot m_{z,y},
\end{equation}
where the product is understood in $C_*(\Omega X)$. We call such a collection \emph{Barraud-Cornea twisting cocycle}, or simply \emph{twisting cocycle} (see Definition \ref{def:general-twisting}). The construction of $m_{x,y}$ proceeds by first considering the evaluation map $\ev_{x,y}:\ol\cL(x;y)\to \cP_{x\to y}X$ with target the space of Moore paths from $x$ to $y$, where the image of each element of $\ol\cL(x;y)$ is parametrized by the levels of the function $f$, see~\cite{BDHO}. We further define $q_{x,y}=\tilde \theta \circ \tilde p\circ \ev_{x,y}$, where $\tilde p:\cP_{x\to y}X\to\Omega (X/\cY)$ is the map induced by $p$, and $\tilde \theta:\Omega (X/\cY)\to \Omega X$ is the map induced by $\theta$. Finally, we set $m_{x,y}=(q_{x,y})_*(s_{x,y})$. 

\emph{The Morse complex of $X$ with DG coefficients in $\cF$ determined by the enriched Morse data $\Xi$}, denoted $C_*(X,\Xi;\cF)$, has as an underlying module $\cF_*\otimes \langle \Crit(f)\rangle$. The differential is defined by 
\begin{equation} \label{eq:DGdifferrential}
\p(\alpha\otimes x)=\p \alpha\otimes x + (-1)^{|\alpha|}\sum_y \alpha\cdot m_{x,y}\otimes y.
\end{equation}
\emph{The Morse homology of $X$ with DG coefficients in $\cF$ determined by the enriched Morse data $\Xi$}, or, for short, \emph{enriched Morse homology of $X$}, is 
$$
H_*(X,\Xi;\cF)=H_*(C_*(X,\Xi;\cF)).
$$
The homotopy type of the chain complex does not depend on the choice of enriched Morse data; see \cite{BDHO}. Therefore, its homology $H_*(X,\Xi;\cF)$ is independent of this choice and we will often drop $\Xi$ from the notation.

\paragraph{Manifolds with boundary.}
  Similarly to the classical Morse homology, one can define the homology $H_*(X,\Xi;\cF)$ when $X$ is a manifold with boundary by using a Morse-Smale pair such that the vector field $\xi$ points inwards along the boundary $\p X$. By working with an outward pointing vector field $\xi$, one defines the \emph{homology rel boundary} of $X$ with DG coefficients.
We denote it by $H_*(X,\partial X, \Xi;\cF)$, or simply $H_*(X,\p X;\cF)$.

\paragraph{Fundamental example.} 
Let $F\to E\to X$ be a Hurewicz fibration. The complex of cubical 
chains $C_*(F)$ carries a natural $C_*(\Omega X)$-module structure induced by the choice of a lifting function. We prove in \cite[Theorem 7.2, (see also Remark 7.9)]{BDHO} that   $H_*(X, \p X; C_*(F))\simeq H_*(E, E|_{\p X})$. 

\paragraph{Functoriality.} In \cite{BDHO} (see \S 8-10 and in particular Theorem~8.2 therein) we give a comprehensive study of the invariance and functoriality properties of this theory. 
Specifically, given two manifolds $X,Y$ and a DG local system $\cF$ on $Y$, any continuous map $\varphi:X\to Y$ induces a canonical linear \emph{direct map}
  \[\varphi_*:H_*(X;\varphi^*\cF)\to H_*(Y;\cF),\]
  and, if $X,Y$ are oriented, a canonical linear \emph{shriek map}
  \[\varphi_!:H_*(Y;\cF)\to H_{*+\dim X-\dim Y}(X;\varphi^*\cF).\]
 The pull-back $\varphi^*\cF$ is defined by endowing $\cF$ with the $C_*(\Omega X)$-module structure obtained by pulling-back the $C_*(\Omega Y)$-module structure by the DGA map $(\Omega\varphi)_*:C_*(\Omega X)\to C_*(\Omega Y)$.
  
These maps are homotopy invariant and satisfy covariant functoriality for the direct map and contravariant functoriality for the shriek map. 

Given a Hurewicz fibration $F\to E\to Y$ and a continuous map $\varphi:X\ri Y$, the  direct map $\varphi_*$ corresponds to the canonical  map $H_*(\varphi^*E)\to H_*(E)$, via the isomorphism above, where $\varphi^*E$ is the total space of the pullback by $\varphi$. We prove this result in   \cite[Proposition 9.14]{BDHO}. In \S\ref{sec:classes_over_fund_class} below we also interpret the shriek map for fibrations in the special case when $X$ is a point. For the general case we refer to~\cite[Theorem~D]{Rie}.

If $X$ and $Y$ have boundary, one can also define direct and shriek maps with similar properties. For instance, if $X$ and $Y$ are oriented the shriek map takes the form 
\begin{equation*}
   \varphi_!:H_*(Y, \partial Y;\cF)\to H_{*+\dim X-\dim Y}(X, \partial X;\varphi^*\cF).
\end{equation*}

Shriek maps can also be defined in the nonorientable case up to twisting the DG local systems by the orientation of the underlying manifold, cf.\@ \emph{loc.\@ cit.}\@ 

The case where $X$ is the basepoint $\star$ is of particular interest in the context of shriek maps, see~\S\ref{sec:classes_over_fund_class}. The homology of a point with coefficients in a DG local system $\calf$ is easy to compute: 
$$
H_*(\star ; \calf) \simeq H_*(\calf).
$$

The Morse complex with DG coefficients is filtered by the index of the critical points of the Morse function. The resulting spectral sequence is canonical starting with the $E^2$-page, and is analogous to the Leray-Serre spectral sequence for Hurewicz fibrations. We refer to~\cite[\S7]{BDHO} for details.

\subsection{Morse homology with (classical) local coefficients}

Another important particular case is that of classical local systems. We denote such an object by $L$,
and the condition of being a local system means that $L$ possesses the structure of an $A[\pi_1]$-module, with $\pi_1=\pi_1(X)$. We note that there is a canonical isomorphism $A[\pi_1]\simeq H_0(\Omega X)$ and we view $A[\pi_1]$ as a DG algebra supported in degree $0$. We view $L$ as being supported in degree $0$ and having trivial differential, and we can view it as a $C_*(\Omega X)$-module via the canonical DG algebra map $C_*(\Omega X)\to A[\pi_1]$ defined as the composition of canonical projections $C_*(\Omega X)\to C_0(\Omega X)\to H_0(\Omega X)$. \emph{Morse homology with local coefficients in $L$} is  defined as the homology of the \emph{Morse complex with local coefficients} $C_*(X,\Xi;L)=L\otimes \langle \Crit(f)\rangle$, with differential 
\begin{equation} \label{eq:nonDGMorsedifferential}
\p(f\otimes x)=\sum_{|y|=|x|-1} f\cdot \hat m_{x,y}\otimes y, 
\end{equation} 
where $\hat m_{x,y}\in H_0(\Omega X)\equiv A[\pi_1]$ is the class of the $0$-chain represented by $m_{x,y}$. That formula~\eqref{eq:DGdifferrential} reduces in this case to~\eqref{eq:nonDGMorsedifferential} is a consequence of the fact that higher dimensional chains act trivially on $L$. 

\subsection{Twist by a rank $1$ local system} \label{sec:twist}

There is a well-defined notion of tensor product of two DG local systems, see~\cite[\S12]{BDHO}. We will focus in this subsection on the particular case where one of the local systems is classical and is free of rank $1$. This material will be used in~\S\ref{sec:DGViterbo}. 

Let $L$ be a rank $1$ classical local system on $X$ supported in degree $0$. By definition this is a rank $1$ free $A$-module with an additional structure of an $A[\pi_1]$-module. The assumption that $L$ is free of rank $1$ implies that $\pi_1$ acts by multiplication with invertible elements $a\in A$, i.e., via a group homomorphism 
$$
\rho^L:\pi_1\to A^\times
$$ 
whose target is the multiplicative group of units of $A$. The group ring $A[\pi_1]$ then acts by multiplication with elements in $A$ via the composed morphism $A[\pi_1]\to A[A^\times]\to A$,  where the first map is induced by $\rho^L$ and the second map is induced by multiplication and addition in $A$. We can view $L$ as a $C_0(\Omega X)$-module via the ring map 
$$
\hat \rho^L : C_0(\Omega X)\to A
$$ 
given by the composition $C_0(\Omega X)\to H_0(\Omega X)\equiv A[\pi_1]\to A[A^\times]\to A$. Here $C_0(\Omega X)\to H_0(\Omega X)\equiv A[\pi_1]$ is the canonical projection. If we view $L$ as a $C_*(\Omega X)$-module, all chains of positive degree act by the zero map. 

Given a cube $\sigma:I^n\to \Omega X$, $I=[0,1]$, we use the notation $\sigma_0=\sigma(0,\dots,0)$, viewed as an element of $C_0(\Omega X)$. We define a linear map 
$$
\Phi^L:C_*(\Omega X)\to C_*(\Omega X)
$$
by setting
$$
\Phi^L(\sigma)=\hat\rho^L(\sigma_0) \sigma
$$
on cubes $\sigma$. The map is then uniquely determined by linearity.  

\begin{lemma}
The map $\Phi^L$ is a DG algebra map. 
\end{lemma}

\begin{proof}
One needs to check the relations $\Phi^L(e)=e$, $\Phi^L(\sigma\cdot\tau)=\Phi^L(\sigma)\cdot\Phi^L(\tau)$ and $\Phi^L(\p\sigma)=\p\Phi^L(\sigma)$ for cubes $\sigma,\tau$ and $e$ the unit given by the constant loop at the baspoint. The first two relations follow from the fact that $\hat\rho^L$ is a ring homomorphism together with the equality $(\sigma\cdot \tau)_0=\sigma_0\cdot \tau_0$. The third relation follows from the fact that $\hat\rho^L$ factors through $A[\pi_1]$ and therefore it takes the same value on all the faces of a cube. 
\end{proof}
    
Given a DG local system $\cF$, the \emph{tensor product DG local system $\cF\otimes L$} is defined degree-wise as ($\cF\otimes L)_k=\cF_k\otimes L$, and chains $c$ on $\Omega X$ act on $\cF\otimes L$ by $(\alpha\otimes \ell)\cdot c= (\alpha\otimes \ell)\cdot \Delta_*c$, where $\Delta_*:C_*(\Omega X)\to C_*(\Omega X)\otimes C_*(\Omega X)$ is the diagonal map. Explicitly, given  a cube $\sigma:I^n\to \Omega X$, $I=[0,1]$ we have $\Delta_*\sigma = \sum_{k=0}^n\sigma|_{I^k\times\{0\}}\otimes \sigma|_{\{0\}\times I^{n-k}}$. Since chains of positive degree act trivially on $L$, it follows that a cube $\sigma$ acts on the tensor product by 
\begin{align*}
(\alpha\otimes \ell)\cdot \sigma & =(\alpha\cdot \sigma)\otimes (\ell\cdot \hat\rho^L(\sigma_0)) \\
& = (\alpha\cdot  \hat\rho^L(\sigma_0) \sigma)\otimes \ell\\
& = (\alpha\cdot \Phi^L(\sigma))\otimes\ell.
\end{align*}
By linearity, the same formula expresses the action of $C_*(\Omega X)$ on $\cF\otimes L$, i.e.,  
$$
(\alpha\otimes \ell)\cdot c = (\alpha\cdot \Phi^L(c))\otimes \ell,\qquad c\in C_*(\Omega X). 
$$
 This shows in particular that we have an isomorphism of $C_*(\Omega X)$-modules
$$
\cF\otimes L\simeq (\Phi^L)^*\cF.
$$

\begin{proposition} \label{prop:FotimesL-mL}
Let $\Xi$ be enriched Morse data with corresponding twisting cocycle $(m_{x,y})$, let $\cF$ be a DG local system and let $L$ be a free rank $1$ local system supported in degree $0$.
\begin{enumerate}
\item The formula   
$$
m^L_{x,y}=\Phi^L(m_{x,y})
$$
defines a twisting cocycle, i.e., the following equation holds: 
\begin{equation} \label{eq:DGMorsedifferential-L}
\p m^L_{x,y}=\sum_z (-1)^{|x|-|z|} m^L_{x,z}\cdot m^L_{z,y}.
\end{equation}
\item There is a chain isomorphism between $C_*(X,\Xi;\cF\otimes L)$ and the complex $\cF\otimes \Crit(f)$ endowed with the differential $\p^L$ given by~\eqref{eq:DGdifferrential} with twisting cocycle $(m^L_{x,y})$, i.e., 
$$
\p^L(\alpha\otimes x)=\p\alpha\otimes x +(-1)^{|\alpha|} \sum_y \alpha\cdot m^L_{x,y}\otimes y.
$$
\end{enumerate}
\end{proposition}

\begin{proof}
1. This is a direct consequence of the fact that $\Phi^L$ is a DG algebra map.

 2. The differential $\p^L$ expresses the differential on $C_*(X,\Xi;(\Phi^L)^*\cF)$. We then conclude using $\cF\otimes L\simeq (\Phi^L)^*\cF$. Compare with~\cite[Remark~4.6]{BDHO}.
\end{proof}

\subsection{Classes living over the fundamental class} \label{sec:classes_over_fund_class}

In this section we study some special classes in the homology of a finite dimensional manifold $Y$ with coefficients in a DG local system $\calf$. They may be defined by the property that the chains representing  them  contain at least one maximum of the Morse function. Their existence will be essential in the proofs of the theorems stated in \S\ref{sec:intro-finiteness-HZ}. These proofs will be given in \S\ref{sec:HZ}.

Throughout this section $Y^n$ is a smooth orientable manifold, possibly with boundary, $\calf$ is a right $C_*(\Omega Y)$-module and $j:\star\rightarrow Y$ is the inclusion of the basepoint. 

\begin{definition}\label{def:living-above-fundamental} We say that a class $\alpha \in H_*(Y,\p Y;\calf)$ \emph{lives above the fundamental class of $Y$} if its image through the shriek map 
$$j_! : H_{*}(Y,\p Y;\calf) \ri H_{*-n}(\star; j^*\calf)\simeq H_{*-n}(\calf)$$
is nonzero.  
\end{definition} 

This terminology is motivated by the fact that, for $\calf =\Z$ the trivial local system, these classes are the nonzero multiples of the fundamental class rel boundary $[Y,\p Y]$.

The above definition applies to fibrations, and we find it important to spell it out in that case. Let 
$$
F\hookrightarrow E\to Y
$$ 
be a fibration, let $\p E=E|_{\p Y}$ and recall 
that $H_*(E,\p E)\simeq H_*(Y,\p Y;C_*(F))$. 
In particular we have a shriek map 
$$
j_! : H_*(E,\p E)\simeq H_{*}(Y,\p Y;C_*(F)) \ri H_{*-n}(\star; j^*C_*(F))\simeq H_{*-n}(F).
$$

\xmpl\label{def:living-above-fundamental-fibration}  In the above situation, a class $\alpha \in H_*(E,\p E)$ lives above the fundamental class of $Y$ if its image through the shriek map 
$$j_! : H_*(E,\p E)\to H_{*-n}(F)$$
is nonzero. 
\lpmx

The next statement provides equivalent characterizations of classes that live above the fundamental class. 

\begin{proposition}\label{prop:living-fundamental} The following are equivalent: 
\begin{itemize} 
\item[(i)] The class $\alpha \in H_k(Y,\p Y;\calf)$ lives above the fundamental class. 
\item[(ii)] For any Morse complex with DG coefficients defining $H_*(Y, \p Y; \calf)$, any chain representative of the class $\alpha$ 
has the form 
\begin{equation}\label{eq:Morse-over-fundamental}
\beta\otimes M + \sum_{x_i\neq M}\beta_i \otimes x_i,
\end{equation}
with $M$ a local maximum of the Morse function and $[\beta]\neq 0\in H_{k-n}(\calf)$. (Note that $\beta\in C_{k-n}(\calf)$ is automatically a cycle.)
\item[(iii)] The class $\alpha$ can be represented in some Morse complex with DG coefficients defining $H_*(Y, \p Y; \calf)$ by a cycle of the form \eqref{eq:Morse-over-fundamental}.
\item[(iv)] The projection of $\alpha$ on the  rightmost column $E_{n,k-n}^\infty$ of the limit of the spectral sequence  is nonzero. 
  \end{itemize}
In particular, there exists a class $\alpha\in H_k(Y,\p Y;\calf)$ that lives above the fundamental class if and only if $E_{n, k-n}^\infty\neq 0$.
\end{proposition} 

\begin{proof} $(iii)\Rightarrow (i)$ The Morse function on $Y$ that underlies a DG Morse complex that computes $H_*(Y,\p Y;\cF)$ has  negative gradient pointing outwards  along $\p Y$.  Thus $f$ has at least one maximum, which we denote $M$.  W.l.o.g.\@ we may suppose that the basepoint $\star$ is located at $M$. By definition~\cite[\S9.2]{BDHO}, the shriek morphism at the level of complexes acts by 
$$ j_! (\beta\otimes M )\, =\,  \beta$$
and $$j_!(\beta_x\otimes x)\, =\, 0$$ for $x\neq M$.
The implication  is then straightforward.

\medskip $(i)\Rightarrow (ii)$ Arguing by contradiction, suppose that for some Morse function $f$, the class $\alpha$  can be represented as $\sum_{|x_i|<n}\beta_i\otimes x_i$, i.e., with no local maximum. We may suppose again that  $f$ has a maximum in $\star$. Considering the definition of $j_!$  given in the proof of  $(iii)\Rightarrow (i)$ we get  $j_!(\alpha)=0$, a contradiction. 

\medskip 
$(ii)\Rightarrow (iii)$ Obvious.

\medskip
$(i)\Leftrightarrow (iv)$ We  use a Morse function  with a unique maximum $M=\star$ to construct the enriched complex of $Y$. (The spectral sequence is independent on the choice of Morse function starting with the second page~\cite[\S6.2]{BDHO}.) Denote by $\pi$ the projection of the enriched homology of degree $k$ on $E_{n,k-n}^\infty$.  By the naturality of the spectral sequence~\cite[Remark~9.3]{BDHO} we have the following commutative diagram:
\begin{equation*}\label{welldefined-shriek2} 
  \xymatrix{
     H_k(Y,\p Y; \calf)\ar[r]^-{\pi_Y}  \ar[d]^{j_!}& E_{n, k-n}^\infty(Y,\p Y;\calf)\ar[d]^{j_{n,k-n,!}^\infty}
     \\
    H_{k-n} (\star; j^*\calf)\ar[r]^-{\pi_\star}  & E_{0, k-n}^\infty(\star;j^*\calf) }
\end{equation*}

We denote by $\pi_Y$ and $\pi_\star$ the two projections, in order to distinguish them.)
We have  
  $H_{k-n}(\star; j^*\calf)= H_{k-n}(\calf)$ and the spectral sequence for $H_*(\star; j^*\calf)$ is supported in the $0$-th column, where it equals
$E_{0,*}^\infty(\star;j^*\calf)= E_{0, *}^1(\star;j^*\calf)=H_{*}(\calf)$.
  Under this identification the projection $\pi_\star$ is the identity and therefore the commuting diagram above means that 
$$
  j_{n,k-n, !}^\infty\circ \pi_Y \, =\,  j_! \, .
$$
  
  Remark that 
$$
E^1_{n, k-n}(Y,\p Y; \calf)  = \Z\langle M\rangle\otimes H_{k-n}(\calf)
$$ and  
$$
E^1_{0, k-n}(\star; j^*\calf)  = \Z\langle \star \rangle\otimes H_{k-n}(\calf).
$$ 
  The map  $j_{n,k-n, !}^1$ between the two is the obvious identification.  Note that  $E^\infty_{n, k-n}(Y,\p Y; \calf)$ is a submodule of  $E^1_{n, k-n}(Y,\p Y; \calf)$ since the $n$-th column of the spectral sequence is the rightmost one. By considering the commutative diagram 
  $$
  \xymatrix{
  E_{n,k-n}^\infty(Y,\p Y;\calf) \ar[d]^{j_{n,k-n,!}^\infty} \ar@{^(->}[r] & 
  E_{n,k-n}^1(Y,\p Y;\calf) 
  \ar@{=}[d]^{j_{n,k-n,!}^1} \\
  E_{0,k-n}^\infty(\star;j^*\calf)\ar@{=}[r]& 
  E_{0,k-n}^1(\star;j^*\calf)
  }
  $$
  we see that the map $j_{n,k-n,!}^\infty$
  corresponds to the inclusion under these identifications, and is therefore injective. This implies that $j_!(\alpha)\neq 0$ if and only if $\pi_Y(\alpha)\neq 0$, completing the proof of the proposition. 
\end{proof}

We now turn to equivalent characterizations of classes  that live above the fundamental class in the case of fibrations. Let $F\hookrightarrow E\to Y$ be a fibration with base a connected oriented compact $n$-manifold, possibly with boundary.

\begin{proposition}\label{prop:living-fundamental-fibrations} The following are equivalent: 
\begin{itemize}
\item[(i)] the class $\alpha\in H_*(E,\p E)$ lives above the fundamental class. 
\item[(ii)] the class $\alpha$ has nonzero projection onto the rightmost column $E^\infty_{n,*}$ of the limit of the Leray-Serre spectral sequence of the fibration $F\hookrightarrow (E,\p E)\to (Y,\p Y)$. 
\item [(iii)] let $Y_{n-1}$ be the $(n-1)$-skeleton of $Y$ for some CW-decomposition relative to the boundary $\p Y$, with a single $n$-dimensional cell. The class $\alpha\in H_*(E,\p E)$ has nonzero image under the map 
$$
H_*(E)\to H_*(E,E|_{Y_{n-1}}).
$$
\end{itemize} 
\end{proposition}

\begin{remark}
Condition~(ii) implies in particular that there exists a class on the $n$-th column $E_{n,*}^2$ of the Leray-Serre spectral sequence of the fibration $F\hookrightarrow (E,\p E)\to (Y,\p Y)$ that survives to $E^\infty$.  
\end{remark}

\begin{remark} Because $E^\infty$ is the graded object associated to a filtration on $(E,\p E)$, all the elements of the rightmost column $E^\infty_{n,*}$ necessarily arise from homology classes in $H_*(E,\p E)$\end{remark}

\begin{proof} $(i)\Leftrightarrow (ii)$ This is the same as (i)$\Leftrightarrow$(iv) in Proposition~\ref{prop:living-fundamental}. 

$(i)\Leftrightarrow (iii)$ This is a trivial consequence of the following lemma from \cite[Lemma 9.18]{BDHO}

\begin{lemma}\label{lem:shriek-point} 
We have a canonical isomorphism $H_*(E,E|_{Y_{n-1}})\simeq H_{*-n}(F)$, and the composition
$$
H_*(E)\to H_*(E,E|_{Y_{n-1}})\stackrel{\simeq}{\longrightarrow} H_{*-n}(F)
$$
is the shriek map induced by the inclusion $\star \hookrightarrow Y$. 
\end{lemma}

The proof of Proposition~\ref{prop:living-fundamental-fibrations} is now complete. \end{proof}

\rmk\label{rmk:shriek-fundamental-class} The proof of the lemma above also shows that, if $E$ is finite dimensional and orientable, then the DG shriek map of the inclusion $\star\hookrightarrow Y$ is precisely the usual shriek map of the inclusion $F\hookrightarrow E$. In particular it maps the fundamental class $[E,\p E]$ to the fundamental class $[F]$. \kmr

\begin{remark} It follows from Lemma~\ref{lem:shriek-point} that the composition
$$
\xymatrix{
H_{q+n}(E,\p E) \ar@{->>}[r] & E^\infty_{n,q} \ar@{^(->}[r] & \ \dots \ \ar@{^(->}[r] & E^2_{n,q}=H_q(F)
}
$$
is identified with the shriek map. This general property of the Leray-Serre spectral sequence can be used in order to give an alternative proof of the equivalence (i)$\Leftrightarrow$(ii) in Proposition~\ref{prop:living-fundamental-fibrations}. 
\end{remark}

We finish this section with the following result:

\begin{proposition}\label{prop:shriek-commute} Let $Y^n$ be a closed orientable manifold endowed with a basepoint $\star$ and two Hurewicz fibrations $F\ri E\stackrel{\pi_E}{\longrightarrow} Y$ and $F'\ri E'\stackrel{\pi_{E'}}{\longrightarrow} Y$.  Consider a fibration map $\sigma: E\ri E'$ (meaning that $\pi_{E'}\circ\sigma=\pi_{E}$) and let $j:\star \ri Y$ be the inclusion of the basepoint. Then the following diagram is commutative: 
$$\xymatrix{H_*(E)\ar[r]^-{\sigma_*}\ar[d]^-{j_!}&H_*(E')\ar[d]^-{j_!}\\
H_{*-n}(F_\star)\ar[r]^-{\sigma_*}&H_{*-n}(F_\star')}$$

\end{proposition}

\begin{corollary}\label{cor:shriek-commute} Using the notation of Proposition~\ref{prop:shriek-commute}, if $E$ is a finite dimensional closed  orientable manifold and if $\sigma: F\ri F'$ satisfies $\sigma_*[F]\neq 0$, then $\sigma_*[E]$ lives over the fundamental class in the fibration $E'\ri Y$. 
\end{corollary} 
\begin{proof}[Proof of the corollary assuming Proposition~\ref{prop:shriek-commute}]

We have $j_![E]=[F]$ by Remark~\ref{rmk:shriek-fundamental-class}, and by Proposition~\ref{prop:shriek-commute} we get $j_!\sigma_*[E]= \sigma_*j_![E]=\sigma_*[F]\neq 0.$
\end{proof}

\begin{proof}[Proof of Proposition~\ref{prop:shriek-commute}]
We consider a CW-decomposition with a single $n$-dimensional cell $M$ whose center is the basepoint $\star$. Let $Y_{n-1}$ be the $(n-1)$-skeleton, and let $D= Y\setminus Y_{n-1}$. We may suppose that $D$ is an embedded $n$-ball centered at $\star$ by slightly thickening  $Y_{n-1}$ with a collar of the boundary of the $n$-cell. We will use the description of $j_!$ given in (the proof of) 
Lemma 9.18 in \cite{BDHO}, namely as the composition of four maps as in the diagram below.
The fibration map $\sigma$ yields the  following diagram, in which the second vertical maps are excision isomorphisms: 
$$\xymatrix{ H_*(E)\ar[d]\ar[r]^-{\sigma_*}\ar@/_4,7pc/[dddd]_-{j_!}&H_*(E')\ar[d]\ar@/^4,7pc/[dddd]^-{j_!}\\
H_*(E, E|_{Y_{n-1}})\ar[d]^-{\mathrm{exc}}\ar[r]^-{\sigma_*}&H_*(E', E'|_{Y_{n-1}})\ar[d]^-{\mathrm{exc}}\\
H_*(E|_D,E|_{\p D})\ar[d]^-{\simeq}\ar[r]^-{\sigma_*}&H_*(E'|_D,E'|_{\p D}\ar[d]^-{\simeq})\\
H_*(F_\star\times(D,\p D))\ar[d]^-{\simeq}\ar[r]^-{(\sigma\times\id)_*}&H_*(F_\star'\times (D, \p D)\ar[d]^-{\simeq})\\
 H_{*-n}(F_\star)\ar[r]^{\sigma_{*}}&H_{*-n}(F'_\star)}$$
 
 Let us show that this diagram is commutative. This is obvious for the upper  and for the second upper  square, as well as for the lower one. As for the third one, recall that the vertical isomorphisms are the respective inverses of the maps induced in homology by $ \Psi^E: (F\times (D,\p D))\ri (E|_D, E|_{\p D})$ and $\Psi^{E'}: (F'\times (D,\p D))\ri (E'|_{D},E'|_{\p D})$ that are defined as follows. For any $a\in D$ we denote by $\gamma^a$ the radial path starting at $\star$ and finishing at $a$; so $\gamma^\star$ is the constant path. Let $\Phi^E: \calp Y\times E\ri E $ be a lifting function for $E\ri Y$. Then $\Psi^E$ is defined by the formula   $\Psi^E(a, f)= \Phi^E(\gamma^a, f)$. The definition of $\Psi^{E'}$ is analogous,  using a lifting function $\Phi^{E'}$ for $E'\to Y$. We prove 
 \begin{lemma} The diagram $$\xymatrix {F_\star\times(D, \p D) \ar[d]^-{\Psi^E} \ar[r]^-{\sigma\times \id}&F'_\star\times (D, \p D)\ar[d]^-{\Psi^{E'}}\\
 (E|_D,E|_{\p D})\ar[r]^-{\sigma}&(E'|_D, E'|_{\p D})}$$
 is commutative up to homotopy.
 
 \end{lemma}
 \begin{proof}[Proof of the lemma] We interpolate between the two lifting functions for $t\in [0,1]$  by using $\Phi^E$ on the radial path $\gamma^{ta}$ in $D$ between $\star$ and $ta$ and the lifting function $\Phi^{E'}$ on the radial path between $ta$ and $a$. Denote the latter by $\delta^{[ta,a]}$. More precisely we define the homotopy  $H: [0,1]\times  F_\star\times (D,\p D)\ri (E'|_D, E'|_ {\p D})$ by 
 $$H(t, f, a) \, =\, \Phi^{E'}(\delta^{[ta,a]}, \sigma(\Phi^E(\gamma^{ta},f))).$$
 We see that $H_0= \Psi^{E'}\circ (\sigma\times\id)$ and $H_1= \sigma\circ\Psi^E$ and the lemma is proved. 
 \end{proof} 
 Since $\Psi^E$ and $\Psi^{E'}$ induce isomorphisms in homology we infer that the third square of the diagram above is commutative and therefore that the whole diagram is commutative, proving thus Proposition~\ref{prop:shriek-commute}. 
\end{proof}

\section{Floer theory with DG coefficients} \label{sec:DGFloer}

In this section we construct the DG Floer toolset announced in~\S\ref{sec:intro-DGFloer}, and we prove all the results announced in that subsection, in particular Theorem~\ref{thm:FH_for_fibrations-intro}.

\subsection{Orientation conventions} \label{sec:orientations}

We follow the orientation conventions of Abouzaid~\cite[\S9.7]{Abouzaid-cotangent}. 

{\bf Orientation lines.} We call a $\Z$-graded free abelian group of rank 1 an \emph{orientation line}. An isomorphism of orientation lines is a graded isomorphism of the underlying free rank 1 abelian groups. In this setting we have the following canonical isomorphisms.
\begin{itemize}
\item Given two orientation lines $\ell_1,\ell_2$, their tensor product $\ell_1\otimes \ell_2$ is by definition supported in degree $\deg \ell_1 +\deg \ell_2$. The abelian group $\Z$ 
supported in degree $0$ is a neutral element for the tensor product.  
We have a canonical \emph{twist isomorphism} 
\begin{equation} \label{eq:ori_lines_twist}
\ell_1\otimes \ell_2\stackrel \simeq \longrightarrow \ell_2\otimes \ell_1
\end{equation}
given by $v_1\otimes v_2\mapsto (-1)^{\deg \ell_1 \cdot \deg \ell_2} v_2\otimes v_1$. 
\item  Given an orientation line $\ell$, we denote by $\ell^{-1}=\Hom_\Z(\ell,\Z)$ its dual, supported in degree $\deg \ell^{-1}=-\deg \ell$. There is a canonical \emph{evaluation isomorphism} 
\begin{equation} \label{eq:ori_lines_evaluation}
\ell^{-1}\otimes \ell\stackrel \simeq \longrightarrow \Z,\qquad \alpha\otimes v\mapsto \alpha(v). 
\end{equation}
Thus $\ell^{-1}$ plays the role of an inverse for $\ell$ with respect to the tensor product. 
There is a canonical isomorphism 
\begin{equation} \label{eq:ori_lines_inverse_of_product}
\ell_2^{-1}\otimes\ell_1^{-1}\simeq (\ell_1\otimes\ell_2)^{-1}, 
\end{equation}  
where a tensor product $\alpha_2\otimes\alpha_1\in\ell_2^{-1}\otimes\ell_1^{-1}$ is seen as an element of $(\ell_1\otimes \ell_2)^{-1}$ via $\langle \alpha_2\otimes \alpha_1,v_1\otimes v_2\rangle = \alpha_2(v_2) \alpha_1(v_1)$. Also, there is a canonical isomorphism 
\begin{equation} \label{eq:ori_lines_inverse_of_inverse}
\ell\simeq \left(\ell^{-1}\right)^{-1}
\end{equation}
given by $v\mapsto \left(\alpha\mapsto \alpha(v)\right)$.
\end{itemize}

{\bf Oriented lines.} To orient an orientation line means to choose one of its two generators, called ``positive". An isomorphism of orientation lines  has a well-defined sign $\pm 1$ if its source and target are oriented.

Given two oriented orientation lines $\ell_1$ and $\ell_2$, we induce canonically an orientation on their tensor product $\ell_1\otimes \ell_2$ as follows: the positive generator is $v_1\otimes v_2$, where $v_1\in\ell_1$ and $v_2\in\ell_2$ are the positive generators. With this convention, the following \emph{Koszul sign rule} holds:
\begin{center}
{\it The twist isomorphism~\eqref{eq:ori_lines_twist} has sign $(-1)^{\deg\ell_1\cdot \deg\ell_2}$.}
\end{center}

The neutral element $\Z$ for the tensor product is canonically oriented by its generator $1$. As a consequence, given an orientation on $\ell$ we induce canonically an orientation on $\ell^{-1}$ by requiring that the evaluation isomorphism~\eqref{eq:ori_lines_evaluation} be orientation preserving. In other words, a generator $\alpha:\ell\to\Z$ is positive if and only if $\alpha(v)=1$ for the positive generator $v\in\ell$. (Note that the twist isomorphism $\ell^{-1}\otimes \ell\simeq \ell\otimes \ell^{-1}$ has sign $(-1)^{\deg \ell}$.)

These conventions for orienting the tensor product and the inverse imply that, given oriented lines $\ell_1$, $\ell_2$, and $\ell$, the isomorphisms~\eqref{eq:ori_lines_inverse_of_product} and~\eqref{eq:ori_lines_inverse_of_inverse} are orientation preserving.

{\bf Orientation lines of real vector spaces.} To a graded $1$-dimensional real vector space $L$ one associates canonically an orientation line $|L|$, defined as the free rank $1$ abelian group generated by the two orientations of $L$ modulo the relation that their sum vanishes, supported in degree $\deg L$. 

Given a finite dimensional real vector space $V$, its \emph{determinant line} is the graded $1$-dimensional real vector space $\det V=\Lambda^{\max}V$ supported in degree $\dim V$. We denote $|V|=|\det V|$ the corresponding orientation line. Since an orientation of $\det V$ is canonically equivalent to an orientation of $V$, we can see $|V|$ as being the free rank 1 abelian group generated by the two orientations of $V$ modulo the relation that their sum vanishes. We have a canonical isomorphism $|0|\simeq \Z$, which sends the positive orientation $+1$ of the $0$-dimensional vector space to $1$.  

{\bf Orientation local systems.} Given a manifold $M$ we denote $|M|$ the local system of graded free rank $1$ abelian groups whose fiber at a point $p$ is the orientation line $|T_pM|$. We call it the \emph{orientation local system of $M$}. Given a real vector bundle $E\to M$, we denote by $|E^{\text{fiber}}|$ the local system on $M$ whose fiber at $p$ is the orientation line $|E_p|$ of the fiber $E_p$. Yet another local system of interest is $|E||_M$, the restriction of $|E|$ to $M$, whose fiber at $p$ is the orientation line of the total tangent space $|T_pE|$.

{\bf Orientation conventions.} 

{\it Direct sum.} Given two finite dimensional real vector spaces $V$ and $W$, we induce an orientation on $V\oplus W$ from orientations of $V$ and $W$ as follows: given positive bases $(v_1,\dots,v_m)$ of $V$ and $(w_1,\dots,w_n)$ of $W$, the basis $(v_1,\dots,v_m,w_1,\dots,w_n)$ is positive. At the level of orientation lines, this is phrased as a canonical isomorphism 
\begin{equation} \label{eq:ori_lines_direct_sum}
|V|\otimes |W| \simeq |V\oplus W|. 
\end{equation}

{\it Short exact sequences.} Given a short exact sequence of finite dimensional real vector spaces 
\begin{equation} 
0\to A\to B\to C\to 0, 
\end{equation}
we induce an orientation on $B$ from orientations of $A$ and $C$ as follows: given positive bases $(a_1,\dots,a_m)$ of $A$ and $(c_1,\dots,c_n)$ of $C$, we choose lifts $\tilde c_1,\dots,\tilde c_n$ for $c_1,\dots,c_n$ and declare the basis $(a_1,\dots,a_m,\tilde c_1,\dots,\tilde c_n)$ of $B$ to be positive. This yields a canonical isomorphism 
\begin{equation} \label{eq:ori_lines_short_exact_sequence}
|A|\otimes |C|\simeq |B|. 
\end{equation}
This isomorphism can be used to induce an orientation on any of the factors $A$, $B$, $C$ from orientations of the two other factors. Note also that this orientation rule is equivalent to the one for the direct sum under the convention that a choice of splitting $C\to B$ gives rise to an isomorphism $B\simeq A\oplus C$.

{\it Fiber products.} Let $f_i : V_i\to V$, $i = 1,2$ be linear maps between oriented vector spaces so that $f:V_1\oplus V_2\to V$, $(v_1,v_2)\mapsto f_1(v_1)-f_2(v_2)$ is surjective. The orientation on the fiber product $V_1 \times_V V_2 = \ker f$ is defined from the short exact sequence $0\to\ker f \to V_1\oplus V_2\stackrel{f}\to V\to 0$.

{\it Vector bundles.} Let $E\stackrel \pi\longrightarrow M$ be a real vector bundle over a manifold $M$. Denote the fiber at a point $p$ by $E_p$ and consider the canonical short exact sequence $0\to E_p\to T_pE \stackrel {d\pi} \longrightarrow T_pM\to 0$. This gives rise to a canonical isomorphism 
$$
 |E_p|\otimes |T_pM|\simeq |T_p E|. 
$$ 
The previous rule for orienting $T_pE$ from orientations of $E_p$ and $T_pM$ can be succinctly rephrased as {\it ``fiber first, base second"}. Pasting these canonical isomorphisms together we get the isomorphism of graded free rank 1 local systems on $M$ 
$$
 |E^{\text{fiber}}|\otimes |M|\simeq |E||_M.
$$

{\it Manifolds with boundary.} Given a manifold $M$ with boundary, consider the normal bundle $\nu\to \p M$ along the boundary. The previous recipe provides an isomorphism $|\nu^{\text{fiber}}|\otimes |\p M|\simeq |M||_{\p M} $ of local systems on $\p M$. 

In this situation the normal bundle is trivial. {\it We trivialize the normal bundle along $\p M$ using an outward pointing vector field along the boundary $\nu^{\text{out}}$.} This determines a canonical isomorphism $|\R|\otimes |\p M|\simeq |M||_{\p M}$. 

Explicitly, we split $T_pM$ as $T_pM\simeq \R\nu_p^{\text{out}}\oplus T_p\p M$, or equivalently we project $T_pM$ onto $T_p\p M$ with kernel $\R\nu_p^{\text{out}}$, leading to the exact sequence $0\to \R\nu_p^{\text{out}}\to T_pM\to T_p\p M\to 0$. This induces a canonical isomorphism 
$$
|\R\nu_p^{\text{out}}|\otimes |T_p\p M|\simeq |T_pM|. 
$$
The canonical isomorphism $\R\nu_p^{\text{out}}\simeq \R$ induces $|\R|\otimes |\p M|\simeq |M||_{\p M}$.

{\it Stabilized Fredholm operators.} Let $D:X\to Y$ be a Fredholm operator between Banach spaces, meaning that the image is closed and the kernel and cokernel are finite dimensional. We define the \emph{determinant line of $D$} to be the graded $1$-dimensional vector space 
$$
\det D = (\det \coker D)^\vee\otimes \det \ker D.
$$ 
Thus $\det D$ is supported in degree $\ind(D)= \dim \ker D - \dim\coker D$. 

Given a finite dimensional vector space $V$ and a linear map $\Phi:V\to Y$, the \emph{stabilization of $D$ (by $V$ and $\Phi$)} is the Fredholm operator 
$$
D^V:V\oplus X\to Y, \qquad (v,x)\mapsto Dx + \Phi v. 
$$ 
In this situation there is a canonical isomorphism (which depends on $\Phi$)
$$
\det D^V\simeq \det V\otimes \det D.
$$
We infer a canonical isomorphism of orientation lines 
$$
|\det D^V|\simeq |V|\otimes |\det D|.
$$

{\it Gluing of stabilized Cauchy-Riemann operators.} Let $D_1,D_2:X\to Y$ be Cauchy-Riemann operators on the cylinder, with the same asymptotic behavior at $-\infty$ for $D_1$ and at $+\infty$ for $D_2$. Denote $D_1\# D_2$ the glued operator for a large enough value of the gluing parameter. Linear gluing theory~\cite{FH-coherent} provides a canonical isomorphism that preserves coherent orientations
$$
|\det D_1| \otimes |\det D_2| \simeq |\det D_1\# D_2|.
$$
Note that one does not need to assume that the operators are surjective. 

Let now $D_1^{V_1}:V_1\oplus X\to Y$ and $D_2^{V_2}:V_2\oplus X\to Y$ be stabilizations of $D_1$ and $D_2$, and denote $D_1^{V_1}\# D_2^{V_2}=(D_1\# D_2)^{V_1\oplus V_2}$. We have isomorphisms $|\det D_i^{V_i}|\simeq |V_i|\otimes |\det D_i|$ for $i=1,2$ and $|\det D_1^{V_1}\# D_2^{V_2}|\simeq |V_1\oplus V_2|\otimes |\det D_1\# D_2|$. Moreover, if $D_1^{V_1}$ and $D_2^{V_2}$ are surjective, then so is $D_1^{V_1}\# D_2^{V_2}$. 

We obtain by composition a canonical isomorphism 
\begin{align*}
|\det D_1^{V_1}|\otimes |\det D_2^{V_2}|& \simeq |V_1| \otimes |\det D_1|\otimes |V_2|\otimes |\det D_2|\\
& \simeq (-1)^{\dim V_2\cdot \ind D_1}|V_1|\otimes |V_2|\otimes |\det D_1|\otimes |\det D_2|\\
& \simeq (-1)^{\dim V_2\cdot \ind D_1}|V_1\oplus V_2|\otimes |\det D_1\# D_2|\\
& \simeq (-1)^{\dim V_2\cdot \ind D_1} |\det D_1^{V_1}\# D_2^{V_2}|.
\end{align*}

Given a system of coherent orientations as in~\cite{FH-coherent}, and given orientations of $V_1$ and $V_2$, all the terms above are canonically oriented and the equations need to be understood as follows.

\begin{lemma} \label{lem:gluing-stabilized-CR}
The gluing of stabilized Cauchy-Riemann operators $D_1^{V_1}$ and $D_2^{V_2}$ as above changes orientations by $(-1)^{\dim V_2\cdot \ind D_1}$.  \qed
\end{lemma}

\subsection{Conventions for Floer theory} \label{sec:conventions_Floer}

We place ourselves in the following setup. We will consider symplectic manifolds $(X,\omega)$ that are \emph{symplectically atoroidal}, meaning that for any continuous map $f:T^2\to X$ with source the $2$-torus $T^2$ we have $f^*[\omega]=f^*(2c_1)=0$, where $c_1=c_1(TX)$ is the first Chern class with respect to some compatible almost complex structure.
\footnote{The relevance of the conditions $f^*[\omega]=0$ and $f^*(2c_1)=0$ for maps $f:T^2\to X$ is the following. The first condition implies that the Hamiltonian action functional is univalued on the free loop space $\cL X$. The second condition allows to define a $\Z$-grading on Floer homology groups in all free homotopy classes by choosing unitary trivializations of the tangent bundle along reference loops, one for each free homotopy class, and using those to induce trivializations along all the Hamiltonian $1$-periodic orbits. 

If one restricts to contractible loops, the condition $f^*[\omega]=f^*(2c_1)=0$ is requested only for maps $f:\SS^2\to X$ defined on the $2$-sphere. We then say that $X$ is \emph{symplectically aspherical}.

If one wishes to study \emph{operations} on Floer homology groups, $\Z$-gradings need to be chosen consistently over all free homotopy classes. The condition $f^*(2c_1)=0$ is then not sufficient anymore and needs to be replaced by $c_1=0$. In the specific case of cotangent bundles we have $c_1=0$ if the base is orientable, and $2c_1=0$ in the general case. This weaker condition is sufficient in that setup in order to define graded operations~\cite{Abouzaid-cotangent}.} 
 
In case the manifold $X$ is not compact (typically for us it will be the symplectic completion of a Liouville domain), we will need to restrict the class of Hamiltonians and almost complex structures under consideration in order to ensure \emph{a priori} $C^0$-bounds for Floer trajectories. We postpone the details to~\S\ref{sec:symplectic_homology} and, in this section and the following  ones, we will simply speak of \emph{admissible} Hamiltonians, almost complex structures, homotopies of Hamiltonians and almost complex structures, and homotopies of homotopies of such. 

Let $S^1=\R/\Z$. We will consider $1$-periodic time-dependent Hamiltonians $H:S^1\times X\to \R$, $H=(H_t)$. The \emph{Hamiltonian vector field} $X_H^t$ is defined by $\omega(X_H^t,\cdot)=-dH_t$, and we denote $\Per(H)$ the set of $1$-periodic orbits of $X_H^t$. We say that $H$ is \emph{nondegenerate} if every $1$-periodic orbit is nondegenerate (and therefore isolated). 

Denote by $\cL X$ the space of smooth free loops on $X$, i.e., smooth maps $f:S^1\to X$.\footnote{Depending on the context one sometimes needs to work with loops of lower regularity, for example $W^{1,p}$ or $C^0$, but this is irrelevant for the topology of $\cL X$ because the resulting spaces of loops are all homotopy equivalent~\cite{Palais1968}, see also~\cite[\S2]{Chataur-Oancea}.} For each free homotopy class of loops $a$ we denote by $\cL_a X$ the connected component of $\cL X$ formed by the loops in the class $a$, and we choose for each such component a reference loop $\gamma_a\in\cL_a X$. If $a=0$, the class of contractible loops, we choose $\gamma_0$ to be constant. We denote by $\Per^a(H)$ the set of $1$-periodic orbits of $H$ in the free homotopy class $a$. Given a free loop $\gamma$ we denote $[\gamma]$ its free homotopy class. 

We define the action functional $\cA_H: \cL X\ri \R$ by 
$$ 
\cA_H(\gamma)\, =\, \int_{S^1\times[0,1]}\bar\gamma^*\omega - \int_{S^1} H(t,\gamma(t))\, dt,
$$
where $\bar\gamma:S^1\times[0,1]\to X$ is a homotopy between $\gamma=\bar\gamma(\cdot,1)$ and the reference loop $\gamma_a=\bar\gamma(\cdot,0)$ with $a=[\gamma]$. By (half of) the atoroidality condition, the value of $\cA_H$ does not depend on the choice of homotopy $\bar\gamma$. The set of critical points of $\cA_H$ is $\crit(\cA_H)=  \Per(H)$.

We will consider $1$-periodic time-dependent compatible almost complex structures $J=(J_t)$, $t\in S^1$. The gradient of $\cA_H$ with respect to the $L^2$-metric on $\cL X$ determined by $J$ is $\nabla \cA_H(\gamma)=-J_t(\gamma(t))(\dot\gamma(t)-X_H^t(\gamma(t)))$ and \emph{positive} gradient trajectories $u:\R\to\cL X$ for $\cA_H$ are equivalent to maps $u:\R\times S^1\to X$ solving the \emph{Floer equation} $\p_s u + J_t(u)(\p_t u-X_H^t(u))=0$. 

As mentioned at the beginning of this subsection, if the manifold $X$ is not compact we need to impose some additional admissibility conditions on $H$ and $J$, to be specified later in~\S\ref{sec:symplectic_homology}. The set of admissibility conditions is void if the manifold $X$ is closed.

An admissible pair $(H,J)$ is said to be \emph{regular} if, for any choice $x^\pm\in\Per(H)$, the {\it space of connecting Floer trajectories from $x^+$ to $x^-$} 
\begin{align*}
\wh \cM(x^+;x^-) & = \wh \cM(x^+;x^-;H,J) \\
&  = \{u:\R\times S^1\to X\, :  \, \p_s u  + J_t(u)(\p_t u -X_H^t(u))=0,\\
& \qquad \qquad \lim_{s\to \pm\infty} u(s,\cdot)=x^\pm\}
\end{align*}
is cut out transversally, i.e., the linearization of the Floer equation at any $u\in \wh \cM(x^+;x^-)$ is surjective. These spaces of Floer trajectories are smooth manifolds of dimension $|x^+|-|x^-|$, with $|x|=\CZ(x)$ (the Conley-Zehnder index of $x$). 
 Here the Conley-Zehnder index is computed with respect to trivializations of the tangent bundle over the elements of $\Per(H)$ constructed as follows: we choose trivializations of the tangent bundle over the reference loops $\gamma_a$, and then we induce trivializations over the elements of $\Per^a(H)$ by choosing homotopies to $\gamma_a$ and extending the trivialization at $\gamma_a$ along these homotopies.  

Since the Floer equation is equivalent to the positive gradient equation, we find that $\cA_H(x^+)-\cA_H(x^-)=\int_{\R\times S^1}\|\p_s u(s,t)\|^2 \, dsdt\ge 0$, with equality if and only if $u$ is constant in $s$ and $x^+=x^-$. Whenever $x^+\neq x^-$, the additive group $\R$ acts freely on $\wh \cM(x^+;x^-)$ by translations $\sigma\cdot u = u(\cdot + \sigma,\cdot)$. The \emph{moduli space of Floer trajectories connecting $x^+$ to $x^-$}, denoted 
$$
\cM(x^+;x^-)=\wh \cM(x^+;x^-)/\R,
$$ 
is then a smooth manifold of dimension $|x^+|-|x^-|-1$. (The case $x^+=x^-=x$ is special: we denote $\cM(x;x)=\wh\cM(x;x)$; this consists of a point and therefore has dimension $0$.)

We will also consider admissible $s$-dependent homotopies $H=(H_s)$, $s\in\R$ that interpolate between Hamiltonians $H_+=(H_{+,t})$ at $+\infty$ and $H_-=(H_{-,t})$ at $-\infty$, and admissible $s$-dependent families of compatible almost complex structures $J=(J_s)$ that interpolate between $J_+=(J_{+,t})$ at $+\infty$ and $J_-=(J_{-,t})$ at $-\infty$. We say that a pair $(H,J)$ is \emph{regular} if, for any choice $x^\pm\in\Per(H_\pm)$ the \emph{moduli space of continuation Floer trajectories from $x^+$ to $x^-$}, defined as
\begin{align*}
\cH(x^+;x^-) & = \cH(x^+;x^-;H_s,J_s)\\
& = \{u:\R\times S^1\to X\, :  \, \p_s u + J_{s,t}(u)(\p_t u -X_{H_s}^t(u))=0,\\
& \qquad \qquad \lim_{s\to \pm\infty} u(s,\cdot)=x^\pm\},
\end{align*}
is cut out transversally. Each of these spaces is then a smooth manifold of dimension $|x^+|-|x^-|$. (In contrast to the case of Floer trajectories discussed in the previous paragraph, the additive group $\R$ does not act on $\cH(x^+;x^-)$ since the Floer data is $s$-dependent.)

Finally, given $(H_+,J_+)$, $(H_-,J_-)$ two regular pairs of admissible Hamiltonians and compatible almost complex structures, we will consider admissible \emph{homotopies of homotopies} $\{(H^\tau,J^\tau)\}$, $\tau\in [0,1]$, i.e., $1$-parameter families of homotopies interpolating from $(H_+,J_+)$ near $+\infty$ to $(H_-,J_-)$ near $-\infty$. Such a family is \emph{regular} if, for any $x^\pm\in\Per(H_\pm)$, the \emph{parametrized moduli space of continuation Floer trajectories} defined by 
\begin{align*}
\cH^{[0,1]}(x^+;x^-) & = \cH^{[0,1]}(x^+;x^-;H^\tau,J^\tau)\\
& = \{(\tau,u)\, : \,  \, \tau\in [0,1],\, u:\R\times S^1\to X \\
& \qquad \p_s u + J^\tau_{s,t}(u)(\p_t u -X_{H^\tau_s}^t(u))=0,\\
& \qquad \lim_{s\to \pm\infty} u(s,\cdot)=x^\pm\}
\end{align*}
is cut out transversally. Each of these spaces is then a smooth manifold of dimension $|x^+|-|x^-|+1$. We denote 
$\cH^i(x^+;x^-)$, $i=0,1$ the moduli spaces of continuation Floer trajectories for the pairs $(H^i, J^i)$, $i=0,1$.

A key ingredient in Floer theory is the construction of coherent orientations for the moduli spaces of Floer trajectories. The classical reference is~\cite{FH-coherent}, but we use in this paper a slight variation borrowed from~\cite{BOauto} and inspired by~\cite{Bourgeois-Mohnke}, which is directly suited for a comparison between Floer theory and Morse theory in the symplectically aspherical case. The construction takes as input the choice, for each $1$-periodic orbit $x\in \Per(H)$, of an orientation of the determinant line bundle over the space of Cauchy-Riemann operators on the Riemann sphere with one positive puncture (identified with $\C$) and asymptotic behavior at that puncture given by the linearization of the Hamiltonian flow at $x$.  This choice induces orientations for moduli spaces of Floer trajectories by requiring that the gluing isomorphisms preserve the orientations. The resulting orientations are coherent as a consequence of an associativity property of the gluing isomorphisms. We refer to~\cite{BOauto} for the Floer theoretic construction, to~\cite[Appendix~A.2]{BDHO} for the Morse theoretic construction, and to~\S\ref{sec:sympl-asph} for a comparison between the two.

\subsection{The twisting cocycle}\label{sec:twisting-cocycle}

In this subsection we define the DG Floer complex by constructing a Maurer-Cartan element $\mathfrak{m}=(m_{x,y})\in \End_{-1}(R_*\otimes C_\bullet)$ as announced in~\S\ref{sec:intro-DGFloer}. Here $R_*=C_*(\Omega \cL X)$ and $C_\bullet=\langle \Per(H)\rangle$, with $H$ an admissible nondegenerate Hamiltonian. We also refer to $\mathfrak{m}$ as a \emph{twisting cocycle}, or as \emph{the Barraud-Cornea twisting cocycle}~\cite{BC07}. 

\subsubsection{Orientations} \label{sec:orientations-twisting-cocycle}
Let $(H,J)$ be a regular admissible pair as in~\S\ref{sec:conventions_Floer}, to which one associates spaces of parametrized Floer trajectories $\wh \cM(x;y)$ and moduli spaces of unparametrized Floer trajectories $\cM(x;y)$ for $x,y\in\Per(H)$.  We recall that, given $u\in\wh \cM(x;y)$, the orbit $x$ is the asymptote of $u$ at its positive puncture and the orbit $y$ is the asymptote of $u$ at its negative puncture. 

The Floer compactification of $\cM(x;y)$ is denoted by $\ol\cM(x;y)$. This is a topological manifold with boundary with corners, 
whose boundary consists of \emph{broken Floer trajectories}~\cite[Appendix A]{BC07}. Without taking into account orientations the boundary can be written as 
$$
\p\ol\cM(x;y)=\bigcup_{z} \ol\cM(x;z)\times \ol \cM(z;y).
$$  
In order to construct the twisting cocycle a key step is to discuss orientations.  

 Let $\p_{x,y}$ be the vector field on $\wh \cM(x;y)$ given by ${\p_{x,y}}\big|_u=-\p_s u$. We refer to $\p_{x,y}$ as \emph{the canonical constant vector field on $\wh\cM(x;y)$}. We have a short exact sequence 
$$
0\to \R\p_{x,y} \to T_u \wh \cM(x;y)\to T_{[u]}\cM(x;y)\to 0
$$
that determines a canonical isomorphism of orientation lines 
\begin{equation} \label{eq:orientMMhat}
|\R\p_{x,y}|\otimes |T_{[u]}\cM(x;y)|\simeq |T_u\wh\cM(x;y)|.
\end{equation}

There is a gluing map
$$
\#_R:\wh \cM(x;z)\times \wh \cM(z;y)\to \wh \cM(x;y)
$$
that is a diffeomorphism onto its image and that induces a canonical isomorphism of orientation local systems  
$$
|\wh \cM(x;z)|\otimes |\wh \cM(z;y)|\simeq |\wh \cM(x;y)|. 
$$

\begin{remark} \label{rmk:remark-on-gluing}
The gluing map $\#_R$ is strictly speaking only partially defined on a compact subset of $\wh \cM(x;z)\times \wh \cM(z;y)$ for any given value of the gluing parameter, but its isotopy class in a neighborhood of any pair of trajectories $(u,v)$ is well-defined, and the resulting isomorphism of local systems is independent of all choices. Also, the gluing map is not surjective onto $\wh \cM(x;y)$. Nevertheless the orientation local system $|\wh \cM(x;y)|$ is trivial~\cite{FH-coherent}, so that a trivialization at a point induces uniquely a global trivialization over the path connected component of that point. 
\end{remark}

 Consider the canonical constant vector fields $\p_{x,z}$, $\p_{z,y}$, and $\p_{x,y}$, defined respectively on $\wh \cM(x;z)$, $\wh \cM(z;y)$, and $\wh \cM(x;y)$. We have the following sequence of canonical isomorphisms, the first one being induced by the inverse of the gluing map $\#_R$ from~\eqref{eq:orientMMhat}:  
\begin{align*}
|\R\p_{x,y}| \otimes |\cM(x;y)| & \simeq |\R\p_{x,z}| \otimes |\cM(x;z)| \otimes | \R \p_{z,y}| \otimes |\cM(z;y)| \\
& \simeq (-1)^{|x|-|z|-1} |\R \p_{x,z}| \otimes |\R \p_{z,y}| \otimes |\cM(x;z)| \otimes |\cM(z;y)| \\
& \simeq (-1)^{|x|-|z|-1} |\R \p_{x,z} \oplus \R \p_{z,y}| \otimes |\cM(x;z)| \otimes |\cM(z;y)| \\
& \simeq (-1)^{|x|-|z|-1}|\R (\p_{x,z}+ \p_{z,y})| \otimes |\R (-\p_{x,z}+ \p_{z,y})| \\
& \qquad \qquad\qquad\qquad \otimes |\cM(x;z)| \otimes |\cM(z;y)|. 
\end{align*}
The gluing map induces an orientation preserving isomorphism 
$$
|\R( \p_{x,z}+\p_{z,y})|\simeq |\R\p_{x,y}| 
$$ 
since it is equivariant with respect to the diagonal $\R$-action on the factors. 

We claim that the gluing map induces an orientation preserving isomorphism 
$$ 
|\R( -\p_{x,z}+ \p_{z,y})| \simeq |\R\nu^\text{in}|,
$$
where $\nu^{\text{in}}$ is a vector field on $\cM(x;y)$ pointing away from the boundary. To prove this we first consider the situation at the level of the pregluing map. Denote by $R$ the gluing parameter, let $u\in\wh\cM(x;z)$, $v\in\wh\cM(z;y)$, write $u(s,t)=\exp_{z(t)}X_u(s,t)$ for $s\le -R\ll 0$ and $v(s,t)=\exp_{z(t)}X_v(s,t)$ for $s\ge R\gg 0$ with $X_u$, $X_v$ uniquely determined vector fields along $z$, fix a nondecreasing cutoff function $\rho:\R\to [0,1]$ such that $\rho\equiv0$ on $(-\infty,0]$, $\rho\equiv1$ on $[1,\infty)$, and consider the vector field $X_R(s,t)=\rho(-s)X_v(s+R,t)+\rho(s)X_u(s-R,t)$ that interpolates on the interval $[-1,1]$ between $X_v(\cdot+R,\cdot)$ and $X_u(\cdot-R,\cdot)$. The \emph{pregluing of the trajectories $u$ and $v$ with gluing parameter $R$} is defined to be 
$$
u\wh\#_Rv(s,t)=\left\{\begin{array}{rl}u(s-R,t), & s\ge 1,\\
\exp_{z(t)}X_R(s,t), & s\in[-1,1],\\
v(s+R,t), & s\le -1.
\end{array}\right.
$$
This defines an approximate solution of the Floer equation for $R\gg 0$, and the gluing $u\#_Rv$ is obtained by applying a projection operator, denoted $\pi_R$, to the approximate solution $u\wh\#_R v$ (see for example~\cite[\S9]{Audin-Damian_English}). It follows from the definition of the pregluing map that  $u(\cdot+\eps,\cdot)\wh\#_R v(\cdot-\eps,\cdot)=u\wh\#_{R-\eps} v$, 
and in particular 
$$
\frac{d}{d\eps}\Big|_{\eps=0} u(\cdot+\eps,\cdot)\wh\#_R v(\cdot-\eps,\cdot)=-\frac{d}{dR}u\wh\#_R v.
$$ 
By definition, the image of the left hand side under the differential of the projection $\pi_R$ is the image of $-\p_{x,z}+\p_{z,y}$ under the differential of the gluing map. The image of the right hand side under the differential of $\pi_R$ points away from the boundary, because the broken trajectory $(u,v)$ is obtained as $\lim_{R\to\infty} u\#_R v$. Since the two coincide, this proves the claim.

Given $\nu^\text{out}$ a vector field on $\cM(x;y)$ that points towards the boundary, the resulting canonical isomorphism $ |\R( -\p_{x,z}+ \p_{z,y})| \simeq |\R\nu^\text{out}|$ is orientation reversing, i.e., it has sign $-1$. In conclusion we obtain a canonical isomorphism 
$$
|\cM(x;y)|\simeq (-1)^{|x|-|z|}|\R\nu^{\text{out}}|\otimes |\cM(x;z)| \otimes |\cM(z;y)|.
$$
In view of our orientation convention for the boundary of a manifold, we have proved:
\begin{proposition}\label{orientationsFloer} 
Let $x,z,y\in\Per(H)$ be $1$-periodic orbits with $|x|>|z|>|y|$. The orientation induced from $\ol \cM(x,y)$ by the outward normal vector on the $1$-codimensional stratum $\cM(x,z)\times\cM(z,y)$ of the boundary differs from the product orientation by $(-1)^{|x|-|z|}$. \qed
\end{proposition} 

In an equivalent formulation, the boundary of the oriented Floer compactification of the moduli spaces of Floer trajectories with respect to a system of coherent orientations as in~\cite{FH-coherent} can be written as 
$$
\p \ol \cM (x;y)=\bigcup_{z} (-1)^{|x|-|z|}\ol\cM(x;z)\times \ol \cM(z;y).
$$

\subsubsection{Representing chain systems} 

The following two statements are Floer counterparts of Propositions 5.6 and 5.8 of \cite{BDHO}. They are proved in exactly the same way, by upgrading with signs the original argument of Barraud-Cornea~\cite[Lemma 2.2]{BC07}, which proceeds by induction on the dimension of the moduli spaces.

\begin{proposition}[Existence of representing chain systems for Floer moduli spaces] \label{representingchainFloer} 
Given $x,y\in\Per(H)$, denote $C_*(\ol\cM(x;y))$ the complex of cubical chains on the space $\ol\cM(x;y)$ of broken Floer trajectories. There exists a collection 
$\{s_{x,y}\, : \, x,y\in \Per(H)\}$ with $s_{x,y}\in C_{|x|-|y|-1}(\ol\cM(x,y))$ satisfying the following properties:  
\begin{enumerate}
\item $s_{x,y}$ is a cycle relative to the boundary and represents the fundamental class $[\ol\cM(x;y)]$. 
\item  
$$\partial s_{x,y}\, =\, \sum_z(-1)^{|x|-|z|}s_{x,z}\times s_{z,y},$$
where the product of chains is defined via the inclusions 
$$\ol\cM(x;z)\times\ol\cM(z;y)\, \subset \, \partial \ol\cM(x;y)\, \subset\, \ol\cM(x;y).$$
\end{enumerate}
\qed
\end{proposition} 

\begin{definition}\label{chainsystemFloer} Following Barraud and Cornea, we will call the family $(s_{x,y})$ a \emph{representing chain system for  the Floer moduli spaces}.
\end{definition} 

\begin{proposition}[Uniqueness of representing chain systems for Floer moduli spaces] \label{prop:uniquenessFloer} Any two representing chain systems $(s'_{x,y})$ and $(s_{x,y})$ are homologous in the following sense: there exists a family $(\kappa_{x,y})$ of chains $\kappa_{x,y}\in C_{|x|-|y|}(\ol\cM(x;y))$ such that:
\begin{enumerate}
\item $\kappa_{x,x}$ is the constant $0$-chain for all $x$ and $\kappa_{x,y}=0$ for $|x|=|y|$, $x\neq y$.
\item for all $x,y$ we have 
\begin{equation}\label{eq:sprime-s-homologous}
\p \kappa_{x,y} = \sum_z s'_{x,z}\times \kappa_{z,y} + (-1)^{|x|-|z|-1}\kappa_{x,z}\times s_{z,y}.
\end{equation}
\end{enumerate}
\qed
\end{proposition}

\subsubsection{Evaluation maps} \label{sec:evaluation_Floer_traj}

Given an admissible Hamiltonian $H$, we complement the data of a regular admissible almost complex structure $J$ and of a representing chain system $(s_{xy})$ for the Floer moduli spaces by the following two additional choices for each connected component $\cL_a X\subset \cL X$: 
\begin{itemize} 
\item an embedded tree $\cY_a$ connecting the base loop $\gamma_a$ to the elements of $\Per^a(H)$. 
\item a homotopy inverse $\theta_a$ for the quotient map $p_a:\cL_a X\to \cL_a X/\cY_a$. 
\end{itemize}
We denote $s^a_{x,y}$ the representing chain system for the moduli spaces of Floer trajectories connecting $1$-periodic orbits of $H$ in the free homotopy class $a$. We denote $\Xi_a=(s^a_{x,y},\cY_a,\theta_a)$ and we refer to such a tuple as an \emph{enriched Floer datum}. 

In the sequel we work componentwise on the free loop space and, in order to alleviate the notation, we do not specify anymore the free homotopy class. Thus we write $\theta$, $p$, $\Xi$, $s_{x,y}$ instead of $\theta_a$, $p_a$, $\Xi_a$, $s^a_{x,y}$. 

Given two orbits $x,y\in \Per(H)$, let $\cP_{x\to y}\cL X$ be the space of Moore paths in $\cL X$ running from $x$ to $y$. Each moduli space $\cM(x;y)$, $x,y\in\Per(H)$ is equipped with a canonical \emph{evaluation map} 
$$
\ev_{x,y}:\cM(x;y)\to\cP_{x\to y}\cL X,
$$ 
that sends a Floer cylinder to the path that it defines in $\cL X$, parametrized backwards on the interval $[0,A_H(x)-A_H(y)]$ by the values of the action functional $A_H$ shifted down by $A_H(y)$. More precisely
\begin{align*}
\ev_{x,y}([u])=\big( \ell\mapsto u(s,\cdot), \quad & \ell\in[0,A_H(x)-A_H(y)],\\
& \text{where }A_H(u(s,\cdot))=A_H(x)-\ell \big).
\end{align*}
This evaluation map extends continuously to the compactification 
$$
\overline\ev_{x,y}:\overline\cM(x;y)\to \cP_{x\to y}\cL X
$$
and we define 
\begin{equation} \label{eq:qxy}
\ol q_{x,y}=\theta\circ p\circ \overline\ev_{x,y}:\overline\cM(x;y)\to \Omega\cL X. 
\end{equation}

The restriction of $\ol q=\ol q_{x,y}$ to the boundary $\p\ol\cM(x;y)$ satisfies
\begin{equation}\label{eq:relation1-Mxy}
\ol q(u,v)\, =\, \ol q(u)\, \# \, \ol q(v)
\end{equation} 
for all $(u,v)\in \ol\cM(x;z)\times \ol\cM(z;y)$,  where $\#$ stands for the concatenation of based Moore loops. 

\subsubsection{Twisting cocycles and DG Floer complex} \label{sec:DGFloercomplex}

We define 
$$
m_{x,y}=(\ol q_{x,y})_* (s_{x,y})\in C_{|x|-|y|-1}(\Omega \cL X).
$$

\begin{definition}[Barraud-Cornea twisting cocycle] A family $(m_{x,y})$ obtained in this way is called \emph{a Barraud-Cornea twisting cocycle}. 
\end{definition}

By abuse of language, and in view of the uniqueness statement below, we will also refer to $(m_{x,y})$ as \emph{the} Barraud-Cornea twisting cocycle. The next result shows that it actually is a twisting cocycle in the sense of Definition \ref{def:general-twisting}. 
It is a straightforward consequence of Proposition~\ref{representingchainFloer} and Equation~\eqref{eq:relation1-Mxy}. 

\begin{proposition}[Existence of twisting cocycles]
\qquad 

Any Barraud-Cornea twisting cocycle satisfies the equation 
$$
\p m_{x,y}-\sum_{z}(-1)^{|x|-|z|}m_{x,z}m_{z,y} = 0.
$$
\qed
\end{proposition}

We also state for the record a uniqueness statement that mirrors Proposition~\ref{prop:uniquenessFloer}.

\begin{proposition}[Uniqueness of twisting cocycles]  \label{prop:uniqueness-BC}
\qquad 

Any two Barraud-Cornea cocycles $(m_{x,y})$ and $(m'_{x,y})$ are homologous in the following sense: there exists a family $\nu_{x,y}\in C_{|x|-|y|}(\Omega \cL X)$ such that $\nu_{x,x}$ is the basepoint for any $x$ and 
$$
\p \nu_{x,y}=\sum_z m'_{x,z}\nu_{z,y} + (-1)^{|x|-|z|-1}\nu_{x,z}m_{z,y}.
$$
\qed
\end{proposition}

\begin{remark}\label{rmk-uniqueness-twisting}
The statement of Proposition~\ref{prop:uniqueness-BC} differs from that of Proposition~\ref{prop:uniquenessFloer} in that it does not guarantee $\nu_{x,y}=0$ if $x\neq y$ and $|x|=|y|$. Indeed, here only the Hamiltonian $H$ is fixed, whereas the auxiliary data, including $J$, is allowed to vary. As such, there may be nontrivial continuation Floer trajectories between periodic orbits that are different but have the same index. 

This  uniqueness property is a particular case of the existence of continuation cocycles stated in Proposition \ref{prop:existence_cont_cocycles}.
\end{remark}

\begin{remark} \label{rmk:twisting-MC}
As explained in~\cite[\S2]{BDHO}, the twisting cocycle $(m_{x,y})$ can be interpreted as a Maurer-Cartan element in the DG Lie algebra $\End(R_*\otimes C_\bullet)$, where $R_*=C_*(\Omega\cL X)$ is the dga of cubical chains with integer coefficients, and $C_\bullet=\langle \Per(H)\rangle$ is the free abelian group generated by the $1$-periodic orbits of $H$, viewed as a complex with zero differential.
\end{remark}

\begin{definition} Let $\cF$ be a DG right $C_*(\Omega \cL X)$-module. Given a regular admissible pair $(H,J)$ with enriched Floer datum $\Xi$, the \emph{DG Floer complex for $(H,J,\Xi)$ with coefficients in $\cF$} is 
$$
FC_*(H,J,\Xi;\cF)=\cF\otimes \langle \Per(H)\rangle,
$$ 
with elements of $\Per(H)$ graded by their Conley-Zehnder index, and with differential 
$$
\p(\alpha\otimes x)=\p\alpha\otimes x + (-1)^{|\alpha|} \sum_y \alpha\cdot m_{x,y}\otimes y.
$$
\end{definition}

Proposition~\ref{prop:uniqueness-BC} and the results from the next section imply the following statement. 

\begin{proposition} \label{prop:FC*htpytype}
Let $\cF$ be a DG right $C_*(\Omega \cL X)$-module. Given an admissible nondegenerate Hamiltonian $H$, the DG Floer complexes $FC_*(H,J,\Xi;\cF)$ and $FC_*(H,J',\Xi';\cF)$ are chain homotopy equivalent for any two choices $(J,\Xi)$, $(J',\Xi')$ consisting of regular almost complex structures $J$, $J'$ and enriched Floer data $\Xi$, $\Xi'$. 
\qed
\end{proposition}

One way to read Proposition~\ref{prop:FC*htpytype} is by saying that, for fixed Hamiltonian $H$ and DG local system $\cF$, the chain homotopy type of the DG Floer complex $FC_*(H,J,\Xi;\cF)$ is well-defined.

The DG Floer complex for $(H,J,\Xi)$ with coefficients in $\cF$ admits a canonical filtration by the Hamiltonian action. For $b\in\ol\R$ we denote $\Per^{<b}(H)$ the set of $1$-periodic orbits of action $<b$, and define for a regular pair $(H,J)$ with enriched Floer datum $\Xi$ the \emph{action filtered DG Floer chain (sub)complex} 
$$
FC_*^{<b}(H,J,\Xi;\cF)=\cF\otimes \langle \Per^{<b}(H)\rangle.
$$

\subsection{Continuation maps} \label{sec:continuation}

Let $H=(H_s)$ and $J=(J_s)$ be a regular pair of admissible homotopies of Hamiltonians and almost complex structures, interpolating between regular pairs $(H_+,J_+)$ at $+\infty$ and $(H_-,J_-)$ at $-\infty$. Following~\S\ref{sec:conventions_Floer}, we associate to $(H,J)$ the moduli spaces of continuation Floer trajectories $\cH(x^+;y^-)$ for $x^+\in\Per(H_+)$ and $y^-\in\Per(H_-)$. In a manner analogous to the construction of the Barraud-Cornea twisting cocycle from~\S\ref{sec:twisting-cocycle}, we will first construct from the compactified moduli spaces $\ol \cH(x^+;y^-)$ a representing chain system $\sigma_{x^+,y^-}\in C_{|x^+|-|y^-|}(\ol\cH(x^+;y^-))$, and then, using appropriate evaluation maps, we will construct a {\it continuation cocycle} $\nu_{x^+,y^-}\in C_{|x^+|-|y^-|}(\Omega \call X)$. The latter will define the continuation map between the Floer complexes with DG coefficients associated to $(H_{+},J_{+})$ and $(H_-,J_-)$. 

\subsubsection{Orientations} \label{sec:orientations_continuation}

The moduli spaces $\cH(x^+;y^-)$ admit Floer compactifications $\ol \cH(x^+;y^-)$ that are manifolds with boundary with corners~\cite{BC07}.  
The codimension $1$ strata in the boundary $\p \ol \cH(x^+;y^-)$ are of the type 
$$
\cM(x^+;z^+)\times \cH(z^+;y^-) \quad \mbox{for }z^+\in\Per(H_+)
$$
and
$$
\cH(x^+;z^-)\times \cM(z^-;y^-) \quad \mbox{for }z^-\in\Per(H_-).
$$
Fixing a coherent set of orientations as in~\cite{FH-coherent}, we now wish to understand for each of these strata what is the difference in sign between the product orientation and the boundary orientation. There are two gluing maps to consider, which correspond to the above two types of strata.

Consider first the gluing map
$$
\wh \cM(x^+;z^+)\times \cH(z^+;y^-) \to \cH(x^+;y^-).
$$
This is a diffeomorphism from its domain to its image and induces an isomorphism of orientation local systems  
$$
|\wh \cM(x^+;z^+)|\otimes |\cH(z^+;y^-)|\simeq |\cH(x^+;y^-)|. 
$$
 Consider the canonical constant vector field $\p_{x^+,z^+}$ on $\wh \cM(x^+;z^+)$, which we view under the gluing map as defining a vector field on $\cH(x^+;y^-)$. A discussion similar to the one in~\S\ref{sec:orientations-twisting-cocycle} shows that this vector field points towards the boundary. Using~\eqref{eq:orientMMhat} we infer a canonical isomorphism 
\begin{align*}
|\cH(x^+;y^-)| & \simeq |\R\p_{x^+,z^+}| \otimes |\cM(x^+;z^+)| \otimes |\cH(z^+;y^-)| \\
& \simeq |\R\nu^{\text{out}}| \otimes |\cM(x^+;z^+)| \otimes |\cH(z^+;y^-)|.
\end{align*}
Here the second isomorphism is simply notational: we record that $\p_{x^+,z^+}$ points towards the boundary by denoting it $\nu^{\text{out}}$. As a consequence the product orientation on $\cM(x^+;z^+)\times \cH(z^+;y^-)$ coincides with its boundary orientation. 

Consider now the gluing map 
$$
\cH(x^+;z^-)\times \wh \cM(z^-;y^-) \to  \cH(x^+;y^-).
$$
This is again a diffeomorphism from its domain to its image and induces a canonical isomorphism of orientation local systems 
$$
|\cH(x^+;z^-)|\otimes |\wh \cM(z^-;y^-)|\simeq |\cH(x^+;y^-)|. 
$$
 Consider the canonical constant vector field $\p_{z^-,y^-}$ on $\wh \cM(z^-;y^-)$, viewed under the gluing map as a vector field on $\cH(x^+;y^-)$. As above, one sees that this vector field points away from the boundary. Using~\eqref{eq:orientMMhat} we infer a canonical isomorphism 
\begin{align*}
|\cH(x^+;y^-)| & \simeq |\cH(x^+;z^-)| \otimes |\R\p_{z^-,y^-}| \otimes |\cM(z^-;y^-)|  \\
& \simeq (-1)^{|x^+|-|z^-|}|\R\p_{z^-,y^-}| \otimes |\cH(x^+;z^-)| \otimes  |\cM(z^-;y^-)| \\
& \simeq (-1)^{|x^+|-|z^-|-1} |\R\nu^{\text{out}}| \otimes |\cH(x^+;z^-)| \otimes  |\cM(z^-;y^-)|.
\end{align*}
The last isomorphism is again notational: the vector field $\p_{z^-,y^-}$ points away from the boundary and, if $\nu^{\text{out}}$ is a vector field pointing towards the boundary, the canonical isomorphism $|\R\p_{z^-,y^-}|\simeq |\R\nu^{\text{out}}|$ reverses orientation and therefore has sign $-1$. 

As a consequence we obtain that the boundary of the oriented Floer compactification of the moduli spaces of continuation Floer trajectories with respect to a system of coherent orientations as in~\cite{FH-coherent} can be written as 
\begin{align*}
\p \ol \cH(x^+;y^-)=\bigcup_{z^+} \ol \cM(x^+;z^+) & \times \ol \cH(z^+;y^-) \\
&  \cup \quad \bigcup_{z^-}(-1)^{|x^+|-|z^-|-1}\ol \cH(x^+;z^-)\times  \ol \cM(z^-;y^-).
\end{align*}

\subsubsection{Representing chain systems} 

\begin{proposition}[Existence of representing chain systems for continuation Floer moduli spaces] \label{representingchainFloercontinuation} Fix a regular admissible homotopy $(H_s,J_s)$ that interpolates between $(H_+,J_+)$ at $+\infty$ and $(H_-,J_-)$ at $-\infty$ as above, and choose representing chain systems $(s_{x^+,y^+})$ for the Floer moduli spaces of $(H_+,J_+)$, and $(s_{x^-,y^-})$ for the Floer moduli spaces of $(H_-,J_-)$. 

Given $x^+\in\Per(H_+)$ and $y^-\in\Per(H_-)$, denote $C_*(\ol\cH(x^+;y^-))$ the complex of cubical chains on the compactified moduli space of continuation Floer trajectories. There exists a collection 
$$
\{\sigma_{x^+,y^-}\, : \, x^+\in \Per(H_+),\, y^-\in\Per(H_-)\}
$$ 
with $\sigma_{x^+,y^-}\in C_{|x^+|-|y^-|}(\ol\cH(x^+;y^-))$ satisfying the following properties:  
\begin{enumerate}
\item $\sigma_{x^+,y^-}$ is a cycle rel boundary and represents the fundamental class\break $[\ol\cH(x^+;y^-)]$. 
\item there holds the relation
$$
\partial \sigma_{x^+,y^-}\, =\, \sum_{z^+} s_{x^+,z^+} \times \sigma_{z^+,y^-}
+ \sum_{z^-}(-1)^{|x^+|-|z^-|-1} \sigma_{x^+,z^-} \times s_{z^-,y^-},
$$
where the product of chains is defined via the inclusions 
$$\ol\cM(x^+;z^+)\times\ol\cH(z^+;y^-) \, \subset \, \partial \ol\cH(x^+;y^-)\, \subset\, \ol\cH(x^+;y^-)$$
and 
$$\ol\cH(x^+;z^-)\times\ol\cM(z^-;y^-) \, \subset \, \partial \ol\cH(x^+;y^-)\, \subset\, \ol\cH(x^+;y^-).$$
\end{enumerate}
\end{proposition} 

\begin{proof}
The proof is similar to that of~\cite[Lemma~6.1(ii)]{BDHO}, and we omit it. 
\end{proof}

\begin{definition}\label{chainsystemFloercontinuation} We call the family $(\sigma_{x^+,y^-})$ a \emph{representing chain system for  the continuation Floer moduli spaces, subordinated to the representing chain systems for Floer moduli spaces $(s_{x^+,y^+})$ and $(s_{x^-,y^-})$}. We will also use the shorthand terminology \emph{continuation representing chain system}. 
\end{definition}

\begin{proposition}[Uniqueness of representing chain systems for continuation Floer moduli spaces] \label{prop:uniquenessFloercontinuation} Fix a regular admissible homotopy $(H_s,J_s)$ between $(H_+,J_+)$ at $+\infty$ and $(H_-,J_-)$ at $-\infty$ as above, and choose representing chain systems $(s_{x^+,y^+})$ for the Floer moduli spaces of $(H_+,J_+)$, and $(s_{x^-,y^-})$ for the Floer moduli spaces of $(H_-,J_-)$. 

Any two representing chain systems $(\sigma'_{x^+,y^-})$ and $(\sigma_{x^+,y^-})$ for the continuation Floer moduli spaces, that are subordinated to $(s_{x^+,y^+})$ and $(s_{x^-,y^-})$, are chain homotopic in the following sense: there exists a family $(\cS_{x^+,y^-})$ with $\cS_{x^+,y^-}\in C_{|x^+|-|y^-|+1}(\ol\cH(x^+;y^-))$ such that, for all $x^+\in\Per(H_+)$ and $y^-\in\Per(H_-)$,  we have 
\begin{align}\label{eq:cprime-c-homologous}
\p \cS_{x^+,y^-} =  \sigma'_{x^+,y^-} - \sigma_{x^+,y^-} 
& + \sum_{z^+} (-1)^{|x^+|-|z^+|}s_{x^+,z^+}\times \cS_{z^+,y^-} \\
& +  \sum_{z^-}(-1)^{|x^+|-|z^-|}\cS_{x^+,z^-}\times s_{z^-,y^-}. \nonumber
\end{align}
\end{proposition}

\begin{proof}
The proof is similar to that of~\cite[Lemma~6.11]{BDHO}, and we omit the details. 
\end{proof}

\subsubsection{Evaluation maps} \label{sec:eval_maps_homotopies}
\label{sec:evaluation-for-continuation}
Let $(H,J)$ be a regular admissible homotopy interpolating from $(H_+,J_+)$ at $+\infty$ to $(H_-,J_-)$ at $-\infty$, and let $(\sigma_{x^+,y^-})$ be a continuation representing chain system subordinated to representing chain systems $(s_{x^+,y^+})$ and $(s_{x^-,y^-})$. Assume further that we have chosen pairs $(\cY_+,\theta_+)$ and $(\cY_-,\theta_-)$ consisting of trees $\cY_\pm$ and homotopy inverses $\theta_\pm$ for the collapsing maps $p_\pm:\cL X\to \cL X/\cY_\pm$, so that $(J_\pm,s_{x^\pm,y^\pm},\cY_\pm,\theta_\pm)$ are enriched Floer data at $\pm\infty$. (We recall that we omit from the notation the free homotopy classes of loops.) We complement the data $(\sigma_{x^+,y^-})$ with the following additional choices: 
\begin{itemize} 
\item  a family of embedded trees $\cY=(\cY_s)$ rooted at the basepoint, which is continuous with respect to Hausdorff distance and connects $\cY_+$ at $+\infty$ to $\cY_-$ at $-\infty$
\item a continuous family $\theta=(\theta_s)$ of homotopy inverses for the quotient maps $p_s:\cL X\to \cL X/\cY_s$, $s\in\R$. Here $\theta $ being continuous means that we can write $\theta_s=\Theta(s,\cdot)$ for a continuous map $\Theta:([0,1]\times \cL X)/\!\sim\,\longrightarrow \cL X$, where $(s,x)\sim (s',x')$ if and only if $s=s'$ and $x,x'\in \cY_s$.
\end{itemize}
We refer to 
$$
\Upsilon=(\sigma_{x^+,y^-},\cY,\theta)
$$ 
as an \emph{enriched Floer continuation datum}. 

We wish to evaluate continuation Floer trajectories into Moore paths in $\cL X$ parametrized by the values of the action functional. To this end, we choose the following additional data:  
\begin{itemize}
\item a smooth cutoff function $\rho:\R\to [0,1]$ such that $\rho(s)=0$ for $s\le s_- -1$, $\rho(s)=1$ for $s\ge s_+ +1$, and $\rho'>0$ on $(s_- -1,s_+ +1)$, where $s_\pm\in\R$ are chosen such that $(H_s,J_s,\cY_s,\theta_s)$ is constant equal to $(H_\pm,J_\pm,\cY_\pm,\theta_\pm)$ for $\pm s\ge \pm s_\pm$.  
\item a constant $C> 0$ such that 
\begin{equation} \label{eq:HrhoC}
\p_s H_s -\rho'(s) C < 0 \mbox{ for } s\in [s_-,s_+].
\end{equation}
\end{itemize}
Let $H_s^{\rho,C}=H_s-\rho(s)C$. As a consequence of the last condition we have $\p_s H_s^{\rho,C}\le 0$ on $\R$.
The existence of such a constant $C> 0$ is clear if $X$ is compact. If $X$ is non-compact (typically the symplectic completion of a Liouville domain), the existence of such a constant follows whenever $\p_sH_s$ is bounded, which is a consequence of the definition of admissibility in~\S\ref{sec:symplectic_homology}. If the homotopy $H_s$ is monotone, i.e., $\p_s H_s\le 0$, any choice of constant $C>0$ is allowed.\footnote{One could allow $C= 0$ if $\p_s H_s\le 0$, or more generally $\p_s H_s -\rho'(s) C \le 0$ for $s\in[s_-,s_+]$, but that would complicate the notation later. We opted for the easiest setup.} We refer to the tuple 
$$
\Xi=(\Upsilon,\rho,C)
$$ 
as a \emph{monotone enriched Floer continuation datum}. 
Upon replacing $H_s$ with $H_s^{\rho,C}$, the continuation Floer trajectories remain the same since $X_{H_s}=X_{H_s^{\rho,C}}$, but the values of the action along a continuation Floer trajectory change as 
$$
A_{H_s^{\rho,C}}(u(s,\cdot)) = A_{H_s}(u(s,\cdot))+\rho(s)C.
$$
Under condition~\eqref{eq:HrhoC} we have 
\begin{align*}
\p_s A_{H_s^{\rho,C}}(u(s,\cdot)) & = \| \nabla A_{H_s}(u(s,\cdot))\|^2 - \int_{S^1} \p_s H_s^{\rho,C}(u(s,\cdot)) \, dt \\
& = \|\p_s u(s,\cdot)\|^2 - \int_{S^1} \p_s H_s^{\rho,C}(u(s,\cdot)) \, dt \\
& \ge 0
\end{align*}
for any $u\in\cH(x^+;y^-)$, with
$$
\p_s A_{H_s^{\rho,C}}(u(s,\cdot))>0 \, \mbox{ for } s\in[s_--1,s_++1].
$$

Given $u\in\cH(x^+;y^-)$, consider the map 
$$
A_u:\ol\R\to [0,A_{H_+}(x^+)+C-A_{H_-}(y^-)], \quad s\mapsto A_{H_s^{\rho,C}}(u(s,\cdot))-A_{H_-}(y^-).
$$
This map is surjective, increasing on $\ol\R$ and strictly increasing on $[s_--1,s_++1]$. On each of the intervals $[-\infty,s_--1]$ and $[s_++1,+\infty]$ it is either constant (in which case $u(s,\cdot)$ is also constant), or strictly increasing (in which case $u(s,\cdot)$ varies with $s$). Denote by $I_u\subset \ol\R$ the maximal closed interval such that $A_u:I_u\to [0,A_{H_+}(x^+)+C-A_{H_-}(y^-)]$ is a bijection, and denote $A_u^{-1}: [0,A_{H_+}(x^+)+C-A_{H_-}(y^-)]\to I_u$ its inverse. The interval $I_u$ is uniquely determined and we have either $I_u=\ol\R$, or $I_u=[s_--1,+\infty]$, or $I_u=[-\infty,s_++1]$, or $I_u=[s_--1,s_++1]$.

Given orbits $x^+\in\Per(H_+)$, $y^-\in \Per(H_-)$, we denote $\cP_{x^+\to y^-}\cL X$ the space of Moore paths in $\cL X$ running from $x^+$ to $y^-$. Each moduli space $\cH(x^+;y^-)$ is equipped with a canonical \emph{evaluation map} 
$$
\ev_{x^+,y^-}:\cH(x^+;y^-)\to\cP_{x^+\to y^-}\cL X,
$$ 
\begin{align*}
u\mapsto \big( \ell\mapsto u(s,\cdot), \quad & \ell\in [0,A_{H_+}(x^+)+C-A_{H_-}(y^-)], \\
& s=A_u^{-1}(A_{H_+}(x^+)+C-\ell)\big).
\end{align*}
This map sends a Floer continuation cylinder to the path that it defines in $\cL X$, parametrized backwards on the interval $[0,A_{H_+}(x^+)-A_{H_-}(y^-)+C]$ by the values of the action functional $A_{H_s^{\rho,C}}$ shifted down by $A_{H_-}(y^-)$. 

Using the projections $p_s:\cL X\to \cL X/\cY_s$ and their homotopy inverses $\theta_s$ we further define an evaluation map into Moore loops on $\cL X$ by 
$$
q_{x^+,y^-}:\cH(x^+;y^-)\to\Omega\cL X,
$$ 
\begin{align*}
u\mapsto \big( \ell\mapsto \theta_s p_s u(s,\cdot), \quad & \ell\in [0,A_{H_+}(x^+)+C-A_{H_-}(y^-)], \\
& s=A_u^{-1}(A_{H_+}(x^+)+C-\ell)\big).
\end{align*}

The above evaluation maps both extend continuously to the compactification $\ol\cH(x^+;y^-)$, and we denote the extension of the second one 
$$
\overline q_{x^+,y^-}:\overline\cH(x^+;y^-)\to \Omega\cL X.
$$

Moreover, denoting by $\ol q_{\pm}:\ol\calm(x^\pm,y^\pm)\ri \Omega\call X$ the evaluation maps   defined in \S\ref{sec:twisting-cocycle}   by the data $(H^\pm,J^\pm,s_{x^\pm,y^\pm },\caly^\pm,\theta^\pm)$,  we have that the restriction of $\ol q=\ol q_{x^+,y^-}$ to the boundary $\p\ol\cH(x^+;y^-)$ satisfies
\begin{equation}\label{eq:relation1}
\ol q(u,v)\, =\, \ol q_+(u)\, \# \, \ol q(v)
\end{equation} 
for all $(u,v)\in \ol\calm(x^+;z^+)\times \ol\cH(z^+;y^-)$, and 
\begin{equation}\label{eq:relation2}
\ol q(u,v)\, =\,  \ol q(u)\, \# \, \ol q_-(v)
\end{equation}
for all $(u,v)\in \ol\cH(x^+;z^-)\times \ol\calm(z^-;y^-)$.

\subsubsection{Continuation cocycles and continuation maps}

Given the continuation representing chain system $(\sigma_{x^+,y^-})$ we define  
$$
\nu_{x^+,y^-}=(\ol q_{x^+,{y^-}})_* (\sigma_{x^+,y^-})\in C_{|x^+|-|y^-|}(\Omega \cL X).
$$

\begin{definition} We call a family $(\nu_{x^+,y^-})$ obtained in this way \emph{a Barraud-Cornea
continuation cocycle subordinated to the twisting cocycles $(m_{x^+,y^+})$ and $(m_{x^-,y^-})$}. 
\end{definition}

The next result states that $(\nu_{x^+,y^-})$ satisfies the equation of continuation cocycles from Definition \ref{def:general-continuation}. It is a straightforward consequence of Proposition~\ref{representingchainFloercontinuation} and of properties~\eqref{eq:relation1} and~\eqref{eq:relation2}.

\begin{proposition}[Existence of continuation cocycles] \label{prop:existence_cont_cocycles}
Any Barraud-Cornea continuation cocycle satisfies the equation 
$$
\p \nu_{x^+,y^-}=\sum_{z^+} m_{x^+,z^+} \nu_{z^+,y^-}+ \sum_{z^-}(-1)^{|x^+|-|z^-|-1} \nu_{x^+,z^-} m_{z^-,y^-}.
$$
\qed
\end{proposition}

We also have a uniqueness result for continuation cocycles. 

\begin{proposition}[Uniqueness of Barraud-Cornea continuation cocycles]  \label{prop:uniqueness_continuation}
\qquad 

Given regular admissible Floer data $(H_\pm,J_\pm,\Xi_\pm)$,  any two associated Barraud-Cornea
continuation cocycles $(\nu^0_{x^+,y^-})$ and $(\nu^1_{x^+,y^-})$ are homologous in the following sense: 
there exists a family $h_{x^+,y^-}\in C_{|x^+|-|y^-|+1}(\Omega \cL X)$ such that  
\begin{align} \label{eq:homotopy-nuprime-nu}
\p h_{x^+,y^-} =  \nu^1_{x^+,y^-} - \nu^0_{x^+,y^-} 
& + \sum_{z^+} (-1)^{|x^+|-|z^+|}m_{x^+,z^+}h_{z^+,y^-} \\
& +  (-1)^{|x^+|-|z^-|}h_{x^+,z^-} m_{z^-,y^-}. \nonumber
\end{align}
\end{proposition}

\begin{proof}
This is the content of Proposition~\ref{prop:cocycle_homotopy_of_homotopies} 
proved in the next section, applied to a homotopy of homotopies $(H^\tau,J^\tau,\Xi^\tau)$ interpolating between the monotone continuation data $(H^i,J^i,\Xi^i)$, $i=0,1$ that determine the continuation cocycles $\nu^i_{x^+,y^-}$. We use the fact that the space of pairs $(\rho,C)$ is convex, the spaces of $s$-families of embedded trees $\cY=(\cY_s)$ and of $s$-families of homotopy inverses $\theta_s$ are path connected, and the space of homotopies $(H_s,J_s)$ interpolating between $(H_+,J_+)$ and $(H_-,J_-)$ is contractible. 
\end{proof}

By abuse of language, and in view of this uniqueness statement, we will also refer to $(\nu_{x^+,y^-})$ as \emph{the} continuation cocycle subordinated to the twisting cocycles $(m_{x^+,y^+})$ and $(m_{x^-,y^-})$.

\begin{remark} \label{rmk:cont-MC}
Following~\cite[\S2]{BDHO}, the continuation cocycle $(\nu_{x^+,y^-})$ can be viewed as a degree $0$ cycle in the complex $\Hom((R_*\otimes C^+_\bullet,D^+),\break(R_*\otimes C^-_\bullet,D^-))$, where $R_*=C_*(\Omega \cL X)$ is the dga of cubical chains with integer coefficients, $C^\pm_\bullet=\langle \Per(H_\pm)\rangle$ is the free abelian group generated by the $1$-periodic orbits of $H_\pm$, and $D^\pm$ are the differentials on $R_*\otimes C^\pm_\bullet$ twisted by the Maurer-Cartan elements given by the corresponding twisting cocycles (see Remark~\ref{rmk:twisting-MC}). 
\end{remark}

\begin{definition}\label{def:continuation-map}
Let $\cF$ be a DG right $C_*(\Omega \cL X)$-module and $(H_\pm,J_\pm)$ be regular admissible pairs with enriched Floer data $\Xi_\pm$. Let $(H_s,J_s)$ be a regular admissible homotopy interpolating between $(H_+,J_+)$ at $+\infty$ and $(H_-,J_-)$ at $-\infty$, and let $\Xi$ be a monotone enriched continuation Floer datum interpolating between $\Xi_+$ and $\Xi_-$. The \emph{continuation map induced by $(H_s,J_s,\Xi)$} is denoted  
$$
\Psi^\cF : FC_*(H_+,J_+,\Xi_+;\cF)\to FC_*(H_-,J_-,\Xi_-;\cF),
$$  
and it is defined with respect to the canonical bases $\Per(H_+)$ and $\Per(H_-)$ by 
$$
\Psi^\cF(\alpha\otimes x^+)=\sum_{y^-} \alpha\cdot \nu_{x^+,y^-}\otimes y^-.
$$
Here $(\nu_{x^+,y^-})$ is the continuation cocycle determined by $(H_s,J_s,\Xi)$.
\end{definition}

\subsubsection{Properties of continuation maps}

In this section we use the notation from Definition~\ref{def:continuation-map}, i.e., $\cF$ is a DG local system, $(H_\pm,J_\pm)$ are admissible pairs with enriched Floer data $\Xi_\pm$, $(H_s,J_s)$ is a regular admissible homotopy interpolating between $(H_\pm,J_\pm)$ at $\pm\infty$, and $\Xi$ is a monotone enriched Floer continuation datum that interpolates between $\Xi_\pm$.

We first discuss the compatibility of continuation maps with action filtrations. 

Recall that, by definition, $H_s$ is constant equal to $H_\pm$ in the neighborhood of $\pm\infty$. As such, we can always write $H_s^t=K_{\beta(s)}^t$ for some nondecreasing function $\beta:\R\to [0,1]$ which is constant equal to $0$ near $-\infty$ and constant equal to $1$ near $+\infty$, and where $K_\vartheta^t$, $\vartheta\in[0,1]$ is a smooth homotopy equal to $H_-^t$ at $0$ and equal to $H_+^t$ at $1$. Then $\p_sH_s^t=\beta'(s)\p_\vartheta K_{\beta(s)}^t$. 

\begin{proposition}[Compatibility of continuation maps with action filtrations] \label{prop:DGcontinuation-filtration} Let $H_s^t=K_{\beta(s)}^t$ as above. The map $$\Psi^\cF : FC_*(H_+,J_+,\Xi_+;\cF)\to FC_*(H_-,J_-,\Xi_-;\cF)$$ has the property that  
\begin{equation} \label{eq:continuation-filtration}
\Psi^\cF(FC_*^{<b}(H_+,J_+,\Xi_+;\cF))\subset FC_*^{<b+E}(H_-,J_-,\Xi_-;\cF) 
\end{equation}
for all $b\in\R$, with $E=\int_0^1\max_{x,\vartheta} \p_\vartheta K_\vartheta^t(x) \,dt$. 
\end{proposition}

\begin{proof} Given a continuation curve $u\in \mathcal H(x^+; y^-)$ associated to $(H_s,J_s)$, a standard computation involving the Stokes formula yields (see also~\S\ref{sec:eval_maps_homotopies})
\begin{align*}
  0&\leq \int_{\R}\int_{S^1}\|\p_su\|^2\,dt\,ds\\
   &=\cA_{H_+}(x^+)-\cA_{H_-}(y^-)+\int_\R\int_{S^1}(\p_sH_s)\circ u\,dt\,ds\\
   & =\cA_{H_+}(x^+)-\cA_{H_-}(y^-)+\int_\R\int_{S^1}\beta'(s)\p_\vartheta K_{\beta(s)}^t(u(s,t))\,dt\,ds \\
   & \leq \cA_{H_+}(x^+)-\cA_{H_-}(y^-) + \int_0^1\max_{x,\vartheta} \p_\vartheta K_\vartheta^t(x) \,dt.
\end{align*}
This implies the conclusion. 
\end{proof}

\begin{remark}
This is a classical computation that underlies the continuity of spectral invariants in Floer theory. As an illustration, the argument in the proof of Corollary~\ref{cor:DGcontinuation-Hofer-norm} below is also used in the proof of the continuity of the spectral invariant from Proposition~\ref{prop:spectral-invariant-W}. One comprehensive reference on this circle of ideas in the context of classical Floer theory is~\cite[Chapter~21]{Oh-Vol2}. 

The key point in our setting is that the modified homotopy $H_s^{\rho,C}$ that is used in order to evaluate the elements of $\cH(x^+;y^-)$ into Moore paths $\cP_{x^+\to y^-}\cL X$ does not interfere with the action filtration. (Recall that $\Xi=(\Upsilon,\rho,C)$ consists of an enriched Floer continuation datum $\Upsilon=(\sigma_{x^+,y^-},\cY,\theta)$, and of a pair $(\rho,C)$ with $\rho$ a cutoff function and $C>0$ a constant. The latter are used in order to define the modified homotopy $H_s^{\rho,C}=H_s-\rho(s)C$, whose continuation Floer trajectories coincide with those of the homotopy $H_s$.)
\end{remark}

\begin{corollary} \label{cor:DGcontinuation-monotone}
The continuation map $\Psi^\cF$ associated to a regular admissible monotone homotopy $H_s$ such that $\p_sH_s\le 0$ preserves the action filtration, i.e., 
$$
\Psi^\cF(FC_*^{<b}(H_+,J_+,\Xi_+;\cF))\subset FC_*^{<b}(H_-,J_-,\Xi_-;\cF) 
$$
for all $b\in\R$.
\end{corollary}

\begin{proof}
We use Proposition~\ref{prop:DGcontinuation-filtration} with $\p_\vartheta K_\vartheta^t(x)\le 0$, so that $E\le 0$.
\end{proof}

\begin{corollary} \label{cor:DGcontinuation-Hofer-norm}
For any two triples $(H_\pm,J_\pm,\Xi_{(H_\pm,J_\pm)})$ and any $\eps>0$, there exists a homotopy $(H_s,J_s,\Xi_{(H_s,J_s)})$ such that the induced  continuation map $\Psi^\cF$ satisfies 
$$
\Psi^\cF(FC_*^{<b}(H_+,J_+,\Xi_{(H_+,J_+)};\cF))\subset FC_*^{<b+E}(H_-,J_-,\Xi_{(H_-,J_-)};\cF),
$$ 
with $E=\int_{0}^1\max (H_+^t-H_-^t)\,dt +\eps$. 
\end{corollary}

\begin{proof}
We apply Proposition~\ref{prop:DGcontinuation-filtration} with $K_\vartheta^t$, $\vartheta\in[0,1]$ a generic homotopy that is $\eps$-close to $H_-^t + \vartheta (H_+^t-H_-^t)$ and $H_s^t=K_{\beta(s)}^t$, where $\beta:\R\to[0,1]$ is a smooth nondecreasing cutoff function as above. 
\end{proof}

We now discuss the chain homotopy type of continuation maps.

\begin{proposition}[Chain homotopy type of continuation maps] \label{prop:continuation_maps_homotopic} 
Given regular admissible Floer data $(H_\pm,J_\pm,\Xi_\pm)$ and a DG local system $\cF$, any two continuation maps 
$$
\Psi^\cF : FC_*(H_+,J_+,\Xi_+;\cF)\to FC_*(H_-,J_-,\Xi_-;\cF)
$$ 
are chain homotopic.
\end{proposition}

\begin{proof}
This is a reformulation of Proposition~\ref{prop:uniqueness_continuation}. 
\end{proof}

We now discuss continuation maps induced by ``constant" homotopies. Given an admissible pair $(H_0,J_0)$ we denote $\mathrm{Id}_{(H_0,J_0)}=(H_s,J_s)$ the constant homotopy equal to $(H_0,J_0)$ for all $s\in\R$. We denote $\star$ the basepoint of (the relevant path-component of) $\cL X$, and we view it as a constant based loop and further as a $0$-chain $\star\in C_0(\Omega \cL X)$.

\begin{proposition}[Continuation map induced by the identity] \label{prop:Id-Id}
Let $(H_0,J_0)$ be a regular admissible pair with enriched Floer datum $\Xi_0=(s_{x,y},\cY_0,\theta_0)$. 
\begin{enumerate}
\item The homotopy $\mathrm{Id}_{(H_0,J_0)}$ is regular, and the moduli spaces of continuation Floer trajectories coincide with the spaces of Floer trajectories for $(H_0,J_0)$, 
$$
\cH(x;y)=\wh\cM(x;y),\qquad x,y\in\Per(H).
$$
\item Any Barraud-Cornea continuation 
cocycle $(\nu_{x,y})$ associated to the homotopy $\mathrm{Id}_{(H_0,J_0)}$ is homologous in the sense of~\eqref{eq:homotopy-nuprime-nu} to the continuation cocycle 
$$
\nu'_{x,y}=\left\{\begin{array}{ll}\star,&x=y,\\0,&\mbox{ else}.\end{array}\right.
$$
\end{enumerate}
\end{proposition}

\begin{proof}
1. This is straightforward from the definitions. 

2. By Proposition~\ref{prop:uniqueness_continuation}, any two continuation cocycles induced by monotone continuation data with the same endpoints are homologous. It is therefore enough to construct one representing chain system $(\sigma_{x,y})$ for $\ol \cH(x;y)$ such that, with monotone enriched Floer datum given by $\cY_s=\cY_0$, $\theta_s=\theta_0$, and an arbitrary function $\rho$ and constant $C>0$, the corresponding continuation cocycle $(\nu_{x,y})$ is homologous to $(\nu'_{x,y})$.

The proof is very similar to~\cite[Lemmas~6.9--6.12]{BDHO} and we spell out the main steps. We denote $\pi:\cH(x;y)=\widehat \cM(x;y)\to\cM(x;y)$ the canonical projection, which induces $\ol\pi:\ol\cH(x;y)\to\ol\cM(x;y)$. 
\renewcommand{\theenumi}{\roman{enumi}}
\begin{enumerate}
\item If $|x|=|y|$, we have $\ol\cH(x;y)=\widehat \cM(x;y)=\varnothing$ if $x\neq y$, and $\ol\cH(x;x)=\widehat \cM(x;x)=\{c_x\}$, the constant Floer trajectory at $x$. The representing chain system $\sigma_{x,y}$ is therefore uniquely determined for such $x,y$. 
\item We complete the above choice to a continuation representing chain system by adding chains $\sigma_{x,y}$ for $|x|>|y|$. By ~\cite[Lemma~6.9]{BDHO}, we can choose $\sigma_{x,y}$ such that $\ol\pi_*(\sigma_{x,y})=0$ for all $|x|>|y|$. The heuristic idea is that, in the cubical complex, the degenerate cubes are by definition equal to $0$.
\item Consider the evaluation map $\ol q'_{x,y}:\ol\cH(x;y)\to \Omega\cL X$ defined by $\ol q'_{x,y}=\ol q_{0,x,y}\circ\ol\pi$, where $\ol q_{0,x,y}:\cM(x;y)\to \Omega \cL X$ is defined in~\eqref{eq:qxy}. These maps satisfy conditions~\eqref{eq:relation1} and~\eqref{eq:relation2} because the maps $\ol q_{0,x,y}$ do. As a consequence $(\ol q'_{x,y})_*(\sigma_{x,y})$ is a continuation cocycle, which clearly coincides with $\nu'_{x,y}$. 
\item One proves exactly as in~\cite[Lemma~6.10]{BDHO} that the two systems of evaluation maps $(\ol q'_{x,y})$ from above and $(\ol q_{x,y})$ determined by the choice of $\cY_s=\cY_0$, $\theta_s=\theta_0$, and some choice of $\rho$, $C$, are homotopic through a family of evaluation maps that each satisfy~\eqref{eq:relation1} and~\eqref{eq:relation2}. As a consequence of~\cite[Lemma~6.12]{BDHO}, the resulting continuation cocycles are homologous.
\end{enumerate}

\end{proof}

\begin{corollary} \label{cor:Id-Id} The continuation map induced by the identity $\mathrm{Id}_{(H_0,J_0)}$ is homotopic to the identity.
\end{corollary}

\begin{proof}
This is a straightforward consequence of the fact that the continuation map induced by the continuation cocycle $(\nu'_{x,y})$ from Proposition~\ref{prop:Id-Id} is the identity at chain level. 
\end{proof}

To close this section, we state a result on the composition of continuation maps. We will prove it in~\S\ref{sec:composition}. 

\begin{proposition}[Composition of continuation maps]  \label{prop:composition-v1}
Given three regular admissible homotopies $(H^{ij},J^{ij})$, $1\le i<j\le 3$ interpolating between regular admissible pairs $(H^i,J^i)$ at $+\infty$ and $(H^j,J^j)$ at $-\infty$, complemented by additional monotone continuation data $\Xi_{(H^{ij},J^{ij})}$, the chain maps $\Psi^{23}\circ\Psi^{12}$ and $\Psi^{13}$ are homotopic. 
\end{proposition}

\subsection{Homotopies} \label{sec:homotopies}
Let $(H_+,J_+)$, $(H_-,J_-)$ be regular admissible pairs, and let $\{(H^\tau,J^\tau)\}$, $\tau\in [0,1]$ be a regular admissible homotopy of homotopies interpolating from $(H_+,J_+)$ near $+\infty$ to $(H_-,J_-)$ near $-\infty$. Following~\S\ref{sec:conventions_Floer}, we associate to $\{(H^\tau,J^\tau)\}$ the parametrized moduli spaces of continuation Floer trajectories $\cH^{[0,1]}(x^+;y^-)$, $x^+\in\Per(H_+)$, $y^-\in\Per(H_-)$. The space $\cH^{[0,1]}(x^+;y^-)$ is a smooth manifold of dimension $|x^+|-|y^-|+1$.  By a procedure similar to the ones in~\S\ref{sec:twisting-cocycle} and~\S\ref{sec:continuation}, we will first construct a representing chain system $(\cS_{x^+,y^-})$ and then, by  evaluating into $\Omega \call X$,  a {\it parametrized continuation cocycle} $h_{x^+,y^-}\in C_{|x^+|-|y^-|+1}(\Omega \call X)$. 
Using the latter we will prove that a homotopy of homotopies as above yields continuation maps that are chain homotopic.

\subsubsection{Orientations} \label{sec:ori-homotopies}
The moduli spaces $\cH^{[0,1]}(x^+;x^-)$ admit Floer compactifications $\ol \cH^{[0,1]}(x^+;x^-)$ that are manifolds with boundary with corners. The codimension $1$ strata in the boundary $\p \ol \cH^{[0,1]}(x^+;y^-)$ are of the type 
$$
\cH^0(x^+;y^-) \mbox{ and } \cH^1(x^+;y^-), 
$$
$$
\cM(x^+;z^+)\times \cH^{[0,1]}(z^+;y^-) \quad \mbox{for }z^+\in\Per(H_+),
$$
and
$$
\cH^{[0,1]}(x^+;z^-)\times \cM(z^-;y^-) \quad \mbox{for }z^-\in\Per(H_-).
$$
Fixing a coherent set of orientations as in~\cite{FH-coherent}, we again need to understand for each of these strata what is the difference in sign between its coherent orientation (given by the product in the last two cases) and the boundary orientation. 

The moduli space $\cH^{[0,1]}(x^+;y^-)$ is the vanishing locus of a section of a Banach bundle as follows. Consider $\cB(x^+;y^-)$ the Banach manifold of maps $\R\times S^1\to X$ locally of class $W^{k,p}$, $p>2$, $k\ge 1$ and converging in $W^{k,p}$-norm to $x^+$ and $y^-$ at $+\infty$ and $-\infty$ respectively. Let $\cE\to \cB(x^+;y^-)$ be the Banach bundle of $W^{k-1,p}$-sections of $\Lambda^{0,1}T^*(\R\times S^1)\otimes u^*TX\simeq u^*TX$. We have a pull-back diagram of bundles 
$$
\xymatrix
@C=60pt
@R=40pt
{
[0,1]\times\cE = \mathrm{pr}_2^*\cE \ar[r]\ar[d]& \cE \ar[d]\\
[0,1]\times \cB(x^+;y^-) \ar[r]_{\mathrm{pr}_2}\ar@/^2pc/[u]^{\ol\p=\{\ol\p_{H^\tau,J^\tau}\}}& \cB(x^+;y^-)
}
$$
and $\cH^{[0,1]}(x^+;y^-)=\ol\p^{-1}(0)$. The regularity of the pair $\{(H^\tau,J^\tau)\}$ means that the vertical differential $D\ol\p^{vert}_{\tau,u}:T_\tau[0,1]\times T_u\cB(x^+;y^-)\to W^{k-1,p}(u^*TX)$ is surjective for any $(\tau,u)\in\ol\p^{-1}(0)$. This is a stabilization by the $1$-dimensional vector space $T_\tau[0,1]\simeq\R$ of the linearized operator $D_u^\tau$ at $u$ for the Floer equation corresponding to the pair $(H^\tau,J^\tau)$.\footnote{We have $D\ol\p^{vert}_{\tau,u}(\ell,\xi)=\frac{d\ol\p(\tau)}{d\tau}\cdot \ell + D_u^\tau\xi$.} As a consequence we infer a canonical isomorphism 
$$
|T_{(\tau,u)}\cH^{[0,1]}(x^+;y^-)|=|\ker D\ol\p^{vert}_{(\tau,u)}|=|\det D\ol\p^{vert}_{(\tau,u)}|\simeq |T_\tau[0,1]|\otimes |\det D_u^\tau|.
$$ 
Note that $T_\tau[0,1]$ is canonically identified with $\R$ and has a canonical orientation. Accordingly $|T_\tau[0,1]|$ is  identified with $\Z$ supported in degree $1$. 

After these preliminaries, we are now ready to discuss the orientation on the boundary of $\ol \cH^{[0,1]}(x^+;y^-)$.

$\bullet$ {\it The component $\cH^0(x^+;y^-)$.} There is a canonical identification $|T_0[0,1]|\simeq |\R\nu^{in}(0)|$, where $\nu^{in}(0)$ is an inward pointing vector in $T_0[0,1]$. Extending $\nu^{in}(0)$ to a neighborhood of $0\in [0,1]$ and lifting it to a neighborhood of $\cH^0(x^+;y^-)\subset \p \cH^{[0,1]}(x^+;y^-)$ we obtain a vector field that points away from the boundary. As a consequence we have canonical orientation preserving isomorphisms
\begin{align*}
|T_{(0,u)}\cH^{[0,1]}(x^+;y^-)|& \simeq |T_0[0,1]|\otimes |T_u\cH^0(x^+;y^-)|\\
& \simeq (-1) \cdot |\R\nu^{out}|\otimes  |T_u\cH^0(x^+;y^-)|,
\end{align*}
where $\nu^{out}$ is a vector field pointing towards the boundary in the neighborhood of $\cH^0(x^+;y^-)\subset \p \cH^{[0,1]}(x^+;y^-)$. Thus the boundary orientation on $\cH^0(x^+;y^-)$ differs from the coherent orientation by a sign $(-1)$. 

$\bullet$ {\it The component $\cH^1(x^+;y^-)$.} There is a canonical identification $|T_1[0,1]|\simeq |\R\nu^{out}(1)|$, where $\nu^{out}(1)$ is an outward pointing tangent vector in $T_1[0,1]$. By a similar reasoning we conclude that the boundary orientation on $\cH^1(x^+;y^-)$ coincides with the coherent orientation. 

$\bullet$ {\it The component $\cM(x^+;z^+)\times \cH^{[0,1]}(z^+;y^-)$, with $z^+\in\Per(H_+)$.} The gluing map 
$$
\wh \cM(x^+;z^+)\times \cH^{[0,1]}(z^+;y^-) \to \cH^{[0,1]}(x^+;y^-)
$$
is a diffeomorphism from its domain of definition to its image. As such, it induces an isomorphism of orientation local systems over the relevant connected components 
$$
|\wh \cM(x^+;z^+)|\otimes |\cH^{[0,1]}(z^+;y^-)|\to  |\cH^{[0,1]}(x^+;y^-)|.
$$
This isomorphism reverses the coherent orientations by a sign $(-1)^{|x^+|-|z^+|}$. Indeed, at a linear level the gluing map is modelled by the gluing of the non-stabilized Cauchy-Riemann operator $D_1$ arising from moduli space problem $\wh \cM(x^+;z^+)$ and of the stabilized Cauchy-Riemann operator $D_2^{V_2}$ arising from the moduli space problem $\cH^{[0,1]}(z^+;y^-)$, with $V_2=\R\simeq T_\tau[0,1]$ carrying its canonical orientation. In view of Lemma~\ref{lem:gluing-stabilized-CR}, the linear gluing map reverses the coherent orientations by a sign $(-1)^{\ind D_1} = (-1)^{|x^+|-|z^+|}$. 

Consider the canonical constant vector field $\p_{x^+,z^+}$ on $\wh \cM(x^+;z^+)$, viewed further under the gluing map as a vector field on $\cH^{[0,1]}(x^+;y^-)$. As in the previous sections one sees that this vector field points towards the boundary. Using~\eqref{eq:orientMMhat} we infer a sequence of orientation preserving isomorphisms 
\begin{align*}
(-1)^{|x^+|-|z^+|}|\cH^{[0,1]}(x^+;y^-)| & \simeq |\wh \cM(x^+;z^+)|\otimes |\cH^{[0,1]}(z^+;y^-)| \\
& \simeq |\R\p_{x^+,z^+}|\otimes |\cM(x^+;z^+)|\otimes |\cH^{[0,1]}(z^+;y^-)| \\
& \simeq |\R\nu^{\text{out}}|\otimes |\cM(x^+;z^+)|\otimes |\cH^{[0,1]}(z^+;y^-)|.
\end{align*}
Here the third isomorphism is notational: we observed that the vector field $\p_{x^+,z^+}$ points towards the boundary, and we record this by denoting it $\nu^{\text{out}}$. As a consequence the product orientation on $\cM(x^+;z^+)\times \cH^{[0,1]}(z^+;y^-)$ differs from its boundary orientation by a sign $(-1)^{|x^+|-|z^+|}$.

$\bullet$ {\it The component $\cH^{[0,1]}(x^+;z^-)\times \cM(z^-;y^-)$ for $z^-\in\Per(H_-)$.} We again consider the gluing map 
$$
\cH^{[0,1]}(x^+;z^-)\times \wh \cM(z^-;y^-) \to \cH^{[0,1]}(x^+;y^-),
$$
which is a diffeomorphism from its domain of definition to its image. As such, it induces an isomorphism of orientation local systems over the relevant connected components 
$$
|\cH^{[0,1]}(x^+;z^-)|\otimes |\wh \cM(z^-;y^-)| \to |\cH^{[0,1]}(x^+;y^-)|.
$$
This isomorphism preserves the coherent orientations. We reason as before: at the linear level the gluing map is modelled by the gluing of the stabilized Cauchy-Riemann operator $D_1^{V_1}$, $V_1=\R\simeq T_\tau[0,1]$ arising from the moduli space problem $\cH^{[0,1]}(x^+;z^-)$ and of the non-stabilized Cauchy-Riemann operator $D_2$ arising from the moduli space problem $\wh\cM(z^-;y^-)$. In view of Lemma~\ref{lem:gluing-stabilized-CR}, and since the second operator is not stabilized, the linear gluing map preserves the coherent orientations.

 We consider the canonical constant vector field $\p_{z^-,y^-}$ on $\wh \cM(z^-;y^-)$, also viewed as a vector field on $\cH^{[0,1]}(x^+;y^-)$ under the gluing map. As before we see that this vector field points away from the boundary. Using~\eqref{eq:orientMMhat} we infer a sequence of orientation preserving isomorphisms 
\begin{align*}
|\cH^{[0,1]}(x^+;y^-)| & \simeq |\cH^{[0,1]}(x^+;z^-)|\otimes |\wh \cM(z^-;y^-)| \\
& \simeq |\cH^{[0,1]}(x^+;z^-)|\otimes |\R\p_{z^-,y^-}|\otimes  |\cM(z^-;y^-)| \\
& \simeq (-1)^{|x^+|-|z^-|} |\R\nu^{out}|\otimes |\cH^{[0,1]}(x^+;z^-)|\otimes |\cM(z^-;y^-)| .
\end{align*}
The last isomorphism is again notational: the vector field $\p_{z^-,y^-}$ points away from the boundary and, if $\nu^{\text{out}}$ is a vector field pointing towards the boundary, the canonical isomorphism $|\R\p_{z^-,y^-}|\simeq |\R\nu^{\text{out}}|$ reverses orientation and therefore has sign $-1$. As a consequence the product orientation on $\cH^{[0,1]}(x^+;z^-)\times \cM(z^-;y^-)$ differs from its boundary orientation by a sign $(-1)^{|x^+|-|z^-|}$. 

Summing up we obtain that the boundary of the oriented Floer compactification of the parametrized moduli spaces of continuation Floer trajectories with respect to a system of coherent orientations as in~\cite{FH-coherent} can be written as 
\begin{align*}
\p \ol \cH^{[0,1]}(x^+;y^-) = & -\cH^0(x^+;y^-) \cup \cH^1(x^+;y^-) \\
& \cup \quad \bigcup_{z^+} (-1)^{|x^+|-|z^+|}\ol \cM(x^+;z^+) \times \ol \cH^{[0,1]}(z^+;y^-) \\
&  \cup \quad \bigcup_{z^-}(-1)^{|x^+|-|z^-|}\ol \cH^{[0,1]}(x^+;z^-)\times  \ol \cM(z^-;y^-).
\end{align*}

\subsubsection{Representing chain systems} 

\begin{proposition}[Existence of representing chain systems for parametrized continuation Floer moduli spaces] \label{representingchainparametrizedFloercontinuation} Fix a regular admissible homotopy of homotopies $\{(H^\tau,J^\tau)\}$, $\tau\in [0,1]$  between $(H_+,J_+)$ at $+\infty$ and $(H_-,J_-)$ at $-\infty$ as above. Choose representing chain systems $(s_{x^\pm,y^\pm})$ for the Floer moduli spaces of $(H_\pm,J_\pm)$, and subordinated representing chain systems $(\sigma^i_{x^+,y^-})$, $i=0,1$ for the continuation Floer moduli spaces determined by the homotopies $(H^i, J^i)$.

Given $x^+\in\Per(H_+)$ and $y^-\in\Per(H_-)$, let $C_*(\ol\cH^{[0,1]}(x^+;y^-))$ be the complex of cubical chains on the compactified moduli space of parametrized continuation Floer trajectories. There exists a collection 
$$
\{\cS_{x^+,y^-}\, : \, x^+\in \Per(H_+),\, y^-\in\Per(H_-)\}
$$ 
with $\cS_{x^+,y^-}\in C_{|x^+|-|y^-|+1}(\ol\cH^{[0,1]}(x^+;y^-))$ satisfying the following properties:  
\begin{enumerate}
\item $\cS_{x^+,y^-}$ is a cycle rel boundary and represents the fundamental class\break $[\ol\cH^{[0,1]}(x^+;y^-)]$. 
\item  there holds the relation
\begin{align*}
\partial \cS_{x^+,y^-}\, =\,  \sigma^1_{x^+,y^-} & - \sigma^0_{x^+, y^-} + \sum_{z^+} (-1)^{|x^+|-|z^+|} s_{x^+,z^+} \times \cS_{z^+,y^-}\\
& + \sum_{z^-}(-1)^{|x^+|-|z^-|} \cS_{x^+,z^-} \times s_{z^-,y^-},
\end{align*}
where the product of chains is defined via the inclusions 
$$\ol\cM(x^+;z^+)\times\ol\cH^{[0,1]}(z^+;y^-) \, \subset \, \partial \ol\cH^{[0,1]}(x^+;y^-)\, \subset\, \ol\cH^{[0,1]}(x^+;y^-)$$
and 
$$\ol\cH^{[0,1]}(x^+;z^-)\times\ol\cM(z^-;y^-) \, \subset \, \partial \ol\cH^{[0,1]}(x^+;y^-)\, \subset\, \ol\cH^{[0,1]}(x^+;y^-).$$
\end{enumerate}
\end{proposition} 

\begin{proof}
The proof is very similar to that of~\cite[Theorem~6.6]{BDHO} and we omit the details.  
\end{proof}

\subsubsection{Evaluation maps}

Let $\{(H^\tau,J^\tau)\}$, $\tau\in [0,1]$ be a regular admissible homotopy of homotopies between $(H_+,J_+)$ at $+\infty$ and $(H_-,J_-)$ at $-\infty$. Assume we have chosen enriched Floer data $\Xi_\pm=(s_{x^\pm,y^\pm},\cY_\pm,\theta_\pm)$, subordinated monotone enriched Floer continuation data $\Xi^i=(\sigma^i_{x^+,y^-}, \cY^i, \theta^i,\rho^i,C^i)$, $i=0,1$, and a subordinated representing chain system $\cS_{x^+,y^-}$ for the parametrized moduli spaces $\ol\cH^{[0,1]}(x^+,y^-)$. We complement this data with the choice of a continuous family $(\cY^\tau_s,\theta^\tau_s,\rho^\tau,C^\tau)$, $\tau\in[0,1]$ that interpolates between $(\cY^0_s,\theta^0_s,\rho^0,C^0)$ and $(\cY^1_s,\theta^1_s,\rho^1,C^1)$, where $\cY^\tau_s$ are embedded trees rooted at the basepoint, $\theta^\tau_s$ are homotopy inverses for the projections $\cL X\to \cL X/\cY^\tau_s$, and the function $\rho^\tau$ together with the constant $C^\tau$ satisfy the requirements from the previous section for the homotopy $H^\tau$. 
Let $\Upsilon^\tau=(\cS_{x^+,y^-},\cY^\tau,\theta^\tau)$, $\tau\in[0,1]$. We refer to $\{\Upsilon^\tau\}$ as an \emph{enriched parametrized Floer continuation datum}, and to $\{\Xi^\tau=(\Upsilon^\tau,\rho^\tau,C^\tau)\}$ as a \emph{monotone enriched parametrized Floer continuation datum}.  

The construction from~\S\ref{sec:eval_maps_homotopies}  
provides continuous evaluation maps 
$$
\ev_{x^+,y^-}:\cH^{[0,1]}(x^+;y^-)\to\cP_{x^+\to y^-}\cL X
$$
and 
$$
q_{x^+,y^-}:\cH^{[0,1]}(x^+;y^-)\to \Omega \cL X.
$$
These extend to the compactification and we denote the extension of the second map by 
$$
\ol q_{x^+,y^-}:\ol\cH^{[0,1]}(x^+;y^-)\to \Omega \cL X.
$$

\subsubsection{Parametrized continuation cocycle and chain homotopy}

Given a representing chain system $(\cS_{x^+,y^-})$ we define 
$$
h_{x^+,y^-}=(\ol q_{x^+,y^-})_*(\cS_{x^+,y^-})\in C_{|x^+|-|y^-|+1}(\Omega \cL X).
$$

\begin{definition} We call a family $(h_{x^+,y^-})$ obtained in this way \emph{a Barraud-Cornea
parame\-trized continuation cocycle subordinated to the twisting cocycles $(m_{x^+,y^+})$, $(m_{x^-,y^-})$ and to the continuation cocycles $(\nu^i_{x^+,y^-})$ for $i=0,1$}. 
\end{definition}

By abuse of language we will also refer to $(h_{x^+,y^-})$ as \emph{the} parametrized continuation cocycle subordinated to the twisting cocycles $(m_{x^+,y^+})$, $(m_{x^-,y^-})$ and to the continuation cocycles $(\nu^i_{x^+,y^-})$ for $i=0,1$. The next result states that it indeed satisfies the equation of Definition \ref{def:general-parametrized}. It is a straightforward consequence of Proposition~\ref{representingchainparametrizedFloercontinuation}.

\begin{proposition} \label{prop:cocycle_homotopy_of_homotopies}
Any parametrized continuation cocycle satisfies the equation 
\begin{align*}
\partial h_{x^+,y^-}\, =\,  \nu^1_{x^+,y^-} & - \nu^0_{x^+, y^-} + \sum_{z^+} (-1)^{|x^+|-|z^+|} m_{x^+,z^+} h_{z^+,y^-}\\
& + \sum_{z^-}(-1)^{|x^+|-|z^-|} h_{x^+,z^-} m_{z^-,y^-}.
\end{align*}
\qed
\end{proposition}

\begin{remark}
One can also phrase a uniqueness statement for parametrized continuation cocycles, similar to the uniqueness statement for continuation cocycles. This would involve continuation cocycles parametrized by the square $[0,1]^2$. The latter are uniquely determined up to continuation cocycles parametrized by the cube $[0,1]^3$, and so on. Higher hierarchies of this kind have already made their appearance in the literature, see for example~\cite{CO} or~\cite{Varolgunes}, and a complete homotopical treatment in the context of Morse theory was given by Mazuir~\cite{Mazuir2}. 
\end{remark}

\begin{remark} As explained in~\cite[\S2]{BDHO}, and recalling the notation from Remark~\ref{rmk:cont-MC}, the parametrized continuation cocycle $(h_{x^+,y^-})$ can be viewed as a degree $1$ element in the $\Hom$ complex $\Hom((R_*\otimes C^+_\bullet,D^+),(R_*\otimes C^-_\bullet,D^-))$ whose boundary is the difference of the degree $0$ continuation cycles at the endpoints of the parametrizing interval. 
\end{remark}

\begin{definition} Let $\cF$ be a DG right $C_*(\Omega \cL X)$-module and $(H_\pm,J_\pm)$ be regular admissible pairs with enriched Floer data $\Xi_\pm$.  Let $(H^i_s,J^i_s)$, $i=0,1$ be two regular admissible homotopies interpolating between $(H_+,J_+)$ at $+\infty$ and $(H_-,J_-)$ at $-\infty$, with monotone enriched continuation Floer data $\Xi^i$ interpolating between $\Xi_\pm$. Let $(H^\tau,J^\tau,\Xi^\tau)$, $\tau\in [0,1]$ be a regular admissible homotopy with monotone parametrized Floer data interpolating between $\Xi^0$ and $\Xi^1$. The \emph{chain homotopy induced by $(H^\tau,J^\tau,\Xi^\tau)$} is denoted by
$$
h:FC_*(H_+,J_+,\Xi_+;\cF)\to FC_{*+1}(H_-,J_-,\Xi_-;\cF),
$$
and it is defined with respect to the canonical bases $\Per(H_+)$, $\Per(H_-)$ by 
$$
h(\alpha\otimes x^+)=(-1)^{|\alpha|}\sum_{y^-} \alpha\cdot h_{x^+,y^-}\otimes y^-.
$$
Here $(h_{x^+,y^-})$ is the parametrized continuation cocycle for $(H^\tau,J^\tau,\Xi^\tau)$. 
\end{definition}

This map satisfies 
$$
\Psi^1-\Psi^0=D^- h + h D^+,
$$
where $\Psi^i$, $i=0,1$ are the continuation maps defined by $(H^i_s,J^i_s, \Xi^i,\rho^i,C^i)$, and $D^\pm$ are the differentials on the corresponding Floer chain complexes with DG coefficients. 

\subsubsection{The composition of continuation maps} \label{sec:composition}

We close this section with a proof of Proposition~\ref{prop:composition-v1}. 

\begin{proof}[Proof of Proposition~\ref{prop:composition-v1}]
For $R\gg 0$ let $H^{13}_R=(H^{13}_{R,s})$ be the homotopy given by 
$$
H^{13}_{R,s}=\left\{\begin{array}{ll}
H^{12}_{s-R},& s\ge 0,\\
H^{23}_{s+R},& s\le 0.
\end{array}\right.
$$
Let $J^{13}_R$ be the homotopy of almost complex structures defined by the analogous formula involving homotopies $J^{12}$ and $J^{23}$.  
 
Fix $R_0\gg 0$ and consider the moduli space 
$$
\cH_{[R_0,\infty)}(x^1;y^3)=\bigcup_{R\ge R_0} \cH(x^1;y^3;H^{13}_R,J^{13}_R).
$$
The moduli spaces $\cH_{[R_0,\infty)}(x^1;y^3)$ admit Floer compactifications that we denote $\ol\cH_{[R_0,\infty)}(x^1;y^3)$. These are manifolds with boundary with corners. The codimension $1$ strata in the boundary $\p \ol\cH_{[R_0,\infty)}(x^1;y^3)$ are of the type 
$$
\cH(x^1;y^3;H^{13}_{R_0},J^{13}_{R_0}), 
$$
$$
\cH(x^1;z^2;H^{12},J^{12})\times \cH(z^2;y^3;H^{23},J^{23}) \quad \mbox{for }z^2\in\Per(H^2),
$$
$$
\cM(x^1;z^1)\times \cH_{[R_0,\infty)}(z^1;y^3) \quad \mbox{for }z^1\in\Per(H^1),
$$
and
$$
\cH_{[R_0,\infty)}(x^1;z^3)\times \cM(z^3;y^3) \quad \mbox{for }z^3\in\Per(H^3).
$$
The first type of boundary corresponds to the boundary $R=R_0$ for the parametrizing interval $[R_0,\infty)$, the second type of boundary corresponds to $R\to\infty$, whereas the third and fourth types of boundary correspond to Floer breaking at values $R\in (R_0,\infty)$. 

Fixing a coherent set of orientations as in~\cite{FH-coherent}, we again need to understand for each of these strata what is the difference in sign between its coherent orientation (given by the product in the last three cases) and the boundary orientation. 
 
The first, third, and fourth kind of boundary strata are analyzed as in~\S\ref{sec:ori-homotopies}, and we find that the sign difference is given by $-1$, $(-1)^{|x^1|-|z^1|}$, and respectively $(-1)^{|x^1|-|z^3|}$.

To analyze the component $\cH(x^1;z^2;H^{12},J^{12})\times \cH(z^2;y^3;H^{23},J^{23})$ we proceed as in~\S\ref{sec:orientations-twisting-cocycle} (and with the same caveat as in Remark~\ref{rmk:remark-on-gluing}). For $R\ge R_0\gg 0$ there is a gluing map
$$
\cH(x^1;z^2;H^{12},J^{12})\times \cH(z^2;y^3;H^{23},J^{23})\to \cH(x^1;y^3;H^{13}_R,J^{13}_R)
$$
that is a diffeomorphism onto its image and that induces a canonical isomorphism of orientation local systems  
$$
|\cH(x^1;z^2;H^{12},J^{12})|\otimes |\cH(z^2;y^3;H^{23},J^{23})|\simeq |\cH(x^1;y^3;H^{13}_R,J^{13}_R)|. 
$$
As a consequence, the relative sign of the orientations is $+1$. 

Summing up we obtain that the boundary of the oriented Floer compactification of the moduli spaces $\cH_{[R_0,\infty)}(x^1;y^3)$  with respect to a system of coherent orientations as in~\cite{FH-coherent} can be written as 
\begin{align*}
\p \ol \cH_{[R_0,\infty)}(x^1;y^3) = & -\ol\cH(x^1;y^3;H^{13}_{R_0},J^{13}_{R_0}) \\
& \cup \quad \bigcup_{z^2} \ol\cH(x^1;z^2;H^{12},J^{12})\times \ol\cH(z^2;y^3;H^{23},J^{23}) \\
& \cup \quad \bigcup_{z^1} (-1)^{|x^1|-|z^1|}\ol \cM(x^1;z^1) \times \ol \cH_{[R_0,\infty)}(z^1;y^3) \\
&  \cup \quad \bigcup_{z^3}(-1)^{|x^1|-|z^3|}\ol \cH_{[R_0,\infty)}(z^1;y^3)\times  \ol \cM(z^3;y^3).
\end{align*}
We choose on the product $\ol\cH(x^1;z^2;H^{12},J^{12})\times \ol\cH(z^2;y^3;H^{23},J^{23})$ as representing chain system the product of representing chain systems for the factors. Upon evaluation, we find that the resulting continuation cocycle is 
$(\nu^{12}\cdot \nu^{23})_{x^1,y^3}=\sum_{z^2}\nu^{12}_{x^1,z^2}\cdot \nu^{23}_{z^2,y^3}$. 
Reasoning as in the proof of Proposition~\ref{prop:cocycle_homotopy_of_homotopies}, we obtain that the continuation cocycle $\nu^{13}_{R_0}$ associated to the homotopy $(H^{13}_{R_0},J^{13}_{R_0})$ is homologous to the cocycle $(\nu^{12}\cdot \nu^{23})$. The former is homologous to the continuation cocycle for the homotopy $(H^{13},J^{13})$ by Proposition~\ref{prop:cocycle_homotopy_of_homotopies} and therefore induces a map chain homotopic to $\Psi^{13}$, whereas the latter induces at chain level the composition of continuation maps $\Psi^{12}\circ \Psi^{23}$. As a consequence, $\Psi^{13}$ is chain homotopic to $\Psi^{12}\circ \Psi^{23}$.
\end{proof}

\subsection{Symplectically aspherical manifolds} \label{sec:sympl-asph}

In this section $(X,\omega)$ is either a closed symplectically atoroidal manifold, in which case we allow ourselves to work in any free loop space component, or a closed symplectically aspherical manifold, in which case we work in the component of contractible loops.

\subsubsection{Definition and computation} 

For a closed symplectic manifold all Hamiltonians and compatible almost complex structures are admissible. As a consequence, Proposition~\ref{prop:continuation_maps_homotopic} can be strengthened as follows. 

\begin{proposition}[Continuation maps are chain homotopy equivalences] \label{prop:continuation_maps_homotopy_equiv} 
\qquad 

Let $X$ be a closed symplectically atoroidal manifold. Given regular Floer data $(H_\pm,J_\pm,\Xi_\pm)$ and a DG local system $\cF$, any continuation map 
$$
\Psi^\cF : FC_*(H_+,J_+,\Xi_+;\cF)\to FC_*(H_-,J_-,\Xi_-;\cF)
$$ 
is a chain homotopy equivalence. 
\end{proposition}

\begin{proof} This is a direct consequence of Propositions~\ref{prop:Id-Id} and~\ref{prop:composition-v1}. Indeed, if the Floer data for the continuation map $\Psi^\cF$ is $(H_s,J_s,\Xi_s)$, we consider $H'_s=H_{-s}$, $J'_s=J_{-s}$, and complement this with monotone Floer continuation data $\Xi'$. Then, by the above results, the continuation map $\Psi^{\prime\cF}$ induced by $(H'_s,J'_s,\Xi')$ is a homotopy inverse of $\Psi^\cF$. 
\end{proof}

Proposition~\ref{prop:continuation_maps_homotopy_equiv} motivates the shorthand notation 
$$
FH_*(X;\cF)
$$
to designate Floer homology with coefficients in a DG-local system $\cF$, without specific mention of the set of Floer data used in the definition.

Let $i_X:X\to \cL_0 X$ be the inclusion of $X$ into the connected component of contractible free loops, which associates to $x\in X$ the constant loop at $x$. This induces a DGA map $i_{X*}:C_*(\Omega X)\to C_*(\Omega \cL_0 X)$, and therefore any DG local system $\cF$ on $\cL_0 X$ induces a DG local system $i_X^*\cF$ on $X$ by viewing it as a $C_*(\Omega X)$-module via $i_{X*}$. Let $H_*(X;i_X^*\cF)$ be the (Morse) homology with DG local coefficients given by $i_X^*\cF$, see~\S\ref{sec:DGMorse} and~\cite{BDHO}.

\begin{theorem} \label{thm:Floer-Morse} Let $X$ be a closed symplectically atoroidal manifold of dimension $2n$, and $\cF$ a DG local system on $\cL X$. We have canonical isomorphisms
$$
FH_*(X;\cF)\simeq H_{*+n}(X;i_X^*\cF)
$$
in the component of contractible loops, and
$$
FH_*(X;\cF)=0
$$
in the other components of $\cL X$. 
\end{theorem}

\begin{proof}
By Proposition~\ref{prop:continuation_maps_homotopy_equiv} it is enough to prove the statement for some appropriate choice of $(H,J,\Xi)$. Let $H$ be time-independent and equal to a $C^2$-small Morse function. For such a Hamiltonian all $1$-periodic orbits are constants given by the critical points of $H$, and in particular contractible. As a consequence, Floer homology vanishes in all the components of $\cL X$ other than that of contractible loops because the Floer complex has an empty set of generators. 

We now focus on the component of contractible loops. By a theorem of Salamon and Zehnder~\cite[Theorem~7.3]{SZ92}, any time-independent almost complex structure $J$ that defines a Morse-Smale gradient flow for $H$ is Floer regular. Moreover, for such a pair $(H,J)$ all Floer trajectories are time-independent. 

If $u:\R\times S^1\to X$ is time-independent, then it solves $\p_su+J(u)(\p_t u -X_H(u))=0$ if and only if it solves $\p_su=J X_H(u)=\nabla (-H)(u)$. Therefore Floer trajectories for $(H,J)$, i.e., \emph{positive} $L^2$-gradient trajectories for the functional $A_H$, are in one-to-one bijective correspondence with \emph{positive} gradient trajectories of $-H$. Accordingly, the twisting cocycle and the Floer differential are obtained from the compactified moduli spaces of \emph{negative} gradient trajectories of $-H$. 

The coherent orientation of the Floer moduli spaces $\cM(x;y)$ that we described in~\S\ref{sec:conventions_Floer} and~\S\ref{sec:orientations-twisting-cocycle} is equivalent to the analytic coherent orientation of the moduli spaces $\cL(x,y)$ of negative gradient trajectories of $-H$, as described in~\cite[Appendix~A.2]{BDHO}. This is explained as follows. 
\begin{itemize}
\item The construction of coherent orientations in Floer theory takes as input the choice, for each $1$-periodic orbit $x\in \Per(H)$, of an orientation of the determinant line bundle over the space $\cD_x^+$ of Cauchy-Riemann operators on the Riemann sphere with one positive puncture (identified with $\C$) and asymptotic behavior at that puncture given by the linearization of the Hamiltonian flow at $x$ (see~\cite{BOauto}, and also~\cite{Bourgeois-Mohnke}). We denote such an operator by $D^+_x$. 
This choice induces orientations for moduli spaces of Floer trajectories by requiring that the gluing isomorphisms preserve the orientations. The resulting orientations are coherent as a consequence of an associativity property of the gluing isomorphisms. 

\item In an analogous manner, the construction of analytic coherent orientations in Morse theory as described in~\cite[Appendix~A.2]{BDHO} takes as input the choice, for each critical point $x$ of a Morse function $f$, of an orientation of the determinant line bundle over the space $\cD^u_x$ of operators $D^u_x:W^{1,2}((-\infty,0],\R^{2n})\to L^2((-\infty,0],\R^{2n})$ of the form $D^u_x=\p_s+A(s)$, where $A:(-\infty,0]\to Sym_{2n}(\R)$ is a path of symmetric matrices such that $\lim_{s\to-\infty} A(s)=A_x$, the matrix that expresses the Hessian of $f$ at $x$ in a given trivialization.  This choice induces orientations for moduli spaces of Morse trajectories by requiring that the gluing isomorphisms preserve the orientations. The resulting orientations are coherent as a consequence of an associativity property of the gluing isomorphisms.\footnote{In~\cite{BDHO} we have included the choice of orientations of $\det \cD^u_x$ as part of the auxiliary data needed to define DG Morse homology. Strictly speaking, we should have included the choice of orientations of $\det \cD^+_x$ as part of the auxiliary data needed to define DG Floer homology, but we omitted it since we did not need to examine the dependence on that choice.}

\item In our situation, by choosing $H$ to be time-independent and $C^2$-small, and $J$ to be time-independent and such that the corresponding Riemannian metric is Morse-Smale for $H$, the following holds: for each $x\in\Crit(-H)=\Per(H)$, there is a canonical correspondence between orientations of the determinant bundle over the space $\cD_x^+$ and orientations of the determinant bundle over the space $\cD^u_x$. This is because, by defining a Cauchy-Riemann operator on $\C$ from radial Floer data obtained by cutting off to zero in the neighborhood of the origin the Hamiltonian $H$, the arguments of~\cite[Theorem~7.3]{SZ92} carry over in order to show that the corresponding operators $D^+_x$ and $D^u_x$ are both surjective and their kernels are in canonical bijective correspondence (elements of $\ker D^+_x$ are radial, and correspond to elements of $\ker D^u_x$). 

\item By choosing orientations of determinant bundles over $\cD^+_x$ and $\cD^u_x$ so that they correspond under the previous equivalence, the recipes to induce orientations on the Floer moduli spaces, resp. on the Morse moduli spaces, are the same, and therefore these orientations are the same under the canonical identification $\cM(x;y)\equiv \cL(x,y)$.  
\end{itemize}

To conclude, we choose the enriched Floer data for $(H,J)$ to be equal to the enriched Morse data 
for the pair consisting of the Morse function $-H$ and its negative gradient vector field, where the Morse representing chain system is oriented by its analytic orientation. Then, with our convention that inputs of Floer cylinders are taken at $+\infty$, the Floer complex with coefficients in $\cF$ is canonically identified with the Morse complex of $-H$ with coefficients in $i_X^*\cF$, and this proves the isomorphism between DG Floer homology and Morse homology up to a shift in the grading. 

The relationship between the two gradings is a general property of the Conley-Zehnder index: the Morse index of a critical point of $-H$ is equal to the Conley-Zehnder index of that critical point seen as a $1$-periodic orbit of $X_H$, plus $n$. This is proved for example in~\cite[Corollary~7.2.2]{Audin-Damian_English} (the Hamiltonian vector field in \emph{loc.\@ cit.}\@ is the opposite of ours). 
\end{proof}

\subsubsection{Canonical spectral sequence} \label{sec:spectral-sequence-sympl-asph}

Let $(X,\omega)$ be a symplectically aspherical manifold of dimension $2n$, and let $\cF$ be a DG local system on $\cL_0 X$. 

Let $(H,J)$ be a regular pair with enriched Floer data $\Xi_{(H,J)}$, and recall the notation $C_\bullet=\langle \Per(H)\rangle$.

\begin{definition} \label{defi:canonical_filtration} The \emph{canonical filtration} on $FC_*(H,J,\Xi_{(H,J)};\cF)=\cF\otimes C_\bullet$ is given by 
\begin{equation} \label{eq:canonical_filtration}
F_p\, =\, \bigoplus_{i\leq p}\cF\otimes \cC_i,\qquad p\in\Z.
\end{equation}
\end{definition}

\begin{theorem}[Canonical spectral sequence] \label{thm:spectral_sequence} The spectral sequence $E^r_{p,q}$ associated to the  filtration~\eqref{eq:canonical_filtration} converges to $FH_*(H,J,\Xi_{(H,J)};\cF)$ and has $E^2$-page 
\begin{equation} \label{eq:second-page}
E^2_{p,q}=FH_p(H,J;H_q(\cF)), 
\end{equation}
the classical Floer homology group of the pair $(H,J)$ with coefficients in the (classical) local system $H_q(\cF)$ on $\cL_0 X$.  

The continuation map induced by a regular homotopy between regular Floer data 
is the limit of a map induced between the corresponding spectral sequences, which moreover is an isomorphism starting with the $E^2$-page. 
\end{theorem}

\begin{proof}
Since $C_\bullet$ is finite dimensional, and therefore supported in a finite range of degrees, 
the spectral sequence $(E_{p,q}^r, d^{r})$ associated to the canonical filtration converges to $FH_*(H,J,\Xi_{(H,J)};\cF)$.

That the second page has the expression~\eqref{eq:second-page} is proved by unravelling the definitions, exactly as in~\cite[Lemma~4.3 and Theorem~7.2]{BDHO}. 

It is clear that any continuation map preserves the filtrations at chain level, and therefore induces a map between the spectral sequences that correspond to the Floer data at the endpoints. That the continuation map is an isomorphism at the $E^2$-page follows from the identification of the second page with Floer homology with local coefficients, and the fact that continuation maps induce isomorphisms in Floer homology with (classical) local coefficients.   
\end{proof}

As explained in~\cite[\S7]{BDHO}, any Serre fibration $F\hookrightarrow E \to \cL_0 X$ determines canonically up to homotopy a DG local system $\cF$ with underlying chain complex $C_*(F)$, the cubical chains on the fiber. Moreover, the DG local system $i_X^*\cF$ is canonically identified with the one determined by the fibration $E|_X=i_X^*E$.

We now have all the ingredients to prove Theorem~\ref{thm:FH_for_fibrations-intro} from the Introduction.

\begin{proof}[Proof of Theorem~\ref{thm:FH_for_fibrations-intro}]
By Theorem~\ref{thm:spectral_sequence} the Floer spectral sequence is canonical, in the sense that continuation maps induce isomorphisms between the spectral sequences that correspond to different Floer data. By Theorem~\ref{thm:Floer-Morse} we can choose the Floer data such that the Floer complex gets identified with the Morse complex, and therefore the Floer spectral sequence gets identified with the Morse spectral sequence. Finally, the result follows from the Fibration theorem for DG Morse homology~\cite[Theorem~A]{BDHO}. 
\end{proof}

\begin{example} Let $(X,\omega)$ be a closed symplectically aspherical manifold of dimension $2n$.
We illustrate Theorem~\ref{thm:FH_for_fibrations-intro} by choosing $\cF=C_*(\Omega \cL_0 X)$ to be the DG local system defined by the path-loop fibration for $\cL_0 X$, i.e., 
$$
\Omega \cL_0 X \hookrightarrow \cP_\star \cL_0 X \to \cL_0 X. 
$$
Here $\cP_\star \cL_0 X$ is the space of  paths in $\cL_0 X$ starting at a basepoint $\star\in X$, and the map to $\cL_0 X$ is given by evaluation at the endpoint. By Theorem~\ref{thm:FH_for_fibrations-intro} we have  
$$
FH_*(H,J,\Xi;\cF_X)\simeq H_{*+n}(\cP_\star \cL_0 X|_X).
$$
\end{example}

\begin{lemma} \label{lem:kpi1} Let $Q$ be a closed manifold. We have 
$$
H_*(\cP_\star\mathcal{L}_0Q|_Q)\simeq H_*(\mathrm{pt})
$$
if and only if $Q$ is a $K(\pi,1)$.
\end{lemma}

\begin{proof}
The space $\cP_\star\mathcal{L}_0Q|_Q$ can be identified with the double based loop space $\Omega^2Q$. This implies 
$
\pi_i \cP_\star\mathcal{L}_0Q|_Q \simeq \pi_{i+2} Q
$ 
for all $i\ge 0$. 

$\Rightarrow$. Assume $H_*(\cP_\star\mathcal{L}_0Q|_Q)\simeq H_*(\mathrm{pt})$ and assume by contradiction that $Q$ is not a $K(\pi,1)$. Let $j\ge 2$ be minimal such that $\pi_jQ\neq 0$. 
\begin{itemize}
\item If $j\ge 4$ then $H_{j-2} (\cP_\star\mathcal{L}_0Q|_Q)=\pi_{j-2} \cP_\star\mathcal{L}_0Q|_Q =\pi_j Q\neq 0$, a contradiction.
\item If $j=3$ then $\pi_1 \cP_\star\mathcal{L}_0Q|_Q = \pi_3Q$ is nontrivial and abelian, therefore $H_1(\cP_\star\mathcal{L}_0Q|_Q)=\pi_1 \cP_\star\mathcal{L}_0Q|_Q\neq 0$, a contradiction. 
\item If $j=2$ then $\pi_0\cP_\star\mathcal{L}_0Q|_Q=\pi_2 Q\neq 0$, hence $\cP_\star\mathcal{L}_0Q|_Q$ has more than one path-connected component and therefore $H_0(\cP_\star\mathcal{L}_0Q|_Q)\neq H_0(\mathrm{pt})$, a contradiction again. 
\end{itemize}

$\Leftarrow$. If $Q=K(\pi,1)$ then the space $\Omega_0 Q$ of contractible based loops has trivial homotopy groups and is therefore contractible. By the homotopy long exact sequence of the fibration $\Omega_0 Q\hookrightarrow \mathcal{L}_0 Q\to Q$, the projection $\cL_0 Q\to Q$ is a weak homotopy equivalence, hence a homotopy equivalence. As a consequence $\cP_\star\mathcal{L}_0Q|_Q\simeq \cP_\star\mathcal{L}_0Q$ is contractible  and its homology is that of a point. 
\end{proof}

\begin{remark} The previous lemma is also the basis of our applications to cotangent bundles, with a more involved analysis of the topology of the space $\Omega^2Q$. We find it useful to present it in its simplest form as above.
\end{remark}

\begin{remark}
Since the DG local system $\cF=C_*(\Omega\cL_0 X)$ is canonically associated to the manifold $X$, the previous example can be interpreted as saying that Floer homology with DG coefficients detects closed symplectically aspherical manifolds that are $K(\pi,1)$'s. However, this is not as strong a statement as it may seem: given a symplectically aspherical manifold $X$ with fundamental group $\pi$, the Floer homology of $X$ with local (classical) coefficients in $\Z[\pi]$ is isomorphic to the homology of the universal cover of $X$, which in turn is isomorphic to the homology of a point if and only if $X$ is a $K(\pi,1)$. Therefore classical local coefficients already detect $K(\pi,1)$'s.
\end{remark}

\subsection{Symplectic homology} \label{sec:symplectic_homology}

Our conventions for symplectic homology are those of~\cite{CO}. The theory goes back to Cieliebak, Floer, Hofer, Wysocki~\cite{FH94,CFH95,CFHW} and, in its current classical version, to Viterbo~\cite{Viterbo99}. We refer to~\cite{CO} for a comprehensive list of references and to~\cite{Seidel07} for an influential early survey. Our goal in this section is to explain how to upgrade symplectic homology theory to DG local coefficients and to emphasize some specific features in that context. This construction has already appeared in a much more restrictive context in the work of Zhou~\cite{Zhou-ring} under the name ``Symplectic homology of sphere bundles". Symplectic homology with classical local coefficients has already been used extensively in symplectic topology, not in the least because such local coefficients appear naturally in the case of cotangent bundles, see~\S\ref{sec:DGViterbo}. 

We work on Liouville domains and their symplectic completions, which we now introduce. A \emph{Liouville domain} is a triple $(W,\omega,\lambda)$ with $W$ a compact manifold with boundary and $\omega$ an exact symplectic form with primitive $\lambda$, such that the restriction of $\lambda$ to $\p W$ is a contact form that defines the boundary orientation of $\p W$. The last condition is equivalent to the requirement that the \emph{Liouville vector field} $Z$ defined by $\iota_Z\omega=\lambda$ points outwards along $\p W$. We denote by $\alpha=\lambda|_{\p W}$ the induced contact form on the boundary, and $R_\alpha$ the Reeb vector field on $\p W$ defined by $\iota_{R_\alpha}\alpha=1$ and $\iota_{R_\alpha}d\alpha=0$. The \emph{symplectic completion of $W$} is defined by gluing the \emph{positive half-symplectization} $((1-\eps,\infty)\times\p W, d(r\alpha), r\alpha)$ to $(W,\omega,\lambda)$ via the diffeomorphism $((1-\eps,1]\times\p W,d(r\alpha),r\alpha)\stackrel\simeq\longrightarrow (\cN(\p W),\omega,\lambda)$, $(r,p)\mapsto \varphi_Z^{\ln r}(p)$, where $\varphi_Z^t$ is the flow of the Liouville vector field $Z$ and $\cN(\p W)$ denotes an unspecified neighborhood of $\p W$. We denote the symplectic completion by $(\hat W,\hat \omega,\hat\lambda)$. Note that the previous map also provides an embedding of the full \emph{negative half-symplectization} $((0,1)\times\p W, d(r\alpha),r\alpha)$ into $W$.  Outstanding examples of Liouville domains are the so-called \emph{Weinstein domains}~\cite[Chapter~11]{Cieliebak-Eliashberg-book}, of which unit disc cotangent bundles of closed manifolds are a particular, and important, class. 

We now discuss admissible Hamiltonians and almost complex structures. 

A Hamiltonian $H:S^1\times \hat W\to \R$ is said to be \emph{linear} if, for $r\ge 1$ large enough, it is a time-independent affine function of $r$. The slope of that affine function is called \emph{the slope at infinity}.\footnote{In the context of symplectic completions of Liouville domains, ``at infinity" means ``for $r\ge 1$ large enough".}
An important example are time-independent Hamiltonians that are constant on $W$ and equal on $[1,\infty)\times\p W$ to a function $h(r)$ that is affine for $r\gg 1$. Our conventions for the Hamiltonian vector field imply that $X_h=h'(r)R_\alpha$, so that $1$-periodic orbits of $h$ are in one-to-one correspondence with closed Reeb orbits of period $h'(r)$ (parametrized backwards if $h'(r)<0$, respectively constant if $h'(r)=0$). A Hamiltonian $H:S^1\times \hat W\to \R$ is \emph{admissible for symplectic homology} if it is linear with positive slope at infinity and nondegenerate. Note that the nondegeneracy condition implies that the slope at infinity cannot be equal to the period of a closed Reeb orbit on $\p W$. 

A time-dependent compatible almost complex structure $J=(J_t)$, $t\in S^1$ is called \emph{admissible} if, at infinity, it is time-independent and cylindrical, i.e., its restriction to the contact distribution $\xi=\ker \alpha$ is independent of $r$, and $J_{(r,p)}r\frac{\p}{\p r}=R_\alpha(p)$. Given an admissible Hamiltonian $H$, the pair $(H,J)$ is \emph{regular} for a generic admissible $J$. 

An \emph{admissible homotopy of Hamiltonians} is a homotopy $H_s$, $s\in\R$ of admissible Hamiltonians that is \emph{monotone}, i.e., such that $\p_sH(t,x)\le 0$ for all $(t,x)\in S^1\times\hat W$.\footnote{To establish \emph{a priori} $C^0$-bounds on Floer continuation trajectories one only needs to assume that, in the region where the Hamiltonians are affine of the form $a_sr+b_s$, we have $\p_s a_s\le 0$. However, we restrict to monotone homotopies in order to obtain action filtered symplectic homology groups. Accessorily, this also simplifies the exposition.} An \emph{admissible homotopy of almost complex structures} is a homotopy $J_s$ such that each $J_s$ is an admissible almost complex structure. Given regular pairs $(H_\pm,J_\pm)$, 
the regularity condition on admissible homotopies $(H_s,J_s)$ that interpolate between $(H_\pm,J_\pm)$ at $\pm\infty$ is generic.
 
Given an admissible pair $(H,J)$, there exists a compact set such that all Floer trajectories for $(H,J)$ are contained in that compact set. Given an admissible homotopy $(H_s,J_s)$, there exists a compact set such that all continuation Floer trajectories for $(H_s,J_s)$ are contained in a compact set. This is the basic fact that makes symplectic homology well-defined, and the vast literature on \emph{a priori} $C^0$-bounds goes back to Viterbo's foundational paper~\cite{Viterbo99}.

Classical symplectic homology groups are defined as 
$$
SH_*(W)=\colim_{(H,J)} FH_*(H,J),
$$
where $(H,J)$ are admissible regular pairs subject to the condition $H\le 0$ on $W$, and the directed system is defined by continuation maps induced by monotone homotopies. Heuristically, the symplectic homology groups compute the Floer homology groups of a Hamiltonian that is $0$ on $W$ and $\infty$ outside $W$. Since the Floer differential and the Floer continuation maps preserve the free homotopy classes of loops, symplectic homology groups split as a direct sum $SH_*(W)=\bigoplus_a SH_{*,a}(W)$ indexed over the free homotopy classes $a$ of loops in $\hat W$. 

Denote by $\cL_a W$ and $\cL_a\hat W$ the connected components of the free loop spaces of $W$, respectively $\hat W$, indexed by free homotopy classes $a$. As in~\S\ref{sec:conventions_Floer}, we pick a basepoint $\gamma_a$ in each $\cL_a W\subset \cL_a \hat W$, and we assume that $\gamma_0$ is a constant loop contained in $W$. Since $W\hookrightarrow \hat W$ is a strong deformation retract, the same holds for the induced inclusion $\cL W\hookrightarrow \cL \hat W$ and $C_*(\Omega \cL W)\hookrightarrow C_*(\Omega \cL \hat W)$ is a chain homotopy equivalence of DGAs. We will therefore identify in the sequel DG local systems on $\cL W$ and on $\cL \hat W$.  

Let $\cF$ be a DG local system on $\cL W$, meaning the data of a DG local system on each connected component $\cL_a W$. We define the \emph{symplectic homology groups of $W$ with coefficients in $\cF$} as 
$$
SH_*(W;\cF)=\colim_{(H,J,\Xi)} FH_*(H,J,\Xi;\cF).
$$
Here the direct limit runs over regular admissible pairs $(H,J)$ such that $H\le 0$, complemented by enriched Floer data $\Xi$. Just as in the case of constant coefficients, we have a direct sum decomposition over free homotopy classes of loops in $\hat W$ given by $SH_*(W;\cF)=\bigoplus_a SH_{*,a}(W;\cF)$.

A special feature of continuation maps induced by monotone homotopies is that they preserve the action filtration, and consequently we can define action filtered symplectic homology groups. Recall from~\S\ref{sec:DGFloercomplex} the action filtered DG Floer chain complex 
$$
FC_*^{<b}(H,J,\Xi;\cF)=\cF\otimes \langle \Per^{<b}(H)\rangle, \quad b\in\ol\R,
$$
where $\Per^{<b}(H)$ is the set of $1$-periodic orbits of $H$ with action $<b$. 
For $-\infty\le b<c\le\infty$ define 
$$
FC_*^{(b,c)}(H,J,\Xi;\cF)=FC_*^{<c}(H,J,\Xi;\cF)/FC_*^{<b}(H,J,\Xi;\cF),
$$
and denote further by $FH_*^{(b,c)}(H,J,\Xi;\cF)$ the homology of this complex.  Then the continuation maps from~\S\ref{sec:continuation} induced by monotone regular admissible homotopies induce \emph{action filtered continuation maps}
$$
\Psi^{(b,c)}:FH_*^{(b,c)}(H_+,J_+,\Xi_+;\cF)\to FH_*^{(b,c)}(H_-,J_-,\Xi_-;\cF).
$$
When $H_+=H_-$, the map $\Psi^{(b,c)}$ is a (canonical) isomorphism for all $b,c$. Therefore, we will often drop $J$ and $\Xi$ from the notation and simply write $FH_*^{(b,c)}(H;\cF)$.
Finally, the \emph{action filtered symplectic homology groups} are defined as 
$$
SH_*^{(b,c)}(W;\cF)=\colim_{H} FH_*^{(b,c)}(H;\cF),\qquad -\infty \le b<c\le \infty. 
$$

Observe that for any Hamiltonian $H$ that is linear at infinity with slope $c$ not equal to the period of a closed Reeb orbit on $\p W$, we have a canonical isomorphism 
  \begin{equation}
    \label{eq:isom-SH-filtered-FH}
    \Psi_H:FH_*(H;\cF) \stackrel{\sim}\longrightarrow SH_*^{(-\infty, c)}(W;\cF).
  \end{equation}
This is because $SH_*^{(-\infty, c)}(W;\cF)$ can be expressed as the Floer homology of a Hamiltonian $H_c:S^1\times \hat W\to\R$ that is linear of slope $c$ for $r\ge 1$ and such that all the $1$-periodic orbits of $H_c$ have action $<c$. For such a Hamiltonian we have $FH_*^{<c}(H_c;\cF)=FH_*(H_c;\cF)$, and the isomorphism~\eqref{eq:isom-SH-filtered-FH} is realized by a continuation map from $H$ to $H_c$. 

If $b\leq b'$ and $c\leq c'$, inclusions of subcomplexes induce maps $FH_*^{(b,c)}(H;\cF)\to FH_*^{(b',c')}(H;\cF)$, which in turn induce $SH_*^{(b,c)}(W;\cF)\to SH_*^{(b',c')}(W;\cF)$. For $b\leq b'< 0$, the map  $SH_*^{(b,c)}(W;\cF)\to SH_*^{(b',c)}(W;\cF)$ is an isomorphism. Therefore, we denote $SH^{<c}(W;\cF)$ the homology $SH_*^{(b,c)}(W;\cF)$ for any $b< 0$. 

Moreover, we have
  \begin{equation}
    \label{eq:SH-limit-filtered}
    SH_*(W;\cF)=\colim_{c\to\infty}SH_*^{<c}(W;\cF).
  \end{equation}

The following two particular cases are especially significant. Let $\eps>0$ be smaller than the period of any closed Reeb orbit on $\p W$, and define \emph{the action zero part of symplectic homology with DG coefficients}
$$
SH_*^{=0}(W;\cF)=SH_*^{(-\eps,\eps)}(W;\cF).
$$
Similarly, define \emph{the positive symplectic homology with DG coefficients}
$$
SH_*^{>0}(W;\cF)=SH_*^{(\eps,\infty)}(W;\cF).
$$

The following results are proved just like their classical counterparts, see for example~\cite{CO}. 

\begin{proposition}[Action zero] \label{prop:SH-action-zero}
We have an isomorphism
$$
SH_*^{=0}(W;\cF)\simeq H_{*+n}(W,\p W; i_W^*\cF), 
$$
where $i_W: W \hookrightarrow \call_0 W$ is the inclusion of the constant loops.
\qed
\end{proposition}

\begin{proposition}[Tautological long exact sequence]
We have a long exact sequence 
\begin{align*}
\cdots \to H_{*+n}(W,\p W& ;i_W^*\cF)\to SH_*(W;\cF)\to SH_*^{>0}(W;\cF) \\
& \to H_{*-1+n}(W,\p W;i_W^*\cF)\to \cdots
\end{align*}
\qed
\end{proposition}

\begin{proposition}[Viterbo transfer]
Associated to an exact codimension $0$ symplectic embedding $\iota:V\hookrightarrow W$ there is a canonical \emph{Viterbo transfer} homomorphism 
$$
\iota_!:SH_*(W;\cF)\to SH_*(V;\iota^*\cF),
$$
with $\iota^*\cF$ the local system on $\cL V$ induced by the inclusion $\cL V\hookrightarrow \cL W$. 
\qed
\end{proposition}

\begin{remark}
One of the remarkable features of symplectic homology in the setting of constant coefficients is that it carries a ring structure. This is also the case for large classes of classical local coefficients of rank $1$, see~\cite{Abouzaid-cotangent} and~\cite[Appendix~A]{CHO-MorseFloerGH}. In the case of DG coefficients such a ring structure can be exhibited provided $\cF$ is a DGA such that the multiplication is suitably compatible with the module structure. This circle of ideas is related to the construction in~\cite{Albers-Frauenfelder-Oancea}  and is explored in ongoing work of Riegel.
\end{remark}

\begin{remark}
We focused in this subsection on symplectic homology with DG local coefficients, but similar definitions can be given in the context of other homology theories for non-compact manifolds, all encompassed by the setup of pairs of Liouville cobordisms with fillings~\cite{CO}. Special cases of particular significance are symplectic cohomology and Rabinowitz Floer homology and cohomology. 

Other homology theories can also be enriched with DG local coefficients, for example the relative symplectic cohomology of Varolgunes~\cite{Varolgunes}, or wrapped Floer cohomology~\cite{Abouzaid-Seidel}.
\end{remark}

An important computational device for symplectic homology with DG coefficients is provided by the following canonical spectral sequence (see also~\S\ref{sec:spectral-sequence-sympl-asph}).  

\begin{theorem} \label{thm:spectral-sequence-SH} Let $W$ be a Liouville domain and $\cF$ a DG local system on $\cL W$. 
\begin{enumerate}
\item[(i)] There is a canonical spectral sequence $E^r_{p,q}$, $r\ge 2$ with second page
\begin{equation} \label{eq:second-page-SH}
E^2_{p,q}=SH_p(W;H_q(\cF)), 
\end{equation}
the symplectic homology group with coefficients in the (classical) local system $H_q(\cF)$ on $\cL W$. 
\item[(ii)] Assume that one of the following conditions holds:  
\begin{center}
(A) The homology $H_*(\cF)$ is bounded from below and the symplectic homology $SH_*(W;H_q(\cF))$ is bounded from below, uniformly in $q$. 

(B) The homology $H_*(\cF)$ is bounded from above, and the symplectic homology $SH_*(W;H_q(\cF))$ is bounded from above, uniformly in $q$. 
\end{center}
Then the spectral sequence converges to $SH_*(W)$. 
\end{enumerate}
\end{theorem}

\begin{proof}
(i) Let $(H,J)$ be a regular admissible pair and $\Xi_{(H,J)}$ be a choice of enriched Floer data. By Theorem~\ref{thm:spectral_sequence} there is a spectral sequence, denoted $E^r_{p,q}(H,J,\Xi_{(H,J)};\cF)$, $r\ge 2$, that converges to $FH_*(H,J,\Xi_{(H,J)};\cF)$ and whose second page is $E^2_{p,q}(H,J,\Xi_{(H,J)};\cF)=FH_p(H,J;H_q(\cF))$. We define
$$
E^r_{p,q}=\colim_H E^r_{p,q}(H,J,\Xi_{(H,J)};\cF).
$$ 
Since the colimit functor is exact, this is a spectral sequence, with the meaning that $E^{r+1}=H(E^r)$. Its second page is 
\begin{align*}
E^2_{p,q}=\colim_H E^2_{p,q}(H,J,\Xi_{(H,J)};\cF) & = \colim_H FH_p(H,J;H_q(\cF)) \\
& = SH_p(W;H_q(\cF)).
\end{align*}
(ii) The hypothesis ensures that the second page of the spectral sequence is contained up to shift either in the first quadrant (condition (A)) or in the third quadrant (condition (B)). This implies convergence for dimensional reasons. 
\end{proof}

\section{DG Viterbo isomorphism} \label{sec:DGViterbo}

We prove in this section a Viterbo isomorphism theorem with DG coefficients, stated in the Introduction as Theorem~\ref{thm:Viterbo_iso_loc_sys_DG}. Our proof is based on Abouzaid's proof for the case of classical local coefficients~\cite{Abouzaid-cotangent}. 

\subsection{Preliminaries} \label{sec:DGViterbo-context}

Let $Q$ be a closed manifold and $T^*Q$ its cotangent bundle. We denote by $D^*Q$ and $S^*Q$ the disc cotangent bundle, respectively the sphere cotangent bundle with respect to some Riemannian metric on $Q$. We describe in this section the local system $\underline{\eta}$ from Theorem~\ref{thm:Viterbo_iso_loc_sys_DG}, and we state a result that explains the behavior of the Viterbo isomorphism at energy zero. 

Following~\cite[Chapter~11]{Abouzaid-cotangent} the \emph{fundamental local system} $\underline\eta$ is defined as the tensor product of three local systems of rank 1 free $\Z$-modules on $\cL Q$ supported in degree $0$,  
$$
\underline\eta=\sigma\otimes \underline\mu\otimes \tilde o,
$$ 
where: 
\begin{itemize}
\item The local system $\sigma$ is obtained by transgressing the second Stiefel-Whitney class $w_2=w_2(TQ)$: the monodromy along a loop $S^1\to \cL Q$, identified by adjunction to a map $S^1\times S^1\to Q$, is $\one$ if $w_2$ evaluates trivially on that $2$-torus, respectively $-\one$ if $w_2$ evaluates nontrivially. 
\item The local system $\underline\mu=\ev_0^*\underline{|Q|}^{-1}$, with $\underline{|Q|}$ the orientation local system of $Q$ seen as supported in degree $0$.\footnote{We have $\underline{|Q|}^{-1}\simeq \underline{|Q|}$, but we use the inverse in the notation in agreement with~\cite{Abouzaid-cotangent}. See also the discussion below about graded classical local systems.} 
\item The local system $\tilde o=\ev_0^*\underline{|Q|}^{-w}$, with 
\begin{equation} \label{eq:w}
w:\cL Q\to \{0,-1\}
\end{equation} 
the locally constant function that, on a connected component $\cL_{[\gamma]}Q$, takes the value $0$ if $\gamma^*TQ$ is orientable, respectively the value $-1$ if $\gamma^*TQ$ is nonorientable. 
\end{itemize}

In light of the discussion on orientation conventions from~\S\ref{sec:orientations}, it is actually useful to allow classical local systems to be supported in arbitrary integer degrees. Given any local system $F$ we denote by $F[k]$ the local system obtained by shifting the degree \emph{down} by $k$, i.e., $F[k]_*=F_{*+k}$. If $F$ is supported in degree $d$ then $F[d]$ is supported in degree $0$ and we denote it by $\underline{F}$. With our conventions the orientation local system $|Q|$ is naturally supported in degree $n=\dim Q$, the local system $|Q|^{-1}$ is supported in degree $-n$ and $\underline{|Q|}^{-1}=(|Q|[n])^{-1}=(|Q|^{-1})[-n]$ is supported in degree $0$. We denote for future reference $\mu=\ev_0^*|Q|^{-1}$ and 
$$
\eta=\sigma\otimes \mu\otimes\tilde o,
$$
both being local systems of rank 1 free $\Z$-modules supported in degree $-n$. 

{\bf Notational convention.} Recall that we denote by $\Pi:\cL T^*Q\to\cL Q$ the map between free loop spaces induced by the projection $\pi:T^*Q\to Q$. 
We will use in the sequel the pull-back $\Pi^*w:\cL T^*Q\to \{0,-1\}$ and, for readability, we will abuse notation and denote it also by $w$. Similarly, we will use the notation $\eta$ for the pull-back $\Pi^*\eta$.

The next result describes the behavior of the Viterbo isomorphism at energy zero. Denote $i_Q:Q\to \cL Q$ the inclusion of constant loops and note that $i_Q^*\ueta\simeq \underline{|Q|}^{-1}\simeq \underline{|Q|}$, the orientation local system of $Q$. 

\begin{proposition} \label{prop:Viterbo-iso-energy-zero} Let $\cF$ be a DG local system on $\cL Q$. The isomorphism $\widetilde\Psi$ from Theorem~\ref{thm:Viterbo_iso_loc_sys_DG} induces an isomorphism 
$$
{\widetilde\Psi}^{=0}:SH_*^{=0}(T^*Q;\Pi^*\cF)\stackrel\simeq\longrightarrow H_*(Q;i_Q^*\cF\otimes \underline{|Q|}) 
$$
that fits into the commutative diagram 
$$
\xymatrix{
SH_*^{=0}(T^*Q;\Pi^*\cF)\ar[r] \ar[d]_\simeq^{{\widetilde\Psi}^{=0}} & SH_*(T^*Q;\Pi^*\cF) \ar[d]_\simeq^{\widetilde\Psi} \\
H_*(Q;i_Q^*\cF\otimes\underline{|Q|}) \ar[r]^-{i_{Q*}} & H_*(\cL Q;\cF\otimes \ueta),
}
$$
where the top horizontal arrow is the tautological map in symplectic homology. 
 
Under the identification $$SH_*^{=0}(T^*Q;\Pi^*\cF)\simeq H_{*+n}(D^*Q,S^*Q;i_{D^*Q}^*\Pi^*\cF)=H_{*+n}(D^*Q,S^*Q; \pi^*i_Q^*\cF)$$ from Proposition~\ref{prop:SH-action-zero}, the isomorphism ${\widetilde\Psi}^{=0}$ coincides with the Thom isomorphism .  \end{proposition} 

We prove Proposition \ref{prop:Viterbo-iso-energy-zero} in Section~\S\ref{sec:DGViterbo-proof}.

\rmk\label{rmk:Thom-explication}  In the DG setting the Thom isomorphism   
is the shriek map $$ j_{Q!} : H_{*+n}(D^*Q,S^*Q; {\cal G})\ri H_*(Q;  j_Q^*{\cal G}\otimes \underline{|Q|})$$ of the inclusion $ j_Q:Q\hookrightarrow D^*Q$, as defined in \cite[\S9.3, \S10.4  and \S12.4]{BDHO} for any DG local system ${\cal G}$ on $D^*Q$.  Since $\pi\circ  j_Q=\Id$, its inverse is $\pi_!$ and, setting ${\cal G}=  i_{D^*Q}^*\Pi^*\cF$, we get
$$\pi_! : H_*(Q;i_Q^*\cF\otimes\underline{|Q|}) \ri H_{*+n}(D^*Q,S^*Q; \pi^*i_Q^*\cF),$$
the inverse of the isomorphism that we identified with $\widetilde{\Psi}^{=0}$ in the proposition above. 
\kmr

\subsection{Canonical grading for Floer homology of cotangent bundles}

Given a closed manifold $Q$, the disc cotangent bundle with respect to any Riemannian metric on $Q$ is a Liouville domain, and the full cotangent bundle is identified with its symplectic completion. In local coordinates $(p,q)$ the symplectic form is $\omega=dp\wedge dq=d(pdq)$, the Liouville form is $pdq$ and the Liouville vector field is $p\frac{\partial}{\partial p}$. The symplectic homology groups with DG coefficients $SH_*(T^*Q;\cF)$ are defined as in~\S\ref{sec:symplectic_homology}. One specific feature is the existence of a canonical $\Z$-grading for symplectic homology of cotangent bundles, which we now briefly discuss following~\cite[\S9.4.5]{Abouzaid-cotangent}. 

There are two key observations that lead to this canonical grading. The first observation is that 
\begin{equation} \label{eq:complexification}
TT^*Q\simeq \pi^*TQ\otimes\C
\end{equation}
is the complexification of a real vector bundle. 
This follows from the fact that $TT^*Q\simeq \pi^*TT^*Q|_Q$ together with $TT^*Q|_Q\simeq TQ\otimes \C$. In particular  $TT^*Q$ is isomorphic to its conjugate bundle and its first Chern class is $2$-torsion, i.e., $2c_1(TT^*Q)=0$. 

Homotopy classes of trivializations of $TT^*Q$ along a given loop are in one-to-one bijective correspondence with homotopy classes of trivializations of $\det TT^*Q$. This holds generally for symplectic vector bundles, and is implied by the fact that the inclusion $U(1)\hookrightarrow \Sp(2n)$ induces an isomorphism on $\pi_1$. 

Passing to determinant lines in~\eqref{eq:complexification} we find $\det TT^*Q\simeq \det(\pi^*TQ)\otimes\C$. Thus any unitary trivialization of $\det TT^*Q$ along a loop determines a \emph{Gauss map} $S^1\to\R P^1$ defined by the fibers of $\det(\pi^*TQ)$. Since the covering map $U(1)\to\R P^1$ has degree $2$, the degree of the Gauss map changes by an even integer upon changing the homotopy class of trivialization, and the second key observation is that there is a unique up to homotopy trivialization such that the degree of the Gauss map is either $0$ or $1$ (the two cases correspond to $\det(\pi^*TQ)$ being orientable or not along the loop). 

The outcome of this discussion is that, for any loop in $\cL TT^*Q$, there is a canonical up to homotopy trivialization of $TT^*Q$ along that loop. 

Given a nondegenerate $1$-periodic orbit $x\in\Per(H)$ one associates to it a Conley-Zehnder index $\CZ(x)$ computed from the path of symplectic matrices determined by the linearization of the Hamiltonian flow in the above trivialization. 
We define \emph{the degree of $x$} by 
\begin{equation} \label{eq:degreex}
|x|=\CZ(x)+w(x),
\end{equation}
where $w:\cL Q\to\{0,-1\}$ is the function defined in~\eqref{eq:w}.
 
\begin{remark}
The remarkable property of this grading is that it is compatible with product structures~\cite{Abouzaid-cotangent}. Although we do not address product structures in this paper, it is important to use the ``correct grading". 
\end{remark}

\begin{remark}
Our degree is equal to $n$ minus the degree defined by Abouzaid in~\cite[Definition~9.4.20]{Abouzaid-cotangent}. We made this choice so that the Viterbo isomorphism is degree preserving. 
\end{remark}

\subsection{Finite dimensional approximation for free loops}  \label{sec:DGViterbo-finite-dim}

We describe in this subsection the construction of a finite dimensional approximation for the space of free loops $\cL Q$ due to Abouzaid. We follow \emph{verbatim}~\cite[\S11.2]{Abouzaid-cotangent}.  

Fix a metric on $Q$ with injectivity radius larger than $4$ and denote by $d$ the associated distance. For each integer $r\ge 1$ and each $1\le i\le r$ we consider the functions 
$$
\rho_i:Q^r\to\R,\qquad \rho_i(q_0,\dots,q_{r-1})=d(q_i,q_{i-1}),
$$
with the convention $q_r=q_0$. 

Fix $1\ge \delta>0$ sufficiently small and a collection $\bolddelta^r=\{\delta_1^r,\dots,\delta_r^r\}$ such that 
\begin{equation} \label{eq:delta}
\frac{\delta}2<\delta_i^r<\delta.
\end{equation}
We define  
$$
\cL_{\bolddelta^r}^rQ=\{\q\in Q^r\, : \, \rho_i(\q)\le \delta_i^r,\ i=1, \dots, r\}.
$$
For a generic choice of the collection $\bolddelta^r$ the space $\cL_{\bolddelta^r}^rQ\subset Q^r$ is a codimension $0$ smooth submanifold with boundary with corners. We choose a family of generic parameters $\bolddelta^r$, $r\ge 1$ such that, in addition, 
$$
\delta_i^r\le \delta_j^{r+1} \quad \mbox{ for all } 1\le i\le r \mbox{ and } 1\le j\le r+1,
$$
and denote 
$$
\cL^rQ=\cL_{\bolddelta^r}^rQ.
$$

There are natural embeddings 
$$
\iota^r:\cL^rQ\to \cL^{r+1}Q,\qquad \iota^r(q_0,\dots,q_{r-1})=(q_0,q_0,\dots,q_{r-1})
$$
that make the collection of spaces and maps $(\cL^r Q,\iota^r)$ into a directed system.\footnote{Here the role of $q_0$ is somewhat artificial. For each $1\le k\le r-1$, the embedding $\cL^rQ\hookrightarrow \cL^{r+1}Q$, $(q_0,\dots,q_{r-1})\mapsto (q_0,\dots,q_{k-1},q_k,q_k,q_{k+1},\dots,q_{r-1})$ is homotopic to $\iota^r$.} 

There are also natural maps 
$$
\geo^r:\cL^rQ\to \cL Q
$$
that send a collection of points to the piecewise geodesic loop connecting them, parametrised at unit speed. (We work with Moore free loops.) Clearly 
\begin{equation} \label{eq:geo-iota}
\geo^{r+1}\circ \iota^r=\geo^r,
\end{equation}
so that the collection of maps $(\geo^r)$ induces a map 
$$
\colim \geo^r:\colim \cL^r Q\to \cL Q. 
$$ 
This map is a homotopy equivalence~\cite[Proposition~11.2.4]{Abouzaid-cotangent}. 

Finally, there is a continuous map 
$$
\ev_r:\cL Q\to Q^r, \qquad \gamma\mapsto (\gamma(0),\gamma(\tfrac{1}{r}L_\gamma),\dots,\gamma(\tfrac{r-1}{r}L_\gamma)),
$$
where the loop $\gamma$ is parametrized on the interval $[0,L_\gamma]$.
Since $\ev_r$ is continuous, its restriction to any compact set takes values in $\cL^r Q$ 
 for $r$ large enough. 

\begin{lemma} \label{lem:cLL1Q}
For $L\ge 0$, we denote by $\cL^{\le L}Q \subset \cL Q$ the subspace of loops of length $\le L$, and by $\cL^{\le L}_1Q\subset \cL^{\le L}Q$ the subspace of loops parametrized at unit speed. Then,
\begin{enumerate}
\item The inclusion $\cL^{\le L}_1Q\subset \cL^{\le L}Q$ is a strong deformation retract. 
\item There exists $r(L)>0$ such that, for $r\ge r(L)$, the image of the map $\ev_r$ restricted to $\cL^{\le L}_1Q$ belongs to $\cL^r Q$, and the composition $\geo^r\circ \ev_r:\cL^{\le L}_1Q\to \cL^{\le L}_1Q$ is homotopic to the identity. 
\end{enumerate}
\end{lemma}

\begin{proof}
1. Denote $\cL^{=L}Q$ the space of loops of length $L$ and $\cL^{=L}_1Q\subset\cL^{=L}Q$ the subspace of loops of length $L$ parametrized at unit speed. Consider the intermediate space $\cL^{=L}_1Q\subset \underline{\cL}^{=L}Q\subset \cL^{=L}Q$ of loops of length $L$ parametrized on the interval $[0,L]$. 
\begin{itemize}
\item The inclusion $\underline{\cL}^{=L}Q\subset \cL^{=L}Q$ is a strong deformation retract since the group $\R^*_+$ of reparametrizations that rescale the speed uniformly is contractible.  The retraction $\cL^{=L}Q\to \underline{\cL}^{=L}Q$ associates to a loop $\gamma$ the unique reparametrization $\gamma(\sigma\cdot)$, $\sigma>0$ whose interval of definition is $[0,L]$.
\item The inclusion $\cL^{=L}_1Q\subset \underline{\cL}^{=L}Q$ is a strong deformation retract because, for a given loop $\gamma$, the possible parametrizations are in bijective correspondence with nondecreasing maps $[0,L]\to [0,L]$ that are equal to $0$, resp. $L$ at the endpoints, and the space of such maps is contractible because it is starshaped with respect to the identity map. (A given parametrization determines such a map by associating to a parameter $t\in[0,L]$ the length of $\gamma|_{[0,t]}$, and the identity map corresponds to the unique reparametrization of the loop at unit speed.)  The retraction $\underline{\cL}^{=L}Q\to \cL^{=L}_1Q$ associates to a loop its unique reparametrization at unit speed.
\end{itemize}
 The above retractions are canonical and, by considering all lengths $\le L$ simultaneously, they fit into a retraction $\cL^{\le L}Q\to \cL^{\le L}_1Q$.

2. It is enough to choose $r(L)$ large enough such that $L/r(L)<\delta/2$, where $\delta>0$ is defined as in~\eqref{eq:delta}. This ensures that the image of $\ev_r|_{\cL^{\le L}_1Q}$ belongs to $\cL^r Q$. That the composition $\geo^r\circ\ev_r$ is homotopic to the identity on $\cL^{\le L}_1Q$ follows from the fact that any path of length $<\delta/2$ is canonically homotopic to the unique geodesic path connecting its endpoints. See also~\cite[Proposition~11.2.4]{Abouzaid-cotangent} for a similar argument. 
\end{proof}

The previous discussion carries over word for word componentwise on the free loop space. For each free homotopy class of loops $a$, which determines a connected component $\cL_aQ$, we set 
$$
\cL^r_aQ=\{\q\in\cL^rQ\, : \, \geo^r(\q)\in\cL_aQ\}\subset \cL^rQ
$$
and 
$$
\geo^r_a=\geo^r|_{\cL^r_aQ}:\cL^r_aQ\to\cL_aQ.
$$
Since $\geo^r$ is continuous over $\cL^rQ$, each subspace $\cL^r_aQ\subset \cL^rQ$ is a union of connected components. Moreover
$$
\colim \geo_a^r:\colim \cL_a^r Q\to \cL_a Q 
$$
is a homotopy equivalence. 

A statement similar to Lemma~\ref{lem:cLL1Q} holds for the map $\ev^a_r=\ev_r|_{\cL_aQ}$. We omit the details.

\subsection{DG Morse homology in the finite dimensional approximation} \label{sec:DGViterbo-Morse-DG}

Our definition of DG Morse homology in~\cite{BDHO} (see~\S\ref{sec:DGMorse}) uses in an essential way basepoints for the connected components of the spaces under study, and we now discuss these. For each connected component $\cL_aQ$ we choose as basepoint a geodesic loop $\gamma_a$ parametrized by arc length on an interval $[0,L_a]$, with $L_a$ the length of $\gamma_a$. If $a=0$ we choose $\gamma_a$ to be constant. There exists an integer $r_a$ such that, for $r\ge r_a$, we have 
$$
(\gamma_a(0),\gamma_a(\tfrac{1}{r}L_a),\dots, \gamma_a(\tfrac{r-1}{r}L_a))\in\cL^r_aQ.
$$
We fix such an integer $r_a$ and we set 
$$
\q_a^{r_a}=(\gamma_a(0),\gamma_a(\tfrac{1}{r_a}L_a),\dots, \gamma_a(\tfrac{r_a-1}{r_a}L_a)),
$$ 
and 
$$
\q_a^r=\iota^{r-1}\circ\dots\circ\iota^{r_a}(\q_a^{r_a})
$$
for $r\ge r_a$. It follows from the construction that $\q_a^r\in\cL_a^rQ$. We henceforth fix this as a basepoint and we note that $\geo_a^r(\q_a^r)=\gamma_a$ since $\gamma_a$ is a geodesic loop. We denote by $\cL_{a*}^rQ$ the connected component of $\q_a^r$ in $\cL_a^rQ$, and $\geo_{a*}^r=\geo_a^r|_{\cL_{a*}^rQ}$. Although $\cL^r_aQ$ may be disconnected for some values of $r$, it still does hold that 
\begin{equation} \label{eq:pointedLra}
\colim \geo^r_{a*}: \cL_{a*}^rQ\stackrel\sim\longrightarrow\cL_aQ
\end{equation}
is a homotopy equivalence. Indeed, any path in $\cL_aQ$ can be deformed into a path lying in $\geo_{a*}^r(\cL_{a*}^r)$ for $r$ large enough.

{\bf Notation.} In view of~\eqref{eq:pointedLra}, we shall write in the sequel $\cL_a^rQ$ instead of $\cL_{a*}^rQ$ and $\geo^r_a$ instead of $\geo_{a*}^r$, discarding from the discussion the components of $\cL_a^rQ$ other than the one containing the basepoint $\q^r_a$. Accordingly, we denote $\cL^rQ=\sqcup_a\cL^r_aQ$. 
The maps $\geo_a^r:(\cL_a^rQ,\q_a^r)\to(\cL_aQ,\gamma_a)$ preserve the basepoints by construction. As a consequence, a DG local system $\cF$ on $\cL Q$ induces DG local systems 
$$
\cF^r=(\geo^r)^*\cF
$$
on the spaces $\cL^r Q$, and we have
$$
(\iota^r)^*\cF^{r+1}=\cF^r
$$
by~\eqref{eq:geo-iota}. We denote by $\cF^r_a$ the DG local system on the component $\cL^r_aQ$ of $\cL^rQ$. 

We choose enriched Morse data $\Xi^r=\sqcup_a\Xi^r_a$ on $\cL^rQ$ for $r\ge 1$, consisting of the following: 
\begin{itemize}
\item a Morse function $f^r:\cL^rQ\to \R$ with outward pointing gradient along $\p \cL^rQ$. We denote $f^r_a$ its restriction to $\cL^r_aQ$. 
\item a Morse-Smale negative pseudo-gradient vector field $\xi^r$ (pointing inwards along the boundary $\p \cL^rQ$). 
\item a representing chain system $(s^r_{x,y})$ for the compactified moduli spaces $\ol\cL(x;y)$ of $\xi^r$-trajectories connecting critical points $x,y\in\Crit(f^r)$.    
\item for each connected component $\cL^r_aQ$, an embedded tree $\cY^r_a\subset\cL^r_aQ$ connecting the basepoint $\q^r_a$ to the critical points of $f^r_a$. 
\item for each connected component $\cL^r_aQ$, a homotopy inverse $\theta^r_a$ for the collapsing map $\cL^r_aQ\to\cL^r_aQ/\cY^r_a$. 
\end{itemize}

We denote Morse homology with coefficients in $\cF^r$ by $H_*(\cL^rQ,\Xi^r;\cF^r)$, and for each connected component $\cL_a^rQ$ of $\cL^rQ$ we denote Morse homology with coefficients in $\cF_a^r$ by $H_*(\cL^r_a,\Xi^r_a;\cF^r_a)$. The corresponding Morse complexes are respectively denoted by $C_*(\cL^r,\Xi^r;\cF^r)$ and $C_*(\cL^r_a,\Xi^r_a;\cF^r_a)$. 

Following~\cite[\S 13]{BDHO}, the embeddings $\iota^r_a:\cL^r_aQ\to\cL^{r+1}_aQ$ admit chain level Morse models 
$$
\iota^r_a:C_*(\cL^r_a,\Xi^r_a;\cF^r_a)\to C_*(\cL^{r+1}_a,\Xi^{r+1}_a;\cF^{r+1}_a),
$$
and these induce in the limit isomorphisms 
$$
\colim_r \, (\iota^r_a)_*:\colim_r \,  H_*(f^r_a,\Xi^r_a;\cF^r_a)\stackrel\simeq\longrightarrow H_*(\cL_aQ;\cF_a).
$$

\subsection{From Floer to Morse: moduli spaces} \label{sec:DGViterbo-moduli}

In~\S\ref{sec:discs-one+-representing-chains}-\S\ref{sec:discs-one+-chains-in-paths} we discuss moduli spaces of discs with 1 positive puncture and boundary on the zero section $Q\subset T^*Q$. In~\S\ref{sec:interpolatingFM-representing-chains}-\S\ref{sec:interpolatingFM-chains} we discuss moduli spaces of configurations consisting of punctured discs and gradient trajectories that interpolate between Floer chains and Morse chains. Finally, in~\S\ref{sec:DGViterbo-proof} we give the proof of Theorem~\ref{thm:Viterbo_iso_loc_sys_DG} and Proposition~\ref{prop:Viterbo-iso-energy-zero}. 

For readability purposes we opted for a bit of redundancy: Proposition~\ref{prop:sx} is not strictly speaking used in the proof of the main theorem, but we feel that clarity is added by leading in parallel the 
discussions of moduli spaces of discs on the one hand, and moduli spaces of interpolating configurations on the other hand. 

\subsubsection{Discs with 1 positive puncture: representing chain systems} \label{sec:discs-one+-representing-chains}

Let $(H,J)$ be a regular admissible pair as in the definition of symplectic homology, see~\S\ref{sec:symplectic_homology}. For each $x\in\Per(H)$ we denote $\Psi_x:[0,1]\to \Sp(2n)$, $\Psi_x(0)=\one$ the linearization of the Hamiltonian flow of $H$ along $x$, read in the canonical trivialization of $TT^*Q$ along $x$. We consider two contractible spaces of Fredholm operators:
\begin{itemize}
\item The space $\cD^+_x$ of Cauchy-Riemann operators on the sphere with one positive puncture and asymptotic data given by $\Psi_x$. We denote by $\det^+_x\to\cD^+_x$ the determinant line bundle, supported in degree $n+\CZ(x)$ equal to the index. This is an orientable line bundle because $\cD^+_x$ is contractible. 
\item The space $\cD^-_x$ of Cauchy-Riemann operators on the sphere with one negative puncture and asymptotic data given by $\Psi_x$. We denote by $\det^-_x\to\cD^-_x$ the determinant line bundle, supported in degree $n-\CZ(x)$ equal to the index. This line bundle is also orientable because $\cD^-_x$ is contractible. 
\end{itemize}
A choice of orientation for $\det^+_x$ is equivalent to a choice of orientation for $\det^-_x$, and the equivalence is induced by requiring that the gluing isomorphisms preserve orientations: the gluing of two operators in $\cD^\pm_x$ is a Cauchy-Riemann operator on the Riemann sphere, and the determinant line over the space of such operators is canonically oriented by the  orientation of the determinant lines of complex linear Cauchy-Riemann operators. 

We choose once and for all orientations $o^+_x$ of the determinant line bundles $\det^+_x$, $x\in\Per(H)$ or, equivalently, orientations $o^-_x$ of the determinant line bundles $\det^-_x$, $x\in\Per(H)$. These in turn determine coherent orientations for the moduli spaces of Floer trajectories $\cM(x;y)$, $x,y\in\Per(H)$, as described in~\S\ref{sec:conventions_Floer}.

We consider a monotone homotopy of Hamiltonians and almost complex structures $(H_s,J_s)$, $s\ge 0$ defined on the positive half-cylinder $[0,\infty)\times S^1$, such that:
\begin{itemize}
\item the Hamiltonians have constant slope at infinity, 
\item $H_0|_Q\equiv0$,
\item $(H_s,J_s)=(H,J)$ for $s\gg0$.
\end{itemize}
Given $x\in\Per(H)$ we define 
\begin{align*}
\cM(x)=\{ u:[0,\infty)\times S^1\to T^*Q\, : \, \, & \p_su+J_{s,t}(u)(\p_t u-X_{H_{s,t}}(u))=0,\\
& u(0,t)\in Q, \\
& \lim_{s\to+\infty}u(s,\cdot)=x(\cdot) \}.
\end{align*}

For a generic choice of Floer data the moduli space $\cM(x)$ is a smooth manifold. Its dimension and its behavior with respect to coherent orientations are explained by the next proposition.

Let $\ev:\cM(x)\to\cL Q$ be the \emph{evaluation map} given by $u\mapsto u(0,\cdot)$. Recall the definition of the degree $|x|=\CZ(x)+w(x)$ from~\eqref{eq:degreex}. 
 
\begin{proposition}[{\cite[Lemmas~12.3.1 and~12.3.2]{Abouzaid-cotangent}}] \label{prop:Mx} 
For any regular pair $(H,J)$, the following hold:
\enumerate{
\item The dimension of the moduli space $\cM(x)$ is 
$$
\dim\cM(x)=|x|.
$$
\item There is a canonical isomorphism 
$$
|{\det}^-_x|[w(x)]\otimes |\cM(x)| \simeq \ev^*\eta^{-1}.
$$}
\qed
\end{proposition}

The isomorphism from Proposition~\ref{prop:Mx} is a consequence of the gluing theorem for Cauchy-Riemann operators: gluing an operator $D^-_x\in\cD^-_x$ to the operator $D_u$ that linearizes the Floer equation at an element $u\in \cM(x)$ results in a Cauchy-Riemann operator on the disc, with Lagrangian boundary condition given by the restriction of $TQ$ to the loop $\ev(u)$. The fact that underlies Proposition~\ref{prop:Mx} is that the local system $\eta^{-1}$ describes the index bundle over the space of such Cauchy-Riemann operators on the disc, see \cite[Proposition~12.2.8 and Lemma~12.3.1]{Abouzaid-cotangent}. 

The Floer compactification of $\cM(x)$, denoted $\ol\cM(x)$, is a manifold with boundary with corners. The evaluation map extends continuously to the compactification, where we denote it $\ol\ev:\ol\cM(x)\to\cL Q$. Without taking into account the orientations, the boundary of $\ol\cM(x)$ can be expressed as 
\begin{equation} \label{eq:pMx}
\p \ol\cM(x)=\bigcup_{y\in\Per(H)} \ol\cM(x;y)\times \ol\cM(y),
\end{equation}
where $\ol\cM(x;y)$ is the compactified moduli space of Floer trajectories connecting $x$ to $y$, see~\S\ref{sec:conventions_Floer}-\S\ref{sec:twisting-cocycle}. 

Note that the isomorphism from Proposition~\ref{prop:Mx} can be improved further to $|{\det}^-_x|[w(x)]\otimes |\cM(x)| \simeq \eta_x^{-1}$, since any loop of the form $\ev(u)$, $u\in\cM(x)$ is homotopic to $x$ along $u$. As a consequence we see that the moduli spaces $\cM(x)$ are orientable, and so are the compactified moduli spaces $\ol\cM(x)$. However, since the canonical trivialization of $|\cM(x)|$ depends on the whole moduli space, in order to understand the algebraic relation implied by the decomposition~\eqref{eq:pMx} it is most convenient to work with fundamental classes with coefficients in $|\ol\cM(x)|[\dim\ol\cM(x)]\simeq \ol\ev^*\ueta^{-1}$. 

\begin{proposition} \label{prop:sx} Assume $(H,J)$ is a regular pair. Let $s_{x,y}$ be a representing chain system for the moduli spaces of Floer trajectories $\ol\cM(x;y)$, $x,y\in\Per(H)$. There exists a collection of chains 
$$
s_x\in C_{|x|}(\ol\cM(x);\ol\ev^*\ueta^{-1}),\quad x\in\Per(H)
$$
with the following properties:
\begin{enumerate}
\item Each $s_x$ is a cycle rel boundary and it represents the fundamental class $[\ol\cM(x)]$ with coefficients in $|\cM(x)|$. (This is canonically defined.) 
\item Each $s_x$ satisfies 
\begin{equation} \label{eq:psx}
\p s_x = \sum_y s_{x,y}\times s_y.
\end{equation}
The product of chains is defined via the inclusion $\ol\cM(x;y)\times\ol\cM(y)\subset \p\ol\cM(x)\subset \ol\cM(x)$, which induces a map 
$$
C_*(\ol\cM(x;y))\, \otimes  \, C_*(\ol\cM(y);\ol\ev^*\ueta^{-1})
\stackrel\times\longrightarrow C_*(\ol\cM(x);\ol\ev^*\ueta^{-1}).
$$
\end{enumerate}
\end{proposition}

\begin{proof}
We argue as in~\S\ref{sec:orientations_continuation}. There is a gluing map
$$
\wh \cM(x;y)\times \cM(y) \to \cM(x)
$$
that is a diffeomorphism from its domain to its image and that induces a canonical isomorphism of orientation local systems  
$$
|\wh \cM(x;y)|\otimes |\cM(y)|\simeq |\cM(x)|. 
$$
Consider the canonical constant vector field $\p_{x,y}$ on $\wh\cM(x;y)$, which we view under the gluing map as defining a vector field on $\cM(x)$. As in~\S\ref{sec:orientations-twisting-cocycle} we see that this vector field points towards the boundary. Using~\eqref{eq:orientMMhat} we infer a canonical isomorphism 
\begin{align*}
|\cM(x)| & \simeq |\R\p_{x,y}| \otimes |\cM(x;y)| \otimes |\cM(y)| \\
& \simeq |\R\nu^{\text{out}}| \otimes |\cM(x;y)| \otimes |\cM(y)|.
\end{align*}
Here the second isomorphism is notational: the vector field $\p_{x,y}$ points towards the boundary, and we record this by denoting it $\nu^{\text{out}}$. 

From this point on the chain representatives $s_x$ can be constructed by induction over $\deg x\ge 0$ (since the moduli spaces are transversely cut out, they are empty for $\deg x<0$). The induction step is similar to~\cite[Proposition~5.6 and Lemma~7.6]{BDHO}.
\end{proof}

\begin{remark}[Behavior under trivializations] \label{rmk:twist-by-eta} 
The local system $\oev^*\ueta^{-1}$ on $\ol\cM(x)$ is canonically isomorphic to $\ueta_x^{-1}$ since the evaluation map $\oev:\ol\cM(x)\to \cL Q\subset \cL T^*Q$ at the endpoint is canonically homotopic to the evaluation map at the initial point, i.e., the constant map equal to $x$. In this trivialization equation~\eqref{eq:psx} keeps the same form, but the product of chains is to be understood via the map  
$$
C_*(\ol\cM(x;y))\, \otimes  \, C_*(\ol\cM(y);\ueta_y^{-1})
 \stackrel\times\longrightarrow C_*(\ol\cM(x);\ueta_x^{-1})
$$
that extends linearly the following action: let $\tau:I^k\to \ol\cM(x;y)$, $\vartheta:I^\ell\to \ol\cM(y)$ be cubes and $e_y\in \ueta_y^{-1}$ be a coefficient for $\vartheta$. We denote $\tau_0=\tau(0,\dots,0)$ and view it as a path in $\cL T^*Q$ from $x$ to $y$. 
We also denote $\rho^\eta:\pi_0(\cP_{x\to y}\cL T^*Q)\to \Iso(\ueta_x^{-1},\ueta_y^{-1})$ the parallel transport for the local system $\ueta^{-1}$. Then
$$
\tau\times (e_y\vartheta) = \rho^\eta(\tau_0)^{-1}(e_y) (\tau\times \vartheta), 
$$
where the product $\tau\times \vartheta$ is understood via the map induced at chain level by the inclusions $\ol\cM(x;y)\times \ol\cM(y)\subset \p\ol\cM(x)\subset \ol\cM(x)$. 
\end{remark}

The takeaway from the previous remark is the following: 
\begin{center}
{\it Upon trivializing $\oev^*\ueta^{-1}$ on $\ol\cM(x)$, $x\in\Per(H)$ as $\ueta_x^{-1}$, the product of chains from the strata of $\p\ol\cM(x)$ gets twisted by the parallel transport of $\ueta^{-1}$.
}  
\end{center}

\subsubsection{Discs with 1 positive puncture: evaluation into path spaces} \label{sec:discs-one+-evaluation}

Given the continuation pair $(H_s,J_s)$, $s\ge 0$, let $s_+\ge 0$ be such that $(H_s,J_s)$ is constant equal to $(H,J)$ for $s\ge s_+$. Following~\S\ref{sec:eval_maps_homotopies} we complement the previous choice of representing chain system $s_x$, $x\in\Per(H)$ to a full set of \emph{monotone enriched Floer continuation data} by making the following additional choices: (i) for each free homotopy class $a$ of loops in $T^*Q$, a collapsing tree $\cY_a$ and a homotopy inverse $\theta_a$ for the projection $p_a:\cL_a T^*Q\to \cL_a T^*Q/\cY_a$. (In the sequel we will drop the mention of the free homotopy class from the notation and simply write $\cY, \theta, p$ instead of $\cY_a,\theta_a,p_a$.) (ii) a smooth cutoff function $\rho:[0,\infty)\to [0,1]$ such that $\rho(0)=0$, $\rho'>0$ on $[0,s_+ +1)$ and $\rho(s)=1$ for $s\ge s_+ +1$. (iii) a constant $C>0$ such that $\p_sH_s-\rho'(s)C<0$ for $s\in[0,s_+]$. 

Following~\S\ref{sec:eval_maps_homotopies} we denote $H_s^{\rho,C}=H_s-\rho(s)C$ and consider the map
$$
A_u:[0,+\infty]\to [0,A_H(x)+C],\qquad s\mapsto A_{H_s^{\rho,C}}(u(s,\cdot)).
$$ 
This map is strictly increasing on $[0,s_+ +1]$ and it is a bijection from an interval $I_u\subset [0,+\infty]$ onto $[0,A_H(x)+C]$, with $I_u=[0,s_+ +1]$ if $u(s,\cdot)$ is constant equal to $x$ for $s\ge s_+ +1$, or $I_u=[0,+\infty]$ if not. 

We denote $\cP_{x\to \cL Q}\cL T^*Q$ and $\cP_{x\to \cL T^*Q}\cL T^*Q$ the spaces of Moore paths in $\cL T^*Q$ with initial point $x$ and endpoint in $\cL Q$, respectively in $\cL T^*Q$. (In the second case, this simply means that we do not impose any constraint on the endpoint.) Following~\S\ref{sec:eval_maps_homotopies} we define the evaluation map 
$$
\ev^\cP_x:\cM(x)\to\cP_{x\to \cL Q}\cL T^*Q,
$$ 
by
\begin{align*}
u\mapsto \big( \ell\mapsto u(s,\cdot), \quad & \ell\in [0,A_H(x)+C], \\
& s=A_u^{-1}(A_H(x)+C-\ell)\big).
\end{align*}
This map sends a Floer continuation cylinder to the path that it defines in $\cL T^*Q$, parametrized backwards on the interval $[0,A_H(x)+C]$ by the values of the action functional $A_{H_s^{\rho,C}}$. 

Using the projection $p:\cL T^*Q\to \cL T^*Q/\cY$ and its homotopy inverse $\theta$ we now define an evaluation map 
$$
q_x:\cM(x)\to\cP_{\star\to \cL Q}\cL T^*Q. 
$$ 
We first define $q'_x:\cM(x)\to\cP_{\star\to \cL T^*Q}\cL T^*Q$ by 
\begin{align*}
u\mapsto \big( \ell\mapsto \theta \circ p (u(s,\cdot)), \quad & \ell\in [0,A_H(x)+C], \\
& s=A_u^{-1}(A_H(x)+C-\ell)\big). 
\end{align*}
This almost works as a candidate for $q_x$ except that it does not preserve the endpoint $u(0,\cdot)$, so that we correct it as follows. Let $H:[0,1]\times \cL T^*Q\to \cL T^*Q$ be a homotopy between $\Id$ and $\theta\circ p$, so that $H(0,\cdot)=\Id$ and $H(1,\cdot)=\theta\circ p$. We define $q_x$ by 
$$
q_x(u)=q'_x(u)\# (t\mapsto H(1-t, u(0,\cdot))).
$$
Then, by construction, the value of $q_x$ at its endpoint is equal to $u(0,\cdot)$. 

The evaluation maps $\ev^\cP_x$ and $q_x$ extend continuously to the compactification $\ol\cM(x)$, and we denote their extensions 
$$
\oev^\cP_x:\ol\cM(x)\to \cP_{x\to \cL Q}\cL T^*Q
$$ 
and 
$$
\overline q_x:\ol\cM(x)\to \cP_{\star\to \cL Q}\cL T^*Q.
$$

Denoting by $\ol q_+:\ol\calm(x;w)\ri \Omega\cL T^*Q$ the evaluation map defined in \S\ref{sec:twisting-cocycle}   by the data $(H,J,\caly,\theta)$, we have that the restriction of $\ol q=\ol q_x$ to the boundary $\p\ol\cM(x)$ satisfies
\begin{equation}\label{eq:relation1-Mx}
\ol q(u,v)\, =\, \ol q_+(u)\, \# \, \ol q(v)
\end{equation} 
for all $(u,v)\in \ol\calm(x;w)\times \ol\cM(w)$.

\subsubsection{Discs with 1 positive puncture: chains in paths spaces} \label{sec:discs-one+-chains-in-paths}

\paragraph{Chains in path spaces $\cP_{x\to \cL Q}\cL T^*Q$, $x\in\Per(H)$.} Consider the commutative diagram 
$$
\xymatrix{
\ol\cM(x)\ar[rr]^{\oev^\cP_x} \ar[dr]_\oev & & \cP_{x\to\cL Q}\cL T^*Q \ar[dl]^{\ev_\text{end}} \\
& \cL Q, &
}
$$
where the map $\ev_\text{end}$ associates to a path the loop at its endpoint. 
We define chains 
$$
m^\cP_x=\oev^\cP_{x*} s_x\in C_{|x|}(\cP_{x\to \cL Q}\cL T^*Q;\ev_\text{end}^*\ueta^{-1}),\quad x\in\Per(H).
$$
We further trivialize $\ev_\text{end}^*\ueta^{-1}$ over $\cP_{x\to \cL Q}\cL T^*Q$ as $\ueta_x^{-1}$ and view $m^\cP_x$ as 
$$
m^\cP_x \in C_{|x|}(\cP_{x\to \cL Q}\cL T^*Q;\ueta_x^{-1}).
$$
Denote as in~\S\ref{sec:evaluation_Floer_traj} by $\oev_{x,y}:\ol\cM(x;y)\to \cP_{x\to y}\cL T^*Q$ the evaluation of the Floer moduli spaces into the spaces of Moore paths connecting $x$ to $y$, and 
$$
m^\cP_{x,y}=(\oev_{x,y})_*s_{x,y}\in C_{|x|-|y|-1}(\cP_{x\to y}\cL T^*Q).
$$
Equation~\eqref{eq:psx} then implies 
\begin{equation} \label{eq:mPx}
\p m^\cP_x = \sum_y m^\cP_{x,y}\cdot m^\cP_y,
\end{equation}
where the product of chains 
$$
C_*(\cP_{x\to y}\cL T^*Q)\, \otimes  \, C_*(\cP_{y\to \cL Q}\cL T^*Q;  \ueta_y^{-1}) \stackrel\cdot\longrightarrow C_*(\cP_{x\to \cL Q}\cL T^*Q;\ueta_x^{-1})
$$
is induced by the composition of paths twisted by the parallel transport of $\ueta^{-1}$ along elements in $\cP_{x\to y}\cL T^*Q$, as in Remark~\ref{rmk:twist-by-eta}. 

\paragraph{Chains in path spaces $\cP_{\star\to \cL Q}\cL T^*Q$.} Consider also the commutative diagram 
$$
\xymatrix{
\ol\cM(x)\ar[rr]^{\ol q_x} \ar[dr]_\oev & & \cP_{\star\to\cL Q}\cL T^*Q \ar[dl]^{\ev_\text{end}} \\
& \cL Q &
}
$$
and define chains
$$
m_x=\ol q_{x*} s_x\in C_{|x|}(\cP_{\star\to \cL Q}\cL T^*Q;\ev_\text{end}^*\ueta^{-1}).
$$
We further trivialize $\ev_\text{end}^*\ueta^{-1}$ as $\ueta_\ast^{-1}\simeq A$ (the ground ring), so that 
$$
m_x\in C_{|x|}(\cP_{\star\to \cL Q}\cL T^*Q).
$$
Then equation~\eqref{eq:mPx} becomes 
\begin{equation} \label{eq:mx}
\p m_x=\sum_y m_{x,y}\cdot m_y,
\end{equation}
where the product of chains is to be understood via the map 
$$
C_*(\Omega \cL T^*Q)\, \otimes  \, C_*(\cP_{\star\to \cL Q}\cL T^*Q)
\longrightarrow C_*(\cP_{\star\to \cL Q}\cL T^*Q),
$$
$$
\sigma\otimes\tau\mapsto \Phi^\eta(\sigma)\cdot \tau,
$$
where we denote $\Phi^\eta=\Phi^{\ueta^{-1}}$ defined as in~\S\ref{sec:twist} and $\cdot$ is the Pontryagin product. In other words, we have proved:

\begin{proposition} 
The chains $m_x\in C_{|x|}(\cP_{\star\to \cL Q}\cL T^*Q)$, $x\in\Per(H)$ satisfy the equation 
$$
\p m_x = \sum_y m^\eta_{x,y}\cdot m_y,
$$
where $\{m^\eta_{x,y}\}$ is the cocycle from~\S\ref{sec:twist} defined by
$$
m^\eta_{x,y}=\Phi^\eta(m_{x,y}).
$$
The product of chains is induced by the concatenation of paths 
$$\Omega \cL T^*Q\times \cP_{\star\to \cL Q}\cL T^*Q\to \cP_{\star\to \cL Q}\cL T^*Q.$$  
\qed
\end{proposition}

\subsubsection{Interpolating from Floer to Morse: representing chain systems} \label{sec:interpolatingFM-representing-chains}

We now consider the map 
$$
\ev_r:\cM(x)\to Q^r,\qquad u\mapsto \Big(\ev(u)(0),\ev(u)(\tfrac{1}{r}),\dots,\ev(u)(\tfrac{r-1}r)\Big), 
$$
and denote $\ol\ev_r:\ol\cM(x)\to Q^r$ its extension to the compactification. For $r$ large enough, the image of $\ol\ev_r$ lies in $\cL^rQ$~\cite[Lemma~12.3.5]{Abouzaid-cotangent}. 

Given $x\in\Per(H)$ and $y\in\Crit(f^r)$ we form the moduli spaces 
$$
\cB(x;y)=\cM(x)\times_{\ev_r} W^s(y),
$$
where $W^s(y)$ is the stable manifold of $y$ with respect to the negative pseudo-gradient $\xi^r$.  Each moduli space $\cB(x;y)$ is endowed with a canonical evaluation map inherited from $\cM(x)$ and denoted also
$$
\ev_r:\cB(x;y)\to Q^r.
$$
The Floer compactification of $\cB(x;y)$ is $\ol\cB(x;y)=\ol\cM(x)\times_{\oev_r} \ol{W}^s(y)$. It inherits from $\ol\cM(x)$ an evaluation map that extends $\ev_r$, denoted also 
$$
\oev_r:\ol\cB(x;y)\to Q^r
$$ 
and whose image is also contained in $\cL^r Q$ for $r$ large enough.

We denote $\pr_1:\cB(x;y)\to \cM(x)$ the first projection and $\pr_2:\cB(x;y)\to W^s(y)$ the second projection.  

\begin{proposition}[{\cite[Lemma~12.3.7]{Abouzaid-cotangent}}]
For a generic choice of pseudo-gradient $\xi^r$ the moduli space $\cB(x;y)$ is a smooth manifold of dimension 
$$
\dim\cB(x;y)=|x|-\ind(y).
$$
Moreover, we have a canonical isomorphism 
\begin{equation} \label{eq:iso-ori-Bxy}
|{\det}^-_x|[w(x)]\otimes |\cB(x;y)|\otimes \pr_2^*|W^u(y)|\simeq\ev_r^*\eta^{-1},
\end{equation}
where we view $|W^u(y)|$ as a trivial bundle over $W^s(y)$, and we view $\eta^{-1}$ as a local system on $\cL^r Q$ by pulling it back from $\cL Q$ by the map $\geo^r$.  
\end{proposition}

\begin{proof}
The dimension formula follows from Proposition~\ref{prop:Mx}.1, since the codimension of $W^s(y)$ is equal to $\ind(y)=\dim W^u(y)$. 

The description of the isomorphism of orientation lines is the following. On the one hand, by definition of the orientation of a fiber 
product we have a canonical isomorphism 
$$
|\cB(x;y)|\otimes |\cL^r Q|\simeq |\cM(x)|\otimes |W^s(y)|. 
$$
On the other hand 
$$
|\cL^rQ|\simeq |W^u(y)|\otimes |W^s(y)|,$$
which gives 
$$
|\cB(x;y)|\otimes |W^u(y)|\simeq |\cM(x)|.
$$
(We suppressed $\pr_2^*$ from the notation of $\pr_2^*|W^u(y)|$ for readability.) The conclusion follows by combining this with Proposition~\ref{prop:Mx}.2. 
\end{proof}

The compactification $\ol\cB(x;y)$ is a manifold with boundary with corners. Without taking into account the boundary orientation we have  
$$
\p\ol\cB(x;y)=\bigcup_z \ol\cM(x;z)\times\ol\cB(z;y)\cup \bigcup_w\ol\cB(x;w)\times\ol\cL(w;y),
$$
and we now discuss the signs for the various strata of the boundary. 

Assume that we have chosen orientations of ${\det}^-_x$, $x\in\Per(H)$ and of $W^u(y)$, $y\in\Crit(f^r)$. We endow the moduli spaces of Floer trajectories $\ol\cM(x;w)$, $x,w\in\Per(H)$ with the corresponding coherent orientation, and the moduli spaces of Morse trajectories $\ol\cL(y;z)$, $y,z\in\Crit(f^r)$ with the corresponding analytic orientation (see~\cite[Appendix~A.2]{BDHO}).

\begin{proposition}  \label{prop:representing-chains-Bxy}
Let $s^{\mathrm{Floer}}_{x,w}$ be a representing chain system for the moduli spaces of Floer trajectories $\ol\cM(x;w)$, $x,w\in\Per(H)$, and $s^{\mathrm{Morse}}_{y,z}$ be a representing chain system for the moduli spaces of Morse trajectories $\ol\cL(y;z)$, $y,z\in\Crit(f^r)$. 

There exists a collection of chains 
$$
\sigma_{x,y}\in C_{|x|-\ind(y)}(\ol\cB(x;y);\ol\ev_r^*\ueta^{-1}),\quad x\in\Per(H), \ y\in\Crit(f^r),
$$
with the following properties:
\begin{enumerate}
\item Each $\sigma_{x,y}$ is a cycle rel boundary and represents the fundamental class $[\ol\cB(x;y)]$ with coefficients in $|\cB(x;y)|$. (This is canonically defined.) 
\item Each $\sigma_{x,y}$ satisfies 
\begin{equation} \label{eq:psigmaxy}
\p \sigma_{x,y} = \sum_{w\in\Per(H)} s^{\mathrm{Floer}}_{x,w}\times \sigma_{w,y} + \sum_{z\in\Crit(f^r)} (-1)^{|x|-\ind(z)-1} \sigma_{x,z}\times s^{\mathrm{Morse}}_{z,y}.
\end{equation}
The product of chains is defined via the inclusions $\ol\cM(x;w)\times\ol\cB(w;y)\subset \p\cB(x;y)\subset \cB(x;y)$ and 
$\ol\cB(x;z)\times \ol\cL(z;y)\subset\p\cB(x;y)\subset \cB(x;y)$, which induce maps 
$$
C_*(\ol\cM(x;w))\, \otimes  \, C_*(\ol\cB(w;y);\ol\ev_r^*\ueta^{-1})
 \stackrel\times\longrightarrow C_*(\ol\cB(x;y);\ol\ev_r^*\ueta^{-1}),
$$
and
$$
C_*(\ol \cB(x;z);\ol\ev_r^*\ueta^{-1})\, \otimes \, C_*(\ol\cL(z;y))  
 \stackrel\times\longrightarrow C_*(\ol\cB(x;y);\ol\ev_r^*\ueta^{-1}).
$$
\end{enumerate}
\end{proposition}

\begin{proof}
The analysis of signs is analogous to that of~\S\ref{sec:orientations_continuation}, and the construction of the chains $\sigma_{x,y}$ proceeds by induction as in Proposition~\ref{representingchainFloercontinuation} and~\cite[Lemma~6.1(ii)]{BDHO}.
\end{proof}

\begin{remark}[Behavior under trivializations]  \label{rmk:behavior-under-triv-Bxy} This remark is the analogue of~\ref{rmk:twist-by-eta} for the moduli spaces $\cB(x;y)$. Having chosen orientations of ${\det}^-_x$, $x\in\Per(H)$ and $W^u(y)$, $y\in\Crit(f^r)$, the isomorphism~\eqref{eq:iso-ori-Bxy} takes the form 
$$
|\cB(x,y)|[|x|-\ind(y)]\simeq \ev_r^*\ueta^{-1}.
$$
Similarly to~\S\ref{sec:discs-one+-representing-chains}, the local system $\ev_r^*\ueta^{-1}$ over $\cB(x;y)$ is canonically isomorphic to the constant local system $\ueta_x^{-1}$, and therefore 
$$
|\cB(x,y)|[|x|-\ind(y)]\simeq \ueta_x^{-1}.
$$
In this trivialization equation~\eqref{eq:psigmaxy} keeps the same form, but the product of the chains $s^{\mathrm{Floer}}_{x,w}\times \sigma_{w,y}$ is to be understood via the map  
$$
C_*(\ol\cM(x;w))\, \otimes  \, C_*(\ol\cB(w;y);\ueta_w^{-1}) \stackrel\times\longrightarrow C_*(\ol\cB(x;y);\ueta_x^{-1})
$$
that extends linearly the following action: let $\tau:I^k\to \ol\cM(x;w)$, $\vartheta:I^\ell\to \ol\cB(w;y)$ be cubes and $e_w\in \ueta_w^{-1}$ be a coefficient for $\vartheta$. Let $\tau_0=\tau(0,\dots,0)$ and view it as a path in $\cL T^*Q$ from $x$ to $w$. 
Let also $\rho^\eta:\pi_0(\cP_{x\to w}\cL T^*Q)\to \Iso(\ueta_x^{-1},\ueta_w^{-1})$ be the parallel transport for the local system $\ueta^{-1}$. Then
$$
\tau\times (e_w\vartheta) = \rho^\eta(\tau_0)^{-1}(e_w) (\tau\times \vartheta), 
$$
where the product $\tau\times \vartheta$ is understood via the map induced at chain level by the inclusions $\ol\cM(x;w)\times \ol\cB(w;y)\subset \p\ol\cB(x;y)\subset \ol\cB(x;y)$. 
\end{remark}

Similarly to~\S\ref{sec:discs-one+-representing-chains}, the takeaway from the previous remark is the following: 
\begin{center}
{\it Upon trivializing $\oev^*\ueta^{-1}$ on $\ol\cB(x;y)$, $x\in\Per(H)$, $y\in\Crit(f^r)$ as $\ueta_x^{-1}$, the product of chains from the strata $\ol\cM(x;w)\times \ol\cB(w;y)$ of $\p\ol\cM(x)$ gets twisted by the parallel transport of $\ueta^{-1}$.
}  
\end{center}

\subsubsection{Interpolating from Floer to Morse: evaluation into $\Omega\cL^r Q$} \label{sec:interpolatingFM-evaluation}

We construct in this subsection evaluation maps 
$$
\ol q_{x,y}:\ol\cB(x;y)\to \Omega\cL^r Q
$$
for $x\in\Per(H)$ and $y\in\Crit(f^r)$. 

\paragraph{Floer evaluation maps} Recall from~\S\ref{sec:discs-one+-evaluation} the choice of monotone enriched Floer continuation data for the pair $(H_s,J_s)$, consisting for each free homotopy class of loops in $T^*Q$ of a collapsing tree $\cY$, a homotopy inverse $\theta$ for the projection $p:\cL T^*Q\to \cL T^*Q/\cY$, a function $\rho:[0,\infty)\to [0,1]$ and a constant $C>0$.

\paragraph{(i) Cylinders.} This data determines as in~\S\ref{sec:evaluation_Floer_traj} evaluation maps 
$$
\ol q_+:\ol\cM(x;w)\to \Omega\cL T^*Q
$$
for $x,w\in\Per(H)$, which satisfy~\eqref{eq:relation1-Mxy}. 

\paragraph{(ii) Half-cylinders.} The above data also determines as in~\S\ref{sec:discs-one+-evaluation} evaluation maps 
$$
\ol q_x:\ol\cM(x)\to \cP_{\star\to\cL Q}\cL T^*Q
$$
for $x\in\Per(H)$, which satisfy~\eqref{eq:relation1-Mx}.

\paragraph{Morse evaluation maps.} 
We now choose enriched Morse data consisting, for each connected component $\cL^rQ$, of a collapsing tree $\cY^r$ for the elements of $\Crit(f^r)$ that belong to that component and of a homotopy inverse $\theta^r$ for the projection $p^r:\cL^rQ\to \cL^rQ/\cY^r$. (As before, we drop from the notation the mention of the connected component, which corresponds to a free homotopy class of loops.)

\paragraph{(i) Lines.} Let $\ol\cL(z;y)$ be the compactified moduli spaces of $\xi^r$-trajectories connecting $z$ at $-\infty$ to $y$ at $+\infty$. The tuple $(f^r,\xi^r,\cY^r,\theta^r)$ determines as in~\cite[Lemma~5.9]{BDHO} evaluation maps 
$$
\ol q_-:\ol\cL(z;y)\to \Omega\cL^r Q
$$ 
for $z,y\in\Crit(f^r)$, which satisfy 
$$
\ol q_-(u,v)=\ol q_-(u)\# \ol q_-(v)
$$
for all $(u,v)\in \ol\cL(z;z')\times\ol\cL(z';y)$. 

\paragraph{(ii) Half-lines.} Let $\cP_{\cL^r Q\to y}\cL^r Q$ be the space of Moore paths starting anywhere on $\cL^r Q$ and ending at $y$. The tuple $(f^r,\xi^r,\cY^r,\theta^r)$ also determines for $y\in\Crit(f^r)$ evaluation maps 
$$
\ol q_y:\ol W^s(y)\to \cP_{\cL^r Q\to \star}\cL^r Q
$$
that satisfy 
\begin{equation} \label{eq:relation1-Wsy}
\ol q_y(u,v)=\ol q_z(u) \# \ol q_-(v)
\end{equation}
for all $(u,v)\in \ol W^s(z)\times\ol\cL(z;y)$. Moreover, the maps $\ol q_y$ can be constructed such that they preserve the initial point of any broken trajectory in $\ol W^s(y)$, i.e., 
$$
\ev_0\circ \ol q_y= \ev_{\text{in}},
$$
where $\oev_\text{in}:\ol W^s(y)\to \cL^r Q$ associates to a broken trajectory its initial point. 
 
The construction of the maps $\ol q_y$ is made in two steps: 

1. There is a canonical evaluation map 
$$
\ev_y:W^s(y)\to \cP_{\cL^rQ\to y}\cL^r Q
$$ 
that associates to a \emph{$\xi^r$-half-trajectory} $\gamma:[0,\infty)\to \cL^r Q$ the path $[0,f^r(\gamma(0))-f^r(y)]\to \cL^r Q$, $\ell\mapsto \gamma(s_\ell)$, where $s_\ell\in[0,\infty]$ is uniquely determined by the condition $f^r(\gamma(s_\ell))=f^r(\gamma(0))-\ell$. In other words, we parametrize any $\xi^r$-half-trajectory in $W^s(y)$ by the values of $f^r$ shifted down by $f^r(y)$. The map $\ev_y$ admits a canonical extension to the compactification, denoted 
$$
\oev_y:\ol W^s(y)\to \cP_{\cL^rQ\to y}\cL^r Q.
$$

2. Let $H^r:[0,1]\times\cL^r Q\to \cL^r Q$ be a homotopy between $\Id$ and $\theta^r\circ p^r$, so that $H^r(0,\cdot)=\Id$ and $H^r(1,\cdot)=\theta^r\circ p^r$. We define 
$$
\ol q_y(u)=(t\mapsto H^r(t,\oev_{\text{in}}(u)))\, \# \, \theta^r\circ p^r (\oev_y(u)).
$$  

\paragraph{Putting together Floer and Morse evaluation maps.}

Recall from~\S\ref{sec:DGViterbo-finite-dim} the map 
$$
\ev_r:\cL Q\to Q^r, \qquad x\mapsto (x(0),x(\tfrac{1}{r}),\dots,x(\tfrac{r-1}{r})).
$$
Recall also the projection $\pi:T^*Q\to Q$.

\begin{proposition}
For $r$ large enough there exists a collection of evaluation maps 
$$
\ol q_{x,y}:\ol\cB(x;y)\to \Omega \cL^r Q
$$
indexed by $x\in\Per(H)$, $y\in\Crit(f^r)$ such that $\ol q=\ol q_{x,y}$ satisfies the relations 
\begin{equation}\label{eq:relation1-Bxy}
\ol q(u,v)\, =\, \ev_r \circ \pi\circ \ol q_+(u)\, \# \, \ol q(v)
\end{equation} 
for all $(u,v)\in \ol\cM(x;w)\times \ol\cB(w;y)$, and 
\begin{equation}\label{eq:relation2-Bxy}
\ol q(u,v)\, =\, \ol q(u)\, \# \, \ol q_-(v)
\end{equation} 
for all $(u,v)\in\ol\cB(x;z)\times \ol\cL(z;y)$. 
\end{proposition}

\begin{proof}
Denote $\Im:\cP_{\star\to \cL Q}\cL T^*Q\to \cL T^*Q^{\{0,1\}}$ the \emph{image} map, which associates to each path in $\cL T^*Q$ its image as a subset in $\cL T^*Q$. The subset 
$$
\bigcup_{x\in\Per(H)} \pi\circ \Im \circ \ol q_x (\ol\cM(x)) \subset \cL Q
$$
is compact since each $\pi\circ \Im \circ \ol q_x (\ol\cM(x))$ is compact and $\Per(H)$ is finite. We choose henceforth $r$ large enough so that the image of this set under $\ev_r$ is contained in $\cL^r Q$.

We define 
$$
\ol q_{x,y}:\ol\cB(x;y)\to \Omega\cL^r Q
$$
by 
$$
\ol q_{x,y}(u,v)= \ev_r \circ \pi \circ \ol q_x(u) \, \# \, \ol q_y(v),
$$
where $u\in \ol\cM(x)$ and $v\in\ol W^s(y)$ are such that $\ev_r(u)=\ev_{\text{in}}(v)$. Equations~\eqref{eq:relation1-Bxy} and~\eqref{eq:relation2-Bxy} are then direct consequences of~\eqref{eq:relation1-Mx} and~\eqref{eq:relation1-Wsy}. 
\end{proof}

\subsubsection{Interpolating from Floer to Morse: chains on $\Omega\cL^r Q$} \label{sec:interpolatingFM-chains}

Given the representing chain system $(\sigma_{x,y})$ from Proposition~\ref{prop:representing-chains-Bxy}, we define 
$$
b_{x,y}=(\ol q_{x,y})_*(\sigma_{x,y})\in C_{|x|-\ind(y)}(\Omega \cL^r Q).
$$

Let $(m^{\mathrm{Floer}}_{x,w})$ be the twisting cocycle associated to the Floer representing chain system $(s^{\mathrm{Floer}}_{x,w})$ via the auxiliary data $(\cY,\theta)$, and let $(m^{\mathrm{Morse}}_{z,y})$ be the twisting cocycle associated to the Floer representing chain system $(s^{\mathrm{Morse}}_{z,y})$ via the auxiliary data $(\cY^r,\theta^r)$. Recall that $m^{\mathrm{Floer}}_{x,w}\in C_{|x|-|w|-1}(\Omega \cL T^*Q)$ and $m^{\mathrm{Morse}}_{z,y}\in C_{\ind(z)-\ind(y)-1}(\Omega \cL^r Q)$. Denote further 
$$
m^{\mathrm{Floer};\, r}_{x,w} = (\ev_r\circ \pi)_*m^{\mathrm{Floer}}_{x,w}\in C_{|x|-|w|-1}(\Omega \cL^r Q).
$$
Recall from~\S\ref{sec:discs-one+-chains-in-paths} the notation $\Phi^\eta=\Phi^{\ueta^{-1}}$, where $\Phi^{\ueta^{-1}}:C_*(\Omega \cL T^*Q)\to C_*(\Omega \cL T^*Q)$ is defined in~\S\ref{sec:twist}. Denote 
$$
\Phi^\eta(m^{\mathrm{Floer};\, r}_{x,w}) = (\ev_{r*}\circ \Phi^\eta\circ \pi_*)(m^{\mathrm{Floer}}_{x,w})\in C_{|x|-|w|-1}(\Omega \cL^r Q).
$$

\begin{proposition} \label{prop:bxy} Each chain $b_{x,y}$ satisfies the equation 
\begin{equation} \label{eq:pbxy}
\p b_{x,y} = \sum_{w\in\Per(H)} \Phi^\eta(m^{\mathrm{Floer};\, r}_{x,w})\cdot b_{w,y} + \sum_{z\in\Crit(f^r)} (-1)^{|x|-\ind(z)-1} b_{x,z}\cdot m^{\mathrm{Morse}}_{z,y}.
\end{equation}
\end{proposition}

\begin{proof}
This is a direct consequence of properties~\eqref{eq:relation1-Bxy} and~\eqref{eq:relation2-Bxy} for the maps $\ol q_{x,y}$, combined with Proposition~\ref{prop:representing-chains-Bxy}, Remark~\ref{rmk:behavior-under-triv-Bxy} and equation~\eqref{eq:mx}.
\end{proof}

\subsection{Proof of the isomorphism theorem} \label{sec:DGViterbo-proof}

\subsubsection{Floer homology in the finite dimensional approximation}

Given a based space $X$ and a DG local system $\cF$ on $X$, we use the \emph{ad hoc} notation 
$$
FH_*(H,\{m_{x,w}\};\cF)
$$
for the homology of $\cF\otimes \langle \Per(H)\rangle$ twisted by some cocycle $\{m_{x,w}\}$ indexed by $x,w\in\Per(H)$ and valued in $C_*(\Omega X)$. (In the sequel $X=\cL T^*Q$ or $X=\cL^r Q$.)

Recall from~\S\ref{sec:interpolatingFM-chains} the cocycles $\{m^{\mathrm{Floer}}_{x,w}\}$ and $\{\Phi^\eta(m^{\mathrm{Floer};\, r}_{x,w})\}$. 
Recall also that $\cF^r = (\mathrm{geo}^r)^*\cF$, where $\mathrm{geo}^r: \call^r Q\ri \call Q$ was defined in~\S \ref{sec:DGViterbo-finite-dim}.

\begin{lemma} \label{lem:FH-finite} Let $\cF$ be a DG local system on $\cL Q$. We have isomorphisms 
$$
FH_*(H,\{m^{\mathrm{Floer}}_{x,w}\};\Pi^*(\cF\otimes\ueta^{-1}))\simeq
FH_*(H,\{\Phi^\eta(m^{\mathrm{Floer};\, r}_{x,w})\};\cF^r),
$$
\end{lemma}

\begin{proof} We claim that we have a sequence of isomorphisms 
\begin{align*}
FH_*(H,\{m^{\mathrm{Floer}}_{x,w}\};\Pi^*(\cF\otimes\ueta^{-1}))
& \simeq FH_*(H,\{\pi_*m^{\mathrm{Floer}}_{x,w}\};\cF\otimes\ueta^{-1})\\
& \simeq FH_*(H,\{\Phi^\eta\circ\pi_*m^{\mathrm{Floer}}_{x,w}\};\cF)\\
& \simeq FH_*(H,\{\Phi^\eta\circ\pi_*m^{\mathrm{Floer}}_{x,w}\};\ev_r^*(\geo^r)^*\cF)\\
& \simeq FH_*(H,\{\ev_{r*}\circ \Phi^\eta\circ\pi_*m^{\mathrm{Floer}}_{x,w}\};(\geo^r)^*\cF)\\
& = FH_*(H,\{\Phi^\eta(m^{\mathrm{Floer};\, r}_{x,w})\};\cF^r).
\end{align*}
The first isomorphism follows from the functoriality property of enriched homology from~\cite[Remark~4.6]{BDHO}. The second isomorphism is Proposition~\ref{prop:FotimesL-mL} applied with $L=\ueta^{-1}$ (we recall that we denote $\Phi^\eta=\Phi^{\ueta^{-1}}$ for readability). 
The third isomorphism holds because $\geo_r\circ \ev_r$ is homotopic to the identity by Lemma~\ref{lem:cLL1Q}, see also \cite[\S 8.2]{BDHO}. The fourth isomorphism is again an instance of functoriality~\cite[Remark~4.6]{BDHO}. Finally, the last equality is notational. 
\end{proof}

\subsubsection{Proof of the isomorphism theorem}
Recall from~\S\ref{sec:interpolatingFM-chains} the cocycles $\{\Phi^\eta(m^{\mathrm{Floer};\, r}_{x,w})\}$, $\{m^{\mathrm{Morse}}_{z,y}\}$, and $\{b_{x,y}\}$. Define 
$$
\widetilde \Psi^r:FC_*(H,\{\Phi^\eta(m^{\mathrm{Floer};\, r}_{x,w})\};\cF^r)\to C_*(f^r,\{m^{\mathrm{Morse}}_{z,y}\};\cF^r),
$$
$$
\widetilde \Psi^r(\alpha\otimes x) = \sum_{y\in\Crit(f^r)} \alpha\cdot b_{x,y}\otimes y. 
$$

\begin{proposition}
The map $\widetilde \Psi^r$ is a chain map. 
\end{proposition}

\begin{proof}
This is a rephrasing of Proposition~\ref{prop:bxy}.
\end{proof}

In view of Lemma~\ref{lem:FH-finite}, we obtain an induced map  
$$
\widetilde \Psi^r_*: FH_*(H,\{m^{\mathrm{Floer}}_{x,w}\};\Pi^*(\cF\otimes\ueta^{-1})) \to H_*(f^r,\{m^{\mathrm{Morse}}_{z,y}\};\cF^r).
$$

\begin{proof}[Proof of Theorem~\ref{thm:Viterbo_iso_loc_sys_DG}]
The map $\widetilde \Psi^r$ and the parameter $r$ depend by definition on the choice of Hamiltonian $H$. As such, we denote $\widetilde \Psi^r=\widetilde \Psi^{r(H)}_H$. We define 
$$
\widetilde \Psi = \colim_{H} \widetilde \Psi^{r(H)}_{H*} : \colim_H FH_*(H;\Pi^*\cF\otimes \ueta^{-1})\to \colim_r H_*(f^r;\cF^r). 
$$
The chain map $\widetilde \Psi^r$ preserves the canonical filtrations on the enriched Floer and Morse complexes, and induces therefore a morphism of spectral sequences. In the direct limit over $H$ the map $\widetilde \Psi^{r(H)}_H$ induces on the $E^2$-page the map from the statement of Abouzaid's theorem from~\S\ref{sec:intro-Viterbo}, which is an isomorphism. As a consequence, the map $\widetilde \Psi^{r(H)}_H$ induces in the direct limit over $H$ an isomorphism of spectral sequences from the second page onwards. The assumption that $H_*(\cF)$ is bounded from below, together with the general fact that $H_*(\cL W;H_q(\cF))$ is supported in nonnegative degrees for all $q$, ensures that both spectral sequences converge for dimensional reasons (see Theorem~\ref{thm:spectral-sequence-SH}(ii)). Therefore
 $\colim_{H} \widetilde \Psi^{r(H)}_{H*}=\widetilde \Psi$ is an isomorphism as well. 
\end{proof}

\begin{proof}[Proof of Proposition~\ref{prop:Viterbo-iso-energy-zero}]
At energy zero the Viterbo isomorphism becomes a continuation map, as described in~\cite[\S10, Proposition~10.3]{BDHO}. This implies the statement of the proposition.
\end{proof}

We will need the following corollary of Theorem~\ref{thm:Viterbo_iso_loc_sys_DG} and Proposition~\ref{prop:Viterbo-iso-energy-zero}: 

\begin{corollary}\label{cor:Viterbo-cotangent} Consider a closed manifold $Q$ and let 
$\calf$ be a DG local system  on  $\call Q$. Then there exists a DG local system $\cal G$ on  $ \call T^*Q$ such that the Viterbo isomorphism $\widetilde{\Psi}$ of  Theorem~\ref{thm:Viterbo_iso_loc_sys_DG} yields the following commutative diagram: 
$$
\xymatrix{
SH_*^{=0}(T^*Q;\cG)\ar[r] \ar[d]_\simeq^{{\widetilde\Psi}^{=0}} & SH_*(T^*Q;\cG) \ar[d]_\simeq^{\widetilde\Psi} \\
H_*(Q;i_Q^*\cF) \ar[r]^-{i_{Q*}} & H_*(\cL Q;\cF).
}
$$
Moreover, there is an identification 
$$
SH_*^{=0}(T^*Q;\cG)\simeq H_{*+n}(D^*Q, S^* Q; \pi^*i_Q^*\calf\otimes \pi^*\underline{|Q|})
$$ 
such that  $\widetilde{\Psi}^{=0}$ becomes the Thom isomorphism 
$$
H_{*+n}(D^*Q, S^* Q; \pi^*i_Q^*\calf\otimes \pi^*\underline{|Q|})\simeq H_*(Q; i_Q^*\calf),
$$ 
i.e., the inverse of the shriek map of the projection $\pi:D^*Q\ri Q$.
\end{corollary} 
\begin{proof} 
Let $\cG=\Pi^*(\cF\otimes\underline{\eta})$, so that the diagram in Proposition~\ref{prop:Viterbo-iso-energy-zero} becomes
$$
\xymatrix{
SH_*^{=0}(T^*Q;\Pi^*\cF\otimes\Pi^*\ueta)\ar[r] \ar[d]_\simeq^{{\widetilde\Psi}^{=0}} & SH_*(T^*Q;\Pi^*\cF\otimes \Pi^*\ueta ) \ar[d]_\simeq^{\widetilde\Psi} \\
H_*(Q;i_Q^*\cF\otimes i_Q^*\ueta\otimes\underline{|Q|}) \ar[r]^-{i_{Q*}} & H_*(\cL Q;\cF\otimes \ueta^{\otimes 2}).
}
$$
 This local system satisfies the first assertion of the corollary because both  $\underline{\eta}^{\otimes 2}$ and $ i_Q^*\ueta\otimes\underline{|Q|} \simeq\underline{|Q|}^2$ are trivial rank $1$ local systems. 
 
 Moreover, again by Proposition~\ref{prop:Viterbo-iso-energy-zero}, the identification 
 $$SH_*^{=0}(T^*Q;\Pi^*\cF\otimes\Pi^*\ueta)\, \simeq \, H_{*+n}(D^*Q, S^* Q; \pi^*i_Q^*\calf\otimes\pi^*i_Q^*\ueta)$$
 transforms $\widetilde{\Psi}^{=0}$ into the Thom isomorphism. Note that $i_Q^*\ueta\simeq \underline{|Q|}$. 
\end{proof}

\section{The  contractible almost existence property}\label{sec:almost-existence}

Let $(W,\omega)$ be a symplectic manifold. Recall from~\S\ref{sec:intro-almost-existence} the  contractible almost existence property, which states that, for any autonomous and proper Hamiltonian function $H:W\to\R$ and any regular level set $\Sigma=H^{-1}(c)$, there exists $\eps>0$ such that, for almost every $c'\in(c-\eps, c+\eps)$, the level set $H^{-1}(c')$ admits a contractible closed characteristic.  

This section is organised as follows. In~\S\ref{sec:homological-condition-W} we establish a criterion that implies this property for Liouville domains -- condition \eqref{eq:criterion-almost-existence}.  Let $U\subset W$ be a relatively compact domain bounded by $\Sigma$. Fix a complement of a small collar neighborhood of $\Sigma$ in $U$, denoted by  $U_{\mathrm{int}}$.
As pointed out in the introduction, the proof of the almost existence property often relies on the finiteness of some Hofer-Zehnder type capacity. Our proof uses a relative Hofer-Zehnder capacity  $c_{HZ}^0(W,\ol{U}_{\mathrm{int}})$.
This was introduced in a slightly more general setting (see below) by Ginzburg and G\"urel in~\cite{Ginzburg-Gurel-relative-capacity}. The content of~\S\ref{sec:homological-condition-W} can be schematically summarised as
$$
\mathrm{Condition}~\eqref{eq:criterion-almost-existence}\ \Longrightarrow\ {c_{HZ}^0(W,{\ol{U}_{\mathrm{int}}})<+\infty}\   \Longrightarrow\ \text{almost existence for }  \Sigma=\p U.
$$ 

The  second subsection \S\ref{sec:almost-existence-cotangent} is devoted to applications to the case when  the ambient manifold is the unit cotangent bundle $D^*Q\subset T^*Q$. We start by reformulating the criterion \eqref{eq:criterion-almost-existence} for the cotangent bundle using the Viterbo isomorphisms from~\S\ref{sec:DGViterbo} and then we show that it holds for a large class of domains $U$ in cotangent bundles. In particular, we prove the almost existence properties stated in  Theorems~\ref{thm:almost-existence-simplified}, \ref{thm:almost-existence-Thom-simplified} and~\ref{thm:almost-existence-nonorientable}.  
We have already recalled in~\S\ref{sec:intro-other_work} the most important results previously proved on this subject.

\subsection{Homological criterion for  contractible almost exist\-ence} \label{sec:homological-condition-W}\label{sec:criterion1}

In this subsection we consider a Liouville domain $(W,\omega,\lambda)$ and its symplectic homology groups $SH_*(W;\cF)$ with coefficients in a DG local system $\cF$ on $\call W$. As in~\S\ref{sec:intro-Liouville}(i), we consider a relatively compact open subset $U\subset W$ with smooth boundary, the inclusions $i_W:W\hookrightarrow \cL W$ and $j_U:U\hookrightarrow W$, as well as the maps that they induce $i_{W*}$ and $j_{U!}$. 

We rewrite for readability the key diagram~\eqref{eq:diagram-criterion-intro}, 
\begin{equation}\label{eq:diagram-criterion}
  \xymatrix{
     H_{*+n}(U, \partial U;j_U^*i_W^*\cF) &\, \\
      H_{*+n}(W,\p W;i_W^*\cF) \ar[r]^-{i_{W*}}\ar[u]^-{j_{U!}}     
    & SH_*(W;\cF),}
\end{equation}

and we restate for readability Theorem~\ref{thm:almost-sure-existence-W-intro} as Theorem~\ref{thm:almost-sure-existence-W} below. This is the main result of this section. 

\begin{theorem}\label{thm:almost-sure-existence-W} Let $\Sigma\subset W$ be a closed hypersurface bounding a relatively compact open subset $U$. We assume that there exists a DG local system $\cF$ on $\cL W$ and a class $\alpha\in H_*(W,\p W;i_W^*\cF)$ such that $j_{U!}(\alpha) \neq 0$ and $i_{W*}(\alpha)=0$.

Then the    contractible almost existence property holds near $\Sigma$.
\end{theorem}

\begin{remark}\label{rmk:criterion-short} The above criterion for the  contractible almost existence property may be written concisely as 
\begin{equation}\label{eq:criterion-almost-existence}
\ker i_{W*}\, \not\subseteq\, \ker j_{U!}.
\end{equation}
\end{remark}

As mentioned above, the proof goes through the intermediate of a relative version of the so-called \emph{$\pi_1$-sensitive Hofer-Zehnder capacity} (cf.~\cite{HZ}), which was introduced in \cite{Ginzburg-Gurel-relative-capacity} and will be denoted by $c_{HZ}^\circ(W,\overline{U}_{\mathrm{int}})$. 
We will show that this relative capacity is finite by defining a suitable spectral invariant that has the following two properties: it is an upper bound for $c_{HZ}^\circ(W,\overline{U}_{\mathrm{int}})$ 
and, under our assumptions, it is finite. The definition of the spectral invariant is an adaptation to the setting of DG coefficients of a construction of Irie~\cite{Irie14}. That the finiteness of the Hofer-Zehnder capacity implies almost existence is a classical fact due to Hofer-Zehnder and Struwe~\cite{Struwe, Hofer-Zehnder_periodic_solutions,HZ}, and this still holds for our relative capacity.

We begin the proof of Theorem~\ref{thm:almost-sure-existence-W} by giving the definition of the relative capacity for a general pair $(V,Z)$ where $V$ is an open subset of $W$ and $Z\subset V$ a compact  subset.  

\begin{definition}[\cite{Ginzburg-Gurel-relative-capacity}]\label{def:relative-capacity} A Hamiltonian $H$ on $W$ is \emph{$(V,Z)$-admissible} if it is autonomous, compactly supported  in $V$ 
  and equal to its minimum on 
  $Z$; see Figure~\ref{fig:admissible-hamiltonian}.  
We define: \begin{align*}
              c_{HZ}^\circ(V,Z) 
             =\sup\{-\min H : & \ H\text{ is $(V,Z)$-admissible and admits no nontrivial}\\
  & \text{contractible closed orbit of period $\leq1$}\}.
\end{align*}
\end{definition}

\begin{remark}\label{rem:monotonicity-capacity} If $(V',Z')$ is any other pair consisting of an open set and a compact  subset, such that $Z\subseteq Z'\subset V'\subseteq V$, then it follows immediately from the definition that
  \[c_{HZ}^\circ(V',Z')\leq c_{HZ}^\circ(V,Z).\]
\end{remark}

We will be particularly interested in the case where $V$ is the whole ambient manifold and $Z$ is the closure of a small retraction of the domain bounded by $\Sigma$, i.e., $(V,Z)=(W,\overline U_{\mathrm{int}})$.

\begin{figure}
  \centering
  \includegraphics[scale=.6]{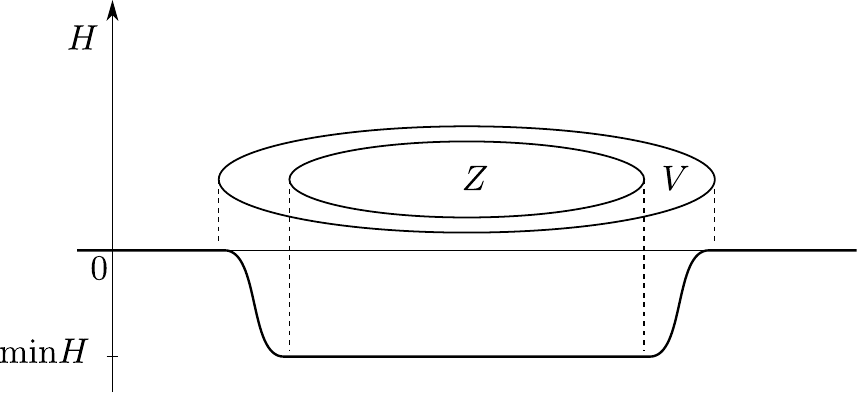} 
  \caption{A  Hamiltonian which is $(V,Z)$-admissible.}
  \label{fig:admissible-hamiltonian}
\end{figure}

Now consider  a nonvanishing class $\alpha\in H_*(W,\p W;i_W^*\cF)$ that belongs to $\ker i_{W*}$ and define its spectral number
\begin{equation}\label{eq:spectral-number}
  c(\alpha)=\inf\{c\in\R : i_{W*}(\alpha)=0\text{ in }SH_*^{<c}(W; \cF)\}.
\end{equation}
The fact that $\alpha$ is nonzero ensures that $c(\alpha)>0$. The finiteness of $c(\alpha)$ is equivalent to the vanishing of $i_{W*}(\alpha)$ in $SH_*(W;\cF)$, by~\eqref{eq:SH-limit-filtered}.

\begin{proposition}\label{prop:finite capacity-W}
  Under the assumptions of Theorem~\ref{thm:almost-sure-existence-W} we have
  \[ c_{HZ}^\circ(W,\overline{U}_{\mathrm{int}}) 
    \leq c(\alpha).\]
  In particular, $c_{HZ}^\circ(W,\overline {U}_{\mathrm{int}})$   
  is finite. 
\end{proposition}

The proof follows  from a slight variation of an argument of Irie~\cite{Irie14}. We will present  it in \S\ref{sec:spectral-invariant-W} and in \S\ref{sec:proof-of-lemma} below.

\begin{proof}[Proof of Theorem~\ref{thm:almost-sure-existence-W} assuming Proposition~\ref{prop:finite capacity-W}] \quad 

  Building on Struwe's enhancement of the Hofer-Zehnder theorem~\cite{Struwe, Hofer-Zehnder_periodic_solutions} in $\R^{2n}$, Hofer and Zehnder explain in~\cite[\S 4.1]{HZ} that, in any symplectic manifold, a hypersurface $\Sigma$ that bounds a relatively compact open subset with finite Hofer-Zehnder capacity satisfies the almost existence property for periodic orbits. The same argument applied to the $\pi_1$-sensitive Hofer-Zehnder capacity gives the almost existence for contractible orbits.  Furthermore, it can also be applied to the relative version of the capacity: finiteness of $c_{HZ}^\circ(V,Z)$ implies almost existence for contractible orbits near $\partial V$ \cite[Theorem~2.14 and Theorem~2.16]{Ginzburg-Gurel-relative-capacity}.

   Recall that  $U_{\mathrm{int}}$ denotes the complement of a small collar neighborhood of $\Sigma$ in $U$. Proposition~\ref{prop:finite capacity-W} applies to $U_{\mathrm{int}}$ and implies that
  $c_{HZ}^\circ(W,\overline U_{\mathrm{int}})<\infty.$
By Remark \ref{rem:monotonicity-capacity}, we deduce that $$c_{HZ}^\circ(U,\overline U_{\mathrm{int}})<\infty,$$ hence the almost existence property for contractible orbits near $\Sigma=\partial U$. 
\end{proof}

  \rmk Since $U_{\mathrm{int}}$ is a small retract of $U$ the criterion \eqref{eq:criterion-almost-existence} is valid for $U_{\mathrm{int}}$ if and only if it is valid for $U$. For readability we will replace $U_{\mathrm{int}}$ by $U$
in the rest of the present section~\S\ref{sec:homological-condition-W}.  
\kmr

This statement of Proposition~\ref{prop:finite capacity-W}  may  then be reformulated as follows: 

\begin{lemma}\label{lemma:existence-fast-orbit-W} Under the assumptions of Theorem~\ref{thm:almost-sure-existence-W}, any $(W,\overline{U})$-admissible Hamiltonian $H$ such that 
  $-\min H > c(\alpha)$ admits a nontrivial contractible closed orbit of period at most $1$. 
\end{lemma}

We now give the proof of  Lemma~\ref{lemma:existence-fast-orbit-W}.

\subsubsection{A spectral invariant}\label{sec:spectral-invariant-W}

In order to prove Lemma~\ref{lemma:existence-fast-orbit-W}, we introduce the following spectral invariant. A similar invariant was introduced by Irie~\cite{Irie14} in the classical case. Given a nondegenerate Hamiltonian $H$ linear at infinity with slope $c$, 
and a nonzero class $\sigma\in SH_*^{<c}(W;\cF)$, we define
\[\rho(H,\sigma)=\inf\{a\, : \, \Psi_H^{-1}(\sigma)\in \im (FH_*^{<a}(H;\cF)\to FH_*(H;\cF))\},\]
where $\Psi_H$ stands for the isomorphism $FH_*(H;\cF)\stackrel{\sim}{\to}SH_*^{<c}(W;\cF)$, see~\eqref{eq:isom-SH-filtered-FH}, where we also explain that $\Psi_H$ can be understood as a continuation map. 

The basic properties of this invariant are listed in the next proposition.

\begin{proposition}\label{prop:spectral-invariant-W} Let $\sigma \in SH_*^{<c}(W;\cF)$ be a nonzero class. The following properties hold:
  \begin{enumerate}
  \item (Continuity) the function $\rho(\cdot,\sigma)$ extends to a continuous function on the set of all (not necessarily nondegenerate) smooth Hamiltonians with slope $c$ at infinity. Here continuity is understood with respect to the Hofer norm.
  \item (Spectrality) for any such Hamiltonian $H$, the value $\rho(H,\sigma)$ belongs to the spectrum $\mathrm{spec}(H)$, i.e., the set of critical values of the Hamiltonian action functional $\cA_H$. 
  \end{enumerate}
\end{proposition}

The proof follows a familiar pattern (see, e.g., \cite{Schwarz, Frauenfelder-Schlenk-convex}) and we only sketch its main steps below. Let $\hat W$ be the symplectic completion of $W$, as in~\S\ref{sec:symplectic_homology}.

\begin{proof} We will establish the inequality
\begin{equation}\label{eq:Hofer-estimate-spectral-inv-W}
    \rho(H^-,\sigma)\leq \rho(H^+, \sigma) + \int_{0}^1\max_{\hat W}(H_t^+-H_t^-)\,dt
  \end{equation}
  for any two nondegenerate Hamiltonians $H^\pm$ that are linear of slope $c$ at infinity. 
  By switching the roles of $H^+$ and $H^-$, this yields 
\[|\rho(H^-,\sigma)-\rho(H^+, \sigma)|\leq \|H^+-H^-\|,\]
where $\|\cdot\|$ stands for the Hofer norm. The first item of the proposition follows.

To prove (\ref{eq:Hofer-estimate-spectral-inv-W}), let 
$(H_s)$ be a regular admissible homotopy which interpolates between $H^+$ and $H^-$ and which is arbitrarily close to a homotopy given by 
  \begin{equation}
H_t^-(x)+\beta(s)(H^+_t(x)-H^-_t(x)),\label{eq:linear-homotopy-W}
\end{equation}
for a smooth cutoff function $\beta:\R\to[0,1]$ that vanishes near $-\infty$ and equals 1 near $+\infty$. Let $u\in \mathcal H(x^+; y^-)$ be a continuation curve associated to $(H_s)$ (for some auxiliary homotopy of almost complex structures $J_s$).
A standard computation involving the Stokes formula yields
\begin{align*}
  0&\leq \int_{\R}\int_{S^1}\|\p_su\|^2\,dt\,ds\\
   &=\cA_{H^+}(x^+)-\cA_{H^-}(y^-)+\int_\R\int_{S^1}(\p_sH_s)\circ u\,dt\,ds\\
   &\leq \cA_{H^+}(x^+)-\cA_{H^-}(y^-) + \int_{0}^1\max_{\hat W}(H_t^+-H_t^-)\,dt +\eps.
\end{align*}
Here $\eps>0$ can be made arbitrarily small by choosing $H_s$ close enough to (\ref{eq:linear-homotopy-W}).
This implies that the continuation map $\Psi^\cF$ from Definition \ref{def:continuation-map} sends the subcomplex $FC_*^{<a}(H^+;\cF)$ into $FC_*^{<a+E}(H^-;\cF)$ (we omit auxiliary data from the notation), where $E=\int_{0}^1\max_{\hat W}(H_t^+-H_t^-)\,dt +\eps$.\footnote{This argument is the same as the one proving Corollary~\ref{cor:DGcontinuation-Hofer-norm}.}
We obtain a commutative diagram
\begin{equation*}
  \xymatrix{FH_*^{<a}(H^+;\cF) \ar[r]\ar[d]_{\Psi^\cF} & FH_*(H^+;\cF) \ar[r]^{\Psi_{H^+}} \ar[d]_{\Psi^\cF} & SH_*^{<c}(W;\cF)\\
    FH_*^{<a+E}(H^-;\cF) \ar[r] & FH_*(H^-;\cF) \ar[ru]_{\Psi_{H^-}} 
    }
\end{equation*}
where all maps in the right triangle are isomorphisms. The triangle is commutative because the isomorphisms $\Psi_{H^+}$ and $\Psi_{H^-}$ are realized by continuation maps, and the composition of two continuation maps is again a continuation map (Proposition~\ref{prop:composition-v1}). 

The inequality (\ref{eq:Hofer-estimate-spectral-inv-W}) readily follows from this diagram, the definition of $\rho(\cdot, \sigma)$ and the fact that $\eps$ can be made arbitrarily small.

\medskip
We now turn to the spectrality property. This property is obvious in the case where the Hamiltonian $H$ is nondegenerate. In the general case, we approximate $H$ in the $C^2$-topology by a sequence of nondegenerate Hamiltonians $(H_i)$. Each number $\rho(H_i,\sigma)$ is the action $\cA_{H_i}(x_i)$ of some 1-periodic orbit $x_i$ of $H_i$. By the Arzela-Ascoli theorem, a subsequence of $(x_i)$ converges to a 1-periodic solution $x$ of the Hamiltonian $H$. By continuity of $\rho(\cdot, \sigma)$, we deduce that $\rho(H,\sigma)=\cA(x)$.
\end{proof}

\subsubsection{Proof of Lemma~\ref{lemma:existence-fast-orbit-W}}\label{sec:proof-of-lemma}

Our proof is a reformulation of an argument of Irie~\cite{Irie14}. All homology groups in this section are understood with coefficients in a fixed DG local system $\cF$, which we henceforth omit from the notation.

\begin{proof} Let $H$ be a $(W,\overline{U})$-admissible Hamiltonian such that 
  $-\min H > c(\alpha)$. We assume that $H$ does not admit any nontrivial contractible closed orbit of period at most $1$. Given $R\ge 1$ we let $W_R = W\sqcup ([1,R]\times \p W$). Equivalently, $W_R$ is the image of $W$ under the time $\log R$ flow of the Liouville vector field in the symplectic completion $\hat W$. Note that we have an identification $\hat W=W_R\sqcup ([R,+\infty)\times \p W)$. 

Recall that $\overline U\subset W=W_1$ and choose $R''> R'>1$. Given $\eps>0$ and $\delta>0$, we consider smooth autonomous Hamiltonians $H_{\eps, \delta}$ of the following form (see Figure \ref{fig:haliltonian-family}): 
  \begin{itemize}
  \item[---] on $W$ the Hamiltonian $H_{\eps,\delta}$ coincides with $\eps H$. 
  \item[---] on $(1,+\infty)\times \p W$, it is equal to a function $h:(1,+\infty)\to\R$ that vanishes on $[1, R']$, that increases and is strictly convex on $(R', R'')$, and that is linear of slope $\delta$ on $(R'',+\infty)$.
\end{itemize}

\begin{figure}
  \centering
  \includegraphics[scale=.6]{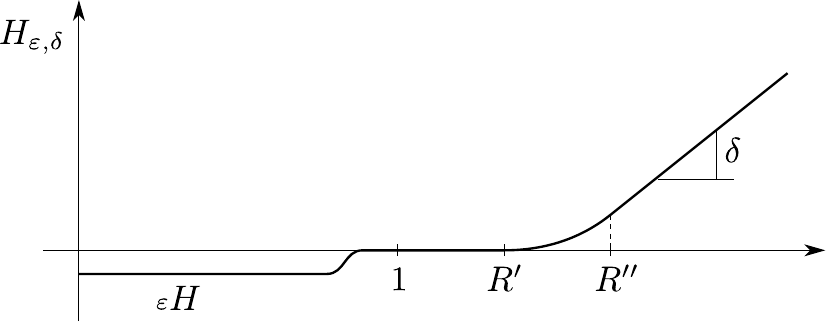} 
  \caption{Deformation.}
  \label{fig:haliltonian-family}
\end{figure}

  Let $i_\delta$ be the natural map $SH^{=0}_*(W)\to SH^{<\delta}_*(W)$.

  \noindent\textbf{Step 1: $\rho(H_{\eps, \delta}, i_{\delta}(\alpha))=-\eps\min H$ for $\eps, \delta$ small enough. } 
  
  In order to compute $\rho(H_{\eps, \delta}, i_{\delta}(\alpha))$ we perturb the Hamiltonian into a sequence of $C^2$-small Morse functions $(H_i)$, while keeping it fixed in the region where it is linear of slope $\delta$.  We choose our perturbations so that a negative gradient of $H_i$ points outwards along the boundary $\partial U$. This implies that for each of these, the Floer complex reduces to the Morse complex of $-H_i$ and the homology $H_*(U,\p U;\cF)$ is computed by a complex
   spanned by critical points arising in the region where $H_{\eps,\delta}$ is equal to its minimum $\eps H$. The assumption ${j_U}_!(\alpha)\neq 0$ in the hypothesis of Theorem~\ref{thm:almost-sure-existence-W} means that $\alpha$ is not represented in the complement of $U$ in these Morse complexes. To see this, recall from \cite[\S 9.1]{BDHO} that the shriek map $j_{U!}: H_*(W,\p W;\calf)\ri H_*(U,\p U; j_{U!}^*\calf )$ was defined at the level of complexes as a projection: we consider a Morse-Smale pair $(f,\xi)$ on $U$ with the negative gradient $\xi$ pointing outwards along the boundary $\p U$ and we extend it to $(W,\p W)$ with negative gradient pointing outwards along $\p W$. If ($F,\ol \xi)$ is the resulting couple on $W$, it is easy to see that the Morse complex defined by $(f,\xi)$ is isomorphic to the quotient of the Morse complex defined by $(F,\ol\xi)$ by its subcomplex spanned by the critical points of $F$ in $W\setminus U$. We thus have a well-defined projection between the two and in the case of DG coefficients this projection is by definition the shriek map $j_{U!}$ above.
  
   Thus, the minimal level at which $\alpha$ is represented in the Morse complex of $-H_i$ is close to $-\eps \min H$. We conclude that  $\rho(H_i, i_{\delta}(\alpha))$ is close to $-\eps \min H$. Letting $i$ go to infinity, we obtain  $\rho(H_{\eps, \delta}, i_{\delta}(\alpha))=-\eps\min H$ as claimed.

    \noindent\textbf{Step 2: $\rho(H_{1, \delta}, i_{\delta}(\alpha))=-\min H$ for $\delta$ small enough. } 
   
   We consider the family of Hamiltonians $H_{s,\delta}$ for $s\in[\eps, 1]$. By our assumption of nonexistence of nonconstant contractible periodic orbits of $H$ of period less than $1$, and also because there are no periodic orbits of period $1$ in the region $(R',\infty)\times \p W$, we infer that the spectrum of $H_{s,\delta}$ is  of the form $s \cE$, where $\cE$ is the set of critical values of $-H$.  
   The conclusion follows from the continuity and the spectrality of $\rho(\cdot, i_{\delta}(\alpha))$.  Indeed, the function $s\mapsto \frac 1s\rho(H_{s,\delta}, i_{\delta}(\alpha))$ is continuous on $[\delta, 1]$ and takes values in $\cE$. Since $\cE$ is totally disconnected by Sard's lemma, this function must be constant.
   
    \noindent\textbf{Step 3: Getting a contradiction. } Recall that we work under the standing assumption $-\min H >c(\alpha)$. Let us choose $-\min H > a> c(\alpha)$. We consider the following commutative diagram, in which we abbreviate $\Psi_\delta=\Psi_{H_{1,\delta}}$ and $\Psi_{a}=\Psi_{H_{1,a}}$: 
\begin{equation*}
      \xymatrix{
        SH_*^{=0}(W) \ar[dr]_-{i_{a}} \ar[r]^{i_{\delta}} & SH_*^{<\delta}(W) \ar[d] & \ar[l]_-{\Psi_\delta}^-\sim FH_*(H_{1,\delta})\ar[d]\ar[r] & FH_*^{(a,+\infty)}(H_{1,\delta})\ar[d]_-{\wr}
        \\
        & SH_*^{<a}(W)  & \ar[l]_-{\Psi_{a}}^-\sim FH_*(H_{1,a}) \ar[r]& FH_*^{( a,+\infty)}(H_{1,a})
      }      
\end{equation*}
The rightmost horizontal arrows are induced by projection maps onto the corresponding Floer quotient complexes in action $\ge a$. The rightmost vertical arrow is an isomorphism because the Hamiltonians $H_{1,\delta}$ and $H_{1,a}$ coincide in the neighborhood of their critical set of action $\ge a$ (which is the locus of critical points at level $-\min H$). Moreover, by the maximum principle the Floer orbits involved in the differential are constrained to $W$, where the Hamiltonians coincide. 

We now look at how the class $\alpha\in SH^{=0}_*$ is spread in the diagram by these maps. We will obtain a contradiction by showing that the image of $\alpha$ in $FH^{\geq a}(H_{1,\delta})$ obtained by following the first line in the diagram is nonzero, whereas the image of $\alpha$ in $FH^{\geq a}(H_{1,a})$ obtained by following the second line in the diagram is zero. 

Consider the first line of the diagram. The image of $\alpha$ in $FH(H_{1,\delta})$ is nonzero by definition of $c(\alpha)$ and the fact that $\delta< c(\alpha)$. To prove that its image in $FH^{(a,+\infty)}(H_{1,\delta})$ is also nonzero we argue by contradiction. Assuming the contrary, we would get that $\alpha$ is represented in $FH^{< a}(H_{1,\delta})$, hence the spectral invariant would satisfy $\rho(H_{1, \delta}, i_{\delta}(\alpha)) \le a$. However, this contradicts Step 2 in view of $-\min H>a$.

We now consider the image of $\alpha$ along the second line of the diagram. By definition of $c(\alpha)$ we have $i_{a}(\alpha)=0$, hence the image of $\alpha$ in $FH^{(a,+\infty)}(H_{1,a})$ is also zero. 
\end{proof}

\subsection{Applications to cotangent bundles}\label{sec:almost-existence-cotangent}

In this subsection we consider the particular case of cotangent bundles and we recast the results of the previous section \S\ref{sec:homological-condition-W} in topological terms using the Viterbo isomorphism from Theorem~\ref{thm:Viterbo_iso_loc_sys_DG}. We then prove Theorem~\ref{thm:almost-existence-simplified} and Theorem~\ref{thm:almost-existence-Thom-simplified}, as a consequence of the more general Theorems~\ref{thm:almost-existence} and~\ref{thm:almost-existence-Thom}.
 
Let $Q$ be a closed manifold and $U\subset D^*Q$ a relatively compact open subset with smooth boundary $\Sigma=\p U$. Recall the notation  $i_Q:Q\to\call_0 Q\subset \call Q$  for the canonical inclusion 
 of the constant loops and $j_U:U\to D^*Q$ for the inclusion of $U$ into $D^*Q$, and let $\pi_U:U\to Q$ be the restriction to $U$ of the canonical projection $\pi:T^*Q\to Q$.

Our criterion \eqref{eq:criterion-almost-existence} established in Theorem~\ref{thm:almost-sure-existence-W}  specializes in this situation to the following: 

\begin{proposition}\label{prop:almost-sure-existence-T*Q} Let $Q^n$ be a smooth closed orientable manifold and let $\Sigma\subset T^*Q$ be a closed hypersurface bounding a relatively compact open subset $U$. We assume that there exists a DG local system $\cF$ on $\cL_0 Q$ such that  
\begin{equation}\label{eq:criterion-almost-existence-cotangent}
\ker i_{Q*} \, \not\subseteq \ker \pi_{U!}
\end{equation} 
where 
$$
i_{Q*}: H_*(Q, i_Q^*\calf)\ri H_*(\call_0 Q;\calf),
$$ 
$$
\pi_{U!}: H_*(Q; i_Q^*\calf)\ri H_{*+n}(U,\p U; \pi_U^*i_Q^*\calf)
$$ 
are, respectively, the direct map defined by $i_Q$ and the shriek map defined by $\pi_U$. Then  $\Sigma = \p U$ has the contractible almost existence property.  
\end{proposition}

\begin{proof} Let $\cal G$ be a DG local system on $T^*Q$ provided by Corollary \ref{cor:Viterbo-cotangent}. We put together the diagram of this corollary with  diagram \eqref{eq:diagram-criterion} and get 
$$
\xymatrix{
H_{*+n}(U,\p U; \pi_U^*i_Q^*\calf)&\, &\, \\
H_{*+n}(D^*Q,S^*Q; \pi^*i_Q^*\calf)\ar[r]_-{\simeq}\ar[u]_-{j_{U!}}\ar@/^2pc/[rr]^- {(i_{D^*Q})_*}
&SH_*^{=0}(T^*Q;\cG)\ar[r] \ar[d]_\simeq^{{\widetilde\Psi}^{=0}} & SH_*(T^*Q;\cG) \ar[d]_\simeq^{\widetilde\Psi} \\
\, & H_*(Q;i_Q^*\cF) \ar[r]^-{i_{Q*}} \ar[ul]^-{\simeq}_-{\pi_!}& H_*(\cL_0 Q;\cF).
}
$$

Recall that the criterion  \eqref{eq:criterion-almost-existence} is $\ker (i_{D^*Q})_*\, \not\subseteq \ker j_{U!}$. Composing these maps with  the Thom isomorphism $\pi_!$ we immediately get \eqref{eq:criterion-almost-existence-cotangent} from the above diagram.
\end{proof} 

\subsubsection{Proof of the almost existence results in the cotangent bundle} 

As we will see, Theorem~\ref{thm:almost-existence-simplified} is a consequence of the following more general result.

    \begin{theorem}\label{thm:almost-existence}  Let $Q^n$ be a closed orientable manifold. Let $\varphi:S\to T^*Q$ be a continuous map, defined on a 
  smooth $n$-dimensional orientable closed manifold $S$, and such that the degree $d=\deg(\pi\circ\varphi)$ is nonzero.

  Assume that either $\pi_2(Q)\neq 0$ or
   one of the following conditions is satisfied: \begin{enumerate}\item The first nonzero homotopy group $\pi_k(Q)$, $k\geq 3$ has an element that is not of $d$-torsion.
   \item There exists $\ell\geq 3$ such that, for any integer $m \geq 1$, the homotopy group $\pi_\ell(Q)$ is not of $d^m$-torsion.
 \end{enumerate}
     
  Then, any hypersurface $\Sigma\subset T^*Q$ that bounds a relatively compact domain $U$ that contains $\varphi(S)$ has the almost existence property for closed characteristics that are contractible in $T^*Q$.
    \end{theorem}

 We will use our criterion  \eqref{eq:criterion-almost-existence-cotangent} for the DG local system $\calf= C_*(\Omega\call_0  Q)$ with module structure over itself given by right multiplication. As a first step, we prove the following proposition which provides information on $\ker(i_{Q*})$.

\begin{proposition}\label{prop:Ker-i}  Let $\pi_k(Q)$  be the first nontrivial group among the higher homotopy groups $\pi_j(Q)$, $j\geq 2$. Then, if $k=2$, $\ker i_{Q*}$ contains an infinite cyclic subgroup, and if $k>2$ we have: 
\begin{enumerate} 
\item  $\ker i_{Q*}$ contains a subgroup that is isomorphic to   $\pi_k(Q)$.
\item let $d\geq 2$ be an integer. If some homotopy group $\pi_\ell(Q)$ for $\ell>2$ 
 is not of $d^m$-torsion for any $m\geq 1$, then $\ker i_{Q*}$ is not of $d$-torsion. 
\end{enumerate}
\end{proposition}

\begin{proof}
  As pointed out in \S\ref{sec:def-MorseDG}, the homology $H_*(\call_0Q;\calf)$ is isomorphic to the homology of  the total space of the path-loop fibration $\calp_{\star\ri\call_0 Q}\call_0Q\ri \call_0 Q$. 
  We refer to \cite[Theorem~7.2 and  Remark~13.19]{BDHO} for the proof of this result.  On the other hand, by \cite[Proposition~9.14 and \S13]{BDHO} we know that  $H_*(Q; i_Q^*\calf)$ is isomorphic to the homology of the total space of  the pullback of the preceding fibration by the inclusion of constant loops $i_Q:Q\ri \call_0Q$. Moreover, via these isomorphisms, the direct map $i_{Q*}$ coincides with the map induced by the inclusion at the level of total spaces in homology (see also the discussion in the paragraph ``Functoriality'' in \S\ref{sec:def-MorseDG}). In other words $$H_*(Q; i_Q^*\calf)\, \simeq\, H_*(\calp_{\star\ri Q}\call_0Q;\Z), $$
 and  $$i_{Q*}: H_*(\calp_{\star\ri Q}\call_0Q;\Z)\ri H_*(\calp_{\star\ri \call_0Q}\call_0Q;\Z)\simeq H_*(\mathrm{point};\Z).$$
In particular $\ker i_{Q*}$ contains $H_*(\calp_{\star\ri Q}\call_0Q;\Z)$ in positive degrees. Let us now analyse this homology. 
 
 We may  identify $\calp_{\star\ri Q}\call_0Q$ with the iterated based loop space $\Omega^2Q$. Indeed, denoting by $\SS^2$ the $2$-sphere, by $\SS^0$ the $0$-sphere, by $\Sigma^2$ the double suspension, and by $\mathrm{Map}_*$ the space of based maps, we have canonical identifications 
$$
\calp_{\star\ri Q}\call_0Q  = \mathrm{Map}_*(\SS^2,Q)=\mathrm{Map}_*(\Sigma^2\SS^0,Q) 
= \mathrm{Map}_*(\SS^0,\Omega^2 Q) = \Omega^2 Q.
$$
This implies 
$$
\pi_j(\calp_{\star\ri Q}\call_0Q) \simeq \pi_{j+2} (Q)
$$
for all $j\ge 0$. 

In particular for $j=0$ we have $\pi_0(\calp_{\star\ri Q}\call_0Q)=\pi_2(Q)$ and the conclusion follows for $k=2$: in degree $0$ we have $i_{Q*}: \Z[\pi_2(Q)]\ri \Z$ and therefore its kernel contains an infinite cyclic subgroup. 

Suppose now  $k>2$ and let us prove the first assertion.
We distinguish two cases. 
\begin{itemize}
\item $k=3$. In this case $\calp_{\star\ri Q}\call_0Q$ is connected,  $\pi_1(\calp_{\star\ri Q}\call_0Q) \simeq\pi_3(Q)$ is abelian, hence by Hurewicz  $H_1(\calp_{\star\ri Q}\call_0Q)\simeq\pi_3(Q)$ and the conclusion follows from $\ker i_{Q*}\supset H_1(\calp_{\star\ri Q}\call_0Q)$. 
\item $k\ge 4$. We have that  $\calp_{\star\ri Q}\call_0Q$ is $(k-3)$-connected, therefore simply connected. Again by Hurewicz we get $$H_{k-2}(\calp_{\star\ri Q}\call_0Q)\simeq \pi_{k-2}(\calp_{\star\ri Q}\call_0Q)\simeq \pi_{k}(Q)$$ 
and the conclusion follows from $\ker i_{Q*}\supset H_{k-2}(\calp_{\star\ri Q}\call_0Q)$. \end{itemize}

We now prove the second assertion of our statement. By the above, it suffices to prove that  $H_{p}(\Omega^2Q;\Z)$ is not of $d$-torsion  for some $p>0$. 

 Denote by $\cal C$ the collection of all the abelian groups that are of $d^m$-torsion for some integer $m\geq 1$. Remark that $\cal C$ is a {\it class} in the sense of Serre~\cite{Serre1952}, meaning that, if $L\ri M\ri N $ is an exact sequence of abelian groups and $L,N \in \cal C$, then $M\in \cal C$. In the aforementioned paper, Serre proves the following generalisation of the Hurewicz theorem.
\begin{theorem}\label{thm:Serre}\cite[Theorem~1]{Serre1952} If $Y$ is a simply connected topological space and $\cal C$ is a class,  then for any integer $k\geq 1$ we have $\pi_i(Y)\in \cal C$ for $i\leq k$ if and only if $H_i(Y;\Z)\in \cal C$ for $i\leq k$. 
\end{theorem} 

We apply it for $Y=\Omega_0 Q$, the space of based contractible loops. The space $Y$ is simply connected because we work under the assumption that $\pi_2 (Q)=0$. Using our assumption on the homotopy groups we get some homology group  $H_\ell(\Omega_0Q;\Z)$ that possesses  $d^m$-torsion-free elements  for any integer $m\geq 1$.

  Arguing by contradiction, let us suppose that  $H_i(\Omega^2Q;\Z)$ is of $d$-torsion for any degree $i>0$. 
  Consider the path-loop fibration 
$$\Omega^2Q\ri \calp_{\star\ri \Omega_0Q}\Omega_0 Q\ri \Omega_0 Q$$
and its corresponding Leray-Serre spectral sequence whose second page is 
$$
E_{p,q}^2=  H_p(\Omega_0Q;H_q(\Omega^2Q;\Z))\simeq H_p(\Omega_0Q;\Z)\otimes H_q(\Omega^2Q;\Z). 
$$
It converges to $H_*(\calp_{\star\ri \Omega_0Q}\Omega_0Q)\simeq H_*(\mathrm{point})$. In particular $E_{p,q}^\infty = 0$ for all $(p,q)\neq (0,0)$. We will get a contradiction by proving that $E_{\ell,0}^r$ is not of $d^m$-torsion for any $m\geq 1$ and  $r\geq 2$.  We proceed by induction over $r$. For $r=2$ the property is true, since $E_{\ell,0}^2= H_\ell(\Omega_0Q;\Z)$. Supposing that it is fulfilled for $r$, notice that $E_{\ell,0}^{r+1}$ is the kernel of the map $E^r_{\ell,0}\to E_{\ell-r, \ell+r-1}^{r}$. By our assumption on $H_i(\Omega^2Q;\Z)$, we know that $E_{p,q}^2$ is of $d$-torsion if $q>0$ and this property obviously persists throughout all the following pages of the spectral sequence. In particular $E_{\ell-r, \ell+r-1}^{r}$ is of $d$-torsion. Now, if $E_{\ell, 0}^{r+1}$ was of $d^m$-torsion for some $m$, we would have that $E_{\ell, 0}^{r}$ is of $d^{m+1}$-torsion thus contradicting the induction hypothesis.
 This finishes our proof.
\end{proof}

Now that we established the existence of some non-trivial classes in $\ker i_{Q*}$ we turn our attention to the other element in the criterion \eqref{eq:criterion-almost-existence-cotangent}, namely $\ker \pi_{U!}$. We prove:

\begin{proposition}\label{propo:d-torsion} Assume there exists $\varphi:S\to U$ as in Theorem~\ref{thm:almost-existence}. Then,  $\ker \pi_{U!}$ is of $d$-torsion.
\end{proposition} 
\begin{proof}

Let $\varphi:S\ri U$ be as in Theorem~\ref{thm:almost-existence}. Consider the diagram 
$$
\xymatrix{\, &H_{*+n}(U,\p U;\pi_U^*i_Q^*\calf)\ar[dl]_-{\varphi_!}&\, \\
H_*(S;(\pi\circ\varphi)^*i_Q^*\calf)&\, &H_*(Q;i_Q^*\calf) \ . \ar[ul]_-{\pi_{U!}}\ar[ll]_-{(\pi\circ\varphi)_!}} 
$$
The diagram commutes since obviously $\pi_U\circ\varphi= \pi\circ\varphi$. In particular we have  $\ker \pi_{U!}\subset \ker (\pi \circ\varphi)_!$. On the other hand, since $\pi\circ\varphi:S\ri Q$ is of degree $d$, Proposition~10.14 of \cite{BDHO} asserts that  $(\pi\circ\varphi)_*\circ(\pi\circ\varphi)_!=d\cdot\Id$ which implies that $\ker(\pi\circ\varphi)_!$ is of $d$-torsion.
\end{proof} 
We are now ready to finish the proof of Theorem~\ref{thm:almost-existence}.

\begin{proof}[Proof of Theorem~\ref{thm:almost-existence}]
We combine Proposition~\ref{prop:Ker-i} and Proposition~\ref{propo:d-torsion} to prove that the criterion \eqref{eq:criterion-almost-existence-cotangent} for the  contractible almost existence property is satisfied,  i.e., $\ker i_{Q*}\not\subseteq \ker \pi_{U!}$. 

 We claim that, under the assumptions of Theorem~\ref{thm:almost-existence}, $\ker i_{Q *}$ contains at least one element that is not of $d$-torsion. Since $\ker \pi_{U!}$ is $d$-torsion by Proposition~\ref{propo:d-torsion}, we find that $\ker i_{Q*}\not\subseteq \ker \pi_{U!}$.

To prove the claim, assume first that $\pi_2(Q)\neq 0$. By the first part of Proposition~\ref{prop:Ker-i} we find that $\ker i_{Q*}$ contains an infinite cyclic subgroup, which is in particular not of $d$-torsion. 

Assume now $\pi_2(Q)=0$. Under condition 1 of  Theorem~\ref{thm:almost-existence}, by applying Proposition~\ref{prop:Ker-i}.1 we find that $\ker i_{Q *}$ contains at least one element that is not of $d$-torsion. Under condition 2, by applying Proposition~\ref{prop:Ker-i}.2  we find again that $\ker i_{Q *}$ is not of $d$-torsion. This proves the claim, and therefore the theorem. 
\end{proof}

\begin{proof}[Proof of Theorem~\ref{thm:almost-existence-simplified} as a consequence of Theorem~\ref{thm:almost-existence}]
\qquad 

Assume first that $d=\pm 1$.  Then $Q$ is not a $K(\pi,1)$ if and only if $Q$ satisfies condition 1. in the statement of Theorem~\ref{thm:almost-existence}, and we therefore conclude.  

Assume now that $d$ is nonzero and $Q$ is not a rational $K(\pi,1)$, i.e, $\pi_\ell(Q)\otimes \Q\neq 0$ for some $\ell\geq 2$. Then $Q$ satisfies condition 2. in the statement of Theorem~\ref{thm:almost-existence}, and we again conclude. 
\end{proof}

Theorem~\ref{thm:almost-existence-Thom-simplified} is a consequence of the following slightly more general statement. 

\begin{theorem} \label{thm:almost-existence-Thom} Let $Q^n$ be a closed orientable manifold, and assume that either $\pi_2(Q)\neq 0$ or some homotopy group $\pi_\ell(Q)$, for $\ell\geq 3$, has an element that is not torsion.  

Given a relatively compact domain  $U\subset T^*Q$ such that the projection $\pi :U\ri Q$ induces a nonzero map $H_n(U;\Z)\ri H_n(Q;\Z)$, the boundary $\Sigma=\p U$  has the almost existence property for closed characteristics that are contractible  in $T^*Q$. 
\end{theorem}

\begin{proof}
Let $\alpha \in H_n(U;\Z)$ be a class such that $\pi_{U*}(\alpha)\in H_n(Q;\Z)$ is not zero, i.e., $\pi_{U*}(\alpha)= b\cdot [Q]$ for some $b\in \N^*$. A theorem of Thom~\cite[Theorem~II.29 and Remark before Theorem~III.1]{Thom} asserts that, for some integer $N\in \N^*$, the class $N\cdot \alpha$ may be represented by a submanifold, i.e., $N\cdot \alpha=\varphi_*([S])$ for some map $\varphi:S\ri U$. We are  in the framework of Theorem~\ref{thm:almost-existence} with the degree of the map $\pi_U\circ\varphi$ being equal to $d=Nb$. We may thus apply Proposition~\ref{propo:d-torsion} and infer that $\ker \pi_{U!}$ is of $d$-torsion. 

It suffices therefore to prove that $\ker i_{Q*}$ is not  torsion in order to be able to apply the criterion \eqref{eq:criterion-almost-existence-cotangent}. The proof is similar to the one of Proposition~\ref{prop:Ker-i}: we apply Serre's generalisation of the Hurewicz theorem for the class of torsion groups and then use the Leray-Serre spectral sequence for the path-loop fibration over $\Omega_0Q$. 
\end{proof}

\begin{proof}[Proof of Theorem~\ref{thm:almost-existence-Thom-simplified} as a consequence of Theorem~\ref{thm:almost-existence-Thom}] 
\qquad 

If $Q$ is not a rational $K(\pi,1)$, meaning that $\pi_\ell(Q)\otimes \Q\neq 0$ for some $\ell\ge 2$, then it obviously satisfies the topological assumption from the statement of Theorem~\ref{thm:almost-existence-Thom}. We therefore conclude using Theorem~\ref{thm:almost-existence-Thom}.
\end{proof}

\begin{proof}[Proof of Theorem~\ref{thm:almost-existence-nonorientable}] The theorem is an immediate consequence of the  following slightly more general statement.

\begin{theorem}\label{thm:C-more-general}
  Let $Q^n$ be a closed non-orientable $n$-dimensional manifold. 
   and let $\varphi:S \ri T^*Q$ be a continuous map defined on a closed $n$-dimensional manifold $S$ such that $\pi\circ\varphi$ has non-zero mod-2 degree. 
   
  Assume  moreover that either $\pi_2(Q)\neq 0$, or $\pi_\ell(Q)\otimes_\Z\Z_2\neq 0$ for some  $\ell\geq 3$. 
  
   Then, any hypersurface $\Sigma\subset T^*Q$ that bounds a relatively compact domain $U$ that contains $\varphi(S)$ has the  contractible almost existence property.
\end{theorem}
\begin{proof} Notice first that  the criterion \eqref{eq:criterion-almost-existence-cotangent} is still valid for $Q$ non-orientable if we use $\Z_2$-coefficients instead of integer coefficients. Passing to $\Z_2$-coefficients throughout the  whole process that led to our criterion simply consists in changing  the DG local system $\calf$ into $\calf\otimes\Z_2$.  

Under our assumption $\mathrm{deg}(\pi\circ\varphi)= 1 \, \, \mathrm{ mod} \, \,  2$ the same argument as in the proof of Proposition~\ref{propo:d-torsion} above yields $\ker(\pi_{U!})=0$. This is true for any DG local system of the form $\calf\otimes \Z_2$ under the hypothesis of the theorem. In order to have  our criterion satisfied, we need to establish that $$\ker \left( i_{Q*}: H_*(Q; i_Q^*\calf\otimes\Z_2)\ri H_*(\call_0Q; \calf\otimes\Z_2)\right)$$ does not vanish for some choice of $\calf$. As above, we choose $\calf=C_*(\Omega\call_0Q)$ and applying  again \cite[Theorem~7.2]{BDHO} we may consider that $i_{Q*}$ is the map $$H_*(\Omega^2Q;\Z_2)\simeq H_*(\calp_{\star\ri Q}\call_0Q; \Z_2)\ri H_*(\calp_{\star\ri \call_0Q}\call_0Q;\Z_2)\simeq H_*(\mathrm{point};\Z_2)$$ induced by the inclusion $\calp_{\star\ri Q}\call_0Q\hookrightarrow \calp_{\star\ri \call_0Q}\call_0Q$. Again, if $\pi_2Q\neq 0$ then $\Omega^2Q$ is not connected and therefore $\ker i_{Q_*}\neq 0$ and the theorem is proved. Then, if $\pi_2 Q=0$, we argue by contradiction and suppose that
$H_i(\Omega^2Q; \Z_2)= 0$ for all $i>0$. As in Proposition~\ref{prop:Ker-i} we consider the path-loop fibration 
$$\Omega^2Q\ri \calp_{\star\ri \Omega_0Q}\Omega_0 Q\ri \Omega_0 Q.$$
 Our assumption implies that the second page $E_{p,q}^2=H_p\big(\Omega_0Q;H_q(\Omega^2Q;\Z_2)\big)$ of the Leray-Serre spectral sequence with $\Z_2$-coefficients associated with this fibration vanishes for $q>0$, and since this spectral sequence converges to the homology of a point we infer that  $E^{2}_{p,0}\simeq H_p(\Omega_0Q;\Z_2)$ vanishes for $p>0$ too. 
 
 On the other hand, denote by $\cal C$ the set of the abelian groups $G$ such that $G\otimes_{\Z}\Z_2= 0$ and notice that it is a class in the sense of Serre. Therefore  we may apply Theorem~\ref{thm:Serre} to the simply connected space $Y=\Omega_0Q$ and deduce (from the fact that $\pi_{\ell}(Q)\otimes \Z_2\neq 0$ for some $\ell\geq 3$) that for some positive  $k$ we have $H_k(\Omega_0Q;\Z)\otimes_{\Z}\Z_2\neq 0$. Since  by the universal coefficients theorem we have an injection $H_k(\Omega_0Q;\Z)\otimes_{\Z}\Z_2\to H_k(\Omega_0Q;\Z_2)$ we infer that the latter does not vanish either and we get a contradiction. 
 
The proof is complete. 
\end{proof}
 Theorem~\ref{thm:almost-existence-nonorientable} is a direct consequence of Theorem \ref{thm:C-more-general}.
\end{proof}

\section{Finiteness of the $\pi_1$-sensitive Hofer-Zehnder capacity} \label{sec:HZ}

\subsection{Homological criterion for finiteness}\label{sec:criterion2}

This subsection is the counterpart of~\S\ref{sec:homological-condition-W}. Our goal is to put into context the criterion for the finiteness of the $\pi_1$-sensitive Hofer-Zehnder capacity of subsets of Liouville domains from Theorem~\ref{thm:finiteness-HZ-intro}.

Let $(W,\omega,\lambda)$ be a Liouville domain of dimension $2n$.

\begin{definition}\label{def:HZ-capacity} Let $U\subset W$ be a relatively compact open subset with smooth boundary. A Hamiltonian $H$ on $W$ is said to be \emph{HZ-admissible for $U$} if it is autonomous, compactly supported in $U$ and equal to its minimum on some open subset of $U$.   

We define the \emph{$\pi_1$-sensitive Hofer-Zehnder capacity of $U$} to be
\begin{align*}
  c_{HZ}^\circ(U)=\sup\{-\min H : & \ H\text{ is HZ-admissible for $U$ and admits no nontrivial}\\
  & \text{contractible closed orbit of period $\leq1$}\}.
\end{align*}
\end{definition} 

The terminology ``$\pi_1$-sensitive'' refers to the fact that we only take into account contractible orbits. There are variations of the definition that take into account some arbitrary set of free homotopy classes in $U$, see~\cite{Irie14}. The following remark will be useful.
\begin{remark}\label{rem:cover-Hofer-Zehnder} Let $\pi:W'\to W$ be a symplectic covering map between Liouville domains and $U\subset W$ be an open set such that  $c_{HZ}^\circ(\pi^{-1}(U))<+\infty$. Then, $c_{HZ}^\circ(U)<+\infty$ too.
\end{remark}

It follows from the definition that the function $c_{HZ}^\circ$ is monotone with respect to inclusion: 
$$
U\subseteq V \qquad \Rightarrow \qquad c_{HZ}^\circ(U)\le c_{HZ}^\circ(V). 
$$
Also, because any Hamiltonian $H$ that is $(V,Z)$-admissible in the sense of Definition~\ref{def:relative-capacity} is also HZ-admissible for $V$, we find 
$$
 c_{HZ}^\circ(V,Z)\le c_{HZ}^\circ(V). 
$$
As a consequence, the finiteness of the $\pi_1$-sensitive Hofer-Zehnder capacity of the Liouville domain $W$ implies finiteness of the $\pi_1$-sensitive Hofer-Zehnder capacity $c_{HZ}^\circ(U)$ of any $U\subset W$, and also the finiteness of the relative $\pi_1$-sensitive Hofer-Zehnder capacity $c_{HZ}^\circ(W,\overline U)$ for any $\ol{U}\subset \mathrm{int}(W)$. 
The next result underpins the dynamical significance of the Hofer-Zehnder capacity, and should be compared with Theorem~\ref{thm:almost-sure-existence-W} and Proposition~\ref{prop:finite capacity-W}.

\begin{proposition}[{\cite[\S 4.1]{HZ},\cite{Struwe}}]  \label{prop:almost-existence-HZ-Struwe}
Assume the  $\pi_1$-sensitive Hofer-Zehnder capacity of $W$ is finite. Then the  contractible almost existence property holds in $W$. \qed
\end{proposition}

We now restate for readability Theorem~\ref{thm:finiteness-HZ-intro} from the Introduction as Theorem~\ref{thm:finiteness-HZ} below. This is our criterion for the finiteness of $c_{HZ}^\circ(W)$. 

Recall from~\S\ref{sec:intro-Liouville} the inclusion of constant loops $i_W:W\hookrightarrow \cL_0W$, the inclusion of a point $j:\star\hookrightarrow W$ and the inclusion of an open domain $j_U:U\hookrightarrow W$, as well as the maps that they induce $i_{W*}$, $j_!$ and $j_{U!}$. We also recall the key diagram~\eqref{eq:diagram-finiteness-HZ-intro}
\begin{equation}\label{eq:diagram-finiteness-HZ}
  \xymatrix{
     H_{*-n}(\cF) &\, \\
      H_{*+n}(W,\p W;i_W^*\cF) \ar[r]^-{i_{W*}}\ar[u]^-{j_{!}}     
    & SH_*(W;\cF).}
\end{equation}

\begin{theorem}\label{thm:finiteness-HZ} Assume that there exists a DG local system $\cF$ on $\cL_0W$ and a class $\alpha\in H_*(W,\p W;i_W^*\cF)$ such that $j_{!}(\alpha) \neq 0$ and $i_{W*}(\alpha)=0$. Then 
$$
c_{HZ}^\circ(W)<\infty. 
$$
In particular, the  contractible almost existence property holds in $W$.
\end{theorem} 

\begin{proof}
Finiteness of the Hofer-Zehnder capacity follows from Proposition~\ref{prop:finite capacity} below. The  contractible almost existence property is ensured by Proposition~\ref{prop:almost-existence-HZ-Struwe} above.
\end{proof}

\begin{remark}\label{rmk:criterion-short-HZ} 
Similarly to~\eqref{eq:criterion-almost-existence}, the above criterion for the contractible almost existence property may be written concisely as 
\begin{equation}\label{eq:criterion-almost-existence-HZ}
\ker i_{W*} \, \not\subseteq\, \ker j_{!}.
\end{equation}
\end{remark}  

Given a class $\alpha\in H_*(W,\p W;i_W^*\cF)$ that belongs to $\ker i_{W*}$, recall from~\eqref{eq:spectral-number} its spectral number $c(\alpha)<\infty$.

\begin{proposition}\label{prop:finite capacity}
  Under the assumptions of Theorem~\ref{thm:finiteness-HZ} we have
  \[c_{HZ}^\circ(W)\leq c(\alpha).\]
  In particular, $c_{HZ}^\circ(W)$   
  is finite. 
\end{proposition}

\begin{proof} Note that, by definition, $c_{HZ}^\circ(W)=\sup c_{HZ}^\circ(W,\overline{U})$, where the supremum runs over all open subsets $U$ with $\overline{U}\subset \mathrm{int}(W)$. By Remark~\ref{rem:monotonicity-capacity}, we may write $c_{HZ}^\circ(W)=\sup c_{HZ}^\circ(W,B)$, where the supremum is taken over all small closed balls $B$ in $\mathrm{int}(W)$.

   By functoriality, the shriek map $ j_!$ of the inclusion $j:\star\hookrightarrow W$ is the composition 
    $$\xymatrix{ H_{*+n}(W,\p W; i_{W}^*\calf) \ar[r]_-{j_{B!}} &H_{*+n}(B, \p B; j_{B}^*i_W^*\calf)\ar[r]^-{\simeq}_-{j_!}&H_{*-n}(\mathrm{pt}; j^*\calf)\ar[d]^-{\simeq}\\
    \, &\, & H_{*-n}(\calf)}$$ 
    the second map being an isomorphism by definition of the shriek. Therefore the class $\alpha$ being nonzero in the DG-homology of the point, it is also nonzero in the DG-homology of $(B,\partial B)$. We may thus apply Proposition~\ref{prop:finite capacity-W} and conclude that $c_{HZ}^\circ(W,B)\leq c(\alpha)$, for any $B$. This yields the result.
\end{proof}

\subsection{The case of cotangent bundles}\label{sec:case-cotangents}

In this subsection we rephrase our previous criterion in topological terms in the case of cotangent bundles. Let $Q$ be a closed $n$-manifold, and let $D^*Q\subset T^*Q$ be the unit disc cotangent bundle for some Riemannian metric on $Q$. The choice of Riemannian metric will not be relevant in the sequel. 

Let $i_Q:Q\hookrightarrow\call_0 Q\subset \call Q$ be the inclusion of constant loops and $j:\star\hookrightarrow Q$ the inclusion of the basepoint. Given a DG local system $\cF$ on $\cL_0 Q$, these induce the direct map $i_{Q*}:H_*(Q;i_Q^*\cF)\to H_*(\cL_0 Q;\cF)$ and the shriek map $j_{!}:H_*(Q;i_Q^*\cF) \to H_{*-n}(\cF\otimes \underline{|Q|})$, where $\underline{|Q|}$ is the orientation local system of $Q$, seen as supported in degree $0$. We consider the diagram 
\begin{equation}\label{eq:diagram-finiteness-HZ-T*Q}
  \xymatrix{
     H_{*-n}(\cF\otimes\underline{|Q|}) &\, \\
      H_*(Q;i_Q^*\cF) \ar[r]^-{i_{Q*}}\ar[u]^-{j_{!}}     
    & H_*(\cL_0 Q;\cF).}
\end{equation}
  
Our criterion from Theorem~\ref{thm:finiteness-HZ} translates into the following.

\begin{proposition}\label{prop:finiteness-HZ-T*Q} Assume that there exists a DG local system $\cF$ on $\cL_0Q$ and a class $\alpha\in H_*(Q;i_Q^*\cF)$ such that $j_{!}(\alpha) \neq 0$ and $i_{Q*}(\alpha)=0$. Then 
$$
c_{HZ}^\circ(D^*Q)<\infty. 
$$
In particular, the  contractible almost existence property holds in $D^*Q$.
\end{proposition}

\begin{remark}\label{rmk:criterion-short-HZ-T*Q} 
The above criterion for the finiteness of the $\pi_1$-sensitive Hofer-Zehnder capacity of cotangent bundles may be written concisely as 
\begin{equation}\label{eq:criterion-almost-existence-HZ-T*Q}
\ker i_{Q*}\, \not\subseteq\, \ker j_{!}.
\end{equation}
\end{remark}  

\begin{proof}[Proof of Proposition~\ref{prop:finiteness-HZ-T*Q}] Let $\cal G$ be the DG local system on $T^*Q$ provided by Corollary \ref{cor:Viterbo-cotangent}. We put together the diagram of that corollary and diagram~\eqref{eq:diagram-finiteness-HZ} to get 
$$
\xymatrix{
H_{*-n}(\calf\otimes\underline{|Q|})&\, &\, \\
H_{*+n}(D^*Q,S^*Q; \pi^*i_Q^*\calf\otimes \pi^*\underline{|Q|})\ar[r]_-{\simeq}\ar[u]_-{j_!}\ar@/^2pc/[rr]^- {(i_{D^*Q})_*}
&SH_*^{=0}(T^*Q;\cG)\ar[r] \ar[d]_\simeq^{{\widetilde\Psi}^{=0}} & SH_*(T^*Q;\cG) \ar[d]_\simeq^{\widetilde\Psi} \\
\, & H_*(Q;i_Q^*\cF) \ar[r]^-{i_{Q*}} \ar[ul]^-{\simeq}_-{\pi_!}& H_*(\cL_0 Q;\cF).
}
$$

Recall that the criterion (\ref{eq:criterion-almost-existence-HZ})
from Theorem~\ref{thm:finiteness-HZ} is $\ker (i_{D^*Q})_*\, \not\subseteq \ker j_!$. Composing these maps in the above diagram  with  the Thom isomorphism $\pi_!$  
we immediately get (\ref{eq:criterion-almost-existence-HZ-T*Q}), and hence the conclusion.
\end{proof}

\begin{remark} One necessary condition for the criterion (\ref{eq:criterion-almost-existence-HZ-T*Q})
  to be applicable is that $j_{!}\neq 0$. In case the base manifold $Q$ is orientable, this exactly means that there is a class $\alpha\in H_*(Q;\cF)$ that ``lives over the fundamental class of $Q$" in the sense of Definition~\ref{def:living-above-fundamental}. Section~\S\ref{sec:classes_over_fund_class} contains a number of equivalent formulations of that property. 
\end{remark}

\subsection{Example: manifolds abundant with $2$-spheres} \label{sec:abundant}

In this subsection we exhibit a class of closed connected orientable manifolds $Q$ that satisfy the criterion from Proposition~\ref{prop:finiteness-HZ-T*Q}. 

{\bf Notation.} Given a space $X$ and subsets $A,B\subset X$, we denote $\cP_{A\to B}X$ the space of Moore paths in $X$ starting in $A$ and ending in $B$. 

This class of examples arises from considering the fibration 
\begin{equation} \label{eq:ice-cream-cones}
\cP_{Q\to \star}\cL_0 Q\hookrightarrow \cP_{Q\to \cL_0 Q}\cL_0 Q \stackrel\ev\longrightarrow \cL_0 Q
\end{equation}
defined by evaluation at the endpoint of a path  (see also Figure~\ref{fig:bananes-cones}). 
This fibration determines via the choice of a lifting function a $C_*(\Omega \cL_0 Q)$-right module structure on the chains on the fiber $\cF=C_*(\cP_{Q\to\star}\cL_0 Q)$ (see~\cite[\S7]{BDHO}). In the situation at hand the module structure
is induced by the concatenation of paths $\cP_{Q\to\star}\cL_0 Q\times \Omega \cL_0Q\to \cP_{Q\to\star}\cL_0 Q$.

The homology of the total space of this fibration is 
$$
H_*(\cL_0 Q;\cF)\simeq H_*(\cP_{Q\to \cL_0 Q}\cL_0 Q)\simeq H_*(Q).
$$
The first isomorphism is an instance of the Fibration Theorem from~\cite[Theorem~A and~\S13.3]{BDHO} (see also \S\ref{sec:def-MorseDG} of the present paper). The second isomorphism holds because $\cP_{Q\to \cL_0 Q}\cL_0 Q$ retracts onto $Q$. 

Restricting the fibration~\eqref{eq:ice-cream-cones} to $Q$ we obtain the fibration  
\begin{equation} \label{eq:bananas}
\cP_{Q\to \star}\cL_0 Q\hookrightarrow \cP_{Q\to Q}\cL_0 Q \longrightarrow Q
\end{equation}
such that
$$
H_*(Q;i_Q^*\cF) \simeq H_*(\cP_{Q\to Q}\cL_0Q) \simeq H_*(\Map(\SS^2,Q)). 
$$
The first isomorphism is again a consequence of the fibration theorem for Morse homology with DG coefficients mentioned above (\cite[Theorem A, Corollary 7.10 and \S13.3]{BDHO}). The second isomorphism follows from the canonical identification between $\cP_{Q\to Q}\cL_0Q$ and $\Map(\SS^2,Q)$, the space of free maps from $\SS^2$ to $Q$. The direct map $i_{Q*}: H_*(Q;i_Q^*\calf)\ri H_*(\call_0Q;\calf)$ is induced by the inclusion map $\calp_{Q\ri Q}\call_0Q\hookrightarrow \calp_{Q\ri\call_0Q}\call_0Q$, by \cite[Proposition~9.14,  Corollary~7.10 and Remark~13.20]{BDHO}, see also the discussion in \S\ref{sec:def-MorseDG}.

The fibration 
$$
\cP_{Q\to \star}\mathcal{L}_0Q\hookrightarrow \Map(\SS^2,Q)\to Q
$$ 
given by evaluation at a fixed basepoint on $\SS^2$ admits a canonical section given by constant spheres. As such, there is an injection $H_*(Q)\hookrightarrow H_*(\Map(\SS^2,Q))$.

\begin{remark} \label{rmk:topological_remark}
  We have seen in 
  Lemma \ref{lem:kpi1} 
that the homology of the fiber $\cP_{Q\to\star}\cL_0Q\equiv \Omega^2Q$ is that of a point if and only if $Q$ is a $K(\pi,1)$. An analogous result holds for the homology of the total space of this fibration: 

(i) If $Q$ is a $K(\pi,1)$, the evaluation map $\ev:\Map(\SS^2,Q)\to Q$ induces an isomorphism in homology.

(ii) If the evaluation map $\ev:\Map(\SS^2,Q)\to Q$ induces an isomorphism in homology and the action of $\pi_1(Q)$ on $\pi_3(Q)$ is trivial, then $Q$ is a $K(\pi,1)$ (and in particular $\pi_3(Q)=0$).\footnote{The assumption that the action of $\pi_1$ on $\pi_3$ is trivial may be redundant.}

Since we will not use them, we leave the proof of these statements to the interested reader. The way to read them is that, for most of the manifolds $Q$ that are not aspherical, the homology of $\Map(\SS^2,Q)$ is strictly larger than that of $Q$. The point of the definition below is to single out situations where classes in $H_*(\Map(\SS^2,Q))$ can be used in order to fulfil our criterion for the finiteness of the $\pi_1$-sensitive Hofer-Zehnder capacity. 
\end{remark}

 Let $j:\star\hookrightarrow Q$ be the inclusion of the basepoint.

\begin{definition} \label{defi:abundant_with_2spheres}
Let $Q$ be a closed connected oriented $n$-manifold. We say that \emph{$Q$ is abundant with $2$-spheres} if there exists a class $\alpha\in H_*(\Map(\SS^2,Q))$ in the  homology of the total space of the fibration $\Map(\SS^2,Q)\to Q$ such that:
\begin{itemize}
\item[(i)] $\alpha$ lives over the fundamental class in the sense of Definition~\ref{def:living-above-fundamental} and Example \ref{def:living-above-fundamental-fibration}, and 
\item[(ii)]  Either $\deg(\alpha)>n$, or $\deg(\alpha)=n$ and the image of $\alpha$ under the composition $H_n(\mathrm{Map}(\SS^2,Q))\stackrel{j_!}{\longrightarrow} H_0(\Omega^2 Q)\to H_0(\Omega^2 Q, \mathrm{pt})$ is nonzero.
\end{itemize} 
\end{definition}

The next result is Theorem~\ref{thm:abundant-intro} from the introduction. 

\begin{theorem} \label{thm:abundant}
Let $Q$ be a closed oriented manifold that is abundant with $2$-spheres. Then 
$$
c_{HZ}^\circ(D^*Q)<\infty.
$$
In particular the  contractible almost existence property holds in $D^*Q$.
\end{theorem}

\begin{proof}
Let $n=\dim\, Q$ and let $\alpha\in H_*(\Map(\SS^2,Q))$ be a class as in Definition~\ref{defi:abundant_with_2spheres}. We prove that there exists a class $\alpha'\in H_*(\Map(\SS^2,Q))=H_*(Q;i_Q^*\cF)$ that satisfies the assumptions of Proposition~\ref{prop:finiteness-HZ-T*Q}, where $\cF=C_*(\cP_{Q\to\star}\cL_0 Q)=C_*(\Omega^2Q)$ is the local system on $\cL_0Q$ that corresponds to the fibration~\eqref{eq:ice-cream-cones}.
That $\alpha$ lives over the fundamental class means precisely that $j_{!}(\alpha)\in H_*(\cF)$ is nonzero. In particular, by Proposition~\ref{prop:living-fundamental-fibrations}, the projection of $\alpha$ on the right-most column $E^\infty_{n,*}$ of the homology spectral sequence for the fibration $\Map(\SS^2,Q)\equiv \cP_{Q\to Q}\cL_0 Q \stackrel\ev\longrightarrow Q$ is nonzero, hence $\alpha$ has degree $\ge n$. 

If $\deg(\alpha)>n$, then $i_{Q*}(\alpha)\in H_*(\cL_0Q;\cF)=H_*(Q)$ vanishes for degree reasons and we can take $\alpha'=\alpha$. 

If $\deg(\alpha)=n$, then we take $\alpha'=\alpha-\sigma_*\ev_*(\alpha)\in H_n(\Map(\SS^2,Q))$,  where $\sigma: Q\hookrightarrow \mathrm{Map}(\SS^2,Q)$ is the section defined by the constant maps.
We need to check that $j_!(\alpha')\neq 0\in H_0(\Omega^2Q)$ and $i_{Q*}(\alpha')=0\in H_n(\cL_0Q;\cF)=H_n(Q)$. That $j_!(\alpha')\neq 0$ follows from the fact that the images of $\alpha$ and $\alpha'$ under the composition $\tilde{j}:H_n(\mathrm{Map}(\SS^2,Q))\stackrel{j_!}{\longrightarrow} H_0(\Omega^2 Q)\to H_0(\Omega^2 Q, \mathrm{pt})$ are the same, hence nonzero by the assumption on $\alpha$. Indeed, remark that  $\sigma_*: H_*(Q) \ri H_*(\mathrm{Map}(\SS^2, Q))$ fits into the commutative diagram 
$$\xymatrix
@C=15pt
{ 
H_n(Q)\ar[r]^-{\sigma_*}\ar[d]_-{j_!}&H_n(\mathrm{Map}(\SS^2,Q))\ar[d]^-{j_!}\ar[dr]^-{\tilde{j}}&\, \\
 H_{0}(\mathrm{pt})\ar[r]^{\sigma_*}&H_{0}(\Omega^2Q)\ar[r]&H_0(\Omega^2 Q, \mathrm{pt})}
 $$

The leftmost square is commutative by Proposition~\ref{prop:shriek-commute}, the map $\sigma$ being seen as a fibration map between the trivial fibration $\id:Q\ri Q$ and the fibration $\mathrm{ev}:\mathrm{Map}(\SS^2,Q)\ri Q$. From the fact that the bottom horizontal line of the diagram is exact we infer that $\tilde{j}\sigma_* =0$, so $\tilde{j}(\alpha')=\tilde{j}(\alpha)\neq 0$, which implies $j_!(\alpha')\neq 0$.  
 
To prove that $i_{Q*}(\alpha')=0$, it is enough to show that $\ev_{0*}i_{Q*}(\alpha')=0$, where $\ev_0:\cP_{Q\to\cL_0Q}\cL_0Q\to Q$ is the evaluation at the initial point of a path, which is a homotopy equivalence. We identify $\Map(\SS^2,Q)$ with $\cP_{Q\to\cL_0Q}\cL_0Q|_Q$ by marking 2 points $z_0$ and $z_1$ on the sphere $\SS^2$, corresponding to the initial, respectively final point of a path in $\cP_{Q\to Q}\cL_0Q$. Let $\ev'_i:\Map(\SS^2,Q)\to Q$ be the evaluation at $z_i$, $i=0,1$. Let $\tilde i_Q:\cP_{Q\to Q}\cL_0 Q\hookrightarrow \cP_{Q\to\cL_0Q}\cL_0 Q$ be the inclusion of the total spaces of the fibrations.
We then have $\ev'_1=\ev$, $\ev_0 \tilde i_Q=\ev'_0$, $\ev_0 {\tilde i_Q} \sigma=\mathrm{Id}_Q$, and we obtain 
\begin{align*}
\ev_{0*}\tilde i_{Q*}(\alpha') & = \ev_{0*}\tilde i_{Q*}(\alpha) - \ev_{0*}\tilde i_{Q*}\sigma_*\ev_*(\alpha) \\
& = \ev_{0*}\tilde i_{Q*}(\alpha) - \ev_*(\alpha)\\
& = \ev'_{0*}(\alpha) - \ev'_{1*}(\alpha) \\
& = 0.
\end{align*}
The last equality holds because the maps $\ev'_0$ and $\ev'_1$ are homotopic. 
\end{proof}

In the next section we exhibit several classes of manifolds that are abundant with $2$-spheres. That discussion is summarized, but not exhausted, by the statement of Theorem~\ref{thm:examples-intro}.

Our examples of manifolds that are abundant with $2$-spheres go well-beyond those of~\cite{Albers-Frauenfelder-Oancea}, but are strictly contained in the class of rationally inessential manifolds of Frauenfelder-Pajitnov~\cite{Frauenfelder-Pajitnov}. 
See also~\S\ref{sec:intro-finiteness-HZ}. 

\subsection{Explicit constructions of manifolds that are abundant with $2$-spheres} \label{sec:abundant_examples}
  
The goal of this section is to present explicit constructions of manifolds that are abundant with $2$-spheres. We will prove in particular Theorem~\ref{thm:examples-intro}. 

We prove that the following manifolds $Q$ are abundant with $2$-spheres. 
\begin{itemize}
\item[(i)] The product $Q=M\times N$, where $M$ is abundant with
  $2$-spheres and $N$ is an arbitrary closed oriented manifold 
  (Proposition~\ref{prop:abundant_products}).
  \item[(ii)] Manifolds that admit fibered families of $2$-spheres (Definition~\ref{defi:fibered_2spheres}). In particular, the spheres $\SS^n$ of even dimension $n=2k$, $k\ge 1$ (Example~\ref{example:spheres}) or odd dimension $n=4k-1$, $k\ge 1$ (Example~\ref{example:spheres-4k-1}), and also total spaces of fiber bundles with such spherical fibers over highly connected bases (Example~\ref{example:fibrations_spherical}). 
\item[(iii)] Manifolds that are not a $K(\pi,1)$ and that are covered by diffeomorphisms in the sense of Definitions~\ref{defi:covered_by_diffeomorphisms} and~\ref{defi:ell_covered_by_diffeomorphisms}, in particular all compact connected Lie groups other than tori (Proposition~\ref{prop:Lie_groups}), and all odd-dimen\-sio\-nal spheres $\SS^{n=2k+1}$, $k\ge 1$ (Proposition~\ref{prop:odd-spheres}).
\item[(iv)]  
The compact rank 1 symmetric spaces $\C P^d$, $\H P^d$, $d\ge 1$, or $\mathrm{Ca} P^2$ (Propositions~\ref{prop:CPd} and~\ref{prop:HPd-CaP2}). 
\end{itemize}
We also make some remarks about homogeneous spaces in~\S\ref{sec:homogeneous-spaces}, and about the use of minimal models in~\S\ref{sec:minimal-models}. 

\subsubsection{Stability under products} 

\begin{proposition} \label{prop:abundant_products}
Let $M$ be a closed, connected and oriented manifold that is abundant with $2$-spheres. Then the product $M\times N$ is abundant with $2$-spheres for any closed oriented manifold $N$. 
\end{proposition}

\begin{proof}
The proof relies on the canonical identification 
$$\Map(\SS^2,M\times N)\simeq \Map(\SS^2,M)\times \Map(\SS^2,N).$$ Given a class $\alpha\in H_*(\Map(\SS^2,M))$ that lives over the fundamental class of $M$ 
 and satisfies Definition~\ref{defi:abundant_with_2spheres}.(ii), 
the class $\alpha\times [N] \in H_*(\Map(\SS^2,M)\times \Map(\SS^2,N))$ lives over the fundamental class $[M\times N]=[M]\times [N]$ and 
also satisfies Definition~\ref{defi:abundant_with_2spheres}.(ii).
\end{proof}

\subsubsection{Manifolds that admit fibered families of $2$-spheres}

The first general geometric framework that guarantees abundance with $2$-spheres is provided by fibrations.  
Let $Q$ be a closed oriented manifold, and let $p$ be the basepoint of $Q$.  

\begin{definition} \label{defi:fibered_2spheres} We say that the manifold $Q$ \emph{admits a fibered family of $2$-spheres} if there exists a smooth fibration $\Sigma_{p}\hookrightarrow\Sigma\to Q$ with fiber $\Sigma_p$ a closed oriented manifold of positive  dimension, endowed with a fibration map $\sigma$ to $\Map(\SS^{2},Q)$,
$$
\begin{tikzcd}
  \Sigma_{p} \arrow[d,hookrightarrow] \arrow[r, "\sigma_p"]& \Omega^{2}Q \arrow[d,hookrightarrow]\\
  \Sigma      \arrow[r,"\sigma"]\arrow[d,"\pi"]& \Map(\SS^{2},Q)\arrow[d,"\mathrm{ev}"]\\
  Q\arrow[r,equal]&Q
\end{tikzcd}
$$
such that
$$
\sigma_{p\,*}[\Sigma_{p}]\neq 0
\quad\text{ in }\quad H_{*}(\Omega^{2}Q).
$$  
\end{definition}

\begin{proposition} \label{prop:class-over-Q}
  In the above setting, the class 
  $$
  \alpha=\sigma_*[\Sigma]\in H_{*}(\Map(\SS^{2}, Q))
  $$ 
  lives above the fundamental class of $Q$ and has degree $>\dim\, Q$.
  In particular, the manifold $Q$ is abundant with $2$-spheres.
\end{proposition}

\begin{proof}  The fact that $\alpha$ lives above the fundamental class of $Q$ is a direct consequence of Corollary \ref{cor:shriek-commute}.

 To conclude, we note that the degree of $\sigma_*[\Sigma]$ is equal to $\dim Q + \dim \Sigma_p$, and this is $>\dim Q$ since we assumed that $\Sigma_p$ has positive dimension.  
\end{proof}

 Let us now describe the first examples of manifolds that satisfy Definition \ref{defi:fibered_2spheres}.

\begin{example}[\bf Even-dimensional spheres of dimension $n\ge 2$ admit fibered families of $2$-spheres] \label{example:spheres}
  Let $Q=\SS^{n}$ be the unit sphere in $\R^{n+1}$, $n\ge 2$ (not necessarily even),  and let 
  $$
  \Sigma =\mathrm{Stiefel}_2(\SS^n)
  $$ 
  be the Stiefel manifold of orthonormal $2$-frames on $\SS^n$. A point in $\Sigma$ consists of a point $p\in\SS^{n}$ and of an ordered pair
  $(u,v)$ of orthogonal unit tangent vectors at $p$.
  The manifold $\Sigma$ fibers naturally over $\SS^n$, and the fiber $\Sigma_{p}$ over the basepoint $p$ of $\SS^n$ is the space of orthonormal $2$-frames in $T_p\SS^n$. As such, the dimension of $\Sigma_p$ is $2n-3$. Let us  define  a  map of fibrations $\sigma:\Sigma\to \Map(\SS^2,\SS^n)$  that satisfies the conditions of Definition~\ref{defi:fibered_2spheres}.

  A point $(p,u,v)\in \Sigma$ determines a $3$-dimensional vector subspace $\mathrm{Vect}(p,u,v)\subset \R^{n+1}$, whose intersection with $\SS^{n}$ is
  a round $2$-sphere. 
  This sphere 
determines a canonical map $\sigma:\Sigma\to \Map(\SS^2,\SS^n)$, $(p,u,v)\mapsto \sigma_{p,u,v}$. 

Let us describe this map explicitly.  Fix a point $p\in\SS^{n}$. 
For any non zero $u\in T_p\SS^n$, we let 
$$\gamma_{u}:
[0,\pi]\to\SS^{n},\quad \varphi  \mapsto  p\cos\varphi +  u\sin\varphi
$$
be the geodesic  (with respect to the round metric) from $p$ to $-p$ starting in the direction of $u$. To any $(u,v)\in \Sigma_p$ 
we associate the map $\Gamma_{u,v}$ defined by 
$$ \Gamma_{u,v}
:  [0,2\pi]^{2}\to\SS^{n},\quad
(\theta,\varphi)\mapsto
(\gamma_{u\cos\theta + v\sin\theta}\, \#\, \bar{\gamma}_{u})(\varphi).
$$
where $\bar{\gamma}$ is $\gamma$ with reversed time (see Figure \ref{fig:loops_of_loops}). 
\begin{figure}
    \centering
    \includegraphics{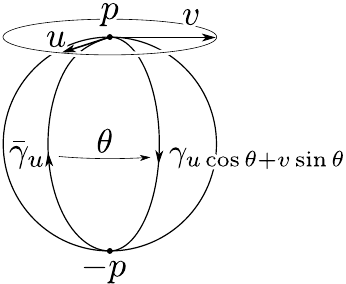}
    \caption{The loop of loops $\Gamma_{u,v}$.}
    \label{fig:loops_of_loops}
\end{figure} 
This is the loop of
loops obtained by moving from $p$ to $-p$ along geodesics associated to a
rotating direction in the plane $(u,v)$, and coming back from $-p$ to $p$
along the fixed geodesic associated to $u$, so $\Gamma_{u,v}\in \Omega_{\gamma_u\, \#\, \bar{\gamma}_u}\Omega_p\SS^n$. Consider now the path $\chi_u: [0,\pi]\to \Omega_p(\SS^n)$ between the constant loop at $p$ and $\gamma_u\, \#\, \bar{\gamma}_u$, defined by
$\chi_u(\theta)=\gamma_u|_{[0,\theta]}\, \#\,
\overline{\gamma_u|_{[0,\theta]}}
$. 
By conjugating with $\chi_u$ we get 
$$\sigma_{p,u,v}= \chi_u\, \#\, \Gamma_{u,v}\, \#\, \bar{\chi}_u\in\Omega_p\Omega_p(\SS^n)$$ 
Now we identify $\Omega_p\Omega_p(\SS^n)$ with $\mathrm{Map}[(\SS^2,\star), (\SS^n,p)]$ and vary $p$ to get a map $\sigma :\Sigma \to\mathrm{Map}(\SS^2, \SS^n)$.

  The map $\sigma$ is a map of fibrations, and we denote by $\sigma_p$ its restriction to the fiber $\Sigma_p$. We prove in Proposition~\ref{prop:nontrivial-class-Sn} below that $\sigma_{p*}[\Sigma_p]$ defines a nontrivial class in $H_{2n-3}(\Omega^2 \SS^n)$ for $n$ even. Therefore $\SS^{n=2k}$ admits a fibered family of $2$-spheres and, as a consequence of Proposition~\ref{prop:class-over-Q}, the class 
  $$
  \alpha=\sigma_*[\Sigma]\in H_{3n-3}(\Map(\SS^2,\SS^n))
  $$
 is nonzero and lives over the fundamental class of $Q= \SS^n$.

\begin{proposition} \label{prop:nontrivial-class-Sn}
With the same notations as above, if $n\geq 2$ is even, the class $\sigma_{p*}[\Sigma_p]$ is nonzero in $H_{2n-3}(\Omega^2 \SS^n)$. 
\end{proposition} 

As a preparation for the proof of Proposition~\ref{prop:nontrivial-class-Sn}, we recall some standard facts about the homology of the based loop spaces of spheres. 

\begin{proposition}\label{2prop:omega-sn}
Let $n\ge 2$ and $p\in\SS^n$ a basepoint. The homology $H_*(\Omega_p\SS^n)$ endowed with the Pontryagin product is a polynomial ring (not graded commutative!) in one variable of degree $n-1$. 
\end{proposition}

\begin{proof}
That $H_*(\Omega_p\SS^n)$ is isomorphic to $\Z$ in degrees $k(n-1)$, $k\ge 0$, and $0$ otherwise, is a straightforward consequence of the existence of the Leray-Serre spectral sequence for the path-loop fibration $\Omega_p\SS^n\to\cP\SS^n\to\SS^n$. Here the total space $\cP\SS^n$ is the space of paths in $\SS^n$ that start at $p$, and is contractible. 

To describe the multiplicative structure is a bit more subtle. One possible reference is~\cite[Example~3.C.7 and~\S4.J]{Hatcher}, where a model for the based loop space is provided by the James construction. 

For completeness, let us show how the multiplicative structure and the generator can be described from Morse homology with DG coefficients. We again consider the path-loop fibration $\Omega_p\SS^n\to \cP\SS^n\to\SS^n$. For this fibration, the construction of \cite[Chapter 5]{BDHO} yields a DG Morse complex $C_*(\Omega_p \SS^{n})\otimes \langle\mathrm{Crit}(f)\rangle$ associated with some Morse data on $\SS^n$, whose homology equals the homology of a point. If we take the height function with maximum at $ p$ and minimum at $q=-p$ as Morse function and consider its gradient with respect to the round metric, this complex is written $C_{*-n}(\Omega_{p}\SS^n)\langle p\rangle\oplus  C_{*}(\Omega_{p}\SS^n)\langle q\rangle$ 
with differential \begin{equation}\label{diff-DG}\partial (a \otimes p+b\otimes q)= \partial a \otimes p + ((-1)^{|a|} a \cdot m_{pq}+\partial b)\otimes q,\end{equation}
with $m_{pq}\in C_{n-1}(\Omega_p\SS^n)$ the chain obtained from the moduli space of connecting trajectories from $p$ to $q$.

Denoting by $1\in C_0(\Omega_p\SS^n)$ the constant loop at $p$, the relation $\partial^2(1\otimes p) =0$ shows that $m_{pq}$ is a cycle. Denote $x= [m_{pq}]\in H_{n-1}(\Omega_p\SS^n)$.

The multiplicative structure on $H_*(\Omega_p\SS^n)$ is then determined by the fact that the differential in the DG Morse complex involves the Pontryagin product.

The DG Morse homology spectral sequence associated with the filtration of the above complex by the Morse indices \cite[\S4.2 and Remark~5.13]{BDHO} satisfies $E^2_{r,s}= H^{\mathrm{Morse}}_{r}(\SS^n)\otimes  
H_s(\Omega_p\SS^n)$.
Also, for degree reasons, we have $E^2_{r,s}=E^n_{r,s}$ and, since this spectral sequence converges to the homology of a point,  $d^n:E^n_{n, k(n-1)}\to E^n_{ 0,(k+1)(n-1)}$ must be an isomorphism for all $k\ge 0$. Moreover, the form of the differential \eqref{diff-DG} implies that 
$d^n([p]\otimes a) = [q]\otimes a\cdot x$. Starting with $d^n([p]\otimes 1)= [q]\otimes x$ and arguing by induction, we infer that $x^k$ is a generator of $H_{k(n-1)}(\Omega_p \SS^n)$ for all $k\ge 1$.
\end{proof}

\begin{proof}[Proof of Proposition~\ref{prop:nontrivial-class-Sn}] 
Consider the composed map  $$\beta : \Sigma_p \times \SS^1\xrightarrow{\sigma_p\times\mathrm{Id}} \Omega_p\Omega_p(\SS^n)\times \SS^1\xrightarrow{\mathrm{ev}}\Omega_p(\SS^n)$$ where the last map represents the evaluation at $\theta\in \SS^1$. 
It suffices to prove that $\beta_*([\Sigma_p]\times [\SS^1])\in H_{2n-2}(\Omega_p\SS^n)$ does not vanish. Taking into account that $H_{2n-2}(\Omega_p\SS^n)$ is isomorphic to $\Z$, we will actually show that $\beta_*([\Sigma_p]\times [\SS^1])$ is equal to twice a generator when $n$ is even.

The map $\beta=\ev\circ (\sigma_p\times\mathrm{Id})$ can be read on the upper side of the following diagram 
$$\xymatrix{\Sigma_p\times \SS^1\ar[r]^-{\sigma_p\times\mathrm{Id}}\ar[d]_-{\tau}&\Omega_p\Omega_p(\SS^n)\times \SS^1\ar[d]^-{\mathrm{ev}}\\
\SS^{n-1}\times\SS^{n-1}\ar[r]^-{\alpha}&\Omega_p(\SS^n)}.$$
The map $\tau$ is defined by the formula $[(u,v), \theta]\mapsto (u,u\cos\theta +v\sin\theta)$, with $\SS^{n-1}$ being understood as the orthogonal of $p$ inside $\SS^n$, and the map $\alpha$ is defined by $(u,w)\mapsto \gamma_w\#\bar{\gamma}_u$.
The diagram above is commutative {\it only up to homotopy}. To prove this fact,  recall that the loop of loops $\sigma_{p,u,v}$ was defined as the concatenation of the path of loops $\chi_u$, the loop of loops  $\Gamma_{u,v}$, and the reverse path $\bar\chi_u$. Therefore, the map $\beta=\ev\circ (\sigma_p\times\mathrm{Id})$ is given by  the formula 
$$
\beta[(u,v),\theta)]= \sigma_{p,u,v}(\theta,\cdot)= 
\begin{cases}
\chi_u(\theta)& \mathrm{if }\ \theta\in [0,\pi],\\
\Gamma_{u,v}(\theta-\pi,\cdot) & \mathrm{if }\ \theta\in [\pi,3\pi],\\
\bar{\chi}_u(\theta-3\pi)& \mathrm{if }\ \theta\in [3\pi,4\pi].
\end{cases}
$$
It is immediate that $\beta$ is homotopic to $\alpha\circ \tau$. A homotopy is given by 
$$
\beta_s[(u,v),\theta)]= \sigma_{p,u,v}(\theta,\cdot)= 
\begin{cases}
\chi_u(\theta+ s)& \mathrm{ if }\ \theta\in [0,\pi-s],\\
\Gamma_{u,v}(\theta+s-\pi,\cdot) & \mathrm{ if }\ \theta\in [\pi-s,3\pi-s],\\
\overline{\chi_u|_{[s,\pi]}}(\theta+2s-3\pi)& \mathrm{ if }\ \theta\in [3\pi-s,4\pi-2s],
\end{cases}
$$
for $s\in [0,\pi]$. 

It suffices therefore to show that $(\alpha\circ\tau)_*([\Sigma_p]\times [\SS^1])\neq 0$.

There is a natural identification $\Sigma_p\equiv T^1T^1_p\SS^n\equiv T^1\SS^{n-1}$. 
We notice first that the map $\tau : \Sigma_p \times\SS^{1}=T^{1}\SS^{n-1}\times\SS^{1}\to \SS^{n-1}\times \SS^{n-1}$ has degree $\pm 2$ when $n$ is even. This can be seen as follows. The map $\tau$ is a map of fibrations over $\SS^{n-1}$ whose restriction to the fiber at $u$ is $\tau_u:\SS^{n-2}\times \SS^1\to \SS^{n-1}$ given by $$ (v,\theta)\mapsto u\cos\theta+v\sin\theta.$$
Thinking of $\SS^{n-2}$ as the equator of $\SS^{n-1}$ that is orthogonal to $u$, we see that $\tau_u(v,\cdot)$ parametrises the meridian passing through the north pole $u$ and through $v$. The preimage of a (regular) value in $\SS^{n-1}\setminus\{\pm u\}$ consists of two points $(v, \theta)$ and $(-v,2\pi -\theta)$. Since, for $n$ even, both the antipodal map of $\SS^{n-2}$ and the reverse map of $\SS^1$ given by $\theta\mapsto 2\pi -\theta $  change orientations, it follows that the self-map of $\SS^{n-2}\times\SS^1$ given by $(v,\theta)\mapsto (-v,2\pi-\theta)$ preserves the orientation, and therefore the degree of $\tau_u$ equals $\pm 2$. This implies that the degree of $\tau$ is also equal to $\pm2$.\footnote{In the case of an odd-dimensional sphere, the same argument shows that the degree of the map $\tau$ is zero. This is the reason why this proof works only for even-dimensional spheres.} 

To finish the proof of the proposition it suffices to show that $$\alpha_*([\SS^{n-1}]\times [\SS^{n-1}])\neq 0$$ in $H_{2n-2}(\Omega_p\SS^n)$ (by Proposition~\ref{2prop:omega-sn}, this group is isomorphic to $\Z$ and therefore has no 2-torsion).
Fix a unit vector $u_0\in T^1_p\SS^n$ and consider the map $\lambda: T^1_p\SS^n\equiv \SS^{n-1}\to \Omega_p\SS^n$ defined by $\lambda(w) = \gamma_w\#\bar{\gamma}_{u_{0}}$. Its composition with the time-reverse map of $\Omega_p\SS^n$ yields the map $\bar{\lambda}:\SS^{n-1}\ri \Omega_p\SS^n$ defined by $w\mapsto \gamma_{u_0}\#\bar{\gamma}_w$. By shrinking $\bar{\gamma}_{u_0}\#\gamma_{u_0}$ to the constant loop at $-p$, we notice that $\alpha$ is homotopic to the composition 
$$\SS^{n-1}\times \SS^{n-1}\xrightarrow{\lambda\times \bar{\lambda}}\Omega_p\SS^n\times \Omega_p\SS^n\xrightarrow{\#}\Omega_p\SS^n.$$
By definition of the DG Morse complex  \cite[Chapter 5]{BDHO} associated to the height function on $\SS^n$, we have that  $\lambda_*([\SS^{n-1}])= x$, the generator of $H_{n-1}(\Omega_p\SS^n)$ described in the proof of Proposition \ref{2prop:omega-sn} above.  Denote  $\bar{\lambda}_*([\SS^{n-1}])$ by $\bar{x}$. In order to complete the proof of the proposition, it suffices now to prove that the Pontryagin product  $x\cdot \bar{x}$ does not vanish. 

We first notice that $\bar{x}\neq 0$ in $H_{n-1}(\Omega_p\SS^{n})$ since it is the image of $x$ by the homeomorphism $\gamma \mapsto \bar{\gamma}$ of $\Omega_p\SS^n$.
Then Proposition~\ref{2prop:omega-sn} implies that $x\cdot \bar{x} \neq 0$, since $\bar x=-x$ hence $x\cdot \bar x=-x^2$.

The proof of Proposition \ref{prop:nontrivial-class-Sn} is now complete.
\end{proof}

\end{example}

\begin{remark}
\label{rmk:generator-S2}
In the case $n=2$ the previous construction can be slightly refined by noticing that $\Sigma=\mathrm{Stiefel}_2(\SS^2)$ has two connected components, corresponding to the two possible orientations of $\SS^2$. By restricting the map $\sigma$ to one of these components, denoted by $\Sigma'$, we obtain a class $\sigma_*[\Sigma']$ such that $\beta_*(\sigma_p)_*[\Sigma'_p]\in H_{2}(\Omega_p\SS^n)\simeq \Z$ is equal to a generator.
\end{remark}

\begin{example}[Spheres of dimension $4k-1$, $k\ge 1$  admit fibered families of $2$-spheres]
\label{example:spheres-4k-1}
A theorem of Adams~\cite{Adams} ensures that spheres of dimension $4k-1$ admit at least 2 (and actually 3) pointwise linearly independent vector fields. (In contrast, spheres of dimension $4k+1$ do not admit a pair of such vector fields.) We will crucially use this feature in the construction below. 

Let $n=4k-1$ and choose on $\SS^n$ two such vector fields $u$, $v$, which we assume w.l.o.g. to be orthogonal and of unit length. Let $\Sigma$ be the $\SS^{n-2}$-bundle over $\SS^n$ with fiber at $p\in\SS^n$ given by $\Sigma_p=\SS^n\cap \langle u(p),v(p)\rangle^\perp$, where the orthogonal of the plane spanned by $u(p)$ and $v(p)$ is considered in $\R^{n+1}$. 
We now define a map $\sigma:\Sigma\to \Map(\SS^2,\SS^n)$.

For the definition we need the following preliminary notation. Let $V$ be a Euclidean vector space of dimension $m+1$ and $\SS_V\subset V$ the unit sphere. (For us, $V$ will be either $\R^{n+1}$, or a subspace of $\R^{n+1}$ of dimension $n$). Given a point $p\in\SS_V$ and a unit vector $u\in T^1_p\SS_V$, let $E_{V,p,u}\subset \SS_V$ be the equator that passes through $p$ and is orthogonal to $u$. Then $E_{V,p,u}=\SS_{u^\perp}$, the unit sphere in the Euclidean $m$-space $u^\perp\subset V$. Let 
$$
\varphi_{V,u}: E_{V,p,u}\to\Omega_p\SS_V
$$
be the map that associates to a point $q\in E_{V,p,u}$, if $q=p$, the constant loop at $p$, and if $q\neq p$, the parametrized planar circle on $\SS_V$ tangent to the direction of $u$ and passing through $q$, with its parametrization proportional to arc length on the interval $[0,2\pi]$ (see Figure \ref{fig:cycle_sphere_4k_1}). (This circle is the intersection with $\SS_V$ of the affine plane containing the direction $u$ and passing through $p$ and $q$.)   

\begin{figure}
    \centering
    \includegraphics{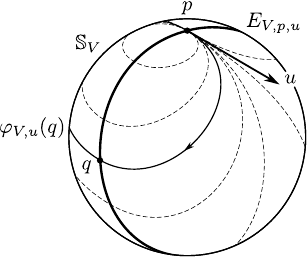}
    \caption{Sweeping a sphere with loops.}
    \label{fig:cycle_sphere_4k_1}
\end{figure}

Let $\sigma_p:\Sigma_p\to \Omega_p\Omega_p\SS^n$ be the composition 
\[
\Sigma_p=E_{u(p)^\perp,p,v(p)}\stackrel{\varphi_{u(p)^\perp,v(p)}}{\longrightarrow} \Omega_p E_{\R^{n+1},p,u(p)} \stackrel{\Omega\varphi_{\R^{n+1},u(p)}}{\longrightarrow} \Omega_p\Omega_p\SS^n.
\]
By varying $p$ we obtain a map of fibrations $\sigma:\Sigma\to\Map(\SS^2,\SS^n)$. 

We prove in Proposition~\ref{prop:nontrivial-class-S4k-1} below that $\sigma_{p*}[\Sigma_p]$ defines a nontrivial class in $H_{n-2}(\Omega^2 \SS^n)$. Therefore $\SS^{n=4k-1}$ admits a fibered familiy of $2$-spheres and, as a consequence of Proposition~\ref{prop:class-over-Q}, the class 
  $$
  \alpha=\sigma_*[\Sigma]\in H_{2n-2}(\Map(\SS^2,\SS^n))
  $$
 is nonzero and lives over the fundamental class of $Q= \SS^n$.

\begin{proposition} \label{prop:nontrivial-class-S4k-1}
The class $\sigma_{p*}[\Sigma_p]\in H_{n-2}(\Omega^2\SS^n)\simeq \Z$ is a generator. 
\end{proposition}

\begin{proof}
We fix a point $p$ and assume without loss of generality that $u(p)=e_{n+1}$ and $v(p)=e_n$, the vectors of the canonical basis of $\R^{n+1}$. The map $\varphi_{\R^{n+1},u(p)}$ represents a generator of $\pi_{n-1}(\Omega_p \SS^n)$: indeed, through the canonical adjunction $[\SS^{n-1},\Omega_p \SS^n]\simeq [\SS^n,\SS^n]$, and after identifying $E_{\R^{n+1},p,u(p)}$ with $\SS^{n-1}$, the map $\varphi_{\R^{n+1},u(p)}$ corresponds to $\mathrm{Id}_{\SS^n}$, which is a generator of $\pi_n(\SS^n)$.
We conclude by the Hurewicz theorem, which ensures $H_{n-2}(\Omega^2\SS^n)\simeq \pi_{n-2}(\Omega^2\SS^n)\simeq \pi_{n-1}(\Omega\SS^n)\simeq \pi_n(\SS^n)\simeq \Z$.
\end{proof}

\end{example}

\begin{example}[\bf Fibrations with spherical fibers over highly-connected bases]  \label{example:fibrations_spherical}

Such fibrations admit a fibered family of 2-spheres. More precisely, we have the following

\begin{proposition}\label{prop:fibrations-spherical-fiber}
Let $Q$ be the total space of a smooth sphere bundle 
$$
\SS^n\hookrightarrow Q\stackrel\pi\longrightarrow B
$$ 
with fiber $\SS^n$, $n\ge 2$ whose structure group can be reduced to $O(n+1)$. Assume one of the following conditions: 
\begin{itemize}
\item $n$ is even and $B$ is $2n$-connected. 
\item $n$ is odd $\equiv \, 3 \, (\mbox{mod } 4)$, $B$ is $(n+1)$-connected, and there exist 
two pointwise linearly independent vector fields on $Q$ tangent to the fibers. 
\end{itemize}
Then $Q$ admits a fibered family of $2$-spheres. 
\end{proposition}

\begin{remark}
Assume $n \equiv \, 3 \, (\mbox{mod } 4)$. Determining whether a given fibration admits two pointwise linearly independent vector fields that are tangent to the fibers is an interesting problem. Here is a class of examples where this condition holds automatically. Assume $\SS^n\hookrightarrow Q \to B$ is the unit sphere bundle of a quaternionic Hermitian vector bundle $E\to B$. In each fiber of $E$ we have three complex structures denoted $I$, $J$, $K$ and a well-defined radial vector field $\rho(z)=z$. Its images via the complex structures $I$, $J$ and $K$ are everywhere tangent to the fibers $\SS^{k}$ and pointwise linearly independent.

It is likely that there exist sphere bundles $\SS^{4k-1}\hookrightarrow Q\to B$ with $4k$-connected base that do not admit a a pair of pointwise linearly independent  vector fields that are tangent to the fibers. We do not know an explicit example though. The following questions arise: 
\begin{enumerate}
\item Under what conditions on the sphere bundle does such a pair of vector fields exist? The primary obstruction for the existence of \emph{one} nonvanishing vector field tangent to the fibers is the Euler class of the rank $n$ vector bundle $T\SS^n\to Q$ whose fiber at a point $p\in Q$ is the tangent space to the fiber through that point. There may be of course higher obstructions.
\item What is the homotopy type of the space of nonvanishing vector fields on an odd-dimensional sphere? Works of G.W. Whitehead~\cite{GWWhitehead1946}, J.H.C. Whitehead~\cite{JHCWhitehead1953}, and V.L. Hansen~\cite{Hansen1974}, are relevant in this direction.
\end{enumerate}
\end{remark}

\begin{remark}
The assumption on the structure group of the fibration in Proposition~\ref{prop:fibrations-spherical-fiber} is equivalent to requiring that it is isomorphic to the unit sphere bundle of a vector bundle of rank $n+1$ endowed with a fiberwise scalar product.
\end{remark} 

\begin{proof}[Proof of Proposition~\ref{prop:fibrations-spherical-fiber}]
We first prove that the conclusion holds if $n$ is even and $B$ is $2n$-connected. Let $\Sigma$ be the fibration whose fiber at a point $p\in Q$ is 
$$
\Sigma_p=\mathrm{Stiefel}_2(T_p\SS^n),
$$ 
the Stiefel manifold of orthonormal $2$-frames in the tangent space $T_p\SS^n$ to the fiber. The map $\sigma$ from Example \ref{example:spheres} can be put in families and gives rise to a diagram 
$$
\begin{tikzcd}
  \Sigma_{p} \arrow[d,hookrightarrow] \arrow[r,"\tilde \sigma_p"]& \Omega^{2}Q \arrow[d,hookrightarrow]\\
  \Sigma      \arrow[r,"\tilde \sigma"]\arrow[d,"\pi"]& \Map(\SS^{2},Q)\arrow[d,"\mathrm{ev}"]\\
  Q\arrow[r,equal]&Q
\end{tikzcd}
$$
where the top horizontal arrow now factors through $\Omega^2\SS^n$. 
 Under our standing assumption that the base $B$ is $2n$-connected, we find that $\pi_i(\Omega^2 B)=0$ for $i\le 2n-2$. By the Hurewicz theorem we have $H_i(\Omega^2 B)=0$ for $1\le i\le 2n-2$, and by examining the spectral sequence of the fibration $\Omega^2\SS^n\hookrightarrow \Omega^2 Q\to \Omega^2 B$ we find that the inclusion $\Omega^2 \SS^n\hookrightarrow \Omega^2 Q$ induces an isomorphism in homology in degrees $i\le 2n-3$. 
As a consequence we find that $(\tilde\sigma_p)_*[\Sigma_p]\neq 0\in H_{2n-3}(\Omega^2 Q)$, which is the conclusion that we were seeking. 

We assume now that $n$ is odd $\equiv\, 3 \, (\mbox{mod } 4)$, $B$ is $(n+1)$-connected and the fibration admits a pair of pointwise linearly independent vector fields $u$, $v$ tangent to the fibers. In this case we put in a family the construction described in Example~\ref{example:spheres-4k-1}.

We find a fibration $\Sigma\to Q$ whose fiber at a point $p\in Q$ is $\Sigma_p=E_{u(p)^\perp,p,v(p)}\subset T_p\SS^n$, the $(n-2)$-dimensional sphere consisting of unit vectors tangent at $p$ to the fiber and orthogonal to both $u(p)$ and $v(p)$. The resulting map ${\tilde \sigma}'_p:\Sigma_p\to \Omega^2Q$ factors through $\Omega^2\SS^n$ and, in order to show that $({\tilde \sigma}'_p)_*[\Sigma_p]\neq 0\in H_{n-2}(\Omega^2 Q)$, it is enough in view of Example~\ref{example:spheres-4k-1} to show that the map $H_{n-2}(\Omega^2 \SS^n)\to H_{n-2}(\Omega^2Q)$ is injective.

Since $B$ is $(n+1)$-connected we have that $\pi_i(\Omega^2 B)=0$ for $i\le n-1$ and $\pi_i(\Omega^2\SS^n)\to \pi_i(\Omega^2 Q)$ is an isomorphism for $i\le n-2$. By the Hurewicz theorem we have $H_i(\Omega^2 B)=0$ for $1\le i\le n-1$ and the spectral sequence of the fibration $\Omega^2\SS^n\hookrightarrow \Omega^2 Q\to \Omega^2 B$ shows that the map $H_i(\Omega^2 \SS^n)\to H_i(\Omega^2 Q)$ induced by the inclusion is an isomorphism for $i\le n-2$. 
\end{proof}

\end{example}

\subsubsection{Manifolds covered by diffeomorphisms}

\begin{definition} \label{defi:covered_by_diffeomorphisms}
Let $M$ be a closed connected oriented manifold of dimension $n$. Denote by $\Diff(M)$ the diffeomorphism group of $M$. Fix a basepoint $\star\in M$ and consider the evaluation map at the basepoint
$$
\ev:\Diff(M)\to M, \qquad \varphi\mapsto \varphi(\star).
$$
We say that $M$ is \emph{covered by diffeomorphisms} if there exists a homology class $A\in H_n(\Diff(M))$ such that 
$$
\ev_*A=[M].
$$
\end{definition}

\begin{remark}
This condition should be compared to the one appearing in work of Viterbo~\cite[Definition~6.1]{viterbo2022inverse}. 
\end{remark}

\begin{example}
Let $\Diff_{\star}(M)\subset \Diff(M)$ be the subgroup of diffeomorphisms that preserve the basepoint. Assume that the 
fibration\footnote{That this is a locally trivial fibration is a particular case of a theorem of Palais~\cite{Palais1960}. Here is a sketch proof: fix a point $p\in M$ and an open neighborhood around $p$ modelled on the ball $B(1)\subset \R^n$ centered at $0$ and of radius $1$. There is a smooth family $\varphi_b$, $b\in B(1/2)$ parametrized by the ball $B(1/2)$ centered at $0$ and of radius $1/2$, consisting of diffeomorphisms that are compactly supported in $B(1)$, which have the property that $\varphi_b(0)=b$. Such a family can be used to trivialize $\Diff(M)$ over $B(1/2)$ by sending $B(1/2)\times \Diff_\star(M)$ to $\Diff(M)|_{B(1/2)}$ via $(b,\varphi)\mapsto \varphi_b\circ\varphi$.} $\Diff_{\star}(M)\hookrightarrow \Diff(M)\to M$ possesses a section. Then $M$ is covered by diffeomorphisms. This happens in particular if $M$ is a compact connected Lie group, in which case it embeds into its group of diffeomorphisms by left translations, see also Example~\ref{example:Lie-groups} below.  
\end{example}

\begin{proposition}\label{prop:covered-by-diffeo-implies-abundant}
Let $M$ be a closed connected oriented manifold of dimension $n$ that is covered by diffeomorphisms. Assume that $M$ is not a $K(\pi,1)$. Then $M$ is abundant with $2$-spheres. 
\end{proposition}

\begin{proof}
Let $\Diff_0(M)$ be the connected component of the identity in $\Diff(M)$, and $\Diff_{0,\star}(M)\subset \Diff_0(M)$ the subgroup of diffeomorphisms that preserve the basepoint. We can assume without loss of generality that $A\in H_n(\Diff_0(M))$. Indeed, the class $A$ can be written as a finite sum $A=\sum_i A_i$, with $A_i\in H_n(\Diff_i(M))$, where the $\Diff_i(M)$ are disjoint components of $\Diff(M)$. Choose $\varphi_i\in \Diff_i(M)$ and let $\underline{\varphi}_i^{-1}:\Diff_i(M)\to \Diff_0(M)$, $\psi\mapsto \varphi_i^{-1}\psi$ be the left translation by $\varphi_i^{-1}$. Then $\ev\circ\underline{\varphi}_i^{-1}=\varphi_i^{-1}\circ \ev$, 
 so that $\ev_* (\underline{\varphi}_i^{-1})_*A_i=(\varphi_i^{-1})_* \ev_*A_i$. Let $\eps_i=\pm 1$ according to whether $\varphi_i$ preserves or reverses the orientation of $M$. Then $\eps_i(\varphi_i^{-1})_*=\Id$ on $H_n(M)$ and, upon replacing $A=\sum_i A_i$ with $A'=\sum_i \eps_i(\underline{\varphi}_i^{-1})_*A_i\in H_n(\Diff_0(M))$, we find that $\ev_*A'=\sum_i \eps_i(\varphi_i^{-1})_* \ev_*A_i=\sum_i \ev_*A_i =\ev_*A= [M]$.  We thus henceforth assume that $A\in H_n(\Diff_0(M))$.

The assumption that $M$ is not a $K(\pi,1)$ is equivalent to the fact that $\Omega^2M$ is not contractible. 
Let $k\ge 0$ be minimal such that $\pi_k\Omega^2M\neq 0$. Choose a class $\alpha\in \pi_k\Omega^2M$ that is nonzero and different from the class of the basepoint $[\mathrm{pt}]$ if $k=0$. By the Hurewicz theorem this class defines a nontrivial element $\alpha\in H_k(\Omega^2M)$, which is not a multiple of $[\mathrm{pt}]$ if $k=0$. 

Consider the map 
$$
\Phi:\Diff_0(M)\times  \Omega^2M\to \Map(\SS^2,M), \qquad (\varphi,u)\mapsto \varphi\circ u. 
$$
This is a map of fibrations (we denote $\Phi'$ the map induced between the fibers)
$$
\xymatrix{
\Diff_{0,\star}(M)\times \Omega^2M \ar@{^(->}[d] \ar[r]^-{\Phi'} & \Omega^2M \ar@{^(->}[d] \\
\Diff_0(M)\times \Omega^2M \ar[r]^-\Phi \ar[d]^{\ev\circ \mathrm{pr}_1} & \Map(\SS^2,M) \ar[d]^\ev \\
M \ar@{=}[r] & M
}
$$

Let $j:\star \hookrightarrow M$ be the inclusion of the basepoint. We have $j_![M]=[\mathrm{pt}]$ and the diagram
$$
\xymatrix{
H_n(\Diff_0(M)) \ar[r]^-{\ev_*} \ar[d]_{j_!} & H_n(M) \ar[d]^{j_!} \\
H_0(\Diff_{0,*}(M)) \ar[r]^-{\simeq} & H_0(\mathrm{pt}) 
}
$$
is commutative by Proposition~\ref{prop:shriek-commute}, therefore $j_!A=[\mathrm{pt}]\in H_0(\Diff_{0,\star}(M))$.   

Since $\Phi$ is a fibration map, the diagram 
$$
\xymatrix{
H_*(\Diff_0(M)\times \Omega^2 M) \ar[r]^-{\Phi_*} \ar[d]_{j_!} & H_*(\Map(\SS^2,M)) \ar[d]^{j_!} \\
H_{*-n}(\Diff_{0,*}(M)\times \Omega^2M) \ar[r]_-{\Phi'_*} & H_{*-n}(\Omega^2M) 
}
$$
is also commutative, again by Proposition~\ref{prop:shriek-commute}  and this implies  \begin{align*}
j_!\Phi_*(A\times \alpha) & = \Phi'_*j_!(A\times \alpha) \\
& = \Phi'_* ((j_!A)\times \alpha) \\
& = \Phi'_* ([\mathrm{pt}]\times \alpha) \\
& = \alpha\neq 0\in H_k(\Omega^2 M).
\end{align*}
The first equality holds by the commutativity of the diagram, the second equality holds because the fibration $\Diff_0(M)\times\Omega^2M\to M$ is trivial in the $\Omega^2M$-factor, the third equality holds because $j_!A=[\mathrm{pt}]$ as proved before, and the last equality holds because all elements of $\Diff_{0,\star}(M)$ are isotopic to the identity and therefore act by the identity on the homology of $\Omega^2M$.

We find that 
$$
\Phi_*(A\times \alpha) \in H_{n+k}(\Map(\SS^2,M))
$$
lives above the fundamental class of $M$.

If $k=0$ we still have to check that condition (ii) from Definition~\ref{defi:abundant_with_2spheres} is satisfied.  In this case $\pi_2(M)\neq 0$ and $H_0(\Omega^2(M))\simeq \Z[\pi_2(M)]$. We may choose $$\alpha \not\in \ker\left( H_0(\Omega^2M)\ri H_0(\Omega^2(M), \mathrm{pt})\right)$$ in order   to fulfil this condition and therefore prove 
that $Q$ is abundant with $2$-spheres. 
\end{proof}

\begin{example}[\bf Lie groups] \label{example:Lie-groups} Let $G$ be a compact connected Lie group. Then $G$ embeds into $\Diff(G)$ via left translations, and the homology class of this embedding represents $[G]$ under the evaluation map at the basepoint $e\in G$. Thus $G$ is covered by diffeomorphisms. If, in addition, $G$ is not a $K(\pi,1)$, then it is abundant with $2$-spheres by Proposition~\ref{prop:covered-by-diffeo-implies-abundant}.  

The proof of abundance can be made very explicit in this case using the homeomorphism
$$
G\times \Omega^2 G\stackrel\simeq\longrightarrow \Map(\SS^2,G),\qquad (g,\sigma)\mapsto (z\mapsto g\sigma(z)).
$$
The space $\Omega^2G$ is homotopy equivalent to a point if and only if $G$ is a $K(\pi,1)$. If $G$ is not a $K(\pi,1)$ then $H_*(\Omega^2G)\neq H_*(\mathrm{pt})$ because $\pi_1\Omega^2G$ is abelian. 
Given \emph{any} nonzero homology class $\alpha\in \widetilde{H}_*(\Omega^2G)= H_*(\Omega^2G, \mathrm{pt})$, 
the class $[G]\times \alpha \in H_*(\Map(\SS^2,G))$ satisfies the requirements of Definition~\ref{defi:abundant_with_2spheres}.  Note that $\alpha$ necessarily has positive degree because $\pi_0\Omega^2G=\pi_2G=0$ by a classical theorem of Cartan-Borel, see for example~\cite{Borel_SHC49-50}.

The classification theorem for compact connected Lie groups states that, up to a finite cover, each such group is the product between a torus and a compact connected and simply connected Lie group. See for example~~\cite[Appendice~1, \S3]{Bourbaki_Lie9}, or~\cite[\S V.8, Theorem~8.1.]{Brocker-TomDieck}, or~\cite[\S10.7.2, Theorem~4]{Procesi_Lie_groups}. In particular, we obtain that a compact connected Lie group is a $K(\pi,1)$ if and only if it  is a torus, which implies the following:

\begin{proposition} \label{prop:Lie_groups}
All compact connected Lie groups other than tori are abundant with $2$-spheres. 
\end{proposition} 

Frauenfelder-Pajitnov~\cite{Frauenfelder-Pajitnov} proved that $c_{HZ}^\circ(D^*Q)<\infty$ for closed manifolds $Q$ that are rationally inessential, i.e., such that their fundamental class vanishes in the rational homology of $K(\pi_{1}(Q),1)$. This includes all compact connected Lie groups $G$ that are not tori. 
To see this, denote by $n$ the dimension of $G$, by $t$ the dimension of a maximal torus of $G$, and by $r$ the rank of $\pi_1(G)$. We have $n> t$ because $G$ is not a torus, and $t\geq r$ by \cite[Theorem~7.1.]{Brocker-TomDieck}. Thus, 
$$H_n(K(\pi_1(G)^{\mathrm{free}},1);\Q)=H_n(\T^r;\Q)=0.$$
On the other hand, $H_i(K(\pi_1(G)^{\mathrm{torsion}},1);\Q)$ vanishes for all $i>0$, because the homology of a finite group is torsion in all positive degrees.
Using $K(\pi_1(G),1)=K(\pi_1(G)^{\mathrm{free}},1)\times K(\pi_1(G)^{\mathrm{torsion}},1)$, and the Künneth formula, we conclude that $H_n(K(\pi_1(G),1);\Q)=0$, hence that $G$ is rationally inessential.
\end{example}

A useful variant of Definition~\ref{defi:covered_by_diffeomorphisms} is the following. 

\begin{definition} \label{defi:ell_covered_by_diffeomorphisms}
Let $M$ be a closed connected oriented manifold of dimension $n$. Denote $\Diff(M)$ the diffeomorphism group of $M$. Fix a basepoint $\star\in M$ and consider the evaluation map at the basepoint
$$
\ev:\Diff(M)\to M, \qquad \varphi\mapsto \varphi(\star).
$$
Let $\ell>0$ be a positive integer. We say that $M$ is \emph{$\ell$-covered by diffeomorphisms} if there exists a homology class $A\in H_n(\Diff(M))$ such that 
$$
\ev_*A=\ell [M].
$$
\end{definition}

The following statement is proved in exactly the same way as Proposition~\ref{prop:covered-by-diffeo-implies-abundant}. 

\begin{proposition} \label{prop:ell-covered-by-diffeo-implies-abundant}
Let $M$ be a closed connected oriented manifold of dimension $n$ that is $\ell$-covered by diffeomorphisms. Assume that $M$ is not a $K(\pi,1)$ and either $\pi_2(M)\neq 0$ or the first nontrivial homotopy group of $M$ in degree $k\ge 3$ is not $\ell$-torsion. Then $M$ is abundant with $2$-spheres. \qed
\end{proposition}

\begin{proposition}\label{prop:odd-spheres}
Odd-dimensional spheres of dimension $n\ge 3$ are abundant with $2$-spheres. 
\end{proposition}

We actually prove that, for any odd $n\ge 3$, the sphere $\SS^n$ is $\ell$-covered by diffeomorphisms for a suitable $\ell>0$. We will give two proofs of this fact. The first proof is of a more algebraic nature and has the merit of emphasizing the interplay between homology and homotopy. The second proof can be interpreted as a geometric realization of the homology class exhibited in the first proof.

The key observation for the first proof is the following. 

\begin{lemma}
  For $n\ge 3$ odd, $\pi_{n-1}(SO(n))$ is torsion.
\end{lemma}

\begin{proof}
We rely on computations of homotopy groups from~\cite[Appendix~A, Tables~6.VII and~6.VIII]{Japanese_Encyclopedic_Dictionary}. From~\cite[Appendix~A, Table~6.VII]{Japanese_Encyclopedic_Dictionary} we read 
\[
\pi_2(SO(3))=0,\quad \pi_4(SO(5))=\Z_2, \quad \pi_6(SO(7))=0,\quad \pi_8(SO(9))=\Z_4,
\]
\[
 \qquad \pi_{10}(SO(11))=\Z_2,\qquad \pi_{12}(SO(13))=\Z_2, \qquad \pi_{14}(SO(15))=\Z_2.
\]
From~\cite[Appendix~A, Table~6.VIII]{Japanese_Encyclopedic_Dictionary} we read that, for $n\ge 16$, 
$$
\pi_{n-1}(SO(n))=\pi_{n-1}(O)\oplus \pi_n(V_{n+2,2}(\R)),
$$
with $\pi_{n-1}(O)$ the stable homotopy group of the orthogonal group, and $V_{n+2,2}(\R)$ the Stiefel manifold of orthonormal $2$-frames in $\R^{n+2}$. 

The Bott periodicity theorem~\cite[Appendix~A, Table~6.VII]{Japanese_Encyclopedic_Dictionary} ensures that
\[
\pi_k(O)=\left\{\begin{array}{ll} \Z,& k\equiv 3, 7 \mbox{ mod } 8,\\
\Z_2,& k\equiv 0, 1 \mbox{ mod } 8,\\
0, & k\equiv 2, 4, 5, 6 \mbox{ mod } 8.
\end{array}\right.
\]
In particular, if $n$ is odd we infer that $\pi_{n-1}(O)$ is either $\Z_2$ or $0$.

Finally, we read from~\cite[Appendix~A, Table~6.VIII]{Japanese_Encyclopedic_Dictionary} that 
\[
\pi_n(V_{n+2,2})=\Z_2, \qquad n \mbox{ odd}.
\]

Summarizing, we find that, for $n\ge 17$ odd, $\pi_{n-1}(SO(n))$ is either $\Z_2 \oplus \Z_2$,  or $\Z_2$, and in all cases it is torsion.  
\end{proof}

\begin{proof}[First proof of Proposition~\ref{prop:odd-spheres}]
We write the relevant portion of the exact sequence of the fibration $SO(n)\hookrightarrow SO(n+1)\stackrel{\ev}\longrightarrow \SS^n$, i.e.,  
$$
\xymatrix{
\dots \ar[r] & \pi_n(SO(n+1))\ar[r]^-{\ev_*} & \pi_n(\SS^n)\ar[r] \ar@{=}[d] & \pi_{n-1}(SO(n)) \ar[r] \ar@{=}[d] & \dots \\
 & & \Z & \mbox{torsion} & 
}
$$
Since $\pi_n(\SS^n)=\Z$ and $\pi_{n-1}(SO(n))$ is torsion, the map $\pi_n(\SS^n)\to \pi_{n-1}(SO(n))$ cannot be injective, hence the map $\ev_*$ cannot be zero. There exist therefore $\alpha\in \pi_n(SO(n+1))$ and $\ell>0$ such that $\ev_*(\alpha)=\ell[\mathrm{Id}_{\SS^n}]$, with $\mathrm{Id}_{\SS^n}:\SS^n\to\SS^n$ the generator of $\pi_n(\SS^n)$. The commutative diagram 
$$
\xymatrix{
\pi_n(SO(n+1)) \ar[r]^-{\ev_*} \ar[d]_{i_*} & \pi_n(\SS^n) \ar[d]^{j_*} \\
H_n(SO(n+1)) \ar[r]_-{\ev_*} & H_n(\SS^n),
}
$$
where $i_*$ and $j_*$ are the canonical maps, and $j_*$ is an isomorphism such that $j_*[\mathrm{Id}_{\SS^n}]=[\SS^n]$, shows that the class $A=i_*\alpha\in H_n(SO(n+1))$ is such that $\ev_*(A)=\ell[\SS^n]$. Since the map $\ev_*:H_n(SO(n+1))\to H_n(\SS^n)$ factors through $\ev_*:H_n(\Diff(\SS^n))\to H_n(\SS^n)$, we find that $\SS^n$ is $\ell$-covered by diffeomorphisms. Since the first nontrivial homotopy group $\pi_n(\SS^n)$ has no torsion, we infer by Proposition~\ref{prop:ell-covered-by-diffeo-implies-abundant} that $\SS^n$ is abundant with $2$-spheres. 
\end{proof}

\begin{proof}[Second proof of Proposition~\ref{prop:odd-spheres}]
We prove that $\SS^n$ is $2$-covered by diffeomorphisms if $n$ is odd. 

\begin{figure}
    \centering
    \includegraphics{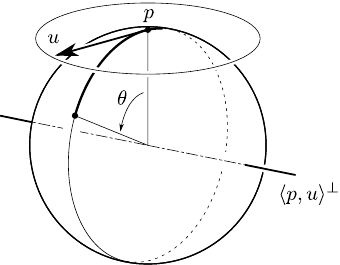}
    \caption{The rotation $f(u,\theta)$ associated to $(u,\theta)\in T^1_p\SS^n\times \SS^1$ }
    \label{fig:family_of_rotations}
\end{figure}

Let $p\in \SS^n$ be fixed. We define the map 
$$
f:T^1_p\SS^n\times \SS^1\to SO(n+1)
$$
such that $f(u,\theta)$ is the rotation that fixes the subspace $\langle p,u\rangle^\perp\subset \R^{n+1}$ and rotates by an angle $\theta$ in the oriented plane $\langle p,u\rangle$ (see Figure \ref{fig:family_of_rotations}).

Let $\ev_p:SO(n+1)\to\SS^n$, $\ev_p(\varphi)=\varphi(p)$ be the evaluation at $p$. 
We claim that the composition 
$$
T^1_p\SS^n\times \SS^1\stackrel{f}\longrightarrow SO(n+1)\stackrel{\ev_p}\longrightarrow \SS^n
$$
has degree $1+(-1)^{n-1}$ for a suitable orientation of the source and target. Indeed, any point in $\SS^n\setminus \{\pm p\}$ has exactly two pre-images, which are regular, of the form $(u,\theta)$ and $(-u,2\pi-\theta)$. Since the degree of the self-map of $T^1_p\SS^n\times\SS^1$ given by $(u,\theta)\mapsto (-u,2\pi-\theta)$ is $1$ if $n$ is odd, and $-1$ if $n$ is even, we infer the claim. 

We now use that $n$ is odd. The class $A=f_*([\SS^{n-1}]\times [\SS^1])\in H_n(SO(n+1))$ is such that $(\ev_{p})_*A=2[\SS^n]$. Since $\ev_p:SO(n+1)\to\SS^n$ factors through the canonical inclusion $SO(n+1)\hookrightarrow \Diff(\SS^n)$, we find that $\SS^n$ is $2$-covered by diffeomorphisms. We conclude using Proposition~\ref{prop:covered-by-diffeo-implies-abundant}. 
\end{proof}

\subsubsection{Homogeneous spaces}
\label{sec:homogeneous-spaces}

Let $M=G/H$ be a homogeneous space with $G$ a compact connected Lie group and $H$ a closed Lie subgroup of positive dimension. (The elements of $M$ are understood to be cosets with respect to the action of $H$ on $G$ by right multiplication.) The starting point of the investigation in this section is the commuting diagram 
$$
\xymatrix{
H \ar[r] \ar[d]& \Diff_*(G/H) \ar[d]\\
G\ar[r]& \Diff(G/H).
}
$$
The group $H$ acts on the left on $G/H$ in the diagram above and therefore 
 on $\Omega^2(G/H)=\Map_*(\SS^2,G/H)$ by 
$$
\Phi_H:H\times \Omega^2(G/H)\to\Omega^2(G/H), \qquad (h,\sigma)\mapsto h\sigma. 
$$

\begin{proposition} \label{prop:homogeneous_spaces}
Assume there exists a class $\alpha\in H_*(\Omega^2(G/H))$ such that 
$$
(\Phi_H)_*([H]\times \alpha)\neq 0\in H_*(\Omega^2(G/H)). 
$$
Then $G/H$ is abundant with $2$-spheres. 
\end{proposition}

\begin{proof}
Denote by
$$
\Phi:G\times \Map(\SS^2,G/H)\to \Map(\SS^2,G/H), \qquad (g,\sigma)\mapsto g\sigma
$$
the action of $G$ on $2$-spheres, and 
$$
j: \star\ri G/H
$$
the inclusion of the basepoint. We then claim that for the fibration $\mathrm{ev}: \mathrm{Map}(\SS^2, G/H)\ri G/H$ we have 
$$
j_!\Phi_*([G]\times \alpha)=(\Phi_H)_*([H]\times \alpha),
$$
which is nonzero in $H_*(\Omega^2(G/H))$ by assumption. This  
means that $\Phi_*([G]\times \alpha)$ 
 lives over the fundamental class of $G/H$.
 It is actually nonzero in the reduced homology $\widetilde{H} _*(\Omega^2(G/H))$ for degree reasons, since we assumed $H$ to have positive dimension, and therefore $G/H$ is abundant with $2$-spheres.

To prove the claim we proceed as in the proof of Proposition~\ref{prop:covered-by-diffeo-implies-abundant}. Let $\pi:G\to G/H$ be the projection. We have a map of fibrations
$$
\xymatrix{
H\times \Omega^2(G/H) \ar@{^(->}[d] \ar[r]^-{\Phi_H} & \Omega^2(G/H) \ar@{^(->}[d] \\
G\times \Omega^2(G/H) \ar[r]^-\Phi \ar[d]^{\pi \circ \mathrm{pr}_1} & \Map(\SS^2,G/H) \ar[d]^\ev \\
G/H \ar@{=}[r] & G/H
}
$$   
We know that  $j_![G]=[H]$ for the fibration $H\hookrightarrow G\stackrel{\pi}\longrightarrow G/H$ by Remark~\ref{rmk:shriek-fundamental-class}. Then, the  following  diagram is commutative by Proposition~\ref{prop:shriek-commute}
$$
\xymatrix{
H_*(G\times \Omega^2 (G/H)) \ar[r]^-{\Phi_*} \ar[d]_{j_!} & H_*(\Map(\SS^2,G/H)) \ar[d]^{j_!} \\
H_{*-n}(H\times \Omega^2(G/H)) \ar[r]_-{(\Phi_H)_*} & H_{*-n}(\Omega^2(G/H))  \ , 
}
$$ 
where $n=\dim G/H$, and we infer that 
\begin{align*}
j_!\Phi_*([G]\times\alpha) & =(\Phi_H)_*j_!([G]\times\alpha)\\
& =(\Phi_H)_*((j_![G])\times\alpha)\\
& = (\Phi_H)_*([H]\times\alpha).
\end{align*} 
This proves the claim, and the proposition.
\end{proof}

\begin{example}[\bf The sphere $\SS^2$] It is unclear what the true range of applicability of the previous proposition is. We can at least recover the fact that $\SS^2$ is abundant with $2$-spheres. We present $\SS^2$ as $SO(3)/SO(2)$ and start with a class $\alpha\in H_0(\Omega^2\SS^2)$ represented by a degree $1$ map. Acting on $\alpha$ with $SO(2)$ turns it into $\alpha'\in H_1(\Omega^2\SS^2)$ that is nonzero. Indeed, this is precisely the nontrivial cycle $(\sigma_p)_*[\Sigma'_p]$ constructed in Remark~\ref{rmk:generator-S2}, which corresponds under the map $\beta_*:H_1(\Omega^2\SS^2)\to H_2(\Omega\SS^2)$ to the Pontryagin square of the generator of $H_1(\Omega\SS^2)$. Denoting $\Phi:SO(3)\times \Map(\SS^2,\SS^2)\to \Map(\SS^2,\SS^2)$, we find that the class $\Phi_*([SO(3)]\times\alpha)\in H_3(\Map(\SS^2,\SS^2))$ provided by Proposition~\ref{prop:homogeneous_spaces} is exactly the nontrivial class $\sigma_*[\Sigma']$ that we constructed in Remark~\ref{rmk:generator-S2}. 
\end{example}

\subsubsection{Compact rank one symmetric spaces (CROSS)}

We generalize in this section Example~\ref{example:spheres}, which deals with spheres, to complex and quaternionic projective spaces, as well as the Cayley plane. Besides spheres, the only other type of CROSS are real projective spaces. Since the finiteness of the $\pi_1$-sensitive Hofer-Zehnder capacity for real projective spaces follows from that of spheres by a covering argument (see Remark \ref{rem:cover-Hofer-Zehnder}), we do not further investigate whether they are abundant with $2$-spheres or not. 

\begin{proposition} \label{prop:CPd} Complex projective spaces are abundant with $2$-spheres. 
\end{proposition}

\begin{proof}
We prove that $\C P^n$ admits a fibered family of $2$-spheres in the sense of Definition~\ref{defi:fibered_2spheres}.

We endow $\C P^n$ with the homogeneous metric, and we denote by $I$ the complex structure on the tangent bundle $T\C P^n$. Let $\Sigma=T^1\C P^n\subset T\C P^n$ be the unit tangent sphere bundle. Given a point $p\in \C P^n$ and a unit tangent vector $v\in T_p\C P^n$, there is a unique complex line through $p$ tangent to $v$. That complex line is a totally geodesic round sphere $\SS^2$ (see~\cite[\S2.A.5 and~2.32]{Gallot-Hulin-Lafontaine}). The orthonormal basis $(v,Iv)$ of its tangent plane at $p$ allows us to parametrize it via a degree $1$ map $\SS^2\to\SS^2$. 
In this way we construct a map of fibrations $\sigma:\Sigma\to \Map(\SS^2,\C P^n)$, and we denote $\alpha=\sigma_*[\Sigma]\in H_{4n-1}(\Map(\SS^2,\C P^n))$. 

The shriek $j_!\alpha$ under the inclusion of a basepoint $p\ri \C P^n$ for the fibration $\Omega^2\C P^n\hookrightarrow \Map(\SS^2,\C P^n)\ri \C P^n$ lives in $H_{2n-1}(\Omega^2\C P^n)$, and is equal to $(\sigma_p)_*[\Sigma_p]$, with $\Sigma_p=T^1_p\C P^n$ the fiber of the fibration $\Sigma\to\C P^n$, and $\sigma_p=\sigma|_{\Sigma_p}$. 

Under the map $\beta_*:H_*(\Omega^2\C P^n)\to H_{*+1}(\Omega\C P^n)$ determined by the adjunction $\Omega^2\C P^n\to \Map(\SS^1,\Omega\C P^n)$, we obtain $\beta_*j_!\alpha= \beta_*(\sigma_p)_*[\Sigma_p]\in H_{2n}(\Omega\C P^n)$, and this class does not vanish. Indeed, it is classical~\cite{Ziller1977} that the energy functional on $\Omega\C P^n$ is perfect, and one can prove as in Example~\ref{example:spheres} that this class corresponds up to sign to
twice the fundamental class of the first
nontrivial critical manifold $T^1_p\C P^n$, of dimension $2n-1$ and of index $1$. (See also the proof of Proposition~\ref{prop:HPd-CaP2} below.) This proves that $\alpha$ lives over the fundamental class, or, equivalently, that $\Sigma$ is a fibered family of $2$-spheres for $\C P^n$.   
\end{proof}

The previous construction amounts to using the family of $2$-spheres on $\SS^2$ described in Remark~\ref{rmk:generator-S2} on each complex line in $\C P^n$. In the next proposition we treat quaternionic projective spaces and the Cayley plane using the families of $2$-spheres on $\SS^4$ and $\SS^8$ from Example~\ref{example:spheres}, by implanting them on each quaternionic line, respectively on each octonionic line.

\begin{proposition} \label{prop:HPd-CaP2}
Quaternionic projective spaces and the Cayley plane are abundant with $2$-spheres.
\end{proposition}

\begin{proof}
We prove that $\H P^n$ and $\mathrm{Ca} P^2$ admit a fibered family of $2$-spheres in the sense of Definition~\ref{defi:fibered_2spheres}. We only give the details for $\H P^n$, since $\mathrm{Ca} P^2$ is treated by an analogous construction.

Consider therefore a quaternionic projective space $\H P^n$ with $n\ge 1$. We endow $\H P^n$ with the homogeneous metric. For any $p\in \H P^n$ and $u\in T^1_p\H P^n$ there is a unique quaternionic projective line $L_{p,u}\subset \H P^n$ passing through $p$ and tangent to $u$. This line is totally geodesic and isometric to a sphere $\SS^4$ of constant curvature (see~\cite[Theorem~3.25]{Besse}). 

We now consider the ``partial'' Stiefel manifold of $2$-frames 
$$
\Sigma=\mathrm{Stiefel}'_2(\H P^n)
$$
whose fiber at point $p\in \H P^n$ is the space of orthogonal unitary $2$-frames $(u,v)$ at $p$ such that $v$ is tangent to $L_{p,u}$. (The vector $u$ is also tangent to $L_{p,u}$, by definition.) The dimension of $\Sigma$ is 
$$
\dim\Sigma=\dim \H P^n + \dim T^1_p\H P^n + 2 = 8n+1.
$$
Given such a $2$-frame $(u,v)$ there is a unique isometric parametrization $\sigma_{p,u,v}:\SS^2\to L_{p,u}$ of the unique round $2$-sphere in $L_{p,u}$ tangent to $u$ and $v$. This determines a map $\sigma:\Sigma\to \Map(\SS^2,\H P^n)$, $\sigma(p,u,v)=\sigma_{p,u,v}$. We denote  $\alpha=\sigma_*[\Sigma]\in H_{8n+1}(\Map(\SS^2,\H P^n))$. 

The shriek $j_!\alpha$ under the inclusion of a basepoint $p\ri \H P^n$ for the fibration $\Omega^2\H P^n\hookrightarrow \Map(\SS^2,\H P^n)\ri \H P^n$ lives in $H_{4n+1}(\Omega^2\H P^n)$, and is equal to $(\sigma_p)_*[\Sigma_p]$, with $\Sigma_p$ the fiber of the fibration $\Sigma\to\H P^n$, 
and $\sigma_p=\sigma|_{\Sigma_p}$. 

Under the map $\beta_*:H_*(\Omega^2\H P^n)\to H_{*+1}(\Omega\H P^n)$ determined by the adjunction $\Omega^2\H P^n\to \Map(\SS^1,\Omega\H P^n)$, we obtain $\beta_*j_!\alpha= \beta_*(\sigma_p)_*[\Sigma_p]\in H_{4n+2}(\Omega\H P^n)$. We claim that $H_{4n+2}(\Omega\H P^n)\simeq \Z$ and $\beta_*(\sigma_p)_*[\Sigma_p]$ equals twice a generator of this group. As a consequence $(\sigma_p)_*[\Sigma_p]\neq 0$, which proves that $\H P^n$ admits a fibered family of 2-spheres. 

For our purposes, the most useful way to infer the isomorphism $H_{4n+2}(\Omega\H P^n)\simeq \Z$ is via Morse theory. We refer to~\cite{Ziller1977} for the following facts concerning the energy functional $\cE$ on $\Omega_p \H P^n$ for the homogeneous metric. Its critical points are either the constant loop at $p$, or, for every $k\ge 1$, the (geodesic) great circles $\gamma_u^k:\SS^1\to L_{p,u}$ based at $p$, tangent to $u$ and traversed $k$ times at constant speed. As such, the energy functional is Morse-Bott, with a ``trivial" critical manifold consisting of the point $p$, and, for each $k\ge 1$, a critical manifold $C_k$ diffeomorphic to $T^1_p\H P^n$ consisting of the (geodesic) great circles $\gamma_u^k$. Each such ``nontrivial'' critical manifold has dimension $4n-1$ and index $2(k-1)(2n+1)+3$~\cite[Theorem~3]{Ziller1977}. The first nontrivial critical manifold, which corresponds to the great circles $\gamma_u^1$ traversed once, has index $3$. By~\cite[Theorem~8]{Ziller1977}, the energy functional $\cE$ is perfect because every critical manifold $C_k$ admits a so-called ``completing manifold''. As a consequence, each critical manifold $C_k$ contributes to $H_*(\Omega_p\H P^n)$ a copy of its own homology shifted in degree by the index. For $C_1$, the completing manifold can be described using the total space of a fibration in 3-spheres $\pi':\Sigma'_p\to T^1_p\H P^n$, whose fiber at $u$ is the equator $E_{p,u}$ orthogonal to $u$ inside the round 4-sphere $L_{p,u}$. There is a canonical map $\sigma'_p:\Sigma'_p\to \Omega_p\H P^n$ that associates to a vector $u$ and a point $q\in E_{p,u}$ the (not necessarily geodesic) circle on $L_{p,u}$ based at $p$, tangent to $u$, and passing through $q$. (If $q=p$ then $\sigma'(q)$ is the constant loop at $p$.) The contribution of $C_1$ to the homology of $H_*(\Omega_p\H P^n)$ is a copy of $H_*(T^1_p\H P^n)$ shifted in degree by $3$, via the map that associates to a homology class $A\in H_*(T^1_p\H P^n)$ the homology class $(\sigma'_p)_*\pi_!A\in H_{*+3}(\Omega_p\H P^n)$ (geometrically, this amounts to apply $\sigma'_p$ to the ``fundamental class'' of the restriction of the fibration $\Sigma'_p$ over a representative of the homology class $A$). 

That $H_{4n+2}(\Omega_p\H P^n)\simeq \Z$ is a consequence of $H_{4n-1}(T^1_p\H P^n)\simeq \Z$. A generator of $H_{4n-1}(T^1_p\H P^n)\simeq \Z$ is given by the fundamental class $[T^1_p\H P^n]$, and a generator of $H_{4n+2}(\Omega_p\H P^n)$ is described fiberwise in $\Sigma'_p$ by associating to each $u$ the family of circles parametrized by the equator $E_{p,u}$. 

The canonical map $\varphi:\Sigma_p\times \SS^1\to \Sigma'_p$, $(u,v,\theta)\mapsto (u,\gamma_v^1(\theta))$ is a map of fibrations over $T^1_p\H P^n$. As seen in Proposition~\ref{prop:nontrivial-class-Sn}, this map has fiberwise degree $\pm2$, and therefore total degree $\pm 2$. The map $\beta\circ(\sigma_p\times \mathrm{Id}):\Sigma_p\times\SS^1\to \Omega_p\H P^n$ factors as a composition $\sigma'_p\circ \varphi$, and the image of the fundamental class $[\Sigma_p]\times [\SS^1]$ is therefore equal to twice a generator of $H_{4n+2}(\Omega_p\H P^n)$. This proves the claim, and the proposition. 
\end{proof}

\subsubsection{Further remarks on abundance with $2$-spheres}
\label{sec:minimal-models}

1. Abundance with $2$-spheres in the case of $\SS^2$ can also be retrieved from computations of Hansen~\cite{Hansen-2-sphere} (use $\Z_2$-coefficients to find a class in the component $\Map_1(\SS^2,\SS^2)$ of degree $1$ maps).

2. Abundance with $2$-spheres in the case of $\C P^n$ follows also from computations of Kallel-Salvatore~\cite{Kallel-Salvatore} (use $\Z_{(n+1)}$-coefficients to find a class in $\Map_0(\SS^2,\C P^n)$ of degree $0$ maps).

3. Minimal models provide a potentially powerful tool for computations. Brown and Szczarba~\cite{Brown-Szczarba} construct a minimal model for function spaces $\Map(X,Y)$, with $X$ a finite complex and $Y$ nilpotent, both of finite type. Their model is particularly simple if $X$ is $\Q$-formal, which is the case of $\SS^2$. In particular, we have a minimal model for $\Map(\SS^2,M)$ where $M$ is any simply connected manifold. 

This procedure works as follows for $\SS^n$ if $n$ is odd. The minimal model for $\SS^n$ is $(\Lambda(x), d=0)$ with $|x|=n$. Accordingly, the Brown-Szczarba minimal model for $\Map(\SS^2,\SS^n)$ is $(\Lambda(x,Sx), d=0)$ with $|x|=n$, $|Sx|=n-2$. The nonzero cohomology class $x\cdot (Sx)$ is dual to a rational homology class that lives on top of the fundamental class in the spectral sequence of the fibration $\Map(\SS^2,\SS^n)\to \SS^n$. 

The exact same argument works for $M=G$ a compact Lie group. It is known that the minimal model of any compact Lie group is $(\Lambda(V),d=0)$ where $V$ is supported in odd degree. This recovers from the perspective of minimal models the geometric computations from Example~\ref{example:Lie-groups}.

However, this procedure does not seem to work for $\SS^n$ if $n$ is even.

\section{Almost existence for hypersurfaces knotted with spheres} \label{sec:knotted}

The motivation for this section is the following. We have previously described two situations in which we can prove contractible almost existence on a closed hypersurface in $T^*Q$:
\begin{itemize}
    \item In~\S\ref{sec:almost-existence} we considered hypersurfaces bounding a domain that contains a submanifold of dimension $n$ that projects with nonzero degree onto the zero section. Such hypersurfaces can be thought of as being ``wide'', in the sense that they cover the zero section under projection. 
    \item In~\S\ref{sec:HZ} we considered arbitrary hypersurfaces, under the assumption that there is a homology class in $H_*(\Map(\SS^2,Q))$ that lives over the fundamental class of the zero section $Q$. Such homology classes can also be thought of as being ``wide'', in the sense that there is an abundance of $2$-spheres through every point.
\end{itemize}

Our purpose in this section is to study a certain intermediate setup. We prove Theorem~\ref{thm:orbits-knotted-case-intro} from the introduction, restated as Theorem~\ref{thm:orbits-knotted-case} below. 

We assume that $Q$ is a closed orientable $n$-manifold such that:
\begin{itemize}
\item 
  $Q$ contains an embedded sphere $X$ of dimension $k\ge 2$, i.e., $X=s(\SS^k)$ for some embedding $s:\SS^k\to Q$,
\item
  $Q$ contains a 
  closed orientable $(n-k)$-submanifold $Y$ that 
  intersects $X$ transversely at a single point $Y\cap X=\{p\}$.
\item $Y$ is $\ell$-connected, with $\ell=2k-2$ if $k$ is even, and $\ell=k-1$ if $k$ is odd. 
\item
  the normal bundle of $Y$ is trivial.
\end{itemize}
Note that, because their intersection number is nonzero, both $X$ and $Y$ are homologically nontrivial, and also not torsion. Also, the assumption that $Y$ is $\ell$-connected implies the following bounds: if $k$ is even, then $n-k\ge 2k-1$, i.e., $n\ge 3k-1$, and if $k$ is odd, then $n-k\ge k$, i.e., $n\ge 2k$.

Let $\pi:T^*Q\to Q$ be the projection. 

\begin{theorem}\label{thm:orbits-knotted-case}
  In the above situation, consider a continuous map $\varphi:S\to
  \pi^{-1}(Y)\subset T^*Q$ defined on a closed orientable $(n-k)$-manifold $S$, and suppose that the map
  $\pi\circ\varphi:S\to Y$ has degree  $d\neq 0$.

  Then, any hypersurface 
  that bounds a relatively
  compact domain $U$ that contains $\varphi(S)$ has the contractible almost 
  existence property.
\end{theorem}

\begin{remark}
  Since $\dim S=n-k<n=\dim Q$, the projection $\pi:U\to Q$
  may fail to be surjective. If that is the case, then $U$ is known to be
  displaceable and hence to support contractible closed characteristics near its boundary. However, no such
  argument is known when the projection is surjective.
\end{remark}

\begin{remark} We suspect that some of the assumptions in Theorem~\ref{thm:orbits-knotted-case} may be weakened (the embededness of the sphere, the triviality of the normal bundle, the existence of a unique point of intersection, which may be replaced by an assumption on the intersection number $[X]\cdot [Y]$). 
\end{remark}

We already mentioned the following in  Example~\ref{example:knotted_with_sphere} from the Introduction.

\begin{example} \label{example:knotted-with-spheres-section-7}
  Let $k\ge 2$ and $Q=Q'\sharp (\SS^k\times Y^{n-k})$, with $Q'$ closed orientable and $Y$ closed orientable and $(2k-2)$-connected (in particular $n\ge 3k-1$). Then $Q$ satisfies the assumptions of Theorem~\ref{thm:orbits-knotted-case}.
  
  Indeed, let $B\subset \SS^k\times Y$ be the ball used to
  construct the connected sum, and let 
  $X=\SS^{k}\times\{q\}$ and $Y\equiv \{p\}\times Y$ for some point $(p,q)\in \SS^{k}\times Y$ such that $X\cap B=Y\cap B=\emptyset$. 
  Then  
  $X$ and $Y$ meet transversely at the single point $(p,q)$ and
  the normal bundle of $Y$ is trivial. Also, $Y$ is connected enough.

  \medskip

  The case where $k\ge 3$ and $Q'$ is a $K(\pi, 1)$, for instance $Q'=\T^{n}$, is particularly interesting. Indeed, no technique from the literature, nor from this paper, allows to prove finiteness of the $\pi_1$-sensitive Hofer-Zehnder capacity for $D^*Q$: the second homotopy group is trivial, hence~\cite{Albers-Frauenfelder-Oancea} does not apply, and the manifold is rationally essential in the sense of Gromov \cite{Gromov-83}, hence~\cite{Frauenfelder-Pajitnov} does not apply either.   
  While we cannot yet prove finiteness of the $\pi_1$-sensitive Hofer-Zehnder capacity for $D^*Q$, our Theorem~\ref{thm:orbits-knotted-case} goes strictly beyond the scope of Theorems~\ref{thm:almost-existence-simplified} and~\ref{thm:almost-existence-Thom-simplified}.
\end{example}

The proof of Theorem~\ref{thm:orbits-knotted-case} is based on
Proposition~\ref{prop:almost-sure-existence-T*Q} and on the existence of a
homology class with suitable properties that is described in the next proposition.

Let $j_Y :Y\to Q$ and  $i_Q:Q\to\cL_{0}Q$ be the inclusions, and
$\cF=C_*(\cP_{Q\to\star}\cL_0Q)$ the DG local system over $\cL_{0} Q$
associated to the fibration
$$
\cP_{Q\to\star}\cL_0Q\hookrightarrow \cP_{Q\to\cL_0Q}\cL_0Q\stackrel{\mathrm{ev}}{\longrightarrow} \cL_0Q.
$$

\begin{proposition}\label{prop:class-alpha-Y}
There exists a class $\alpha\in H_*(Q;i_Q^*\cF)$ such that $i_{Q*}\alpha=0$ and $j_{Y!}\alpha$ is not torsion.
\end{proposition}

\begin{proof}[Proof of Theorem~\ref{thm:orbits-knotted-case} assuming Proposition~\ref{prop:class-alpha-Y}]
Let $U$ be the relatively compact domain in $T^*Q$ that contains $\varphi(S)$, and let $\pi_U:U\to Q$. Consider the commutative diagram 
$$
\begin{tikzcd}
  S\arrow[r,"\varphi"]\arrow[d,"\pi\circ\varphi"'] & U\arrow[d,"\pi_{U}"]\\
  Y\arrow[r,"j_Y"] & Q
\end{tikzcd}
$$
Then 
$$
\ker \pi_{U!}\subset \ker \varphi_!\pi_{U!} = \ker(\pi\circ\varphi)_!j_{Y!} \subset 
\ker (\pi\circ\varphi)_*(\pi\circ\varphi)_!j_{Y!} =\ker (d\cdot j_{Y!}). 
$$
The last equality follows from $(\pi\circ\varphi)_*(\pi\circ\varphi)_! = d\cdot\mathrm{Id}$, with $d=\deg (\pi\circ\varphi)$ (see~\cite[Proposition~10.14]{BDHO}). 

Consider the class $\alpha$ in $H_{*}(Q;i_Q^{*}\cF)$ from Proposition \ref{prop:class-alpha-Y}. 
Since $j_{Y!}\alpha$ is not torsion, we find that $\alpha \notin \ker (d\cdot j_{Y!})$,
and therefore
$$
\alpha\notin\ker  \pi_{U!}.
$$
On the other hand $i_{Q*}\alpha=0$, and we conclude by applying Proposition~\ref{prop:almost-sure-existence-T*Q}. 
\end{proof}

\begin{proof}[Proof of Proposition~\ref{prop:class-alpha-Y}] 
By~\cite[Corollary~7.10 and Remark~13.20]{BDHO} we have
$$
H_*(Q;i_Q^*\calf)\simeq H_*(i_Q^*\cP_{Q\to\cL_0Q}\cL_0Q)=H_*(\calp_{Q\ri Q}\call_0Q)=
H_*(\mathrm{Map}(\SS^{2}, Q)).
$$
The embedding $s:\SS^{k}\ri Q$ induces canonically a map $$
s_{\call_0}: \calp_{\SS^{k}\ri
  \SS^{k}}\call_0\SS^{k}\ri \calp_{Q\ri Q}\call_0Q,
  $$ 
  or, equivalently,  
$s_{\call_0}:\Map(\SS^2,\SS^k)\to \Map(\SS^2,Q)$. 
Let $\alpha_0\in H_*(\Map(\SS^2,\SS^k))$ be the class of degree $*=3k-3$ if $k$ is even, respectively $*=2k-2$ if $k$ is odd, constructed in Example~\ref{example:spheres}, respectively in Proposition~\ref{prop:odd-spheres}. We define 
$$
\alpha= s_{\cL_0 *}\alpha_0
$$
and claim that $i_{Q*}\alpha=0$ and $j_{Y!}\alpha$ is not torsion. 

We recall some important properties of the class $\alpha_0$. Let $i_{\SS^k}:\cP_{\SS^{k}\to \SS^{k}}\cL_0 \SS^{k}\to \cP_{\SS^{k}\to \cL_0 \SS^{k}}\cL_0 \SS^{k}$ be the canonical inclusion, and let $j:\star\to \SS^k$ be the inclusion of a basepoint. We proved that $i_{\SS^k*}\alpha_0=0$ and $j_!\alpha_0\neq 0\in H_*(\Omega^2\SS^k)$. Moreover, the same proof shows that $j_!\alpha_0$ is not torsion.   

We prove $i_{Q*}\alpha=0$. We consider the commuting diagram 
  $$
  \begin{tikzcd}
    \cP_{\SS^{k}\to \SS^{k}}\cL_0 \SS^{k}\arrow[d,"i_{\SS^{k}}"]
    \arrow[r,"s_{\call_0}"]&
    \cP_{Q\to Q}\cL_0 Q\arrow[d,"i_Q"]\\
    \cP_{\SS^{k}\to \cL_0 \SS^{k}}\cL_0 \SS^{k}\arrow[r,"s_{\cL_{0}}"]
    & \cP_{Q\to \cL_0 Q}\cL_0 Q,
  \end{tikzcd}
  $$
  with $i_{\SS^k}$ and $i_Q$ being the canonical inclusions. 
    Then $i_{Q*}\alpha=i_{Q*}s_{\cL_{0}*}\alpha_0=s_{\cL_{0}*}i_{\SS^{k}*}\alpha_{0}=0$ as a consequence of $i_{\SS^k*}\alpha_0=0$, which in turn holds because $\deg \alpha_0>k$ (see also the proof of Theorem~\ref{thm:abundant}).

We now prove that $j_{Y!}\alpha$ is not torsion.

The map of fibrations $s_{\call_0}: \calp_{\SS^{k}\ri
  \SS^{k}}\call_0\SS^{k}\ri \calp_{Q\ri Q}\call_0Q$ factors
through the pullback by $s:\SS^{k}\ri Q$ of the target fibration, i.e., 
we have a diagram
\begin{equation}\label{eq:last-diagram}
    \begin{tikzcd}
  \Omega^2 \SS^k\arrow[d]\arrow[r,"\ol{s}"]&
  \Omega^2 Q\arrow[d]\arrow[r,equal]&
  \Omega^2 Q\arrow[d]&
  \\
  \cP_{\SS^{k}\to\SS^{k}}\cL_0\SS^{k} 
  \arrow[d, "\mathrm{ev}"]\arrow[r,"\ol{s}"]&
  s^{*}\big(\cP_{Q\to Q}\cL_0Q\big)
  \arrow[d]\arrow[r,"\widetilde{s}"]&
  \cP_{Q\to Q}\cL_0 Q
  \arrow[d,"\mathrm{ev}"]\\
  \SS^{k}\arrow[r,equal]&\SS^{k}\arrow[r,"s"]& Q\\
\end{tikzcd}
\end{equation}
with
$$
s_{\call_0} = \tilde{s}\circ\bar{s}.
$$

  Let $\beta=\bar{s}_*\alpha_0 \in H_*(s^*(\calp_{Q\ri Q}\call_0Q))$. In order to show that $j_{Y!}\alpha$ is not torsion, we first prove that $j_!\beta$ is not torsion,  with  $j: \star\ri \SS^k$ the inclusion of the basepoint.

\begin{lemma}\label{lemma:last-lemma}
  If the normal bundle of $Y$ is trivial, then $j_!\beta$ is not torsion in $H_*(\Omega^2Q)$.
\end{lemma}

\begin{proof}
  We have that $j_! \beta= \bar{s}_* j_! \alpha_0$ by Proposition \ref{prop:shriek-commute}. 
  We also know that $j_!\alpha_0$ is not torsion. It therefore suffices to show that $\bar{s}_* : H_*(\Omega^2\SS^k)\to H_*(\Omega^2Q)$ is injective. 

  For this purpose, we use the Thom-Pontryagin construction. The normal bundle of $Y$ being
  trivial, there is a tubular neighborhood of $Y$ that can be identified
  with $Y\times \D^{k}$, and we have a projection $Q\xrightarrow
  p\overline{Q}=(Y\times\D^k)/(Y\times\partial\D^{k})$, where $\overline{Q}$ is the
  space in which the complement of this neighborhood is collapsed to a
  point.

  We can assume without loss of generality that $X$ intersects the tubular neighborhood of $Y$
  along a fiber, and that $X$ intersects $Y$ positively.

  The image of the projection of $X$ in $\overline{Q}$ is then a single
  sphere $\SS^{k}=\D^{k}/\partial\D^{k}$. Moreover, the second
  projection $Y\times  \D^{k}\to \D^k$ induces a projection $\overline{Q}\xrightarrow{ p_{2}} \SS^{k}$, and we have the following diagram
  \begin{equation*}
  \xymatrix{
    \SS^k\ar[r]^-s \ar[drr]^>>>>>>{\delta}|!{[r];[dr]}\hole & Q \ar[d]^<<<{p} \\
    & \overline{Q} \ar[r]_-{p_2} & \SS^k
   }
  \end{equation*}
  where the diagonal map $\delta$ has topological degree $1$. This in turn induces the 
  diagram
  \begin{equation*}
  \xymatrix{
    \Omega^2\SS^k\ar[r]^-{\bar{s}} \ar[drr]^>>>>>>{\bar\delta}|!{[r];[dr]}\hole & \Omega^2 Q \ar[d]^<<<{\bar{p}} \\
    & \Omega^2\overline{Q} \ar[r]_-{\bar p_2} & \Omega^2\SS^k
   }
  \end{equation*}
Since $\delta$ has
topological degree $1$, it is homotopic to $\Id_{\SS^k}$ and therefore $\bar{\delta}$
is also homotopic to $\Id$. It follows that 
${{\ol{p}}_{2}}_{*} {\bar{p}}_{*} {\bar{s}}_{*}= \mathrm{\Id}
$
and in particular that $\bar{s}_*$ is injective. 

This completes the proof of the lemma. 
\end{proof}

We look now at the rightmost part of the diagram \eqref{eq:last-diagram}. We prove
\begin{lemma}\label{lem:very-last-lemma}
    Let $d\geq 2$ be an integer and suppose that the manifold $Y$ is $(d+1)$-connected. 
    Given a class $\beta\in H_{d+k}(s^*(\calp_{Q\to Q}\call_0Q))$ such that $j_!\beta$ is not torsion in $H_d(\Omega^2Q)$, the class $j_{Y!}\tilde{s}_*\beta$ is not torsion in $H_{d}(\calp_{Q\to Y}\call_0 Q)$.
\end{lemma}

\begin{proof}
  Since $X$ and $Y$ intersect transversely at a single point, we may assume w.l.o.g.  that 
  the fiber product $\SS^{k} \ {_s}\!\times_{j_Y} Y$  equals the basepoint $\star$.
  Consider the following diagram
  $$
  \begin{tikzcd}
   \star
   \arrow[r,"j^{(Y)}"]\arrow[d,"j=j^{(\SS^k)}"']&
    Y\arrow[d,"j_Y"]\\
    \SS^{k} \arrow[r,"s"]& Q
  \end{tikzcd}
  $$

  By \cite[Proposition~10.18]{BDHO}, we have $j_{Y!}s_*= j^{(Y)}_*j_!$. This relation is valid for any DG coefficients on $Q$. For the coefficients $C_*(\Omega^2 Q)$  defined by the fibration $\Omega^2Q\to \calp_{Q\to Q}\call_0Q\to Q$  the fibration theorem \cite[Theorem 7.2]{BDHO} transforms the relation above into the commutative diagram 
  $$
  \begin{tikzcd}
   H_*(\Omega^2Q)
   \arrow[r,"j^{(Y)}_*"]&
    H_{*}(\calp_{Q\to Y}\call_0 Q)  \\
    H_{*+k}(s^*(\cP_{Q\to Q}\cL_0Q)) \arrow[r,"\tilde{s}_*"] \arrow[u,"j_!"]& 
    H_{*+k}(\cP_{Q\to Q}\cL_0Q)\arrow[u,"j_{Y!}"']
  \end{tikzcd}
  $$

  We infer that $j_{Y!}\tilde{s}_*\beta = j_*^{(Y)}j_!\beta$ and since by hypothesis $j_!\beta$ is not torsion it suffices to prove that $j_*^{(Y)}$ is injective in degree $d=\mathrm{deg}(j_!\beta)$. This map is induced in homology by the inclusion of the fiber $\Omega^2Q\hookrightarrow \calp_{Q\ri Y}\call_0Q$ for the fibration $\calp_{Q\ri Y}\call_0Q\stackrel{\mathrm{ev}}{\longrightarrow}Y.$ Examining the Leray-Serre spectral sequence of this fibration we notice that, since $Y$ is $(d+1)$-connected, this inclusion induces an isomorphism in homology in degrees $\leq d$. The proof of the lemma is now complete.
  \end{proof}

\noindent {\it End of the proof of Proposition~\ref{prop:class-alpha-Y}.} We recall the notation $\beta=\bar s_*\alpha_0$. By Lemma~\ref{lemma:last-lemma} we know that $j_!\beta$ is not torsion. We now apply  Lemma~\ref{lem:very-last-lemma} to the class $\beta$ for $d=k-2$ if $k$ is odd, respectively $d=2k-3$ if $k$ is even, and infer that
$$
j_{Y!}\tilde s_*\beta =j_{Y!}\tilde s_* \bar s_*\alpha_0 = j_{Y!} s_{\cL_0*} \alpha_0=j_{Y!} \alpha
$$
is not torsion. This concludes the proof of the proposition.
\end{proof}

\bibliographystyle{alpha}
\bibliography{000_dgfloer2}

\bigskip

{\small

\medskip
\noindent Jean-François BARRAUD \\
\noindent Institut de mathématiques de Toulouse, Université Paul Sabatier – Toulouse
III, 118 route de Narbonne, F-31062 Toulouse Cedex 9, France.\\
{\it e-mail:} barraud@math.univ-toulouse.fr
\medskip

\medskip
\noindent Mihai DAMIAN \\
\noindent Université de Strasbourg, Institut de recherche mathématique avancée, IRMA, Strasbourg, France.\\
{\it e-mail:} damian@math.unistra.fr 
\medskip

\medskip
\noindent Vincent HUMILI\`ERE \\
\noindent Sorbonne Université and Université de Paris, CNRS, IMJ-PRG, F-75006 Paris, France.\\
\noindent \& Institut Universitaire de France.\\
{\it e-mail:} vincent.humiliere@imj-prg.fr
\medskip

\medskip
\noindent Alexandru OANCEA \\
\noindent Université de Strasbourg, Institut de recherche mathématique avancée, IRMA, Strasbourg, France.\\
{\it e-mail:} oancea@unistra.fr 
\medskip

}

\end{document}